\documentclass[11pt]{article}

\usepackage{amsmath,amsthm,amssymb,amsfonts,mathtools,mathrsfs,enumerate}
\usepackage[usenames,dvipsnames]{color}
\RequirePackage[dvips]{graphics}
\usepackage{graphicx,psfrag,epsfig,ifpdf}
\usepackage{mathabx}
\usepackage{appendix}
\usepackage{leftidx}

\RequirePackage{pst-all}

\usepackage[hyperindex=true, 
					colorlinks=false]{hyperref}

\def\jD{{\mathscr D}}
\def\jB{{\mathscr B}}
\def\jW{{\mathscr W}}
\def\jS{{\mathscr S}}

\def\jK{{\mathscr K}}

\def\jO{{\mathscr O}}
\def\jC{{\mathscr C}}
\def\jT{{\mathscr T}}
\def\jV{{\mathscr V}}
\def\jP{{\mathscr P}}
\def\jU{{\mathscr U}}

\def\al{\alpha}
\newcommand{\bet}{\beta}	
\def\ga{\gamma}

\def\de{\delta}

\def\De{\Delta}
\def\eps{{\varepsilon}}
\def\ka{\kappa}
\def\la{\lambda}
\def\La{\Lambda}

\def\Om{\Omega}

\def\sig{{\sigma}}
\def\Sig{{\Sigma}}

\def\th{{\theta}}

\def\ph{\varphi}

\newcommand{\Demi}{\frac{1}{2}}
\newcommand{\dem}{{\tfrac{1}{2}}}
\newcommand{\quart}{\frac{1}{4}}
\newcommand{\tquart}{\frac{3}{4}}
\newcommand{\un}[1]{{\underline{#1}}}
\newcommand{\ov}{\overline}

\newcommand{\dist}{\operatorname{dist}}
\newcommand{\area}{\operatorname{area}}
\newcommand{\mult}{\operatorname{mult}}
\newcommand{\Int}{\operatorname{Int}}

\newcommand{\Id}{\mathop{\hbox{{\rm Id}}}\nolimits}
\newcommand{\pa}{\partial}
\newcommand{\na}{\nabla}
\newcommand{\ii}{^{-1}}
\newcommand{\ti}{\tilde}

\def\ov{\overline}
\def\ha{\widehat}

\def\pg{\leaders\hbox to 5mm{\hfil.\hfil}\hfill}

\newcommand{\ie}{{\it i.e.}\ }
\newcommand{\cf}{{\it cf.}\ }
\newcommand{\eg}{{\it e.g.}\ }

\newcommand{\wrt}{{with respect to\ }}
\newcommand{\lhs}{{left-hand side\ }}

\newcommand{\ens}{\enspace}

\def\dst{\displaystyle}
\def\tst{\textstyle}

\def\A{\mathbb A}
\def\R{\mathbb R}

\def\Z{\mathbb Z}
\def\C{\mathbb C}

\def\T{\mathbb T}
\def\N{{\mathbb N}}
\def\Dm{{\mathbb D}}

\def\Cal{\mathscr}

\def\P{\Cal P}

\def\D{\Cal D}
\newcommand{\cF}{\mathcal{F}}
\def\cS{{\Cal S}}
\def\cG{{\Cal G}}

\def\cB{{\Cal B}}

\def\cE{{\Cal E}}

\def\gd{{\mathfrak d}}

\def\bdef{\begin{definition}}
\def\endef{\end{definition}}
\def\bthm{\begin{theorem}}
\def\ethm{\end{theorem}}
\newcommand{\blm}{\begin{lemma}}
\newcommand{\elm}{\end{lemma}}
\def\brm{\begin{remark}}
\def\erm{\end{remark}}
\def\bprop{\begin{prop}}
\def\eprop{\end{prop}}
\def\bcor{\begin{corollary}}
\def\ecor{\end{corollary}}
\newcommand{\beq}{\begin{equation}}
\newcommand{\eeq}{\end{equation}}
\def\beal{\begin{aligned}}
\def\enal{\end{aligned}}

\def\B{{\bf L}}

\def\norm#1{\Vert #1\Vert}
\def\abs#1{\vert #1\vert}

\def\abs#1{\left\vert#1\right\vert}
\def\norm#1{\Vert#1\Vert}

\def\norm#1{\left\Vert#1\right\Vert}

\def\abs#1{\left\vert#1\right\vert}

\def\Im{{\rm Im\,}}
\def\Re{{\rm Re\,}}

\def\Inf{\mathop{\rm Inf\,}\limits}

\def\log{{\rm log\,}}

\def\ppdemi{{\frac{1}{2}}}

\def\norm#1{\left\Vert#1\right\Vert}

\def\abs#1{\left\vert#1\right\vert}

\def\un{\underline}

\def\ha{\widehat}
\def\setm{\setminus}
\def\d{\partial}
\def\ov{\overline}

\def\cN{{\Cal N}}

\def\B{{\Cal B}}

\newtheorem{thm}{Theorem}

\newtheorem*{thmnonb}{Theorem}
\newtheorem*{thmCp}{Theorem~\ref{th:lowerbounds}'}

\newtheorem{Def}{Definition}[section]
\newtheorem{lemma}[Def]{Lemma}
\newtheorem*{sublem}{Sub-Lemma}
\newtheorem{prop}[Def]{Proposition}

\newtheorem{cor}[Def]{Corollary}
\newtheorem*{Not}{Notation}
\newtheorem{rem}[Def]{Remark}

\numberwithin{equation}{section}

\reversemarginpar

\newcounter{parag}[subsection]

\newcounter{paraga}[section]
\renewcommand{\theparaga}{{\bf\arabic{section}.\arabic{paraga}}}
\newcommand{\paraga}{\medskip \addtocounter{paraga}{1} 
\noindent{\theparaga\ } }

\newcounter{pparag}

\def\Bbibitem#1#2{\bibitem[#1]{#2}}

\def\Frac#1#2{{\frac{\displaystyle\strut#1}{\displaystyle\strut#2}}}
\def\pfrac#1#2{{\scriptstyle{#1\over#2}}}
\def\demi{{\Frac{1}{2}}}
\def\pdemi{{\tst{1\over2}}}

\newcommand{\be}{\mathbf{e}}

\def\h{\hbar}

\newcommand{\ex}{\mathrm{e}}
\newcommand{\ima}{\mathrm{i}}
\newcommand{\dd}{\mathrm{d}}

\newcommand{\defeq}{\coloneqq} 

\newcommand{\col}{\colon\thinspace}          

\newcommand{\GNm}{G_{N,\mu}}
\newcommand{\GMm}{G_{M,\mu}}
\newcommand{\Gpj}{G_{p_{j+2},\mu_{j,\ell}}}
\newcommand{\DNm}{D_{q,N,\mu}}
\newcommand{\DMp}{D_{p,M,\mu}}
\newcommand{\Dpj}{D_{p,p_{j+2},\mu_{j,\ell}}}
\newcommand{\DNmj}{D_{q,N_j,\mu_{j,q}}}
\newcommand{\BdeN}{\cB_{\de/N}}
\newcommand{\Bdep}{\cB_{1/p_{j+2}}}
\newcommand{\Bd}{\cB_{d}}

\newcommand{\BdeeN}{\cB_{\de/(2N)}}
\newcommand{\Bdeep}{\cB_{1/(2p_{j+2})}}

\newcommand{\Bdd}{\cB_{d/2}}
\newcommand{\FNqm}{F_{q,N,\mu}}

\newcommand{\FF}{\mathbb{F}}
\def\jF{{\mathscr F}}
\newcommand{\cM}{\mathcal{M}}

\newcommand{\Paal}{\jP^{\al,L}}
\newcommand{\Pal}{\jP^{\al,L}(\Phi^h)}
\newcommand{\Palm}{\jP_m^{\al,L}(\Phi^h)}
\newcommand{\Paa}[1]{\jP_{#1}^{\al,L}(\Phi^h)}
\newcommand{\Pa}[1]{\jP_{#1}^{\al,L}}

\newcommand{\cth}{^{[\th]}}
\newcommand{\ccr}{^{[r]}}
\newcommand{\zz}[1]{^{[#1]}}

\newcommand{\CG}{\operatorname{C_G}}

\DeclarePairedDelimiter\absD{\lvert}{\rvert}%
\DeclarePairedDelimiter\normD{\lVert}{\rVert}%
\DeclarePairedDelimiter\NormD{\lvert\lvert\lvert}{\rvert\rvert\rvert}%
\DeclarePairedDelimiter\angD{\langle}{\rangle}

\newcounter{paragaD}[subsection]
\renewcommand{\theparagaD}{{\bf\arabic{section}.\bf\arabic{subsection}.\arabic{paragaD}}}
\newcommand{\paragaD}{\medskip \addtocounter{paragaD}{1} 
\noindent{\theparagaD\ } }

\begin{document}


\title{Measure and capacity of wandering domains\\
in  Gevrey near-integrable exact symplectic systems
}
\author{%
Laurent Lazzarini\thanks{
Universit\'e Paris VI, UMR 7586, {\em Analyse alg\'ebrique},
4 place Jussieu, 75252 Paris {\sc cedex} 05.
E-mail: laurent.lazzarini@imj-prg.fr},\ \ 
Jean-Pierre Marco\thanks{
Universit\'e Paris VI, UMR 7586, {\em Analyse alg\'ebrique},
4 place Jussieu, 75252 Paris {\sc cedex} 05.
E-mail: jean-pierre.marco@imj-prg.fr},\ \ 
David Sauzin\thanks{
CNRS UMI 3483 -- Laboratorio Fibonacci,
Collegio Puteano, Scuola Normale Superiore di Pisa,
Piazza dei Cavalieri 3,
56126 Pisa.
E-mail: david.sauzin@sns.it}
}

\maketitle

\vspace{-.75cm}

\begin{abstract}
A wandering domain for a diffeomorphism~$\Psi$ of $\A^n=T^*\T^n$ is an
open connected set~$W$
such that $\Psi^k(W)\cap W=\emptyset$ for all $k\in\Z^*$.  We
endow~$\A^n$ with its usual exact symplectic structure. 
An integrable diffeomorphism, \ie the time-one map~$\Phi^h$ of a
Hamiltonian $h \col \A^n\to\R$ which depends only on the action
variables, has no nonempty wandering domains.
The aim of this paper is to estimate the size (measure and Gromov
capacity) of wandering domains in the case of an exact symplectic
perturbation of~$\Phi^h$, in the analytic or Gevrey category.
Upper estimates are related to Nekhoroshev theory, 
lower estimates are related to examples of Arnold diffusion.
This is a contribution to the ``quantitative Hamiltonian
perturbation theory'' initiated in previous works on the optimality
of long term stability estimates and diffusion times;
our emphasis here is on discrete systems because this is the natural
setting to study wandering domains.

We first prove that the measure (or the capacity) of these wandering domains
is exponentially small, with an upper bound of the form 
$\exp \big( - c \big(\frac{1}{\eps}\big)^{\frac{1}{2n\al}} \big)$, 
where~$\eps$ is the size of the perturbation, $\al\geq1$ is the Gevrey exponent ($\al=1$ for
analytic systems) and~$c$ is some positive constant depending mildly on~$h$. 
This is obtained as a consequence of an exponential stability theorem
for near-integrable exact symplectic maps, in the analytic or Gevrey category,
for which we give a complete proof based on the most recent
improvements of Nekhoroshev theory for Hamiltonian flows,
and which requires the development of specific Gevrey suspension
techniques.

The second part of the paper is devoted to the construction of
near-integrable Gevrey systems possessing wandering domains, for which
the capacity (and thus the measure) can be estimated from below. 
We suppose $n\ge2$, essentially because KAM theory precludes Arnold
diffusion in too low a dimension.
For any $\al>1$, we produce examples with lower bounds of the
form
$\exp \big( - c \big(\frac{1}{\eps}\big)^{\frac{1}{2(n-1)(\al-1)}} \big)$.
This is done by means of a ``coupling'' technique, involving
rescaled standard maps possessing wandering discs in~$\A$ and near-integrable systems
possessing periodic domains of arbitrarily large periods in~$\A^{n-1}$.
The most difficult part of the construction consists in obtaining a
perturbed pendulum-like system on~$\A$ with periodic islands of
arbitrarily large periods, whose areas are explicitly estimated from
below. Our proof is based on a version due to Herman of the translated
curve theorem.
\end{abstract}

\vfill
\pagebreak

\setcounter{tocdepth}{2}
\tableofcontents

\newpage


\addtocounter{section}{-1}
\section{Introduction}


\paraga Let  $\A^n=\T^n\times\R^n$ be the cotangent bundle of the torus~$\T^n$,
endowed with its usual angle-action coordinates $(\th,r)$ and the
usual exact symplectic form.
What we call an integrable diffeomorphism is the time-one map of the
flow generated by an integrable Hamiltonian, \ie a Hamiltonian function which
depends only on the action variables~$r$;
the phase space is then foliated into invariant tori
$\T^n\times\{r^0\}$, $r^0\in\R^n$,
carrying quasiperiodic motions.
We are interested in \emph{near-integrable systems}, \ie exact
symplectic perturbations of an integrable diffeomorphism,
and their {\em wandering sets}. 

A wandering set for a diffeomorphism~$\Psi$ of~$\A^n$ is a subset
$W\subset\A^n$ whose iterates $\Psi^k(W)$, $k\in\Z$, are pairwise
disjoint.
The Poincar\'e recurrence theorem shows that any wandering set of an
integrable diffeomorphism has zero Lebesgue measure.
It is more difficult to prove that wandering sets with positive measure
may exist for near-integrable systems.
In fact KAM theory shows that they cannot exist for $n=1$ (at least
when restricting to perturbations of non-degenerate integrable maps).
In \cite{ms}, examples of near-integrable systems possessing nonempty
wandering domains\footnote{We call \emph{domain} an open connected
  subset of~$\A^n$.}, thus with positive measure, have been constructed
for any $n\ge2$.

The aim of the present work is to estimate, from above and from below,
the possible ``size'' of the wandering sets of a near-integrable
system as a function of the ``size of the perturbation''.
We define the size of the perturbation to be the distance between the
near-integrable system and the integrable map of which it is a
perturbation, 
assuming that all our functions belong to a Gevrey class and measuring
this distance in the Gevrey sense.
For Gevrey classes, we use the notation~$G^{\al}$, where $\al\geq 1$ is a
real parameter;
recall that~$G^1$ coincides with the analytic class, while the
classes~$G^{\al}$, $\al>1$, are larger and more flexible, in
particular they contain bump functions.

As for the size of wandering sets, we consider two natural
candidates: the Lebesgue measure and the Gromov capacity.  The former
can be seen as the ``maximal'' possible one and is related to the
theory of transport.  The latter is the ``minimal'' possible one and
is directly related to the symplectic character of our problem.

Our aim is to find explicit upper bounds for the measure of wandering
Borel subsets and ``test'' their optimality by constructing examples
of near-integrable systems possessing wandering sets whose capacity is
estimated from below.

Our upper estimates are closely related to the long term stability
estimates in perturbation theory, initiated by Nekhoroshev. Since
usual estimates deal with continuous systems, we first have to
transfer the whole theory to the discrete setting, which is done by
adapted suspension techniques.  The resulting estimates hold for
near-integrable systems in all Gevrey classes~$G^\al$, $\al\geq 1$:
for such systems, the actions remain almost constant during
exponentially long times. As a consequence, taking into account the
measure-preserving character of symplectic diffeomorphisms, we prove
that for any~$G^\al$ near-integrable system the measure of a wandering
Borel subset is exponentially small with respect to the size of the
perturbation.

Our lower estimates deal with the capacity of the wandering sets.
The interest of this is twofold:
first, the measure of a set is always larger than a positive power of
its capacity (up to an explicit multiplicative factor), so lower
bounds for the capacity entail lower bounds for the measure;
second, capacity is a truly symplectic notion.
The capacity of a set is in general extremely difficult to compute,
however we will design our ``unstable'' examples so that they admit
wandering {\em polydiscs} (\ie products of discs in each factor
of~$\A^n$). In that case the capacity is just the minimum of the areas
of the factors.
%
%
%
%
Our constructions rely strongly on the existence of bump
functions. So, as in \cite{hms} and \cite{ms}, we produce our examples
in the classes~$G^\al$ for $\al>1$ only.
A striking fact is that the lower bounds on the capacity admit the
same exponential form as the upper estimates deduced from Nekhoroshev
theory.  This is the main result we get here about instability in
near-integrable systems.  Related results on the existence of
periodic domains with large periods will also be obtained in the
course of the proof of the main instability result.
 
Before stating our results more precisely, let us now give a brief
overview of the evolution of results in perturbation theory, in order
to introduce the main tools we will use in the sequel.


\paraga We begin with stability results. 
The stability problem for perturbations of integrable Hamiltonian
systems originated in the first investigations on the secular
stability of the solar system. At this early stage, the main question
was to understand the behaviour of the linearised equations along a
particular solution.  Then, under the influence of Poincar\'e, this
purely local vision experienced a drastic metamorphosis towards a
global qualitative understanding of the asymptotic behaviour of the
orbits.  He introduced a fundamental tool---amongst many others---for
such a qualitative description: the method of {\em normal forms}, that
is, the construction of simplified systems which nevertheless exhibit
the pre-eminent features of the initial one.  The theory of normal
forms was further developed by Birkhoff, and thoroughly investigated
since then by a number of authors.

It was also a fundamental contribution of Poincar\'e to distinguish
between two different modes of ``convergence'' of series:
convergence {\em au sens des astronomes} and convergence {\em au sens
  des g\'eom\`etres}.  The latter coincides with our
usual notion of convergence, while the former is related to the notion
of asymptotic expansion and does not exclude the possibility of
performing ``least term summation'' (which is itself intimately
related to the Gevrey nature of the series at hand).

This bunch of ideas was applied in particular to the {\em fundamental
  problem of dynamics}, that is, the study of the qualitative
dynamical behaviour of analytic Hamiltonian systems on~$\A^n$ which
are perturbations of integrable Hamiltonians.  The various problems of
convergence of the series giving rise to the solutions of the
perturbed problem were extensively examined by Poincar\'e, without
however reaching a definitive conclusion.

The next major breakthrough was due to Kolmogorov, who understood in
the 1950s how to take advantage of stability properties exhibited by
the unperturbed quasiperiodic tori, provided that their fundamental
frequencies are sufficiently nonresonant (\ie satisfy Diophantine conditions).
This approach was then generalised by Arnold and Moser and gave rise
to the so-called KAM theory: in appropriate function spaces (analytic,
$C^\infty$, finitely differentiable) the surviving tori form a subset
whose relative measure tends to~$1$ when the size of the perturbation
tends to~$0$. 
See \cite{KAMstory} for a non-technical historical
account and appropriate references on KAM theory.

Besides the KAM theorem, one major achievement occured in the 1980s
with the direct proof by Eliasson of the convergence of the
perturbative series under the usual assumptions of KAM theory.  This
yields directly the quasiperiodic solutions, and the KAM tori are
nothing but their closure. So, after a somewhat suprising detour, the
Poincar\'e convergence {\em au sens des g\'eom\`etres} indeed gives
rise to invariant geometric objects. 
%

In a different direction, the Poincar\'e convergence {\em au sens des
  astronomes} 
can be considered as the mechanism at work behind ``exponential stability''.
Exponential stability means that, for a perturbation of an integrable
Hamiltonian flow, the action variables of an arbitrary solution vary little
over a time interval of the order of $\exp \big( \frac{1}{\eps^a}
\big)$, where~$\eps$ is the size of the perturbation and~$a$ is a
positive exponent independent of the perturbation.
This was established by Nekhoroshev in the 1970s for the analytic case.
Again, the theory of normal forms revealed itself to
be of crucial importance in this setting. The main idea there is to
cover the whole phase space by a patchwork of domains with
various resonant structures and perform in each of them a finite---but
long---sequence of adapted normalising transformations.  An additional
geometric argument (steepness) then proves the confinement of the
actions for all initial conditions over an exponentially long
timescale.

In this text 
%
%
we will take advantage of the numerous improvements of the stability estimates
after Nekhoroshev's initial work, beginning with the work by Lochak
\cite{L92} where the question of the optimality of the stability
exponent first appeared, with a first conjecture on its value.
Then the ``likely optimal'' stability exponent was derived more
precisely in the case of quasi-convex unperturbed systems by Lochak,
Neishtadt and Niederman, and P\"oschel (see \cite{LNN,P} and
references therein).
Finally, based on Herman's ideas, these works were later generalised
to general Gevrey classes in \cite{hms}---which also clarified the
connection with Gevrey asymptotics and least term summation---and the
stability exponent was improved in \cite{BM} to reach the probably
optimal value.
Our present study relies on these latter two works to produce upper
estimates of the measure of wandering sets.

More surprisingly, the construction of our {\em unstable} examples
will also heavily rely on KAM techniques (in the form developed by
Herman for invariant curves on the annulus).  So stability results may
also help produce unstable behaviour and transport in phase
space. This has already been noticed in the context of Arnold
diffusion (see below), but we will deal here with new phenomena.


\paraga Let us now pass to the description of some unstable systems,
beginning with the seminal and highly inspiring Arnold example. In
parallel with the evolution of stability theory, Arnold
introduced in the 1960s a paradigm perturbed angle-action system
exhibiting unstable behaviour \cite{Ar}. In his example (a non-autonomous
nearly integrable Hamiltonian flow on~$\A^2$), the action variables
drift over intervals of fixed length {\em whatever the size of the
  perturbation}.  
Arnold conjectured that this instability phenomenon
(now called Arnold diffusion) should occur in the complement of the KAM tori for
``typical'' systems.  Of course, due to Nekhoroshev theory, Arnold
diffusion in analytic or Gevrey systems has to be exponentially slow
with respect to the size of the perturbation.

The key idea in Arnold's example is the possibility that a
perturbation of an integrable Hamiltonian can create a
continuous family of hyperbolic tori in a given energy level, whose
invariant manifolds also vary continuously.
An additional perturbation then makes the stable and unstable
manifolds of each torus intersect transversely in their energy
level. It is therefore possible to exhibit ordered families
$(T_m)_{1\leq m\leq m_*}$ of hyperbolic tori, extracted from the
continuous one, such that the unstable manifold $W^u(T_m)$
intersects transversely the stable manifold $W^s(T_{m+1})$, and such
that the distance between the extremal tori~$T_1$ and~$T_{m_*}$ (in
the action space) is independent of the size of the perturbation
(the number~$m_*$ of tori in such a family tends to~$+\infty$
when the size of the perturbation tends to~$0$). 
Finally, one constructs orbits which shadow the consecutive
heteroclinic orbits between the tori and pass close to both~$T_1$
and~$T_{m_*}$. The action variables of such orbits therefore
experience a drift which is independent of the size of the
perturbation.

Arnold's example has been generalised in many ways, particularly in
view of proving the ``generic'' occurrence of Arnold diffusion in
nearly integrable systems on~$\A^3$.
Notice that the existence of hyperbolic KAM tori (or
more generally hyperbolic Mather sets) is an important tool to
implement the previous scheme: this is a first example of how stability
induces instability.

Another important development was the possibility of computing the
drifting time of unstable orbits in examples of Arnold diffusion.
This program was achieved in \cite{hms} for~$G^\al$ Gevrey systems
with $\al>1$, and then in \cite{LM05,Z11} for analytic systems.


In this work, to produce examples of near-integrable systems
possessing wandering polydiscs, we develop the method of \cite{ms},
which itself builds on the techniques of \cite{hms}.
The construction of unstable orbits in \cite{hms} is rather different
from that of Arnold, even if, \emph{a posteriori}, one can see that
these drifting orbits too shadow families of heteroclinically connected
tori, which clearly shows the intrinsic complexity of their dynamics.

The key idea in \cite{hms} is to embed a well-controlled discrete
dynamical system of~$\A$, namely a renormalised standard map, into a
high iterate of a specific~$G^\al$ near-integrable system of~$\A^n$,
assuming $\al>1$ (and then into a nearly integrable non-autonomous
Hamiltonian flow).
Taking advantage of the drifting points of the standard map, one can
produce drifting orbits in this near-integrable system (and then in
the corresponding Hamiltonian flow).
In contrast with Arnold's example, such a construction yields systems
with orbits {\em biasymptotic to infinity} in action.  
This proves to be a crucial feature of the construction when one
builds on it to produce wandering sets with positive measure, since
they cannot be confined inside compact subsets due to the preservation
of volume.

In \cite{ms}, the KAM theorem was used to produce wandering polydiscs
surrounding drifting orbits in near-integrable systems of the same
kind as those of \cite{hms}, but without any quantitative estimate.
It seems that such a coexistence of stable geometric objects
(invariant tori) and highly unstable open sets had not been observed
before, although it is reminiscent of the existence of the
``periodic islands'' in ``chaotic seas'' which are ubiquitous in the
theory of two-dimensional symplectic maps.
Our wandering domains coexist with (and are contained in the
complement of) all the invariant compact subsets, including 
Lagrangian invariant tori, 
lower-dimensional invariant tori (like
the hyperbolic ones used in Arnold's mechanism) 
or Mather sets.

The novelty of the present paper is that, using more refined versions
of the KAM theorem and a (much) better control of the normal forms, we
are now able to estimate the capacity of the wandering polydiscs that
we construct.


\paraga We can now informally describe the content of this paper.
Our interest in wandering sets makes it essential that we deal with
diffeomorphisms rather than flows,
like in \cite{ms} and in contrast with most of the literature on
Hamiltonian perturbation theory;
in fact, it is the first time that wandering sets of near-integrable
discrete systems are the object of such detailed investigation.

So we first have to transfer the known stability results for
Hamiltonian flows to the setting of near-integrable discrete systems. 
%
%
The result is \textbf{Theorem~\ref{th:NekhFevMaps}}, whose simplified statement
is the following. 
We fix a real $\al\ge1$ and an integrable diffeomorphism of~$\A^n$
\[
\Phi^h \col (\th,r) \mapsto \big( \th+\na h(r), r \big),
\]
where $h \col \R^n\to\R$ is a~$G^\al$ convex function. 
Then, given $\rho>0$, for any $G^\al$ exact symplectic
diffeomorphism~$\Psi$ having
$\eps \defeq \dist_\al(\Psi,\Phi^h)$ small enough 
(see Section~\ref{secthmNekhMapsA} for the precise definition of the
distance in Gevrey classes),
the iterates $\Psi^k(\th\zz0,r\zz0)=(\th\zz k,r\zz k)$ of any initial condition satisfy
\[
\normD{r\zz k-r\zz 0}\leq \rho 
\quad\text{for}\ens 
0\leq k\leq \exp\Big({c\Big(\frac{1}{\eps}\Big)^{\frac{1}{2n\al}}}\Big),
\]
where~$c$ is a suitable positive constant depending mildly on~$h$. 

Other and more refined estimates are also available, for which the
confinement radius of the action variables tends to~$0$ with~$\eps$.

Our proof makes use of a new suspension result, \textbf{Theorem~\ref{thmsusp}},
which allows one to embed a~$G^\al$ near-integrable system into a
non-autonomous near-integrable Hamiltonian flow;
the analytic case is essentially done in \cite{KP}, while for the case
$\al>1$ we had to devise specific Gevrey techniques to adapt
quantitatively Douady's method \cite{Douady} based on generating
functions.


Next, taking into account the preservation of the Lebesgue
measure~$\mu$ by symplectic diffeomorphisms, we prove in
\textbf{Theorem~\ref{th:upperbounds}} that a wandering Borel set~$W$
of~$\Psi$,
when it is contained in a bounded region of~$\A^n$,
must satisfy
\[
\mu(W)\leq \exp\Big( -c \Big( \frac{1}{\eps} \Big)^{\frac{1}{2n\al}} \Big),
\]
where $c>0$ is a suitable constant depending mildly on~$h$. 


The rest of the paper is devoted to the construction of examples with
wandering domains, with estimates of their Gromov capacity.  
This is the content of \textbf{Theorem~\ref{th:lowerbounds}}: 
there is a sequence $(\Psi_j)_{j\ge1}$ of $G^\al$~diffeomorphisms
of~$\A^n$, with $\al>1$, such that
\[
\eps_j \defeq \dist_\al \! \Big(\Psi_j, \Phi^{\pdemi(r_1^2+\cdots+r_n^2)}\Big)
\]
tends to~$0$ when $j\to\infty$, each of which admits a wandering
polydisc~$\jW_j$ whose Gromov capacity satisfies the inequality
\[
\CG(\jW_j) \ge \exp\Big( -c' \Big( \frac{1}{\eps_j} \Big)^{\frac{1}{2(n-1)(\al-1)}} \Big),
\]
where $c'>0$ is a suitable constant.


The proof of Theorem~\ref{th:lowerbounds} is based on a version of the
``coupling lemma'' introduced in \cite{hms} and \cite{ms}, whose
application requires the construction of several controlled dynamics
on distinct subfactors of the annulus~$\A^n$. 
Here is, in a few lines, the strategy:

\begin{enumerate}[--]
\item
  On the one hand, \textbf{Theorem~\ref{th:perdomains}} provides us
  with a near-integrable system~$G$ arbitrarily close to
  $\Phi^{\dem(r_2^2+\cdots+r_n^2)}$, possessing a $q$-periodic
  polydisc $\jV \subset \A^{n-1}$ of arbitrarily large period~$q$.
  Moreover the orbit of~$\jV$ under~$G$ is controlled well enough and
  its capacity can be explicitly bounded from below.
\item
  On the other hand, we choose a Gevrey function~$U$ so that the ``rescaled standard map''
\[
\psi \col (\th_1,r_1)\mapsto \big(\th_1 + q r_1,\ r_1- \frac{1}{q}U'(\th_1+qr_1)\big)
\]
has a wandering domain $\jU\subset\A$ whose area is of order
$1/q$. We observe that this map can be written
%
%
as the composition of the time-one map of the Hamiltonian
${\frac{1}{q}U(\th_1)}$ with the $q$th iterate of the integrable map
$\Phi^{\pdemi r_1^2}$.
\item
The aforementioned coupling lemma then produces an exact
symplectic perturbation~$\Psi$ of
$\Phi^{\pdemi r_1^2} \times \Phi^{\dem(r_2^2+\cdots+r_n^2)} = \Phi^{\dem(r_1^2+r_2^2+\cdots+r_n^2)}$
whose $q$th iterate coincides with $\psi \times G^q$ when restricted
to $\A\times\jV$.
In that situation, $\jU\times\jV$ is easily seen to be a wandering set of~$\Psi$.
\end{enumerate}

Choosing an appropriate function~$U$ and estimating the area of the
wandering domain~$\jU$ are rather easy.

The application of the coupling lemma requires a Gevrey function~$g$ on~$\A^{n-1}$ satisfying a
``synchronization condition'' \wrt the orbit of~$\jV$ under~$G$.
For this, we use a bump function,
whose Gevrey norm is large, unavoidably, but somehow this
can be compensated by choosing~$q$ large enough, so as to ensure
that~$\Psi$ is indeed arbitrarily close to integrable.

Much more work will be needed to prove Theorem~\ref{th:perdomains}.
Another use of the coupling lemma will first allow us to reduce the problem to
proving the two-dimensional version of the statement, which is
essentially \textbf{Theorem~\ref{th:varpseudopend}}.
The proof of Theorem~\ref{th:varpseudopend} is then the most technical part of
the construction.
It first necessitates the introduction of a ``pseudo-pendulum'' on~$\A$,
of the form
\[
P(\th,r)=\pdemi r^2+\pfrac{1}{N^2}V(\th),
\]
where the potential~$V$ is a flat-top bump function on~$\T$ and~$N$ is a
large parameter. This pseudo-pendulum is then perturbed to produce
elliptic periodic points of any period,
surrounded by elliptic islands with controlled areas. 
The main task consists in estimating these areas, which requires the
use of Herman's quantitative version of the two-dimensional KAM
theorem and necessitates the computation of high-order parametrised
normal forms.  
Theorem~\ref{th:varpseudopend} is interesting in itself; see for instance \cite{Liv} for
related questions on standard maps.


\paraga To conclude this introduction, let us mention that the present
work can be seen as a contribution to the development of a
``quantitative Hamiltonian perturbation theory'', 
focused on the question of the size of wandering domains.
Other studies should be devoted to the numerous quantities one can
associate with a nearly integrable exact symplectic diffeomorphism or
Hamiltonian flow:
one may think of the separatrix splitting, the angles of Green bundles,
the topological entropy, the growth of isolated periodic orbits, and
so on; each of them should be estimated from above and below in an
optimal way.
While the upper estimates may be based on normal form theory, the
construction of ``optimal'' examples may reveal itself to be extremely
rich and difficult, as illustrated by the case of wandering domains
in this work.  
We see this problem as a challenging motivation to pursue these
quantitative studies and get a more developed vision of this
domain---still in its infancy.






\paraga The paper is organized as follows.

\begin{itemize}

\item Section~\ref{secpresentres} is dedicated to a precise formulation of our
  assumptions, notations and results.
  Theorem~\ref{th:NekhFevMaps} gives long time stability
  estimates for near-integrable systems of~$\A^n$.
  Theorem~\ref{th:upperbounds} and Theorem~\ref{th:lowerbounds} 
  respectively state our main results about the upper and lower bounds
  for the measure and capacity of wandering sets of near-integrable systems.
  Theorem~\ref{th:perdomains}, on the construction of near-integrable
  systems of~$\A^{n-1}$ possessing periodic domains with explicit
  lower bounds for their capacity, is stated.
  It splits into two parts: Theorem~\ref{th:perdomains}(i) deals with
  the two-dimensional case, \ie periodic domains in~$\A$, while
  Theorem~\ref{th:perdomains}(ii) is dedicated to systems on~$\A^m$,
  $m\geq 2$.

\item In Section~\ref{secGevNekhMaps}, we state and prove
  Theorem~\ref{thmsusp}, on the suspension of Gevrey near-integrable systems.
  This enables us to deduce the stability theory for Gevrey diffeomorphisms
  from the stability theory for Gevrey Hamiltonian flows and thus prove
  Theorem~\ref{th:NekhFevMaps}, and then Theorem~\ref{th:upperbounds}.

\item Section~\ref{sec:proofpseudopend} contains the most technical
  part of the paper, that is, the construction of examples of
  near-integrable systems of~$\A$ with periodic islands of
  arbitrarily large period, whose area we are able to estimate from below.
  This is the content of Theorem~\ref{th:varpseudopend}, which is a
  parametrised version of Theorem~\ref{th:perdomains}(i). The proof of
  Theorem~\ref{th:perdomains}(i) is in Section~\ref{secPfThperi}, the
  rest of Section~\ref{sec:proofpseudopend} is devoted to the proof of
  Theorem~\ref{th:varpseudopend}.

\item In Section~\ref{SecPfWander} we explain the coupling lemma and
  its use to produce periodic or wandering polydiscs. 
  The proof of Theorem~\ref{th:perdomains}(ii) is thus obtained, by
  coupling the periodic domains of Theorem~\ref{th:varpseudopend}
  (suitably rescaled) with periodic domains of an elementary
  perturbation of $\Phi^{\dem(r_3^2+\cdots+r_n^2)}$.
  Then, Theorem~\ref{th:lowerbounds} is obtained by coupling the
  wandering domain of a rescaled standard map and the periodic domains
  of Theorem~\ref{th:perdomains}.

\item The paper ends with four appendices, dealing with some
  technicalities needed in the course of the various proofs.
%

\end{itemize}


\section{Presentation of the results}   \label{secpresentres}


\subsection{Perturbation theory for analytic or Gevrey
  near-integrable maps---Theorem~\ref{th:NekhFevMaps}}


\paragaD
Let $\T \defeq \R / \Z$.
For $n\geq1$ we denote by $\A^n=\T^n\times\R^n$ the $2n$-dimensional
annulus, viewed as the cotangent bundle of~$\T^n$, with coordinates
$\th=(\th_1,\ldots,\th_n)$, $r=(r_1,\ldots,r_n)$, 
respectively called ``angles'' and ``actions'',
and Liouville exact symplectic form $\Om=- \dd \la$,
$\la \defeq \sum_{i=1}^n r_i \dd \th_i$.
Recall that a map~$\Psi$ is said to be exact symplectic (or globally canonical) if the
differential $1$-form $\Psi^*\la-\la$ is exact.
Examples of exact symplectic maps are provided by the flows of
Hamiltonian vector fields.

When the Hamiltonian vector field generated by a function~$H$ on
$\A^n$ 
\[
\frac{\pa H}{\pa r_1} \frac{\pa\ens\;}{\pa \th_1} +
\cdots + \frac{\pa H}{\pa r_n} \frac{\pa\ens\;}{\pa \th_n}
- \frac{\pa H}{\pa \th_1} \frac{\pa\ens\;}{\pa r_1} -
\cdots - \frac{\pa H}{\pa \th_n} \frac{\pa\ens\;}{\pa r_n}
\]
is complete, we denote by
$\Phi^{H} \col \A^n \righttoleftarrow$ the time-one map of the
Hamiltonian flow.
We say that a diffeomorphism of $\A^n$ is {\em integrable} when it is of
the form $\Phi^h$, where the function $h$ depends only on the action variable
$r$, thus
\begin{equation} \label{eqPhih}
\Phi^h(\th,r)=\big(\th+\langle \nabla h (r)\rangle,\,r\big),
\end{equation}
where $\angD{\,\cdot\,} \col \R^n \to \T^n$ is our notation for the canonical
projection.

We are interested in \emph{near-integrable maps}%
\footnote{
We also often call them ``near-integrable systems'' to emphasize that we are interested
in the discrete dynamical systems consisting in iterating these maps.
When we say ``near-integrable'', the exact symplectic character is understood.
}, \ie exact symplectic maps close to an integrable map~$\Phi^h$,
closeness being intended in the analytic sense or Gevrey sense.


\paragaD
Let us introduce notations for the spaces of Gevrey {functions}. 
Given $n\geq 1$, we use the Euclidean norm in~$\R^n$ and,
for~$R$ positive real or infinite, denote by $\ov B_R$ the
closed ball of radius~$R$ centred at~$0$ (so $\ov B_\infty=\R^n$).
We set
\begin{equation}
\A^n_R \defeq \T^n\times\ov B_R,
\end{equation}
in particular $\A^n_\infty=\A^n$. 
Given $\al\ge1$ and $L>0$ real, and $0<R\le\infty$, we define the
Banach spaces of real-valued functions
\begin{align}
\label{eq:defGalLR}
G^{\al,L}(\ov B_R) &\defeq \{ h\in C^\infty(\ov B_R) \mid \normD{h}_{\al,L,R} <\infty \},
\quad
\normD{h}_{\al,L,R} \defeq \sum_{\ell\in\N^{n}} \frac{L^{\absD{\ell}\al}}{\ell !^\al} \normD{\pa^\ell h}_{C^0(\ov B_R)}
\\
\label{eq:defGalL}
G^{\al,L}(\A^n_R) &\defeq \{ f\in C^\infty(\A^n_R) \mid \normD{f}_{\al,L,R} <\infty \},
\quad
\normD{f}_{\al,L,R} \defeq \sum_{\ell\in\N^{2n}}
\frac{L^{\absD{\ell}\al}}{\ell !^\al} \normD{\pa^\ell f}_{C^0(\A^n_R)}.
\end{align}
%
%
We have used the standard notations
$\absD{\ell} = \ell_1+\cdots+\ell_{2n}$, $\ell! = \ell_1!\ldots\ell_{2n}!$,
$\pa^\ell = \pa_{1}^{\ell_1}\ldots\pa_{2n}^{\ell_{2n}}$, 
$\pa_k=\pa_{x_k}$ for $k=1,\ldots,2n$,
where $(x_1,\ldots,x_{2n})=(\th_1,\ldots,\th_n,r_1,\ldots,r_n)$,
and
\[
\N \defeq \{0,1,2,\ldots\}.
\]
%
%
We shall make use of the natural inclusion 
$G^{\al,L}(\ov B_R) \hookrightarrow G^{\al,L}(\A^n_R)$
without further notice, treating an $h \in G^{\al,L}(\ov B_R)$
indifferently as an element of any of the two spaces.


For $\al=1$, one recovers real analytic functions of $\A^n_R$:
any function $f \in G^{1,L}(\A^n_R)$ is real analytic in~$\A^n_R$ and
admits a holomorphic extension in
$\jV_L\T^n \times \jV_L \ov B_R$,
with complex neighbourhoods of~$\T^n$ and~$\ov B_R$ defined by
\beq   \label{eqdefjV}
\begin{aligned}
\jV_L\T^n &\defeq \{\, 
\th \in (\C/\Z)^n \mid \abs{ (\Im\th_1,\ldots,\Im\th_n) }_\infty < L
\,\}, \\[1ex]
\jV_L \ov B_R &\defeq \bigcup_{r^*\in \ov B_R} 
\{\, r \in \C^n \mid \abs{ r-r^* }_\infty < L
\,\},
\end{aligned}
\eeq
where
$\abs{\xi}_\infty \defeq \max\{ \abs{\xi_1}, \ldots, \abs{\xi_n} \}$
for $\xi \in \R^n$ or~$\C^n$;
conversely, for any function~$f$ real analytic in~$\A^n_R$, there
exists $L>0$ such that $f \in G^{1,L}(\A^n_R)$.
%
%
For $\al>1$, one gets non-quasianalytic spaces of Gevrey functions.


Recall that $\normD{\,\cdot\,}_{\al,L}$ is an agebra norm for every $\al\ge1$:
$\normD{fg}_{\al,L} \le \normD{f}_{\al,L} \normD{g}_{\al,L}$;
see \cite{hms}---some other useful properties of these norms are recalled in
Appendix~\ref{secappremind}.


\paragaD
Our {maps} will be analytic, \ie Gevrey-$1$, or
more generally Gevrey-$\al$ for some $\al\ge1$, \ie elements of one of
the sets
\begin{multline*}
G^{\al,L}(\A^n_R,\A^n) \defeq
\{\, \Psi \col \A^n_R \to \T^n\times\R^n \mid 
\exists \ti\Psi 
\col \A^n_R \to \R^n\times\R^n \; \text{lifting $\Psi$} \\ 
\text{so that}\;
\ti\Psi_1,\ldots,\ti\Psi_{2n} 
\in G^{\al,L}(\A^n_R)
\,\},
\end{multline*}
with $L>0$ and $0 < R \le \infty$.
We set, for any $\De \in G^{\al,L}(\A^n_R,\A^n)$,
\begin{equation}	\label{eqdefNotNorm}
\NormD{\De}_{\al,L,R} \defeq
\inf\Big\{\,
\normD{\ti\De_{1}}_{\al,L,R} + \cdots + \normD{\ti\De_{2n}}_{\al,L,R} 
\mid \; \ti\De \col \A^n_R \to \R^n \times \R^n 
\; \text{lift of~$\De$} \,\Big\}
\end{equation}
(in fact the infimum is always attained);
one can check that the formula
\[ (\Psi_1,\Psi_2) \mapsto \NormD{\Psi_2-\Psi_1}_{\al,L,R} \]
defines a translation-invariant distance which makes
$G^{\al,L}(\A^n_R,\A^n)$ a complete metric space.


\paragaD 
Our first result is a version of the Nekhoroshev Theorem for Gevrey near-integrable
exact symplectic maps in the convex case:
\label{secthmNekhMapsA}
\begin{thm}[Exponential stability for maps]	\label{th:NekhFevMaps}
Let $n\geq1$ be an integer. Let $\al\ge1$ and $L,R,R_0>0$ be reals such
that $R<R_0$.
Let $h\in G^{\al,L}(\ov B_{R_0})$ have positive definite Hessian matrix on $\ov B_{R_0}$.
Then there exist positive reals $\eps_*,c_*$,
and, for each positive $\rho<R_0-R$, positive reals $\eps'_\rho \le \eps_*$ and
$c'_\rho \le c_*$,
and, for each positive $\sig < \frac{1}{n+1}$, positive reals $\eps''_\sig \le \eps_*$ and
$c''_\sig \le c_*$,
satisfying the following:

For each exact symplectic map $\Psi \in G^{\al,L}(\A^n_{R_0},\A^n)$
such that $\eps \defeq \NormD{\Psi-\Phi^h}_{\al,L,R_0} \le \eps_*$,
every point $(\th\zz0,r\zz0)$ of $\A^n_R$ has well-defined iterates
$(\th\zz k, r\zz k) \defeq \Psi^k(\th\zz0,r\zz0) \in \A^n_{R_0}$ 
for all $k\in\Z$ such that
$\absD{k} \le 
\exp \big( c_* \big(\frac{1}{\eps}\big)^{\frac{1}{2n\al}} \big)$,
and
\begin{enumerate}[(i)]
\item
$\dst \eps \le \eps'_\rho \ens\text{and}\ens
\absD{k} \le \exp \Big( c'_\rho \Big(\frac{1}{\eps}\Big)^{\frac{1}{2n\al}} \Big)
\quad\ens \Rightarrow \quad
\normD{r\zz k - r\zz0} \leq \rho$,
\item
$\dst \eps \le \eps''_\sig \ens\text{and}\ens
\absD{k} \le \exp \Big( c''_\sig \Big(\frac{1}{\eps}\Big)^{\frac{1-\sig}{2n\al}} \Big)
\quad\ens \Rightarrow \quad
\normD{r\zz k - r\zz0} \leq \begin{cases}
\tfrac{1}{c''_\sig} \eps^{\frac{\sig}{2}} & \text{if $\al=1$,} \\[1.5ex]
\tfrac{1}{c''_\sig} \eps^{\frac{\sig}{5 n^2}} & \text{if $\al>1$,} 
\end{cases}$
\item
$\dst \eps \le \eps_* \ens\text{and}\ens
\absD{k} \le \exp \Big( c_* \Big(\frac{1}{\eps}\Big)^{\frac{1}{2(n+1)\al}} \Big)
\ens \Rightarrow \quad
\normD{r\zz k - r\zz0} \leq \tfrac{1}{c_*} \eps^{\frac{1}{2(n+1)}}$.
\end{enumerate}
\end{thm}

\noindent
The proof of Theorem~\ref{th:NekhFevMaps} is in Section~\ref{secGevNekhMaps}.


The case $\al=1$ of~(iii) is due to S.~Kuksin and~J.~P\"oschel
\cite{KP}.
The rest of the statement is, to the best of our knowledge, new.
It relies on the most recent version of the Nekhoroshev Theorem for
Gevrey near-integrable quasi-convex Hamiltonian vector fields due to A.~Bounemoura
and J.-P.~Marco \cite{BM},
which improves the possible exponents for the stability time (at the
price of a less good control of the confinement of the orbits)---we
reproduce Bounemoura-Marco's statement in Section~\ref{secPfThmNekhMaps}.
To transfer it to the discrete dynamics induced by a near-integrable
exact symplectic map~$\Psi$, we will have to construct a
non-autonomous time-periodic Gevrey Hamiltonian function, defined for $(\th,r,t) \in
\A^n_{R_1}\times\T$ with $R<R_1<R_0$, whose flow interpolates the discrete
dynamics---this is the content of Theorem~\ref{thmsusp} of
Section~\ref{secGevNekhMaps}.

The hypothesis that the Hessian matrix of the integrable part~$h$ is
positive definite is a strict convexity assumption:
it amounts to the existence of a positive real~$m$ such that $h$ is $m$-convex, in
the sense that
\beq    \label{eqdefmconv}
\ltrans{v}\,\dd\na h(r) v \ge m \normD{v}^2
\qquad\text{for all $r\in \ov B_{R_0}$
and $v\in \R^n$,}
\eeq
where $\dd\na h(r)$ is the Hessian matrix of~$h$ ar the point~$r$.
In fact, the reals $\eps_*,c_*,\eps'_\rho,c'_\rho,\eps''_\sig,c''_\sig$
depend on the integrable part~$h$ only through~$m$ and
$\normD{h}_{\al,L,R_0}$.


\begin{rem}[On the time exponents]
Beware that, as far as Nekhoroshev theory is concerned, exact
symplectic maps in~$\A^n$ behave like $N$-degree of freedom
autonomous Hamiltonian systems with $N=n+1$.
So the ``time exponent'' and the ``confinement exponent'' in the
case~(iii) are simply $\frac{1}{2N\al}$ and $\frac{1}{2N}$, which have been familiar since
the works by Lochak-Neishtadt and P\"oschel in the analytic case, or
Marco-Sauzin in the Gevrey case.

Bounemoura-Marco's novel result of \cite{BM} was the obtention of
better stability times at the price of releasing part of the control on
the confinement property.
The counterpart for discrete systems, as demonstrated by
Theorem~\ref{th:NekhFevMaps}, is that the time exponent can be taken
as large as $a_\sig = \frac{1-\sig}{2n\al}$ with arbitrary~$\sig$ such
that $0\le\sig\le\frac{1}{n+1}$, and the corresponding confinement
radius still tends to~$0$ as $\eps\to0$ if $\sig>0$, while we only get
a fixed (but arbitrarily small) confinement radius~$\rho$ if $\sig=0$;
observe that $a_\sig$ is a decreasing function of~$\sig$, so~$\sig$
close to~$\frac{1}{n+1}$ yields worse stability exponents~$a_\sig$,
close to the exponent $\frac{1}{2(n+1)\al}$ of case~(iii) (but better
confinement properties), while $\sig=0$ gives the best time exponent,
namely $\frac{1}{2n\al}$, for general orbits.
See Remark~\ref{remmorecomm} for more comments.
\end{rem}

\begin{rem}[Stabilization by resonances]
  We leave it to the reader to devise a refined statement for orbits
  starting $O(\eps^{1/2})$-close to a resonance of multiplicity
  $m \in \{1,\ldots, n\}$) by
  exploiting the well-known stabilizing effect of resonances available
  for Hamiltonian flows---see Remark~\ref{remmreson}.
  The time and confinement exponents then jump to $\frac{1}{2(n+1-m)\al}$
  and $\frac{1}{2(n+1-m)}$.

In particular, the time exponent $\frac{1}{2n\al}$ given in~(i) for
general orbits coincides with the time exponent available for the
orbits starting close to a simple resonance, but the latter have a
better confinement property (described by the positive exponent
$\frac{1}{2n}$) than general orbits.
\end{rem}


\begin{rem}[About the steep case]
  The original Nekhoroshev theorem was proved in the analytic case for
  a wider class of near-integrable Hamiltonian flows than just those
  with quasi-convex integrable part. Nekhoroshev only needed a
  non-degeneracy assumption called \emph{steepness}, which turns out
  to be generic in quite a strong sense.
This allowed S.~Kuksin and J.~P\"oschel to give an exponential
stability theorem for analytic near-integrable maps in the case
where~$h$ is supposed to be steep but not necessarily convex \cite{KP}. 
The same could be done for Gevrey near-integrable maps if the original
Nekhorohev statement could be generalised to the Gevrey steep case.
\end{rem}


\begin{rem}[KAM theorem for analytic or Gevrey near-integrable maps]    \label{remKAM}
The assumption that~$h$ be non-degenerate in the sense of
Kolmogorov 
(\ie that $\nabla h$ be a local diffeomorphism, which is a weaker
condition than strict convexity)
is sufficient to apply the KAM theorem, in its analytic
version if $\al=1$, or in its $C^\infty$ version if $\al>1$. 
For each $r_* \in B_{R_0}$ such that $\nabla h(r_*)$ is Diophantine,
we obtain for the discrete dynamics~$\Psi$ an invariant quasi-periodic
torus $\simeq\T^n$ located close to $\T^n\times\{r_*\}$
as soon as $\NormD{\Psi-\Phi^h}_{\al,L,R_0}$ is small enough.
If $\al =1$, then such a torus is known to be analytically embedded
in~$\A^n_{R_0}$.
If $\al>1$, then the embedding is known to be $C^\infty$ and 
\emph{one can prove that the embedding is in fact Gevrey-$\al$} 
by applying Popov's KAM theorem for Gevrey near-integrable
Hamiltonians \cite{Popov} to the interpolating Hamiltonian flow constructed
in Theorem~\ref{thmsusp} of Section~\ref{secGevNekhMaps}.
\end{rem}


\subsection{Wandering sets of near-integrable
  systems---Theorems~\ref{th:upperbounds} and~\ref{th:lowerbounds}}


\paragaD
The other results of this paper deal with wandering sets for
near-integrable systems.

\begin{Def}
%
%
Given a diffeomorphism $\Psi$ of a manifold~$M$, we say that 
$W\subset M$ is {\it wandering}  if
\[
\Psi^k(W)\cap W=\emptyset \qquad \text{for all $k\in\Z\setm\{0\}$}
\]
or, equivalently, if $\Psi^k(W)\cap\Psi^\ell(W)=\emptyset$ for all
$k,\ell \in\Z$ with $k\neq\ell$. 
\end{Def}

Notice that if $W' \subset W$ and $W$ is wandering, then~$W'$ is
wandering too.
Beware that, when~$W$ is reduced to a single point~$x$,
saying that the set $W=\{x\}$ is wandering is a less stringent
condition than saying that the point~$x$ is wandering in the usual
sense (which amounts to the existence of a neighborhood~$V$ of~$x$
such that $\Psi^k(V)\cap V=\emptyset$ for $k\in\Z\setm\{0\}$).


\begin{rem}     \label{remPoinRec}
If $\Psi$ preserves a finite measure, then obviously any measurable wandering
set must have zero measure (this is the key argument in the Poincar\'e
recurrence theorem).
%
%
\end{rem}


\paragaD 
We denote the canonical Lebesgue measure on~$\A^n$ by~$\mu$.
Recall that a \emph{domain} of $\A^n$ is a connected open subset of~$\A^n$.

Before going further, we notice that, 
given $h$ of class $C^2$ on an open set~$\Om$ of $\R^n$,
\emph{any measurable wandering set~$W$ of the integrable
  diffeomorphism $\Phi^h \col \T^n \times \Om \righttoleftarrow$ has zero
Lebesgue measure}.
Indeed, formula~\eqref{eqPhih} shows that each
torus $\jT(r)=\T^n\times\{r\}$ is invariant,
with the restriction of $\Phi^h$
%
preserving the Haar measure~$\mu_r$ of $\jT(r)$, which is finite.
Thus Remark~\ref{remPoinRec} implies that the wandering set
$W\cap\jT(r)$ has zero $\mu_r$-measure for each $r\in\Om$
and, by Fubini,
\[
\mu(W)=\int_{\Om}\mu_r(W\cap \jT(r) )\,dr=0.
\]
In particular,  the only wandering domain for $\Phi^h$ is the
empty set.  


\paragaD 
Another preliminary remark concerns the case $n=1$: 
\emph{any measurable wandering set~$W$ of a near-integrable
system of~$\A$ has zero Lebesgue measure}.
More precisely, if $0<R<R_0<\infty$ and 
$h\in G^{\al,L}([-R_0,R_0])$ is Kolmogorov non-degenerate
(\ie its second derivative does not vanish),
then for any exact symplectic diffeomorphism~$\Psi$ of~$\A$
with a restriction to $\A_{R_0} = \T\times[-R_0,R_0]$ such that
$\NormD{\Psi-\Phi^h}_{\al,L,R_0}$ is small enough,
any measurable wandering set contained in~$\A_R$ has zero Lebesgue measure.

Indeed, the KAM theorem yields two invariant circles,
one contained in $\T \times ]R,R_0[$
and the other one in $\T \times ]-R_0,-R[$,
which bound a finite measure invariant region; any measurable
wandering set contained in that region must have zero measure
according to Remark~\ref{remPoinRec}.
(This is the same argument which forbids Arnold diffusion in two
degrees of freedom.)


\paragaD 
We thus assume $n\geq2$ from now on.
The first examples of near-integrable systems possessing
wandering sets of positive Lebesgue measure, namely wandering domains,
were constructed in \cite{ms}.\footnote{%
Notice that it is the exactness of near-integrable systems which makes the existence of such examples
not obvious.
If exactness is relaxed, then one trivially gets arbitrarily close to
integrable symplectic maps with wandering domains by considering
$\Psi_\eps \col (\th,r) \in \A \mapsto (\th+r,r+\eps) \in \A$, with small
$\eps>0$, and $W_\eps \defeq \T \times ]0,\eps[$.
}
Although the construction was quite explicit, no estimate was given for the
``size'' of these wandering domains.

\emph{In this paper, we show that, for a perturbation~$\Psi$ of an
integrable diffeomorphism~$\Phi^h$ with $\eps \defeq
\NormD{\Psi-\Phi^h}_{\al,L,R}$ small, the wandering sets have an 
exponentially small size.
We shall provide general upper bounds and examples with explicit lower bounds.}


\paragaD 
We shall use two natural but essentially different notions of ``size'': the
Lebesgue measure and the Gromov capacity. Recall that the Gromov
capacity (or width, or depth) $\CG(W)$ of a subset~$W$ of a symplectic manifold is the supremum
of the numbers $\pi r^2$, where $r\geq 0$ is such that the Euclidean ball
$B^{2n}(r)$ of radius~$r$ in $\R^{2n}$ can be symplectically embedded in~$W$.
As a consequence, for measurable subsets~$W$ of~$\A^n$, 
%
\begin{equation}\label{eq:meascap}
\CG(W)\leq \pi\Big(\frac{\mu(W)}{{\rm Vol\,}(B^{2n}(1))}\Big)^{1/n}.
\end{equation}
The capacity of a domain in the $2$-dimensional annulus~$\A$ equals
its Lebesgue measure (\ie its area in this case), 
but they are in general distinct for higher dimensional domains. 
As an extreme case, given a disc~$D$ in $\A$, the capacity of $W
\defeq D\times \A^{n-1}\subset \A^n$ is the area of~$D$, while the Lebesgue measure
of~$W$ is infinite. We refer to \cite{McS} for a more complete exposition of the
notion of Gromov capacity.

We are interested in estimates of the size of wandering subsets from above
and from below. In view of inequality~\eqref{eq:meascap}, we may content
ourselves with using the Lebesgue measure for upper estimates and the
Gromov capacity for the lower ones.


\paragaD 
Our upper bound result consists in general exponentially small estimates, with
explicit exponents stemming from Theorem~\ref{th:NekhFevMaps}:
\label{secthmupperbdsB}
\begin{thm}[Upper bounds for wandering sets]	\label{th:upperbounds}
Let $n\geq2$ be integer. Let $\al\ge1$ and $L,R_0>0$ be real.
Let $h\in G^{\al,L}(\ov B_{R_0})$ have positive
definite Hessian matrix on $\ov B_{R_0}$.
%
%
Then for $0 < R < R_0$ there exist $\eps_*,c_*>0$ such that,
for each exact symplectic diffeomorphism~$\Psi$ of~$\A^n$ whose
restriction to $\A^n_{R_0}$ satisfies
\[
\Psi_{|\A^n_{R_0}} \in G^{\al,L}(\A^n_{R_0},\A^n), \quad
\eps \defeq \NormD{\Psi_{|\A^n_{R_0}} - \Phi^h}_{\al,L,R_0} < \eps_*,
\] 
any measurable wandering set~$W$ of~$\Psi$ contained in $\A^n_{R}$ 
has Lebesgue measure
\begin{equation}	\label{inequpper}
\mu(W)\leq \exp \Big( - c_* \Big(\frac{1}{\eps}\Big)^{\frac{1}{2n\al}} \Big).
\end{equation}
\end{thm}

\noindent
The proof is in Section~\ref{secPfThmUpperBds}.
It is a pretty direct consequence of Theorem~\ref{th:NekhFevMaps} and the preservation of the Lebesgue measure
by symplectic maps.
(It works for the case $n=1$ as well but, as already mentioned,
$\mu(W)=0$ in that case.)
Again, the reals~$\eps_*$ and~$c_*$ depend on~$h$ only through
$\normD{h}_{\al,L,R_0}$ and~$m$ such that~$h$ is $m$-convex in the
sense of~\eqref{eqdefmconv}.


\paragaD
Our lower bound result consists in constructing examples which possess
wandering domains whose Gromov capacity is estimated from below by an
exponentially small quantity with explicit exponents:
\label{secthmlowbdsC}
\begin{thm}[Lower bounds in examples of wandering domains]	\label{th:lowerbounds}
Let $n\ge2$ be integer. Let $\al>1$ and $L>0$ be real.
Let $h(r) \defeq \frac{1}{2}(r_1^2+\cdots+r_n^2)$. 
Then there exists a sequence $(\Phi_j)_{j\geq0}$ of exact symplectic
diffeomorphisms of~$\A^n$ such that
\begin{itemize}
\item
each~$\Phi_j$ has a wandering domain~$\jW_j$ contained in~$\A^n_3$,
\item
for $0<R<\infty$ the maps~$\Phi_j$ belong to $G^{\al,L}(\A^n_R,\A^n)$ 
and there exists $c>0$ such that
\begin{equation} 	\label{ineqlower}
\eps_j \defeq \NormD{\Phi_j - \Phi^h}_{\al,L,R}
\xrightarrow[j\to\infty]{}0
\quad \text{and} \quad
\CG(\jW_j)\ge \exp\Big(-c \Big(\frac{1}{\eps_j}\Big)^{\frac{1}{2(n-1)(\al-1)}}\Big)
\end{equation}
for all integers $j$.
\end{itemize}
\end{thm}

\noindent
The proof is in Sections~\ref{sec:proofpseudopend}
and~\ref{SecPfWander};
see Section~\ref{sec:descriptstrctpfC} for a description of the structure of the proof.


Observe that, putting together \eqref{eq:meascap}
and~\eqref{inequpper}, we get
\[
\CG(\jW_j) \le K \, \mu(\jW_j)^{1/n} \le
\exp \Big( - c^* \Big(\frac{1}{\eps_j}\Big)^{\frac{1}{2n\al}} \Big),
\]
with appropriate $K,c^*>0$, for $j$ large enough,
which is compatible with~\eqref{ineqlower} because 
$\frac{1}{2n\al} \le \frac{1}{2(n-1)(\al-1)}$.
Notice also that our examples are constructed only in the non-quasianalytic case $\al>1$.
See Section~\ref{par:optimal} for more comments on the previous inequalities and possible
extensions to the analytic case.


Our method is related to the one developed in \cite{hms} for
estimating the maximal speed of Arnold diffusion orbits and in
\cite{ms} for constructing the first examples of near-integrable systems
with wandering domains.
A common feature of the examples in \cite{ms} and in
Theorem~\ref{th:lowerbounds} is that these wandering domains follow
complicated paths in the phase space, located in the complement of the
set of KAM tori.


\subsection{Specific form of our examples and elliptic islands---Theorem~\ref{th:perdomains}}
\label{sec:descriptstrctpfC}


We now indicate the structure of the proof of
Theorem~\ref{th:lowerbounds} to be found in Sections~\ref{sec:proofpseudopend}
and~\ref{SecPfWander}.


\paragaD
%
%
Given a function $h\in C^\infty(\R^n)$ and
real constants $\al\ge1$ and $L>0$, 
we set
\begin{align}	\label{eq:defpertset}
\Palm & \defeq
\Big\{ \Phi^{u_m}\circ\cdots\circ\Phi^{u_1}\circ\Phi^{h+u_0} \mid
u_0,u_1,\ldots, u_m \in G^{\al,L}(\A^n) \Big\}, \qquad m\ge1 \\[1ex]
\Pal & \defeq \bigcup_{m\ge1} \Palm
\end{align}
(observe that the Hamiltonian functions $h+u_0,u_1,\ldots,u_m$ generate
complete vector fields because their partial derivatives with respect to the angles are bounded;
the notation is well-defined, since giving the diffeomorphism~$\Phi^h$
allows one to compute the gradient $\nabla h$ mod $\Z^n$ and therefore
the diffeomorphism $\Phi^{h+u_0}$ for any smooth~$u_0$).
Any $\Psi \in \Pal$ is an exact symplectic map which can be viewed as
a perturbation of~$\Phi^h$, with a ``deviation'' defined as
\begin{equation}	\label{eq:defdist}
\de^{\al,L}(\Psi,\Phi^h) \defeq
\Inf\Big\{\sum_{k=0}^m\normD{u_k}_{\al,L,\infty}\mid 
m\ge1, \;
(u_0,u_1,\ldots, u_m)\in  \jU_m^{\al,L}(\Psi,\Phi^h)\Big\},
\end{equation}
where 
$\jU_m^{\al,L}(\Psi,\Phi^h) \defeq
\big\{(u_0,u_1,\ldots, u_m)\in\big(G^{\al,L}(\A^n)\big)^{m+1}\mid 
\Phi^{u_m}\circ\cdots\circ\Phi^{u_1}\circ\Phi^{h+u_0}=\Psi\big\}$.
One can check that the deviation vanishes if and only if $\Psi=\Phi^h$.


\paragaD
If $h\in G^{\al,L}(\ov B_{R_0})$, then the elements of $\Pal$ are
Gevrey maps and the deviation can be compared to the distances
$\NormD{\Psi-\Phi^h}_{\al,L_*,R}$. More precisely,
%
%
\begin{prop}	\label{propPsiPhih}
Let $n\ge1$ be integer. Let $\al\ge1$ and $L,R_0>0$ be real.
Let $h\in C^\infty(\R^n) \cap G^{\al,L}(\ov B_{R_0})$.
Then for $0<R<R_0$ there exist $\eps_*,L_*,C_*>0$, with $L_*<L$, such that
\begin{equation}	\label{ineqPsiPhih}
\Psi \in \Pal \ens\text{and}\ens
\de^{\al,L}(\Psi,\Phi^h) < \eps_* 
\ens\Rightarrow\ens
\NormD{\Psi-\Phi^h}_{\al,L_*,R} \le C_*\, \de^{\al,L}(\Psi,\Phi^h).
\end{equation}
\end{prop}
%
%
\noindent
The proof of Proposition~\ref{propPsiPhih} is in Appendix~\ref{subsecpfpropPsiPhih}
($\eps_*,L_*,C_*$ depend on~$h$ only through $\normD{h}_{\al,L,R_0}$).


As a consequence,
Theorems~\ref{th:NekhFevMaps} and~\ref{th:upperbounds} apply to the
maps of $\Pal$ with $\de^{\al,L}(\Psi,\Phi^h)$ small enough,
and the role of~$\eps$ in the statements can be played by
$\de^{\al,L}(\Psi,\Phi^h)$ instead of
$\NormD{\Psi-\Phi^h}_{\al,L,R_0}$.


\paragaD    \label{paragabetterC}
Theorem~\ref{th:lowerbounds} will follow from a more precise statement,
\textbf{Theorem~\ref{th:lowerbounds}'} stated in
Section~\ref{sec:statementthmCp}.
The unperturbed system $h \defeq \dem(r_1^2+\cdots+r_n^2)$ and the
constants $\al>1$, $L>0$ being fixed, 
Theorem~\ref{th:lowerbounds}' will yield very explicit maps
$\Phi_j\in\Pa{2}(\Phi^h)$ when $n=2$, $\Phi_j\in\Pa{3}(\Phi^h)$ when
$n\geq3$, 
with wandering domains $\jW_j \subset \A^n_3$,
and a real $c_*>0$ such that
\begin{equation}	\label{ineqbetterlower}
\eps'_j \defeq \de^{\al,L}(\Phi_j,\Phi^h)
\xrightarrow[j\to\infty]{}0,
\qquad
\CG(\jW_j)\geq \exp(-c_* \Big(\frac{1}{\eps'_j}\Big)^{\frac{1}{2(n-1)(\al-1)}}).
\end{equation}
By Proposition~\ref{propPsiPhih},  
\eqref{ineqbetterlower} implies the property~\eqref{ineqlower} for
every finite~$R$, hence Theorem~\ref{th:lowerbounds} is an immediate consequence
of Theorem~\ref{th:lowerbounds}'.

%

The domains~$\jW_j$ will be polydiscs, \ie product sets of the form 
$\jD\zz1_j \times \cdots \times \jD\zz n_j$ 
with discs\footnote{When we use the word ``disc'', unless otherwise
  specified, we mean any bounded and simply
  connected domain in~$\A$ or in~$\R^2$.}
$\jD\zz1_j, \ldots, \jD\zz n_j \subset \A$.
This product structure is an essential  feature in the use of the
``coupling lemma'' of Section~\ref{Sec:Couplem}, which is a basic ingredient of the proof of
Theorem~\ref{th:lowerbounds}'.

Note that the Gromov capacity of a polydisc is given by the formula
\beq   \label{eqGromovCapProduct}
\CG(\jD\zz 1 \times \cdots \times \jD\zz n) =
\min\big\{
\area(\jD\zz 1), \ldots, \area(\jD\zz n)
\big\}.
\eeq
(One inequality follows from the fact that, in dimension~$2$, Gromov
capacity and area coincide;
the reverse inequality is a consequence of Gromov's ``non-squeezing
theorem''---see \cite{McS}.)


\paragaD 
As another ingredient of the proof of
Theorem~\ref{th:lowerbounds}', we shall have to devise an additional result %
on the construction of examples with periodic domains,
%
%
which is interesting in itself and connected with other aspects of transport
phenomena in near-integrable Hamiltonian systems.

To ease the comparison with Section~\ref{SecPfWander}, we present this result
in~$\A^{n-1}$ (still with $n\ge2$), labelling the coordinates as
$(\th_2,r_2),\ldots,(\th_n,r_n)$, and set
\[
h(r) \defeq \Demi (r_2^2 + \cdots + r_n^2).
\]
For an integer $q\ge1$, we call \emph{$q$-periodic polydisc} of a
diffeomorphism~$\phi$ of~$\A^{n-1}$ a polydisc~$\jD$ of~$\A^{n-1}$
such that $\phi^q(\jD)=\jD$.
We introduce the notation
\beq   \label{eqdefApdBd}
\A^+_d \defeq \T\times [0,d] \subset \A, \quad
	\Bd \defeq \big\{\, \big( \angD{\th}, r \big) \mid \th \in
        [-d,d],\; r \in \R \,\big\} \subset \A
\quad \text{for any real $d>0$.}
\eeq


\label{secthmperdomD}
\begin{thm}[Periodic domains in $\A^{n-1}$]	\label{th:perdomains}
Let $\al>1$ and $L>0$ be real, and let $n\ge 2$ be integer.
Then there exist real numbers $c, C_1, C_2, C_3>0$, a non-negative
integer~$j_0$ and a sequence
$(\Psi_{j,q})$ of exact symplectic diffeomorphisms of $\A^{n-1}$
belonging to $\Pal$
defined for 
\beq
j,q\in\N, \quad 
j \ge j_0, \quad
q\ge C_1 N_j,
\eeq
with deviations
\beq\label{eq:devPsij}
\de^{\al,L}(\Psi_{j,q},\Phi^h) \le \frac{C_2}{N_j^2},
\eeq
where
\beq
N_j \defeq  p_{j+2} \cdots p_{j+n},
\eeq
$(p_j)_{j\ge0}$ denoting the prime number sequence,
so that:
\begin{enumerate}[(i)]
\item
If $n=2$,
each~$\Psi_{j,q}$ is in $\Paa1$ and has a $q$-periodic disc
$\jD := \jD_{j,q} \subset \A_3$ with all its iterates also contained
in~$\A_3$, such that
\begin{equation}
\label{eq:capaDdeux}
\CG(\jD) \ge C_3 \min \bigg\{%
\frac{N_j^2}{q^5}, \,
\exp\Big( -c N_j^{\frac{1}{\al-1}} \Big)
\bigg\}
\end{equation}
and
\begin{equation}
\label{eq:disjointdeux}
\jD \subset \A^+_{\frac{4}{N_j}} \cap  \jB_{\frac{1}{2p_{j+2}}}, \qquad
\Psi_{j,q}^k(\jD) \cap \jB_{\frac{1}{p_{j+2}}} = \varnothing 
\quad \text{for $1\le k \le q-1$}.
\end{equation}
\item
If $n\ge3$,
each~$\Psi_{j,q}$ is in $\Paa2$ and,
for~$q$ integer multiple of~$N_j$, 
$\Psi_{j,q}$ has a $q$-periodic polydisc $\jD := \jD_{j,q}$ whose iterates are polydiscs:
\[
 \Psi_{j,q}^k(\jD) = \jD\zz{2,k} \times \cdots \times \jD\zz{n,k},
\qquad k\in \Z,
\]
with $\Psi_{j,q}^k(\jD) \subset \A_3^{n-1}$ for all~$k$,
such that
\begin{equation}
\label{eq:capaD}
\CG(\jD) \ge C_3 \min \bigg\{%
\frac{1}{q^5} N_j^{4-\frac{2}{n-1}}, \,
\exp\Big( -c N_j^{\frac{1}{(n-1)(\al-1)}} \Big)
\bigg\},
\end{equation}
the projections of the polydisc~$\jD$ satisfy 
\begin{equation}
\label{eq:disjoint1}
\jD\zz{2,0} \subset \A^+_{\frac{4}{N_j}} \cap  \jB_{\frac{1}{2p_{j+2}}}, \quad
\jD\zz{3,0} \subset \jB_{\frac{1}{2p_{j+3}}}, \ldots, \quad
\jD\zz{n,0} \subset \jB_{\frac{1}{2p_{j+n}}}, 
\end{equation}
and, for $1\le k \le q-1$, those of the polydisc $\Psi_{j,q}^k(\jD)$ satisfy
\begin{equation}
\label{eq:disjoint2}
\text{$\exists \ell\in\{2,\ldots,n\}$ such that}\ens
\jD^{[\ell,k]} \cap \jB_{\frac{1}{p_{j+\ell}}} = \varnothing. 
\end{equation}
\end{enumerate}
\end{thm}
\noindent
The proof of Theorem~\ref{th:perdomains} is spread over
Sections~\ref{sec:proofpseudopend} and~\ref{SecPfWander}.
More precisely, Case~(i), \ie the two-dimensional case, is proved in
Section~\ref{secPfThperi}, based on an auxiliary result;
this auxiliary result is also used in Section~\ref{ssec:perdomains},
together with the ``coupling lemma'' (Lemma~\ref{Lem:cl}), to prove
Case~(ii).

Theorem~\ref{th:perdomains} is used in Section~\ref{ssec:proofn}
(again with the help of the coupling lemma) to prove
Theorem~\ref{th:lowerbounds}',
with an appropriate choice of $q=q_j$ exponentially large with respect to~$N_j$.

\begin{rem}
Fix $j\ge j_0$ and~$q$ as in the statement of
Theorem~\ref{th:perdomains}. 
Because of condition~\eqref{eq:disjointdeux} or
conditions~\eqref{eq:disjoint1}--\eqref{eq:disjoint2}, the sets
$\Psi_{j,q}^k(\jD_{j,q})$, $k=0,1,\ldots,q-1$, are pairwise disjoint.
This implies that~$q$ is the minimal period of the periodic
polydisc~$\jD_{j,q}$.
This also implies an upper bound for the Lebesgue measure of this
polydisc:
\[
\mu(\jD_{j,q}) \le \frac{\mu(\A_3^{n-1})}{q}.
\]
Indeed, the $q$ pairwise disjoint sets $\Psi_{j,q}^k(\jD_{j,q})$ have the
same Lebesgue measure and are all contained in~$\A_3^{n-1}$.
It follows that the lower bound in~\eqref{eq:capaDdeux} or~\eqref{eq:capaD} has to depend
on~$q$, it cannot depend on~$j$ alone, because $q$ is allowed to be
arbitrarily large and~\eqref{eq:meascap} implies
\[
\CG(\jD_{j,q}) \le \frac{\pi}{q^{\frac{1}{n-1}}} 
\bigg(\frac{\mu(\A_3^{n-1})}{{\rm Vol\,}(B^{2(n-1)}(1))}\bigg)^{\frac{1}{n-1}}.
\]
\end{rem}

\paragaD 
The aforementioned auxiliary result on which the proof of
Theorem~\ref{th:perdomains}(i) is based is \textbf{Theorem~\ref{th:varpseudopend}};
this much more precise statement is the object of Section~\ref{sec:proofpseudopend},
it is the analytical core of our
method. 

The (quite lengthy) proof of Theorem~\ref{th:varpseudopend}
relies on the construction of a suitable perturbation of the time-one
map of a ``pseudo-pendulum'' on~$\A$, of the form
\beq   \label{eqdefpseudoPendP}
P(\th,r)=\dem r^2+\frac{1}{N_j^2} V(\th),
\eeq
where~$V$ is a (specially designed) potential function on $\T$. 
Both~$V$ and the perturbation can be made very explicit.
The effect of the perturbation is to create
elliptic islands around the periodic points located near the separatrix of the
pseudo-pendulum.  The main difficulty in estimating the size of these islands is that
one has to use Herman's quantitative version of the two dimensional KAM 
theorem~(\cite{herman2}), whose implementation requires
the computation of high order parametrized normal forms, the parameters
being the size of the perturbation and the period of the island.%
\footnote{We insist on being able to take the period of the elliptic island arbitrarily
  large.
If this requirement were dropped, a much simpler construction
would be available---see the auxiliary Proposition~\ref{prop:perellipse}.}

Another peculiarity of our systems is that the potential $V$ has {\em degenerate} 
maxima, which create degenerate stationary points for the Hamiltonian
vector field generated by~\eqref{eqdefpseudoPendP}. This is crucial in order to find elliptic islands 
with ``exponentially small'' area: a nondegerate situation would yield a double
exponential in the estimates.


\subsection{Further comments}


\paragaD  
Observe that in Theorem~\ref{th:upperbounds}, we impose a priori that the
wandering set~$W$ be contained in a fixed compact $\A^n_{R}$.

Suprisingly enough, as soon as $n\geq 3$, this is necessary to ensure that the measure
of~$W$ is finite.
Indeed, given $\al>1$ and $L>0$, for any $\eps,R_0>0$ we can exhibit
(by \cite{ms} or by Theorem~\ref{th:lowerbounds}) a near-integrable
system~$\Psi$ on~$\A^2$ with a non-empty wandering domain~$W$, such that
$\NormD{\Psi - \Phi^{\dem(r_1^2+r_2^2)}}_{\al,L,R_0}<\eps$.
Therefore, when $n\geq 3$, the direct product $\ha\Psi=\Psi\times
\Phi^{\dem(r_3^2+\cdots+r_n^2)}$ on $\A^n$ admits the wandering domain
$W\times\A^{n-2}$, which is of infinite measure, while
$\NormD{\ha\Psi - \Phi^{\dem(r_1^2+\cdots+r_n^2)}}<\eps$.
As a consequence, by taking subsets of $W\times\A^{n-2}$, \emph{one may obtain 
for the near-integrable system~$\ha\Psi$
wandering domains of arbitrary measure between~$0$ and~$\infty$
inclusive}.



\paragaD  
In any case, this leaves open the question of the existence of upper
bounds for the Gromov capacity of an arbitrary wandering set~$W$
(without the restriction $W \subset \A^n_R$):
is it always finite? is it exponentially small?

Notice that a wandering set has empty intersection with the set of KAM
tori, so a related question is the question of the finiteness of
$\CG(\T^n\times(\R^n\setminus\jK)\big)$, where $\jK$ is the set of all
vectors satisfying a fixed Diophantine condition. Due to the intricate
structure of this set, it could be worthwile to produce a simpler model
for this line of questions. For instance,  what can be said on the finiteness
of {\em any} symplectic
capacity of the open subset
\[
\T^n\times(\R^n\setminus\Z^n) \subset T^*\T^n?
\]
This question seems to be completely open.


\paragaD\label{par:optimal}
Another open question is that of the optimal exponents that one could obtain in
inequalities such as~\eqref{inequpper} and~\eqref{ineqlower}:
to sharpen Theorem~\ref{th:upperbounds} would mean to replace
the exponent $\frac{1}{2n\al}$ by a larger exponent~$a_{\text{up}}$
in~\eqref{inequpper}, 
and to sharpen Theorem~\ref{th:lowerbounds} would mean to replace
the exponent $\frac{1}{2(n-1)(\al-1)}$ by a smaller exponent~$a_{\text{low}}$
in~\eqref{ineqlower};
how large can one take the first exponent and how small can one take the second?
Of course, one would still have $a_{\text{up}} \le a_{\text{low}}$; if the
equality $a_{\text{up}} = a_{\text{low}}$ could be realised, the resulting
exponent should certainly be called ``optimal''.

The problem is clearly related to the possibility of contructing
examples in the analytic category $\al=1$, since the factor $\al-1$
(whose appearance is directly linked to our use of Gevrey bump functions)
creates a major discrepancy between our lower and upper bounds when
$\al\to 1$.  We believe that such constructions are possible, at the
cost of relaxing the constraint that our wandering subsets be {\em
  domains}.
%


\vfil

\pagebreak

\section{Stability theory for Gevrey near-integrable maps}
\label{secGevNekhMaps}


We develop in this section a perturbation theory for Gevrey
\emph{discrete} dynamical systems, based on the corresponding theory available
for Gevrey Hamiltonian \emph{flows}. 
To transfer the results from the latter to the former, we first prove
a Gevrey suspension theorem (Theorem~\ref{thmsusp}), according to
which any Gevrey near-integrable map can be viewed as the time-one
map of a Gevrey near-integrable Hamiltonian vector field.
%
%
This will allow us to prove in Section~\ref{secPfThmNekhMaps} the Nekhoroshev Theorem for Gevrey maps
(Theorem~\ref{th:NekhFevMaps}),
from which we will derive upper bounds for the measure of their wandering sets
(Theorem~\ref{th:upperbounds}) in Section~\ref{secPfThmUpperBds}.

\subsection{Embedding in a Hamiltonian flow---Theorem~\ref{thmsusp}}	\label{subsecsusp}


\begin{Def}
Given an exact symplectic map $\Psi \col \A^n_R \to \T^n\times\R^n$, we call
suspension of~$\Psi$ any $1$-periodic time-dependent Hamiltonian function
$H \col \Om\times\T \to \R$, where~$\Om$ is a neighbourhood of~$\A^n_R$,
for which the flow map between the times $t=0$ and $t=1$ is well-defined on
$\A^n_R$ and coincides with~$\Psi$.
\end{Def}


We adapt the definitions~\eqref{eq:defGalLR} and~\eqref{eq:defGalL} to
deal with $C^\infty$ functions depending on an extra variable $t\in\T$
or $t\in[0,1]$:
\begin{gather}
\label{eqdefGalLT}
G^{\al,L}(\T) \defeq \{ \eta\in C^\infty(\T) \mid \normD{\eta}_{\al,L} <\infty \},
\quad
\normD{\eta}_{\al,L} \defeq \sum_{\ell\in\N} \frac{L^{\ell\al}}{\ell !^\al} \normD{\pa^\ell\eta}_{C^0(\T)}
\\[1ex]
\begin{multlined}[t][.88\displaywidth]
G^{\al,L}(\A^n_R\times\T) \defeq \{ f\in C^\infty(\A^n_R\times\T) \mid
\normD{f}_{\al,L,R} <\infty \}, 
\\[-.5ex]
\normD{f}_{\al,L,R} \defeq \sum_{\ell\in\N^{2n+1}}
\frac{L^{\absD{\ell}\al}}{\ell !^\al} \normD{\pa^\ell f}_{C^0(\A^n_R)}
\end{multlined}
\end{gather}
and similarly for $G^{\al,L}([0,1])$ and  $G^{\al,L}(\A^n_R\times[0,1])$.


\begin{thm}[Suspension theorem]	\label{thmsusp}
Let $n$ be a positive integer. Let $\al\ge1$, $L_0,R,R_0,E>0$ be reals such that
$R<R_0$.
Then there exist $\eps_*,L_*,C_*>0$ such that,
for every $h\in G^{\al,L_0}(\ov B_{R_0})$ with
$\normD{h}_{\al,L_0,R_0} \le E$,
the restriction to~$\A^n_R$ of any exact symplectic map
$\Psi \in G^{\al,L_0}(\A^n_{R_0}, \A^n)$ such that
\[
\eps \defeq \NormD{\Psi-\Phi^h}_{\al,L_0,R_0} \leq \eps_*
\] 
admits a suspension 
$H = H(\th,r,t) \in G^{\al,L_*}(\A_R^n\times\T)$
for which
\begin{equation}	\label{ineqnormHh}
\normD{H-h}_{\al,L_*,R} \le C_* \eps.
\end{equation}
\end{thm}


\begin{rem}
In view of Proposition~\ref{propPsiPhih}, 
Theorem~\ref{thmsusp} applies to the maps of
$\Pal$ with $\de \defeq \de^{\al,L}(\Psi,\Phi^h)$ small enough, 
and the role of~$\eps$ in the statement can be played by~$\de$ instead of
$\NormD{\Psi-\Phi^h}_{\al,L,R_0}$.

In fact, the resulting statement can be proved directly if one restricts oneself
to $\Psi \in \Palm$ with a fixed~$m$ (upon which the implied constants may depend)
and $\al>1$, by adapting the ideas of \cite[\S\,2.4.1]{hms} and \cite[\S\,5.2]{ms}.
Indeed, use the hypothesis $\al>1$ to find non-negative functions
$\ph_0,\ph_1,\ldots,\ph_m \in G^{\al,L}(\T)$ such that each~$\ph_j$
has total mass~$1$ and is supported on
$[\frac{j}{m+1},\frac{j+1}{m+1}]$ mod~$\Z$
(use \eg Lemma~A.3 of \cite{hms}),
and set $\ti\ph_0(t) \defeq \int_0^t\big( \ph_0(s)-1\big)\,\dd s$.
Then, for any $u_0,u_1,\ldots, u_m \in G^{\al,L}(\A^n)$, the map
$\Psi = \Phi^{u_m}\circ\cdots\circ\Phi^{u_1}\circ\Phi^{h+u_0}$
admits an explicit suspension given by
\[
H(\th,r,t) \defeq h(r) +
\ph_0(t) u_0\big( \th + \ti\ph_0(t)\na h(r), r \big) +
\sum_{j=1}^m \ph_j(t) u_j\big( \th + (1-t)\na h(r), r \big),
\]
and one can find $\la\in(0,1)$ and $C>0$ independent of $u_0,\ldots,u_m$ such that
\[
\normD{H-h}_{\al,\la L,R} \le C \big( \normD{u_0}_{\al,L} + \cdots +
\normD{u_m}_{\al,L} \big).
\]
\end{rem}


We now briefly indicate how to prove Theorem~\ref{thmsusp} in the
analytic case, \ie when $\al=1$; the case $\al>1$ is dealt with in Section~\ref{secsuspGev}.


\begin{proof}[Proof of Theorem~\ref{thmsusp} in the case $\al=1$]
This is due to Kuksin \cite{Kuksin} and Kuksin-P\"oschel \cite{KP}.
There is only a slight difference in the way norms are measured, but this is
immaterial:
for a real analytic function $\ph \col \A^n_R \to \R$, \cite{KP} defines
$\absD{\ph}_\rho$ as the sup-norm of the holomorphic extension of~$\ph$ to a
complex domain $V_\rho\T^n \times V_\rho \ov B_R$ defined as
in~\eqref{eqdefjV} but with $\abs{\,\cdot\,}_\infty$ replaced by
$\norm{\,\cdot\,}$, with
$\norm{\xi} \defeq \sqrt{ \abs{\xi_1}^2 + \cdots + \abs{\xi_n}^2 }$
for $\xi \in \R^n$ or~$\C^n$;
this is related to our Gevrey-$1$ norms by
\[
c \normD{\ph}_{1,L,R} \le \absD{\ph}_\rho \le \normD{\ph}_{1,\rho,R}
\]
for $0 < L < \rho/\sqrt{n}$, with $c\defeq (1-L\rho\ii\sqrt n)^{2n}$.
With this in mind, when $\al=1$, our Theorem~\ref{thmsusp} follows from
Theorem~4 of \cite{KP} by isoenergetic reduction, with the help of the Implicit
Function Theorem (the same way their Theorem~1 follows from their Theorem~3).
\end{proof}

\subsection{Proof of Theorem~\ref{thmsusp} in the Gevrey non-analytic case}
\label{secsuspGev}

For the Gevrey non-analytic case, the proof will consist in a Gevrey
quantitative adaptation of Douady's method \cite{Douady}.

In all this section we fix a positive integer~$n$ and a real $\al>1$.
When dealing with a map~$\Psi$ taking its values in $\T^n\times\R^n$
or $\R^n\times\R^n$, we shall often denote its components by
$\Psi_1,\ldots,\Psi_{2n}$ and use the notation
\beq \label{notacompPsiundeux}
\Psi\zz1 \defeq (\Psi_1,\ldots,\Psi_n), \quad
\Psi\zz2 \defeq (\Psi_{n+1},\ldots,\Psi_{2n}).
\eeq
Similarly, we shall make use of the partial gradient
operators~$\na\zz1$ and~$\na\zz2$ defined by~\eqref{eqdefnazz}.

\setcounter{subsubsection}{-1}
\subsubsection{Overview}

The construction is based on the classical formalism of
generating functions for exact symplectic $C^\infty$ maps, with mixed
set of variables:
we use the notation~$\jF_A$ whenever we have a $C^\infty$ function~$A$
defined on an open subset of~$\A^n$ such that the equation
\[
r = \th + \na\zz1 A(\th,r')
\]
implicitly defines $r' \in \R^n$ in terms of $\th\in\T^n$
and~$r\in\R^n$, so that we can set
\[
\jF_A(\th,r) \defeq (\th',r'), \qquad
\th' \defeq \th + \na\zz2 A(\th,r').
\]
When it is defined, the map~$\jF_A$ is automatically an exact
symplectic local diffeomorphism;
moreover, all exact symplectic $C^\infty$ maps close enough to
identity are of this form.
The reader is referred to Appendix~\ref{secGenfcns} for more
details.\footnote{
Up to sign, the function~$A$ corresponds to what is called
``generating function of type~$V$'' in \cite{McS} \S9.2.
}

Here is an overview of the construction of a suspension for a given
exact symplectic Gevery map~$\Psi$ close enough to~$\Phi^h$:
following \cite{Douady}, we write our map as
\beq   \label{eqPsiPhihjFA}
\Psi = \Phi^h \circ \jF_A,
\eeq
while we pick $\eta \in C^\infty([0,1])$ such that $\eta \equiv 0$ on a neighbourhood
of~$0$, $\eta \equiv 1$ on a neighbourhood of~$1$, and $0 \le \eta \le
1$ on $[0,1]$;
then the formula
\[
\Psi_t \defeq \Phi^{th} \circ \jF_{\eta(t)A}
\]
defines an isotopy between the identity and~$\Psi$, which can be shown to
be the flow map between time~$0$ and time~$t$ for a time-periodic Hamiltonian vector
field~$H$ which is close to~$h$.

We will repeat the arguments in detail to check that one can find a
small Gevrey function~$A$ such that~\eqref{eqPsiPhihjFA} holds and
that, provided we take a Gevrey function for~$\eta$ (which is possible
because $\al>1$), we can find a suspension~$H$ Gevrey close to~$h$.
The last point will follow from the very explicit formula that we
shall obtain for~$H$: 
with the notation~\eqref{notacompPsiundeux},
\[
H(\th,r,t) = h(r) + 
\eta'(t) A\big( (\jF_{\eta(t)A}\ii)\zz1(\th-t\na h(r),r), r \big)
\]
(formula~\eqref{eqexplicitsusp} below).

\subsubsection{First step: finding a generating function}

\bprop    \label{propfindAfromPsi}
Let $L_0,R,R_0,E>0$ be reals such that $R<R_0$.
Then there exist $\eps_*,L,C_*>0$ such that, 
for any $h\in G^{\al,L_0}(\ov B_{R_0})$ such that
$\normD{h}_{\al,L_0,R_0}\le E$
and any exact symplectic map $\Psi \in G^{\al,L_0}(\A^n_{R_0}, \A^n)$ such that
\beq    \label{eqepsPsiphihepsst}
\eps \defeq \NormD{\Psi-\Phi^h}_{\al,L_0,R_0} \leq \eps_*,
\eeq
there exist open subsets~$\Om$ and~$\Om'$ of~$\A^n_{R_0}$ which
contain~$\A^n_R$ and a function $A \in C^\infty(\Om')$ such
that 
\begin{itemize}
\item $\jF_A \col \Om \to \A^n_{R_0}$ is a well-defined exact
  symplectic map,
\item $\Psi_{|\Om} = \Phi^h \circ \jF_A$,
\item $A_{| \A_R^n} \in G^{\al,L}(\A_R^n)$ and
$\normD{A_{| \A_R^n}}_{\al,L,R} \le C_*\, \eps$.
\end{itemize}
\eprop


The proof of Proposition~\ref{propfindAfromPsi} relies on two
auxiliary results. 
The first one is a straightforward Gevrey adaptation in~$\A^n_R$ of
the Poincar\'e lemma, 
the second one is a technical inversion result that will be needed in the
second step too and whose proof is given in Appendix~\ref{apppfTechnicLem}.


\blm    \label{lemPoincGev}
Let $R,L>0$ and $\bet_1,\ldots,\bet_{2n} \in G^{\al,L}(\A^n_R)$.
We denote the variables in~$\A^n_R$ by 
$(\th,r) = (x_1,\ldots,x_{2n})$
and assume that 
\begin{itemize}
\item $\pa_{x_i}\bet_j = \pa_{x_j}\bet_i$ for $i,j=1,\ldots,2n$,
\item for each $r\in \ov B_R$ and $i = 1,\ldots,n$, the function 
$\bet_i(\,\cdot\,,r)$ has mean value zero on~$\T^n$.
\end{itemize}
Then there exists $A \in G^{\al,L}(\A^n_R)$ such that 
\[
\sum_{i=1}^{2n} \bet_i \,\dd x_i = \dd A
\quad\text{and }\quad
\normD{A}_{\al,L,R} \le C \big(
\normD{\bet_1}_{\al,L,R} + \cdots + \normD{\bet_{2n}}_{\al,L,R}
\big),
\]
where $C \defeq \max\big\{ \dem, R, L^\al \big\}$.
\elm


\blm     \label{sublemminvers}
Let $R,R_0,L_0>0$ be reals such that $R<R_0$, and let $\eta \in
G^{\al,L_0}([0,1])$ be a non-trivial function.
Then there exist $\eps_*,L>0$ such that, for any
$\psi = (\psi_1,\ldots,\psi_n) \in G^{\al,L_0}(\A^n_{R_0},\R^n)$ satisfying
\[
\eps \defeq \sum_{i=1}^n \normD{\psi_i}_{\al,L_0,R_0} \le \eps_*
\]
and for any $t\in[0,1]$, the map
\beq   \label{eqthemapetapsi}
(\th,r) \in \A^n_{R_0} \mapsto 
(\th,r') = \big( \th, r + \eta(t) \psi(\th,r) \big) \in \A^n
\eeq
induces a $C^\infty$ diffeomorphism 
from $\T^n\times B_{R_0}$ onto an open subset~$\Om_t$ of~$\A^n$ 
which contains~$\A^n_{R}$, 
with an inverse map of the form
\beq     \label{eqinversetapsi}
(\th,r') \in \Om_t \mapsto (\th,r) = \big( \th, r' + \chi(\th,r',t) \big)
\in \T^n\times B_{R_0},
\eeq
where $\chi = (\chi_1,\ldots,\chi_n)$ is $C^\infty$ and restricts to
$\chi_{| \A^n_R\times[0,1]} \in G^{\al,L}(\A^n_R\times[0,1],\R^n)$ with
\beq    \label{ineqnormchiiepseta}
\sum_{i=1}^n \normD{\chi_i}_{\al,L,R} \le \eps \normD{\eta}_{\al,L_0}.
\eeq
For~$\eps_*$ and~$L$, one can take the values indicated
in~\eqref{eqdefepsst} and~\eqref{eqdefeeLmun}.
\elm


\begin{proof}[Proof of Lemma~\ref{lemPoincGev}]
The function
\[
\ti A(x) \defeq \int_0^1 \sum_{i=1}^{2n} x_i \bet_i(t x) \,\dd t
\]
is well defined on $\R^n \times \ov B_R$. An easy computation yields 
$\pa_{x_i} \ti A = \bet_i$ for $i=1,\ldots,2n$.
In particular, for each $r\in\ov B_R$, the functions $\pa_{\th_1}\ti A(\,\cdot\,,r), \ldots,
\pa_{\th_n}\ti A(\,\cdot\,,r)$ are $\Z^n$-periodic and have mean value
zero,
whence it follows that~$\ti A(\,\cdot\,,r)$ is itself $\Z^n$-periodic.
Thus~$\ti A$ induces a function $A\in C^\infty(\A^n_R)$,
and the differential of~$A$ is $\bet_1\,\dd x_1 + \cdots +
\bet_{2n}\,\dd x_{2n}$.

Choosing $\big[\!-\dem,\dem\,\big)^n\times\ov B_R$ as a fundamental domain in
$\R^n\times\ov B_R$, we get
$\normD{A}_{C^0(\A^n_R)} \le \max\big\{ \dem, R \big\}
\big(
\normD{\bet_1}_{C^0(\A^n_R)} + \cdots + \normD{\bet_{2n}}_{C^0(\A^n_R)}
\big)$.
Any $\ell \in \N^{2n}$ such that $\absD{\ell} \ge1$ can be written
(usually in more than one way) as $\ell = m + \be_i$ with
$m\in\N^{2n}$ and $i\in\{1,\ldots,2n\}$,
moreover $\pa^\ell A = \pa^m\bet_i$ and $(m+\be_i)! \ge m!$,
hence 
\[
\sum_{\absD{\ell}\ge1} \frac{L^{\absD{\ell}\al}}{\ell!^\al}
\normD{\pa^\ell A}_{C^0(\A^n_R)} \le
\sum_{i=1}^{2n} \sum_{m\in\N^{2n}}
\frac{L^{(1+\absD{m})\al}}{(m+\be_i)!^\al} \normD{\pa^m\bet_i}_{C^0(\A^n_R)}
\le L^\al \sum_{i=1}^{2n} \normD{\bet_i}_{\al,L,R},
\]
which completes the proof.
\end{proof}


\begin{proof}[Proof of Lemma~\ref{sublemminvers}]
See Appendix~\ref{apppfTechnicLem}.
\end{proof}


\begin{proof}[Proof of Proposition~\ref{propfindAfromPsi}]
Given $L_0,L,R>0$ such that $R<R_0$, we set $R' \defeq
\frac{R+R_0}{2}$ and
\beq    \label{eqdefepsstRpL}
\eps_* \defeq \min\Big\{
\frac{R_0-R'}{2}, \frac{L_0^\al}{2^{\al+1} (2n+1)^{\al-1}}
\Big\}, \quad
L \defeq \frac{L_0}{(2^{\al+1}(2n+1)^{\al-1})^{1/\al}}.
\eeq
Let $h\in G^{\al,L_0}(\ov B_{R_0})$ and let $\Psi \in
G^{\al,L_0}(\A^n_{R_0}, \A^n)$ be exact symplectic and satisfy~\eqref{eqepsPsiphihepsst}.
Let us choose a lift $\xi\in C^\infty(\A^n_{R_0},\R^n\times\R^n)$ of $\Psi-\Phi^h$ so that
$\normD{\xi_1}_{\al,L_0,R_0} + \cdots + \normD{\xi_{2n}}_{\al,L_0,R_0}
= \eps$.
Since $\Phi^h(\th,r) = (\th+\angD{\na h(r)},r)$, we have
\[
\Psi\zz1(\th,r) = \th + \angD{ \na h(r) + \xi\zz1(\th,r) },
\quad \Psi\zz2(\th,r) = r + \xi\zz2(\th,r).
\]
We apply Lemma~\ref{sublemminvers} with $\eta\equiv1$ and $\psi = \xi\zz2$:
in view of~\eqref{eqdefepsst} and~\eqref{eqdefeeLmun}, our
choice~\eqref{eqdefepsstRpL} of~$\eps_*$ and~$L$ implies the existence
of an open subset~$\Om_1$ of~$\A^n$ containing~$\A^n_{R'}$ such that
\beq   \label{eqdefdiffepPsizzd}
(\th,r) \in \T^n\times B_{R_0} \mapsto 
(\th,r') = \big( \th, \Psi\zz2(\th,r) \big) \in \Om_1
\eeq
is a $C^\infty$ diffeomorphism, the inverse of which can be written
\[
\Phi \col (\th,r') \in \Om_1 \mapsto 
(\th,r) = \big( \th, r'+\chi(\th,r') \big) \in \T^n\times B_{R_0},
\]
with $\normD{\chi_1}_{\al,L,R'} + \cdots + \normD{\chi_n}_{\al,L,R'} \le \eps$.
We set $\Om' \defeq \Om_1 \cap (\T^n\times B_{R_0})$ and 
$\Om \defeq \Phi(\Om') \subset \T^n \times B_{R_0}$.
Notice that $\A^n_R \subset \A^n_{R'} \subset \Om'$
and $\A^n_R \subset \Om$
(because $\normD{\xi\zz2(\th,r)} \le R'-R$ for all $(\th,r) \in
\A^n_{R_0}$, thus $\Phi\ii(\A^n_R) \subset \A^n_{R'}$).


We now consider 
\[
F(\th,r) \defeq \Phi^{-h}\circ \Psi(\th,r) = \big(
\Psi\zz1(\th,r) - \angD{\na h\circ\Psi\zz2(\th,r)}, \Psi\zz2(\th,r) \big)
\]
for $(\th,r) \in \Om$ (which is possible since $\Psi\zz2(\Om) \subset B_{R_0}$).
This is an exact symplectic $C^\infty$ local diffeomorphism, which can
be written
\[
F(\th,r) = \big( \th + \angD{f(\th,r)}, \Psi\zz2(\th,r) \big), 
\qquad
f \defeq \xi\zz1 + \na h - \na h\circ\Psi\zz2,
\]
and the map~\eqref{eqdefdiffepPsizzd} induces a $C^\infty$
diffeomorphism from~$\Om$ onto~$\Om'$;
therefore, following the recepee of Lemma~\ref{lemfindAfromF}, we know
that the $1$-form
\[
\bet \defeq \sum_{i=1}^n \chi_i(\th,r')\,\dd\th_i +
\sum_{i=1}^n f_i\circ\Phi(\th,r')\,\dd r'_i
\]
is exact and $F = \jF_A$ on~$\Om$, where $A \in C^\infty(\Om')$ is any
primitive of~$\bet$.


We conclude by checking that we can apply Lemma~\ref{lemPoincGev} and
get a primitive $A \in G^{\al,L}(\A^n_R)$ whose norm we can bound.
On the one hand, we have $\chi_i \in G^{\al,L}(\A^n_R)$ for each~$i$ and
$\normD{\chi_1}_{\al,L,R} + \cdots + \normD{\chi_n}_{\al,L,R} \le \eps$.
On the other hand, since $\Psi\zz2\circ\Phi(\th,r') = r'$, we can
write
\begin{multline*}
f_i \circ \Phi(\th,r') = \xi_i\circ\Phi(\th,r') + g_i(\th,r'),\\
\qquad g_i(\th,r') \defeq \pa_i h\circ\Phi\zz2(\th,r')-\pa_i h(r')
= \sum_{j=1}^n \int_0^1 \pa_i\pa_j h(r' + s\chi(\th,r') \big)
\chi_j(\th,r')\,\dd s.
\end{multline*}
Let $L_1 \defeq L_0/2$.
We can apply Proposition~A.1 of \cite{hms} to the composition with 
$\Phi(\th,r') = \big(\th, r' + \chi(\th,r') \big)$
or, more generally, with $U_s(\th,r') \defeq \big( \th, r' +
s\chi(\th,r') \big)$ for $s \in [0,1]$, 
because
\[
\sum_{\ell\in\N^{2n},\,\ell\neq0} \frac{L^{\absD{\ell}\al}}{\ell!^\al}
\normD{\pa^\ell U_{s,i}}_{C^0(\A^n_R)} \le
\frac{L_1^\al}{(2n)^{\al-1}}, 
\qquad i = 1,\ldots,2n
\]
(indeed: this follows from $L^\al + \normD{\chi_i}_{\al,L,R} \le \frac{L_0^\al}{2^\al(2n)^{\al-1}}$), 
and we get
\[ \normD{\xi_i\circ\Phi}_{\al,L,R} \le \normD{\xi_i}_{\al,L_1,R_0} \] and
$\normD{\pa_i\pa_j h\circ U_s}_{\al,L,R} \le 
\normD{\pa_i\pa_j h}_{\al,L_1,R_0}$,
whence \[ \normD{g_i}_{\al,L,R} \le \sum_j 
\normD{\pa_i\pa_j h}_{\al,L_1,R_0} \normD{\chi_j}_{\al,L,R} \]
by the algebra norm property.
Thus Lemma~\ref{lemPoincGev} gives us $A\in G^{\al,L}(\A^n_R)$ with 
\[
\normD{A}_{\al,L,R} \le C \bigg( 
\sum_i \normD{\xi_i}_{\al,L_1,R_0} + \sum_{i,j} \normD{\pa_i\pa_j h}_{\al,L_1,R_0} \normD{\chi_j}_{\al,L,R}
\bigg)
\le C \bigg( 1 + \sum_{i,j} \normD{\pa_i\pa_j h}_{\al,L_1,R_0} \bigg) \eps,
\]
and, using~\eqref{ineqGevCauch}, we get the desired estimate with 
$C_* \defeq C \big(1+\frac{2^{3\al}}{L_0^{2\al}}\norm{h}_{\al,L_0,R_0} \big)$.
\end{proof}

\subsubsection{Second step: constructing a Hamiltonian isotopy}

\bprop    \label{propHamIsot}
Let $L_0,R,R_0>0$ be reals such that $R<R_0$.
Let $\eta\in G^{\al,L_0}([0,1])$.
Then there exist $\eps_*,L,C>0$ satisfying the following:
for any $A \in G^{\al,L_0}(\A^n_{R_0})$ such that
$\normD{A}_{\al,L_0,R_0} \le \eps_*$
and for any $t\in[0,1]$,
there exists an open subset~$\Om_t$ of~$\A^n$ containing~$\A^n_R$
such that
\[
\jF_{\eta(t)A} \col \Om_t \to \T^n\times B_{R_0}
\]
is a well-defined exact symplectic $C^\infty$ diffeomorphism, 
and for each $(\th,r) \in \A^n_R$,
\beq   \label{eqEDOXf}
\frac{\dd\,}{\dd t} \jF_{\eta(t)A}(\th,r) = X_f \big( \jF_{\eta(t)A}(\th,r), t \big),
\qquad t \in [0,1],
\eeq
where $X_f$ is the non-autonomous Hamiltonian vector field associated
with
\beq   \label{eqdefHsuspetatA}
f(\th,r,t) \defeq \eta'(t) A\big( (\jF_{\eta(t)A}\ii)\zz1(\th,r), r \big),
\qquad (\th,r,t)\in \T^n\times B_{R_0} \times [0,1],
\eeq
which is a $C^\infty$ Hamiltonian function whose restriction to $\A^n_R\times[0,1]$ is Gevrey-$(\al,L)$, with
\beq    \label{ineqnormHalLRA}
\normD{f}_{\al,L,R} \le 2^\al L_0^{-\al}\normD{\eta}_{\al,L_0} \normD{A}_{\al,L_0,R_0}.
\eeq
\eprop


\begin{proof}
Let $L_1 \defeq L_0/2$ and 
\[
\eps_* \defeq \frac{(L_0-L_1)^\al}{\normD{\eta}_{\al,L_0}} \min\Big\{
\sqrt{n}, \frac{R_0-R}{2}, \frac{L_1^\al}{2^{\al+1} (2n+1)^{\al-1}}
\Big\}, \qquad
L \defeq \frac{L_1}{(2^{\al+1}(2n+1)^{\al-1})^{1/\al}}.
\]
Let $A \in G^{\al,L_0}(\A^n_{R_0})$ such that
$\eps \defeq \normD{A}_{\al,L_0,R_0} \le \eps_*$.


By~\eqref{ineqGevCauch}, we have
$\sum_{i=1}^n \normD{\pa_i A}_{\al,L_1,R_0} \le
\frac{1}{(L_0-L_1)^\al} \eps_*$,
thus we can apply Lemma~\ref{sublemminvers} and we get for each
$t\in[0,1]$ an open subset~$\Om_t$ of~$\A^n$ containing~$\A^n_R$ such
that the map
\[
(\th,r') \in \T^n \times B_{R_0} \mapsto
(\th,r) = \big( \th, r' + \eta(t) \na\zz1 A(\th,r') \big) \in \Om_t
\]
is a diffeomorphism whose inverse is $C^\infty$ on $\A^n_R\times[0,1]$ in the variables $\th$,
$r$ and~$t$. 
%
%
%
%
%
%
%
By Lemma~\ref{lemfindSigfromA}, $\eta(t)A$ is thus a generating
function for~$\Om_t$, inducing an exact symplectic local
diffeomorphism from~$\Om_t$ to $\T^n\times B_{R_0}$:
given $(\th,r) \in \Om_t$ and $(\th',r') \in \T^n\times B_{R_0}$,
\beq   \label{eqequivFetatA}
(\th',r') = \jF_{\eta(t)A}(\th,r) 
\quad\Longleftrightarrow\quad
\left\{ \begin{aligned}
r &= r' + \eta(t)\na\zz1 A(\th,r') \\
\th' &= \th + \eta(t) \na\zz2 A(\th,r').
\end{aligned} \right.
\eeq
Moreover, $(\th,r,t) \in \A^n_R\times[0,1] \mapsto \jF_{\eta(t)A}(\th,r)$ is $C^\infty$.


In order to check that $\jF_{\eta(t)A}$ is in fact a diffeomorphism
from~$\Om_t$ onto $\T^n\times B_{R_0}$,
we consider the map
\beq    \label{eqdefdiffeononper}
(\th,r') \in B_{2\sqrt{n}} \times B_{R_0} \mapsto
(\th',r') = \big( \th + \eta(t) \na\zz2 A(\th,r'), r' \big) \in \R^n \times \R^n.
\eeq
By~\eqref{ineqGevCauch}, we have
$\sum_{i=1}^n \normD{\pa_{n+i} A}_{\al,L_1,R_0} \le
\frac{1}{(L_0-L_1)^\al} \eps$,
thus we can apply of Lemma~\ref{sublemminvers} (or rather a variant of
it in which the $\Z^n$-periodicity assumption is removed and the roles
of~$\th$ and~$r$ are exchanged): we get an open subset~$\ti\Om_t$ of
$\R^n\times B_{R_0}$ containing $B_{\sqrt{n}}\times B_{R_0}$ such that the
map~\eqref{eqdefdiffeononper} is a $C^\infty$ diffeomorphism from
$B_{2\sqrt{n}} \times B_{R_0}$ to~$\ti\Om_t$,
with an inverse of the form
\[
(\th',r') \in \ti\Om_t \mapsto (\th,r') = 
\big( \th'+\ti g(\th',r',t), r' \big) \in B_{2\sqrt{n}} \times B_{R_0},
\]
with Gevrey-$(\al,L)$ estimates on $B_{\sqrt{n}}\times B_{R_0} \times
[0,1]$ for the components of~$\ti g$.
Since~$\na\zz2 A$ is $\Z^n$-periodic in~$\th$ and $B_{\sqrt{n}}$
contains $\big[-\dem,\dem\big]^n$, the vector-valued function~$\ti g$ is
$\Z^n$-periodic in~$\th'$ and extends by periodicity to the whole of
$\R^n\times B_{R_0}$; we thus get a $C^\infty$ diffeomorphism
$(\th',r') \in \T^n \times B_{R_0} \mapsto (\th,r') = 
\big( \th'+ g(\th',r',t), r' \big) \in \T^n \times B_{R_0}$
with
\[
g_1,\ldots,g_n \in G^{\al,L}(\A^n_R\times[0,1]), \qquad
\sum_{i=1}^n \normD{g_i}_{\al,L,R} \le
\frac{\normD{\eta}_{\al,L_0}}{(L_0-L_1)^\al} \eps.
\]
In view of~\eqref{eqequivFetatA}, we conclude that~$\jF_{\eta(t)A}$ is
invertible, with inverse
\[
\jF_{\eta(t)A}\ii(\th',r') =
\big( \angD{\th'+ g(\th',r',t)}, 
r' + \eta(t)\na\zz1A(\th'+ g(\th',r',t), r') \big),
\qquad (\th',r') \in \T^n \times B_{R_0}.
\]
%

Let us now consider the $C^\infty$ function 
\[
f(\th,r,t) \defeq \eta'(t) A\big( \th+ g(\th,r,t), r \big),
\qquad (\th,r,t)\in \T^n\times B_{R_0} \times [0,1].
\]
By Proposition~A.1 of \cite{hms} (\cf also Appendix~\ref{secappremind}), it is Gevrey-$(\al,L)$ on
$\A^n_R\times[0,1]$ because
\[
L^\al + \sum_{\ell\in\N^{2n+1},\,\ell\neq0} \frac{L^{\absD{\ell}\al}}{\ell!^\al}
\normD{\pa^\ell g_i}_{C^0(\A^n_R\times[0,1])} \le
\frac{L_0^\al}{(2n)^{\al-1}}, 
\qquad i = 1,\ldots,n,
\]
and
$\normD{f}_{\al,L,R} \le \normD{\eta'}_{\al,L} \normD{A}_{\al,L_0,R_0}
\le \frac{1}{(L_0-L)^\al} \normD{\eta}_{\al,L_0} \eps$
(thanks to the algebra norm property and~\eqref{ineqGevCauch}), which
yields~\eqref{ineqnormHalLRA}.
%
%
It only remains to be shown that, for each $(\th,r) \in \A^n_R$, the
$C^\infty$ curve $t\in[0,1] \mapsto 
\big( \th(t),r(t) \big) \defeq \jF_{\eta(t)A}(\th,r)$
satisfies the system of ordinary differential equations
\beq   \label{eqEDOthtrt}
\th'(t) = \na\zz2 f\big( \th(t),r(t),t \big), \quad
r'(t) = -\na\zz1 f\big( \th(t),r(t),t \big).
\eeq


On the one hand, the relations
\[
r = r(t) + \eta(t) \na\zz1 A\big( \th,r(t) \big), \qquad
\th(t) = \th + \eta(t) \na\zz2 A\big(\th,r(t) \big)
\]
\vspace{-1ex}
entail
\begin{align}
\label{eqrptfrmrel}
r'(t) &= - \eta'(t) \Big( 1_n + \eta(t) \dd\zz2\na\zz1 A\big(\th,r(t)\big)
\Big)\ii \na\zz1 A\big(\th,r(t)\big), \\[1ex]
\label{eqthptfrmrel}
\th'(t) &= \eta'(t) \na\zz2 A\big(\th,r(t)\big) +
\eta(t) \dd\zz2\na\zz2 A\big(\th,r(t)\big) r'(t).
\end{align}
On the other hand, 
with the notation $\cG_t \defeq \jF_{\eta(t)A}\ii$,
the formula~\eqref{eqdefHsuspetatA} yields
\begin{align*}
\na\zz1 f(\th,r,t) &= \eta'(t) \ltrans{\Big( \dd\zz1\cG_t\zz1(\th,r) \Big)}
\na\zz1 A\big( \cG_t\zz1(\th,r), r \big) \\[1ex]
\na\zz2 f(\th,r,t) &= \eta'(t) \na\zz2 A\big( \cG_t\zz1(\th,r), r \big) +
\eta'(t) \ltrans{\Big( \dd\zz2\cG_t\zz1(\th,r) \Big)}
\na\zz1 A\big( \cG_t\zz1(\th,r), r \big)
\end{align*}
for any $(\th,r,t)$. We rewrite this at the point $\big( \th(t),r(t), t
\big)$, using the fact that the Jacobian matrix of~$\cG_t$ at $\big(
\th(t),r(t) \big)$ is the inverse Jacobian matrix of~$\jF_{\eta(t)A}$
at~$(\th,r)$, whose first~$n$ lines are given by~\eqref{eqinvJacFA},
thus
\begin{align*}
\dd\zz1\cG_t\zz1\big( \th(t), r(t) \big) &=
\Big( 1_n + \eta(t)\dd\zz1\na\zz2 A\big( \th,r(t) \big) \Big)\ii,
\\[1ex]
\dd\zz2\cG_t\zz1\big( \th(t), r(t) \big) &=
\dd\zz1\cG_t\zz1\big( \th(t), r(t) \big) \dd\zz2\na\zz2 A\big( \th,r(t) \big),
\end{align*}
%
%
and
\[ 
\na\zz1 f\big(\th(t),r(t),t\big) = \eta'(t) 
\Big( 1_n + \eta(t)\dd\zz2\na\zz1 A\big( \th,r(t) \big) \Big)\ii
\na\zz1 A\big( \th, r(t) \big) 
= - r'(t)
\]
by~\eqref{eqrptfrmrel}, and 
\[ 
\na\zz2 f\big(\th(t),r(t),t\big) =
\eta'(t) \na\zz2 A\big( \th, r(t) \big) + \eta(t)\dd\zz2\na\zz2 A\big( \th,r(t) \big)r'(t)
= \th'(t)
\] 
by~\eqref{eqthptfrmrel},
hence~\eqref{eqEDOthtrt} is proved.
\end{proof}


\begin{rem}
  The fact that $t \mapsto \jF_{\eta(t)A}$ is a Hamiltonian isotopy is
  standard result of basic symplectic topology, however the explicit
  formula~\eqref{eqdefHsuspetatA} for the non-autonomous Hamiltonian
  function~$f$ is new.
This explicit formula was needed to obtain the estimate~\eqref{ineqnormHalLRA}.
\end{rem}

\subsubsection{Completion of the proof of Theorem~\ref{thmsusp}}

We now prove Theorem~\ref{thmsusp}.
We thus give ourselves reals $L_0,R,R_0>0$ such that $R<R_0$
and a function $h\in G^{\al,L_0}(\ov B_{R_0})$.
We pick $R_1 \in (R,R_0)$ and $\eta \in G^{\al,L_0}([0,1])$ such that
$\eta \equiv 0$ on a neighbourhood of~$0$, $\eta \equiv 1$ on a
neighbourhood of~$1$, and $0 \le \eta \le 1$ on $[0,1]$
(\eg $\eta(t) = G(t)/G(1)$ with $G(t) = \int_0^t F(s)\,\dd s$, where
$F\in G^{\al,L_0}([0,1])$ satisfies $F\ge0$, $F(\dem)=1$,
$F_{|[0,\quart]\cup[\tquart,1]}\equiv0$; such a function~$F$ is
constructed in Lemma~A.3 of \cite{hms}---see also Lemma~3.3 of
\cite{ms} quoted in Appendix~\ref{secBumpGev}).

Applying Proposition~\ref{propfindAfromPsi} with $L_0,R_1,R_0$ and~$h$,
we get constants $\eps_1,L_1,C_1$ such that,
for any exact symplectic map $\Psi \in G^{\al,L_0}(\A^n_{R_0}, \A^n)$ with
\beq   \label{eqdefepsleepsun}
\eps \defeq \NormD{\Psi-\Phi^h}_{\al,L_0,R_0} \leq \eps_1,
\eeq
there exists $A \in G^{\al,L_1}(\A^n_{R_1})$ such that 
$\jF_A \col \A^n_{R_1} \to \A^n_{R_0}$ is a well-defined exact
symplectic map, 
\beq   \label{eqPsiPhihjFAnormA}
\Psi_{|\A^n_{R_1}} = \Phi^h \circ \jF_A, \qquad
\normD{A}_{\al,L_1,R_1} \le C_1\, \eps.
\eeq

Applying Proposition~\ref{propHamIsot} with $L_1,R,R_1$
and~$\eta$:
we get constants $\eps_2,L_2,C_2$ such that,
for any $A \in G^{\al,L_1}(\A^n_{R_1})$ with $\normD{A}_{\al,L_1,R_1} \le \eps_2$
and for any $t\in[0,1]$,
there exists an open subset~$\Om_t$ of~$\A^n$ containing~$\A^n_R$
such that
$\jF_{\eta(t)A} \col \Om_t \to \T^n\times B_{R_1}$
is a well-defined exact symplectic $C^\infty$ diffeomorphism, 
$t\mapsto \jF_{\eta(t)A}(\th,r)$ satisfies the ordinary differential
equation~\eqref{eqEDOXf} for each $(\th,r) \in \A^n_R$,
with $f \in C^\infty(\T^n\times B_{R_1}\times[0,1])$ such that
\beq   \label{eqdeffnormf}
f(\th,r,t) \defeq \eta'(t) A\big( (\jF_{\eta(t)A}\ii)\zz1(\th,r), r \big),
\qquad
\normD{f}_{\al,L_2,R} \le 2^\al L_1^{-\al}\normD{\eta}_{\al,L_1} \normD{A}_{\al,L_1,R_1}.
\eeq


Let us set
\[
\eps_* \defeq \min \Big\{ \eps_1, \frac{1}{C_1}\eps_2 \Big\}, \qquad
C_* \defeq \frac{2^\al C_1}{L_1^\al} \normD{\eta}_{\al,L_0}
\]
and choose $L_*>0$ small enough so that
\beq    \label{eqchoiceLst}
L_*^\al + \frac{L_*^\al (1+L_2^\al)}{(L_2-L_*)^\al(L_0-L_2)^\al} \normD{h}_{\al,L_0,R_0}
\le \frac{L_2^\al}{(2n+1)^\al}.
\eeq

Given an exact symplectic map $\Psi \in G^{\al,L_0}(\A^n_{R_0}, \A^n)$
such that~\eqref{eqdefepsleepsun} holds, we get from
Proposition~\ref{propfindAfromPsi} a function~$A$
satisfying~\eqref{eqPsiPhihjFAnormA}.
Since $C_1\,\eps \le \eps_2$, we can then apply
Proposition~\ref{propHamIsot} to the generating function~$A$ and get a
non-autonomous Hamiltonian function $f \in C^\infty(\T^n\times
B_{R_1}\times[0,1])$ as in~\eqref{eqdeffnormf},
for which the flow between time~$0$ and time~$t$ on~$\A^n_R$ coincides
with $\jF_{\eta(t)A}$
(because the differential equation~\eqref{eqEDOXf} is satisfied and
$\eta(0)=0$, $\jF_0 = \Id$).
Notice that $\normD{f}_{\al,L_2,R} \le C_*\,\eps$.

For $t\in[0,1]$, we define $\Psi_t \defeq \Phi^{th} \circ \jF_{\eta(t)A}$ on~$\A^n_R$:
this is an isotopy from~$\Id$ to~$\Psi$, and one checks easily that it
gives the flow between time~$0$ and time~$t$ on~$\A^n_R$ for the
Hamiltonian function
\[
H(\th,r,t) \defeq h(r) + f\circ\Phi(\th,r,t),
\qquad (\th,r,t) \in \T^n \times B_{R_1} \times [0,1],
\]
where $\Phi(\th,r,t) \defeq \big( \th - t\na h(r), r, t\big)$
(because $\Phi^{th}$ is symplectic and $\Phi(x,t) = (
\Phi^{-th}(x,t), t )$, hence
$\dd\Phi^{th}(x) X_f(x,t) = X_{f\circ\Phi}( \Phi^{th}(x), t )$).
Since $\eta'(t)\equiv0$ in neighbourhoods of~$0$ and~$1$, the formula
\beq   \label{eqexplicitsusp}
H(\th,r,t) = h(r) + 
\eta'(t) A\big( (\jF_{\eta(t)A}\ii)\zz1(\th-t\na h(r),r), r \big)
\eeq
shows that~$H$ can be extended by $\Z$-periodicity in~$t$, so that we
get
$H \in C^\infty(\T^n\times B_{R_1} \times \T)$,
which can be viewed as a suspension of~$\Psi_{|\A^n_R}$.

By Proposition~A.1 of \cite{hms} (\cf also
Appendix~\ref{secappremind}), we have
$H_{|\A^n_R\times\T} \in G^{\al,L_*}(\A^n_R\times\T)$ and
\[
\normD{H-h}_{\al,L_*,R} = \normD{f\circ\Phi}_{\al,L_*,R} 
\le \normD{f}_{\al,L_2,R} \le C_*\,\eps
\]
because the components of~$\Phi$ satisfy
\beq   \label{ineqcompPhi}
\sum_{\ell\in\N^{2n+1},\,\ell\neq0} \frac{L_*^{\absD{\ell}\al}}{\ell!^\al}
\normD{\pa^\ell \Phi_i}_{C^0(\A^n_R\times[0,1])} \le
\frac{L_2^\al}{(2n+1)^{\al-1}}, 
\qquad i = 1,\ldots,2n+1
\eeq
(indeed, the \lhs is $\le L_*^\al +
\frac{L_*^\al}{(L_2-L_*)^\al}\normD{t\pa_i h}_{\al,L_2,R}$ by a
$(2n+1)$-variable variant of Lemma~\ref{lemusefulNst}, which is 
$\le L_*^\al + \frac{L_*^\al (1+L_2^\al)}{(L_2-L_*)^\al(L_0-L_2)^\al} \normD{h}_{\al,L_0,R_0}$
by~\eqref{ineqGevCauch}, hence~\eqref{ineqcompPhi} follows from~\eqref{eqchoiceLst}).
This ends the proof of Theorem~\ref{thmsusp}.

\subsection{Proof of Theorem~\ref{th:NekhFevMaps} (Nekhoroshev Theorem
  for maps)}
\label{secPfThmNekhMaps}

We now prove Theorem~\ref{th:NekhFevMaps} of Section~\ref{secthmNekhMapsA}.
To this end, we first recall the exponential stability theorem for
near-integrable quasi-convex Hamiltonian \emph{flows} in its most recent formulation.
Theorem~\ref{thmsusp} will then allow us to transfer this result to
near-integrable \emph{maps}.


\begin{thmnonb}[Bounemoura-Marco \cite{BM}]
Let $N\geq2$ be an integer. Let $\al\ge1$ and $L,R,R_0,m,E>0$ be reals such
that $R<R_0$.
Then there exist positive reals $\eps_*,c_*$,
and, for each positive $\rho<R_0-R$, positive reals $\eps'_\rho \le \eps_*$ and
$c'_\rho \le c_*$,
and, for each positive $\sig < \frac{1}{N}$, positive reals $\eps''_\sig \le \eps_*$ and
$c''_\sig \le c_*$,
satisfying the following:
\medskip

For each $h\in G^{\al,L}(\ov B_{R_0})$ such that
$\normD{h}_{\al,L,R_0} \le E$,
$\na h(r) \neq0$ for every $r\in\ov B_{R_0}$
and
\beq    \label{eqcondGCm}
\ltrans{v} \,\dd\na h(r) v \ge m \normD{v} 
\quad \text{for all $v \in \R^N$ orthogonal to~$\na h(r)$,}
\eeq
and for each $H \in G^{\al,L}(\A^N_{R_0})$
such that $\eps \defeq \normD{H-h}_{\al,L,R_0} \le \eps_*$,
every initial condition $(\th\zz0,r\zz0)$ in $\A^N_R$ gives rise to a solution
$t \mapsto \big(\th(t), r(t) \big)$ of~$X_H$ which is defined at least for
$\absD{t} \le 
\exp \big( c_* \big(\frac{1}{\eps}\big)^{\frac{1}{2(N-1)\al}} \big)$,
and
\begin{enumerate}[(i)]
\item
$\dst \eps \le \eps'_\rho \ens\text{and}\ens
\absD{t} \le \exp \Big( c'_\rho \Big(\frac{1}{\eps}\Big)^{\frac{1}{2(N-1)\al}} \Big)
\ens \Rightarrow \quad
\normD{r(t) - r\zz0} \leq \rho$,
\item
$\dst \eps \le \eps''_\sig \ens\text{and}\ens
\absD{t} \le \exp \Big( c''_\sig \Big(\frac{1}{\eps}\Big)^{\frac{1-\sig}{2(N-1)\al}} \Big)
\ens \Rightarrow \quad
\normD{r(t) - r\zz0} \leq \begin{cases}
\tfrac{1}{c''_\sig} \eps^{\frac{\sig}{2}} & \text{if $\al=1$,} \\[1.5ex]
\tfrac{1}{c''_\sig} \eps^{\frac{\sig}{5(N-1)^2}} & \text{if $\al>1$,} 
\end{cases}$
\item
$\dst \eps \le \eps_* \ens\text{and}\ens
\absD{t} \le \exp \Big( c_* \Big(\frac{1}{\eps}\Big)^{\frac{1}{2N\al}} \Big)
\quad\ens\; \Rightarrow \quad
\normD{r(t) - r\zz0} \leq \tfrac{1}{c_*} \eps^{\frac{1}{2N}}$.
\end{enumerate}
\end{thmnonb}


\begin{rem}   \label{remmorecomm}
This result is given in \cite{BM} in a slightly different presentation and
we took the opportunity of correcting a slight mistake in the time exponent in
the case $\al>1$ of~(ii): in \cite{BM}, it should be
$\frac{1}{2(N-1)\al}-\frac{\de}{\al}$
with a parameter $\de \in \big(0,\frac{1}{2N(N-1)}\big)$
(and not $\frac{1}{2(N-1)\al}-\de$ as is written there),
and we introduced $\sig = 2(N-1)\de$.

It is a refined version of the Nekhoroshev theorem for analytic or
Gevrey Hamiltonians in the case of an $m$-quasi-convex integrable part, \ie
in the case of a function~$h(r)$ satisfying the condition~\eqref{eqcondGCm}.
The article \cite{BM} is the last of a series of attempts to
obtain the largest possible exponents~$a$ in the stability time
$\exp\Big( \text{const} \, \Big(\frac{1}{\eps}\Big)^{a} \Big)$
and~$b$ in the corresponding confinement radius $\text{const} \,
\eps^b$,
after the original work of Nekhoroshev in 1977 for analytic steep
Hamiltonians, 
the refinement by Lochak-Neishtadt and P\"oschel in 1992--94
for analytic quasi-convex Hamiltonians (which gave the exponents
$a=b=\frac{1}{2N}$ as in~(iii) in the case $\al=1$), 
and the first Gevrey stability theorem by Marco-Sauzin in 2002 still
in the quasi-convex case (which gave the exponents $a=\frac{1}{2N\al}$
and $b=\frac{1}{2N}$ as in~(iii) in the case $\al\ge1$).

Bounemoura-Marco's article \cite{BM} focuses on the stability time
(rather than the confinement radius, which is anyway a less important
issue), for the Gevrey case ($\al\ge1$); their discovery is that one
can obtain a time exponent~$a$ arbitrarily close to
$\frac{1}{2(N-1)\al}$ at the price of a smaller exponent~$b$, or even
equal to that value at the price of accepting a weaker notion of
confinement: there is an arbitrarily small confinement radius~$\rho$ but it
does not tend to~$0$ with~$\eps$.
This weaker confinement property is all we need when studying
wandering domains (see Section~\ref{secPfThmUpperBds}).
\end{rem}


\begin{rem}[Stabilization by resonances]   \label{remmreson}
The phenomenon of stabilization by resonances for quasi-convex
Hamiltonians was first proved by Lochak-Neishtadt and P\"oschel in the
analytic case;
Marco-Sauzin's article \cite{hms} contains a generalization to the
Gevrey case $\al\ge1$ obtained by adapting Lochak's periodic method.
The result can be formulated as follows:

For any submodule~$\cM$ of~$\Z^n$ of rank $\mult(\cM) \in
\{1,\ldots,N-1\}$, consider the resonant surface
\[
S_\cM \defeq \big\{\, r \in \ov B_{R_0} \mid \sum_{i=1}^n k_i \pa_i h(r) = 0
\;\text{for all $k\in\cM$} \,\big\},
\]
which is a $\mult(\cM)$-codimensional submanifold of~$\ov B_{R_0} \subset \R^N$.
Then there is an improvement of the stability property whenever the
initial condition lies at a distance $O(\eps^{1/2})$ of~$S_\cM$:
for any real $\sig>0$, there exist $\ti\eps, \ti c>0$ (which depend
on $\al,L,R,R_0,m,E,\sig,\cM$) such that,
for any $m$-quasi-convex $h\in G^{\al,L}(\ov B_{R_0})$ such that
$\normD{h}_{\al,L,R_0}\le E$,
for any $H\in G^{\al,L}(\A^N_{R_0})$ such that $\eps \defeq \normD{H-h}_{\al,L,R_0} \le \ti\eps$,
for any initial condition $(\th\zz0,r\zz0)$ in $\A^N_R$ such that
\[
\dist(r\zz0, S_\cM) \le \sig\, \eps^{1/2},
\]
the solution $\big(\th(t), r(t) \big)$ of~$X_H$ satisfies
\[
\absD{t} \le \exp \Big( \ti c \Big(\frac{1}{\eps}\Big)^{a} \Big)
\quad\ens\; \Rightarrow \quad
\normD{r(t) - r\zz0} \leq \tfrac{1}{\ti c} \eps^{b}
\]
with $\dst a \defeq \frac{1}{2\big(N-\mult(\cM)\big)\al}$
and $\dst b \defeq \frac{1}{2(N-\mult(\cM))}$.
\end{rem}


\begin{proof}[Proof of Theorem~\ref{th:NekhFevMaps}]
  Let us give ourselves $n\ge1$ integer and $\al\ge1, L,R,R_0,m,E>0$ reals
  such that $R<R_0$.
Let $R_1 \defeq \frac{R+R_0}{2}$.


Applying Theorem~\ref{thmsusp} with $n,\al,L,R_1,R_0,E$, we get
positive reals $\eps_1,L_1,C_1$ such that, 
for every $h\in G^{\al,L}(\ov B_{R_0})$ such that $\normD{h}_{\al,L,R_0} \le E$
and every exact symplectic map
$\Psi \in G^{\al,L}(\A^n_{R_0}, \A^n)$ such that
$\eps \defeq \NormD{\Psi-\Phi^h}_{\al,L,R_0} \leq \eps_1$,
there is a suspension $H\in G^{\al,L_1}(\A_{R_1}^n\times\T)$
such that $\normD{H-h}_{\al,L_1,R_1} \le C_1\, \eps$.
Without loss of generality, we can assume $L_1\le L$ and $C_1\ge1$.


Let $N\defeq n+1$, $E_1 \defeq R_1+L_1+E$, $m_1 \defeq \frac{m}{1+(L-L_1)^{-2\al}E^2}$.
Applying Bounemoura-Marco's theorem with $N,\al,L_1,R,R_1,m_1,E_1$,
we get positive reals $\ti\eps_*,\ti c_*$,
and, for each positive $\rho<R_1-R$, positive reals $\ti\eps'_\rho \le \ti\eps_*$ and
$\ti c'_\rho \le \ti c_*$,
and, for each positive $\sig < \frac{1}{n+1}$, positive reals $\ti\eps''_\sig \le \ti\eps_*$ and
$\ti c''_\sig \le \ti c_*$,
such that,
for any $m_1$-quasi-convex~$\ti h$ and any~$\ti H$ in
$G^{\al,L_1}(\A^{n+1}_{R_1})$ such that
$\normD{\ti h}_{\al,L_1,R_1} \le E_1$ and 
$\ti\eps\defeq\normD{\ti H-\ti h}_{\al,L_1,R_1} \le \ti\eps_*$,
every initial condition $\big( \ti\th\zz0, \ti r\zz0 \big)$ in
$\A_{R}^{n+1}$ gives rise to a solution of~$X_{\ti H}$ defined at
least for
$\absD{t} \le \exp \big(\ti c_* \big(\frac{1}{\ti\eps}\big)^{\frac{1}{2n\al}} \big)$,
which satisfies the properties (i), (ii) and~(iii) of Bounemoura-Marco's theorem.

We now check the statement of Theorem~\ref{th:NekhFevMaps} for an
$m$-convex function $h \in G^{\al,L}(\ov B_{R_0})$ such that $\normD{h}_{\al,L,R_0} \le E$
and an exact symplectic $\Psi \in G^{\al,L}(\A^n_{R_0},\A^n)$ such
that
\[
\eps\defeq\NormD{\Psi-\Phi^h}_{\al,L,R_0} \le
\eps_* \defeq \min\{\eps_1,\ti\eps_*/C_1\}.
\]
Let $H\in G^{\al,L_1}(\A_{R_1}^n\times\T)$ denote the suspension
of~$\Psi$ obtained from Theorem~\ref{thmsusp}.
We introduce the $(n+1)$-degree of freedom autonomous Hamiltonian
functions
\[
\ti h(r,r_{n+1}) \defeq r_{n+1}+h(r), \qquad
\ti H(\th,r,\th_{n+1},r_{n+1}) \defeq r_{n+1} + H(\th,r,\th_{n+1})
\]
for $(\th,r,\th_{n+1},r_{n+1}) \in \A^n_{R_1}\times \T\times \R \simeq
\T^{n+1}\times\ov B_{R_1}\times \R$, which contains $\A_{R_1}^{n+1}$.
One easily checks that $\normD{\ti h}_{\al,L_1,R_1} \le E_1$ and~$\ti h$ is $m_1$-quasi-convex.
Since $\normD{\ti H-\ti h}_{\al,L_1,R_1} = \normD{H-h}_{\al,L_1,R_1}
\le C_1\, \eps \le \ti\eps_*$,
Bounemoura-Marco's theorem ensures stability properties for all the solutions
of~$X_{\ti H}$ starting in~$\A^{n+1}_R$.
The conclusion stems from the fact that the solutions of the
autonomous vector field~$X_{\ti H}$ are related to the solutions of the
non-autonomous vector field~$X_H$, which, in turn, interpolate the
discrete dynamics induced by~$\Psi$; in particular, if the initial
condition is of the form 
$\big( \th\zz0, r\zz0, \th_{n+1}\zz0, r_{n+1}\zz0 \big) = (\th,r,0,0)$,
then the value of the solution at any integer time~$k$ such that 
$\absD{k} \le \exp \big(\ti c_* \big(\frac{1}{C_1\,\eps}\big)^{\frac{1}{2n\al}} \big)$
satisfies
\[
\big( \th(k), r(k) \big) = \Psi^k(\th,r), \ens
\th_{n+1}(k) = k, \ens
r_{n+1}(k) = H(\th,r,0) - H\big( \Psi^k(\th,r), k\big),
\]
hence the properties~(i), (ii) and~(iii) in Bounemoura-Marco's theorem imply
the desired properties for the discrete orbits of~$\Psi$ starting
in~$\A^n_R$, with 
$c_* \defeq \ti c_* \big(\frac{1}{C_1} \big)^{\frac{1}{2n\al}}$,
$\eps'_\rho \defeq \min\{\eps_1,\ti\eps'_\rho/C_1\}$,
$c'_\rho \defeq \ti c_* \big(\frac{1}{C_1} \big)^{\frac{1}{2n\al}}$,
$\eps''_\sig \defeq \min\{\eps_1,\ti\eps''_\sig/C_1\}$,
$c''_\sig \defeq \ti c_* \big(\frac{1}{C_1} \big)^{\frac{1-\sig}{2n\al}}$.
\end{proof}


\subsection{Proof of Theorem~\ref{th:upperbounds} (upper bounds for
  wandering sets)}	\label{secPfThmUpperBds}

We now prove Theorem~\ref{th:upperbounds} of Section~\ref{secthmupperbdsB}.
Let us give ourselves $n\ge1$ integer and $\al\ge1$, $L,R,R_0,m,E>0$ such
that $R<R_0$.
We take~$\eps_*$ and~$c_*$ as in Theorem~\ref{th:NekhFevMaps}.

Given an arbitrary $m$-convex function $h\in G^{\al,L}(\ov B_{R_0})$
such that $\normD{h}_{\al,L,R_0}\le E$, and a map~$\Psi$ as in the
statement of Theorem~\ref{th:upperbounds}, with a measurable wandering
set $W \subset \A^n_R$, we can apply Theorem~\ref{th:NekhFevMaps} to
$\Psi_{|\A^n_{R_0}}$.
This shows that for a point $(\th,r) \in \A^n_R$, all the iterates
$\Psi^k(\th,r)$ with $\absD{k} \le k_* \defeq 
\exp \big( c_* \big(\frac{1}{\eps}\big)^{\frac{1}{2n\al}} \big)$
stay in~$\A^n_{R_0}$.
In particular, all the sets $\Psi^k(W)$ with $\absD{k}\le k_*$ are
contained in~$\A^n_{R_0}$.
But these sets are pairwise disjoint and they all have the same
Lebesgue measure, therefore
$(2k_*+1) \mu(W) \le \mu(\A^n_{R_0})$,
which yields the desired estimate (diminishing the value of~$\eps_*$
and~$c_*$ if necessary).


\vfil

\pagebreak

\section{A quantitative KAM result -- proof of Part~(i) of
  Theorem~\ref{th:perdomains}} 
\label{sec:proofpseudopend}


As announced in Section~\ref{sec:descriptstrctpfC}, this section
contains the proof of the two-dimensional case of
Theorem~\ref{th:perdomains} stated there.
This proof is based on an auxiliary result, Theorem~\ref{th:varpseudopend},
which will also be instrumental in the obtention of the full proof of
Theorem~\ref{th:perdomains} in Section~\ref{ssec:perdomains}.


\subsection{Elliptic islands in~$\A$ with a tuning parameter --
  Theorem~\ref{th:varpseudopend}}
\label{sec:proofpseudopend-def}


We will take the liberty of identifying a $1$-periodic function
on~$\R$ with a function on~$\T$.
Here is the auxiliary result which has been alluded to:

%
\label{secthmperdompseudopF}
\begin{thm}\label{th:varpseudopend} 
Let $\al>1$ and $L>0$ be real numbers.
Suppose, on the one hand, that $V \in C^\infty(\R)$ is a $1$-periodic
function and that $L_0,\th^\star,\rho_0$ are positive reals such that
$L_0<\Demi-\th^\star$ and
\\[-1ex]
\begin{tabular}{p{9cm}c}
\begin{enumerate}[(i)]
\item 
$\quad -L_0 \le \th \le L_0 \quad\ens
\ens \Rightarrow \ens\ens
V(\theta)=-\frac{1}{2}\rho_0^2$
\item 
$\dem-\th^\star \le \th \le \dem+\th^\star
\ens \Rightarrow \ens\ens
V(\theta)=-(\theta-1/2)^4$
\item 
$\qquad \th-\dem \not\in \Z \qquad\ens
\ens \Rightarrow \ens\ens
V(\theta)<0$.
\end{enumerate}
&
\psset{xunit=1cm,yunit=0.8cm}
\begin{pspicture}(-0.8,-1.8)(3,-1.5)
\psline[linewidth=1pt]{->}(0,-3)(4,-3)
\psline[linewidth=1pt]{->}(2,-4.3)(2,-2.5)
\def\f{ x  4 exp  1   x   1 sub -4 mul  add  x 1 sub 2 exp 10 mul add x 1 sub 3 exp -20 mul add
 mul   -1 mul 3 sub}
\def\g{4 x sub  4 exp  1   4 x sub    1 sub -4 mul  add  4 x sub  1 sub 2 exp 10 mul add  4 x sub  
1 sub 3 exp -20 mul add mul   -1 mul   3 sub }
\psplot[plotstyle=curve,linewidth=1.5pt,plotpoints=30]{0}{1}{\f}
\psplot[plotstyle=curve,linewidth=1.5pt,plotpoints=30]{3}{4}{\g}
\psplot[plotstyle=curve,linewidth=1.5pt,plotpoints=30]{1}{3}{-4}
\psline[linewidth=1pt,linestyle=dashed](1,-4)(1,-3)
\psline[linewidth=1pt,linestyle=dashed](3,-4)(3,-3)
\rput[l](2,-2.2){$V(\theta)$}
\rput[l](0.7,-2.6){$-L_0$}
\rput[l](3,-2.6){$L_0$}
\rput[l](-0.2,-2.6){$-\frac{1}{2}$}
\rput[l](4,-2.6){$\frac{1}{2}$}
\rput[l](1.3,-4.5){$-\frac{\rho_0^2}{2}$}
\end{pspicture}
\end{tabular}
\\
We use the notation
\begin{equation}
\label{eq:P_V}
P_{V/N^2}(\th,r) \defeq \demi r^2 + \frac{1}{N^2} V(\th)
\quad \text{for $N\in\N^*$.}
\end{equation}
Suppose, on the other hand, that~$\de$ is a real number such that
$0<\delta<\rho_0/2$, and that $(W_N)_{N\in\N^*}$ is a sequence of
$1$-periodic functions of $C^\infty(\R)$ such that
\begin{enumerate}[(i)] \setcounter{enumi}{3}
\item
$-\frac{\de}{2N} \le \th \le \frac{\de}{2N} \,
\ens \Rightarrow \ens 
W_N(\theta) = \Demi \th^2$
\item
$\frac{\de}{N} \le \th \le 1 - \frac{\de}{N}
\ens \Rightarrow \ens 
W_N(\theta)=0$.
\end{enumerate}


Then there exist positive reals $C_1,C_2,C_3,C_4$ such that, 
for any integers $q, N \in \N^*$ such that $q\ge C_1 N$
and for any real $\mu\in\big( 0, C_2 {N^4}/{q^5} \big)$,
the exact symplectic map of~$\A$
\beq
\label{eqdefGNm}
\GNm \defeq \Phi^{\mu W_N} \circ \Phi^{P_{V/N^2}} 
\eeq
admits a $q$-periodic disc $\DNm \subset \A_3$, with all its iterates
also contained in $\A_3$, such that
\beq   \label{ineqareaDNm}
	C_3 \frac{\mu}{N^2} \le \area(\DNm) \le C_4 \frac{\mu}{N^2}
\eeq
and
\beq   \label{eqlocalizDNm}
\DNm \subset \BdeeN\cap \A^+_{4\delta/N}, \qquad
\GNm^k(\DNm) \cap \BdeN = \varnothing 
\quad \text{for $1\le k \le q-1$.}
\eeq
\end{thm}
%

The proof of Theorem~\ref{th:varpseudopend} is given in
Sections~\ref{sec:overvpseudopend}--\ref{secconclupfThmF}.
Recall that the notations $\A^+_d$ and~$\Bd$ were introduced in~\eqref{eqdefApdBd}.

Observe that, when $V, W_N \in G^{\al,L}(\R)$,
we have $\GNm \in \Pa{1}(\Phi^{\Demi r^2})$; 
this map can be viewed as a perturbation of the ``pseudo-pendulum''
$\Phi^{P_{V/N^2}}$.
Here, we have two external parameters, $N$ and~$\mu$ (changing them
amounts to changing the discrete dynamical system we are dealing
with), and one internal parameter, $q$ (we may vary it, \eg taking it
larger and larger, while keeping the same system~$\GNm$).
We call~$\mu$ the ``tuning parameter'';
an appropriate choice of~$\mu$ will yield
Theorem~\ref{th:perdomains}(i) in Section~\ref{secPfThperi}
and Theorem~\ref{th:perdomains}(ii) in Section~\ref{ssec:perdomains}.


\subsection{Theorem~\ref{th:varpseudopend} implies Part~(i) of Theorem~\ref{th:perdomains}}
\label{secPfThperi}


Taking for granted Theorem~\ref{th:varpseudopend}, we now show how
Theorem~\ref{th:perdomains}(i) follows.

Let $\al>1$ and $L>0$.
With the help of ``bump functions'' as in Appendix~\ref{secBumpGev},
we can easily choose $V \in G^{\al,L}(\R)$ satisfying conditions
(i)--(iii) (for whatever choice of $L_0,\th^\star,\rho_0$).
We choose $\rho_0 >2$ and $\de=1$.
For the choice of the sequence $(W_N)_{N\in\N^*}$ we apply
Lemma~\ref{lembump}, which produces a real $c(\al,L)>0$ and a sequence of
$1$-periodic functions $(\eta_N)_{N\in\N^*}$ in $G^{\al,L}(\R)$ such that
\[
-\frac{1}{2N} \le \th \le \frac{1}{2N}
\ens \Rightarrow \ens 
\eta_N(\theta) = 1, \qquad
\frac{1}{N} \le \th \le 1 - \frac{1}{N}
\ens \Rightarrow \ens 
\eta_N(\theta)=0
\]
and
\beq
\norm{\eta_N}_{\al,L} \le \exp\Big(c(\al,L)\, N^{\frac{1}{\al-1}} \Big)
\quad \text{for all $N \in \N^*$.}
\eeq
We then set
$W_N(\th) \defeq \Demi \eta_N(\th) \big(\dist(\th,\Z)\big)^2$, so that
\beq
\norm{W_N}_{\al,L} \le C_0\, \exp\Big(c(\al,L)\, N^{\frac{1}{\al-1}} \Big)
\quad \text{for all $N \in \N^*$}
\eeq
with some constant $C_0>0$,
and we apply Theorem~\ref{th:varpseudopend}.


We get $C_1,C_2,C_3,C_4>0$ fulfilling the conclusions of Theorem~\ref{th:varpseudopend}.
Observe that formula~\eqref{eqdefGNm} defines
$\GNm \in \Pa{1}(\Phi^{\ppdemi r^2})$ with
\begin{equation}	\label{eqdeGNm}
\de^{\al,L}\Big(\GNm,\Phi^{\dem r^2}\Big) \le
\frac{1}{N^2}\norm{V}_{\al,L} + 
C_0 \,\mu\, \exp\Big( c(\al,L)\, N^{\frac{1}{\al-1}} \Big) 
\end{equation}
for any integer $N\ge1$ and real $\mu>0$.
Recall that $N_j \defeq p_{j+2}$ is given by the prime number
sequence. We set
\beq   \label{enavantlamujq}
\mu_{j,q} \defeq \min\bigg\{
\frac{C_2 N_j^4}{2q^5}, 
\frac{1}{N_j^2} \exp\Big( -c(\al,L) N_j^{\frac{1}{\al-1}} \Big)
\bigg\}, \quad
\Psi_{j,q} \defeq G_{N_j,\mu_{j,q}}, \quad
\jD_{j,q} \defeq \DNmj
\eeq
for all $j,q \in \N^*$ such that $q \ge C_1 N_j$
(notice that $\DNmj$ is well defined because 
$\mu_{j,q} \in \big( 0, C_2 {N_j^4}/{q^5} \big)$
for such values of~$j$ and~$q$).

Let us check that the conclusions of Theorem~\ref{th:perdomains}(i)
are fulfilled.
Since~\eqref{enavantlamujq} entails 
$\mu_{j,q} \exp\Big( c(\al,L)\, N^{\frac{1}{\al-1}} \Big) \le 1/N_j^2$,
we deduce from~\eqref{eqdeGNm} that
\[
\Psi_{j,q} \in \Pa{1}(\Phi^{\ppdemi r^2}), \qquad
\de^{\al,L}\big(\Psi_{j,q},\Phi^{\dem r^2}\big) \le 
\frac{\norm{V}_{\al,L} + C_0}{N_j^2}.
\]
According to Theorem~\ref{th:varpseudopend}, $\jD_{j,q}$ is a
$q$-periodic disc for $\Psi_{j,q}$, whose orbit is localized precisely
as desired, in particular~\eqref{eqlocalizDNm} amounts exactly to~\eqref{eq:disjointdeux}.
Now, by~\eqref{ineqareaDNm}, 
\begin{align*}
\CG(\jD_{j,q}) = \area(\jD_{j,q}) & \ge C_3 \frac{\mu_{j,q}}{N_j^2}
= \Demi C_2 C_3 \min\bigg\{
\frac{N_j^2}{q^5}, 
\frac{2}{C_2 N_j^4} \exp\Big( -c(\al,L) N_j^{\frac{1}{\al-1}} \Big)
\bigg\} \\[1ex]
& \ge C'_3 \min \bigg\{%
\frac{N_j^2}{q^5}, \,
\exp\Big( -c N_j^{\frac{1}{\al-1}} \Big)
\bigg\}
\end{align*}
with $C'_3 \defeq \Demi C_2 C_3$ and $c \defeq 2\, c(\al,L)$ for $j$ large enough.
This ends the proof of Theorem~\ref{th:perdomains}(i).


\subsection{Overview of the proof of Theorem~\ref{th:varpseudopend}}
\label{sec:overvpseudopend}


The rest of Section~\ref{sec:proofpseudopend} is devoted to the proof of
Theorem~\ref{th:varpseudopend}.

We thus give ourselves once for all $\al,L,V$ and $(W_N)_{N\in\N^*}$ as
in the statement.
Let us begin with a brief overview of the method.

The pseudo-pendulum $\Phi^{P_{V/N^2}}$ has a degenerate equilibrium point at
$(\dem,0)$, with an ``upper separatrix'' 
$\{\, (\th,r) \in \A \mid r>0 \ \text{and}\  P_{V/N^2}(\th,r) =
0\,\}$;
see the figure on p.~\pageref{figdegpend}.
It also has periodic points of arbitrarily high period located near this
upper separatrix,
and the effect of the perturbation $\Phi^{\mu W_N}$ in~$\GNm$
is to create elliptic islands around these periodic points.

We will estimate the size of these islands by means of Herman's quantitative version of the two-dimensional KAM 
theorem \cite{herman2} recalled in Section~\ref{sec:stateHermanthm}. 
To do this, we have to compute parametrized normal forms of high
order, the parameters being the size of the perturbation (measured
by~$1/N$ and~$\mu$)
and the period~$q$ of the island.

More precisely, for $q$ large enough, we will study the $q$-th iterate
of~$\GNm$ in a neighbourhood of a $q$-periodic point in
Section~\ref{sec:PreliminGq}
and, in Section~\ref{sec:normalisGq}, find normalizing coordinates in
which Herman's version of the invariant curve theorem can be applied to~$\GNm^q$.

\subsection{Preliminary study near a \texorpdfstring{$q$-periodic point~$a_q$ of~$\GNm$}
{q-periodic point}}

\label{sec:PreliminGq}

\subsubsection{Localization}   

\label{sec:Defqbox}

We begin with defining a suitable notion of {\em adapted box}, to be
used in this section as well as in Section~\ref{ssec:perdomains}.


\begin{Def}
\label{def:q-box}
Let  $q$ be an integer $\geq2$ and fix $d\in (0,1/2)$. Consider a Hamiltonian function $H:\A\to\R$.
A {\em $q$-adapted box for $H$ and ${\Bd}$} is a rectangle
$B=I\times [a,b]\subset \T\times \R$ contained in $\Bd$ such that
\begin{enumerate}
\item\label{Cond:disjonction-q-box}
 for $1\leqslant t \leqslant q-1$,  $\Phi^{tH}(B)\cap \overline{\Bd}=\varnothing$;
\item \label{Cond:injection-q-box} $\Phi^{qH}(B)\subset\overline{\Bdd}$.
\end{enumerate}
\end{Def}


This section is devoted to the construction of an explicit 
$q$-adapted box centered at a $q$-periodic point $a_{q,N}$, as defined below,  with respect 
to the system $P_{V/N^2}$ and $\BdeN$.~\vspace{5pt}

\noindent
In the same way as  for the classical pendulum,  the integral curve 
\[
\cS_e:=\{(\theta,r)\in\A\mid r>0,  
P_{V/N^2}(\theta,r)=e \}
\] 
is a closed curve if $e>0$, so the flow $t\mapsto \Phi^{tP_{V/N^2}}(\theta,r)$ is periodic 
on $\cS_e$, with a period $T_{V/N^2}(e)$ given by 
\begin{equation}\label{eq:T_V}
	T_{V/N^2}(e)=\int_0^1\frac{du}{\sqrt{2(e- V(u)/N^2)}}.
\end{equation}
\begin{Def}
Let $q>0$ be a real number. Let $a_{q,N}=(0,r_{q,N})$ denote the unique point that satisfies 
$r_{q,N}>0$ and $T_{V/N^2}(a_{q,N})=q$. We also set $e_{q,N}=P_{V/N^2}(a_{q,N})$ 
and $\rho_N={\rho_0}/{N}$.
\end{Def}
\psset{xunit=1cm,yunit=0.8cm}
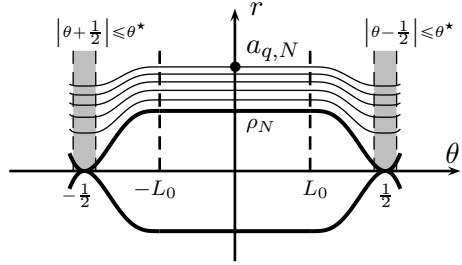
\label{figdegpend}
\begin{pspicture}(-4,-1.5)(3,2.7)
\psframe*[linecolor=lightgray](-0.15,0)(0.15,2)
\psframe*[linecolor=lightgray](3.85,0)(4.15,2)
\psline[linewidth=1pt]{->}(-1,0)(5,0)
\psline[linewidth=1pt]{->}(2,-1.5)(2,2.7)
\def\f{ x  4 exp  1   x   1 sub -4 mul  add  x 1 sub 2 exp 10 mul add x 1 sub 3 exp -20 mul add
 mul sqrt}
\def\fa{ x  4 exp  1   x   1 sub -4 mul  add  x 1 sub 2 exp 10 mul add x 1 sub 3 exp -20 mul add
 mul 0.4 add sqrt}
\def\fb{ x  4 exp  1   x   1 sub -4 mul  add  x 1 sub 2 exp 10 mul add x 1 sub 3 exp -20 mul add
 mul 0.8 add sqrt}
\def\fc{ x  4 exp  1   x   1 sub -4 mul  add  x 1 sub 2 exp 10 mul add x 1 sub 3 exp -20 mul add
 mul 1.2 add sqrt}
\def\fd{ x  4 exp  1   x   1 sub -4 mul  add  x 1 sub 2 exp 10 mul add x 1 sub 3 exp -20 mul add
 mul 1.6 add sqrt}
\def\fe{ x  4 exp  1   x   1 sub -4 mul  add  x 1 sub 2 exp 10 mul add x 1 sub 3 exp -20 mul add
 mul 2 add sqrt}
\def\g{4 x sub  4 exp  1   4 x sub    1 sub -4 mul  add  4 x sub  1 sub 2 exp 10 mul add  4 x sub  
1 sub 3 exp -20 mul add mul sqrt}
\def\ga{4 x sub  4 exp  1   4 x sub    1 sub -4 mul  add  4 x sub  1 sub 2 exp 10 mul add  4 x sub  
1 sub 3 exp -20 mul add mul  0.4 add sqrt}
\def\gb{4 x sub  4 exp  1   4 x sub    1 sub -4 mul  add  4 x sub  1 sub 2 exp 10 mul add  4 x sub  
1 sub 3 exp -20 mul add mul  0.8 add sqrt}
\def\gc{4 x sub  4 exp  1   4 x sub    1 sub -4 mul  add  4 x sub  1 sub 2 exp 10 mul add  4 x sub  
1 sub 3 exp -20 mul add mul  1.2 add sqrt}
\def\gd{4 x sub  4 exp  1   4 x sub    1 sub -4 mul  add  4 x sub  1 sub 2 exp 10 mul add  4 x sub  
1 sub 3 exp -20 mul add mul  1.6 add sqrt}
\def\ge{4 x sub  4 exp  1   4 x sub    1 sub -4 mul  add  4 x sub  1 sub 2 exp 10 mul add  4 x sub  
1 sub 3 exp -20 mul add mul  2 add sqrt}
\def\fm{ x  4 exp  1   x   1 sub -4 mul  add  x 1 sub 2 exp 10 mul add x 1 sub 3 exp -20 mul add
 mul sqrt -1 mul}
\def\gm{4 x sub  4 exp  1   4 x sub    1 sub -4 mul  add  4 x sub  1 sub 2 exp 10 mul add  4 x sub  
1 sub 3 exp -20 mul add mul  sqrt  -1  mul }
\psplot[plotstyle=curve,linewidth=1.5pt,plotpoints=100]{-0.2}{1}{\f}
\psplot[plotstyle=curve,linewidth=0.5pt,plotpoints=100]{-0.2}{1}{\fa}
\psplot[plotstyle=curve,linewidth=0.5pt,plotpoints=100]{-0.2}{1}{\fb}
\psplot[plotstyle=curve,linewidth=0.5pt,plotpoints=100]{-0.2}{1}{\fc}
\psplot[plotstyle=curve,linewidth=0.5pt,plotpoints=100]{-0.2}{1}{\fd}
\psplot[plotstyle=curve,linewidth=0.5pt,plotpoints=100]{-0.2}{1}{\fe}
\psplot[plotstyle=curve,linewidth=1.5pt,plotpoints=30]{-0.2}{1}{\fm}
\psplot[plotstyle=curve,linewidth=1.5pt,plotpoints=30]{3}{4.2}{\g}
\psplot[plotstyle=curve,linewidth=0.5pt,plotpoints=100]{3}{4.2}{\ga}
\psplot[plotstyle=curve,linewidth=0.5pt,plotpoints=100]{3}{4.2}{\gb}
\psplot[plotstyle=curve,linewidth=0.5pt,plotpoints=100]{3}{4.2}{\gc}
\psplot[plotstyle=curve,linewidth=0.5pt,plotpoints=100]{3}{4.2}{\gd}
\psplot[plotstyle=curve,linewidth=0.5pt,plotpoints=100]{3}{4.2}{\ge}
\psplot[plotstyle=curve,linewidth=1.5pt,plotpoints=30]{3}{4.2}{\gm}
\psplot[plotstyle=curve,linewidth=1.5pt,plotpoints=30]{1}{3}{-1}
\psplot[plotstyle=curve,linewidth=1.5pt,plotpoints=30]{1}{3}{1}
\psplot[plotstyle=curve,linewidth=0.5pt,plotpoints=30]{1}{3}{1 0.4 add sqrt}
\psplot[plotstyle=curve,linewidth=0.5pt,plotpoints=30]{1}{3}{1 0.8 add sqrt}
\psplot[plotstyle=curve,linewidth=0.5pt,plotpoints=30]{1}{3}{1 1.2 add sqrt}
\psplot[plotstyle=curve,linewidth=0.5pt,plotpoints=30]{1}{3}{1 1.6 add sqrt}
\psplot[plotstyle=curve,linewidth=0.5pt,plotpoints=30]{1}{3}{3 sqrt}
\psline[linewidth=1pt,linestyle=dashed](1,0)(1,2)
\psline[linewidth=1pt,linestyle=dashed](3,0)(3,2)
\psline[linewidth=0.5pt,linestyle=dashed](3.85,0)(3.85,2)
\psline[linewidth=0.5pt,linestyle=dashed](4.15,0)(4.15,2)
\psline[linewidth=0.5pt,linestyle=dashed](-0.15,0)(-0.15,2)
\psline[linewidth=0.5pt,linestyle=dashed](0.15,0)(0.15,2)
\rput[l](3.7,2.3){$\scriptscriptstyle\abs{\theta-\frac{\scriptscriptstyle 1}{\scriptscriptstyle 2}}
\leqslant \theta^\star$}
\rput[l](-0.4,2.3){$\scriptscriptstyle\abs{\theta+\frac{\scriptscriptstyle 1}{\scriptscriptstyle 2}}
\leqslant \theta^\star$}
\rput[l](2.2,2.7){$r$}
\rput[l](2.15,2){$ a_{q,N}$}
\rput[l](2.15,0.75){$\scriptstyle\rho_N$}
\rput[l](1.91,1.73){$\bullet$}
\rput[l](4.8,0.3){$\theta$}
\rput[l](0.65,-0.3){$\scriptstyle-L_0$}
\rput[l](2.9,-0.3){$\scriptstyle L_0$}
\rput[l](-0.3,-0.4){$\scriptscriptstyle-\frac{\scriptscriptstyle 1}{\scriptscriptstyle 2}$}
\rput[l](3.9,-0.4){$\scriptscriptstyle\frac{\scriptscriptstyle 1}{\scriptscriptstyle 2}$}
\end{pspicture}

\noindent
Note that $\rho_N$ is the height of the separatrix $\cS_0$ 
above $[-L_0,L_0]\times\{0\}$.
%
We identify the time-$1$ flow $\varphi=\Phi^{P_{V/N^2}}$ and its lift  to $\R^2$ satisfying
$\varphi^q(0,r_{q,N})=(1,r_{q,N})$. 
We set 
\[
	B_q(\ell,\ell'):=[-\ell,\ell]\times [r_{q,N}-\ell',r_{q,N}+\ell'].
\]
If there is no source of confusion, we may   identify ${B}_q(\ell,\ell')\subset\R^2$ with 
its image in $\A$. We also denote $\theta$ the projection to the first coordinate, 
from $\A$ to $\T$, or from $\R^2$ to $\R$.
\begin{prop}\label{prop:quantitative-disjoint-boxes} 
Assume $N$ is a positive  integer and $q$ is real number.  
 If   $N\geqslant N_0$ and $q\geqslant q_0 N$ then we have
\[
	\rho_N\leqslant r_{q,N}\leqslant 2\rho_N, \quad C_1 \frac{N^2}{q^4}\leqslant e_{q,N}
	\leqslant  C_2 \frac{N^2}{q^4}\quad~\text{and}~\quad C_3 \frac{N^3}{q^4}\leqslant 
	r_{q,N}-\rho_N\leqslant  C_4 \frac{N^3}{q^4};
\]
if $q\geqslant 2$ and $0<\delta<\rho_0/2$, if we set $\ell=\delta/(4N)$ and
 $\ell'=C_5\frac{N^3}{q^5}\delta$ then $B_q(\ell,\ell')$ is a $q$-adapted box with respect to  
$P_{V/N^2}$ and $\BdeN$, where $N_0$, $q_0$, $C_1$, $C_2$, $C_3$, $C_4$ and $C_5$ 
are positive constants (which depend only on $V$).
\\
Furthermore,  if $q\geqslant q_0'N$, with $(q_0'-q_0)N_0\geqslant 1$,  then  we have 
\[
\rho_N< r_{q+1,N}<r <r_{q-1,N}\leqslant 2\rho_N,\quad\text{for all $(\theta,r)\in 
B_q(\ell,\ell')$.}
\]
\end{prop}

For the convenience of the reader,  the proof of the proposition is split as follows. 
Lemma~\ref{lem:qualitative-disjoint-boxes}  provides a criterium for a rectangle to be 
$q$-adapted for the system $P_{V/N^2}$ and Lemma~\ref{lem:system-P_V}  extracts from 
the system $P_V$ all the properties we need.  This eventually  gives an explicite estimate of 
the size of a $q$-adapted box in Proposition~\ref{prop:size-box-N=1} when $N=1$. 

In order to generalise these resuts to the system when $N\geqslant N_0$ is arbitrary, we note 
that the systems $N^2P_{V/N^2}$ and $P_V$ are equal up to a  dilatation in the coordinate 
$r$. From this, Lemma~\ref{lem:V-V/N^2} gives explicit  dependances  between  the two 
systems  $P_V$  and $P_{V/N^2}$ for the main quantities we may compute. 

\begin{proof}
Assuming Proposition~\ref{prop:size-box-N=1}, Lemma~\ref{lem:system-P_V} and 
Lemma~\ref{lem:V-V/N^2}, we prove Proposition~\ref{prop:quantitative-disjoint-boxes}. 
\\[3pt]
$\bullet$~Since $e_{q,1}=(r_{q,1}^2-\rho_0^2)/2$, we have $\rho_0<r_{q,1}\leqslant 
2\rho_0$ if and only if $0<e_{q,1}\leqslant \frac{3}{2}\rho_0^2$. Let $q_0>0$ be the unique 
real number such that $e_{q_0,1}=3\rho_0^2/2$; 
Lemma~\ref{lem:system-P_V}.~\ref{Cond:period-asympt}) with $k=0$ shows that 
\[
	c_0/q\leqslant e_{q,1}^{1/4}\leqslant c_1/q,\quad \text{for $q\geqslant q_0$, }
\]
where $c_0$ and $c_1$ are positive constants (which depend only on $V$). Therefore, for 
$q\geqslant q_0$, we have $0<e_{q,1}\leqslant e_{q_0,1}$, so 
\[
	\rho_0<r_{q,1}\leqslant 2\rho_0\quad;\quad \frac{c_0^4}{q^4}\leqslant e_{q,1}\leqslant 
	\frac{c_1^4}{q^4}.
\]
Furthermore, we have $e_{q,1}=\frac{1}{2}(r_{q,1}+\rho_0)(r_{q,1}-\rho_0)$. 
Since  $\rho_0<\frac{1}{2}(r_{q,1}+\rho_0)\leqslant 2\rho_0$, it follows that
\[
	\frac{c_0^4}{2\rho_0}\frac{1}{q^4}\leqslant\frac{e_{q,1}}{2\rho_0}\leqslant r_{q,1}-
	\rho_0\leqslant \frac{e_{q,1}}{\rho_0}\leqslant \frac{c_1^4}{\rho_0}\frac{1}{q^4}.
\] 
Since Lemma~\ref{lem:V-V/N^2} shows that $r_{q,N}=r_{q/N,1}/N$, we obtain that 
\[
	\rho_N<r_{q,N}\leqslant 2\rho_N~\text{and}~
	C_3 \frac{N^3}{q^4}\leqslant r_{q,N}-\rho_N\leqslant  C_4 \frac{N^3}{q^4}, 
	~\quad\text{provided $q/N\geqslant q_0$,}
\]
with $C_3:=c_0^4/(2\rho_0)$ and $C_4:=c_1^4/\rho_0$.\\
In a similar way, since Lemma~\ref{lem:V-V/N^2} shows that 
$e_{q,N}=e_{q/N,1}/N^2$, we obtain that 
\[
	C_1\frac{N^2}{q^4}\leqslant e_{q,N}\leqslant C_2 \frac{N^2}{q^4},
	~\quad\text{provided $q/N\geqslant q_0$,}
\]
with $C_1:=c_0^4$ and $C_2:=c_1^4$.
\\[5pt]
$\bullet$~We now construct the  $q$-adapted box.  Proposition~\ref{prop:size-box-N=1} with 
 $V/N^2$ substituted to $V$ shows that the rectangle $B_q(\ell,\ell')$ is $q$-adapted with 
respect to $P_{V/N^2}$ and $\BdeN$, provided we have  $0<\delta/N<\delta_{0,N}$, 
$\ell=(\delta/N)/4$ and
 \[
\ell'=\frac{\delta}{N}\cdot\bigg(\frac{r_{q,N}-\rho_N}{r_{0,N}}\bigg)^{5/4}.
\]
Here $r_{0,N}$ and $\delta_{0,N}$ are the corresponding constants for $P_{V/N^2}$
given by~\eqref{subeq:r_0-d_0},  that is 
\begin{subequations}\label{subeq:r_0-d_0-N}
\begin{align}
r_{0,N}:=2 \cdot (40 C_{0,N}\rho_N^2)^{4/5},   \\[3pt]
\delta_{0,N}=\min\bigg(L_0;{\rho_N}/{2};\frac{r_{0,N}^{5/4}}{2\rho_N^{1/4}}\bigg)=
\min\bigg(L_0;{\rho_N}/{2};20\cdot 2^{5/4}\rho_N^{7/4} C_{0,N}\bigg),
\end{align}
\end{subequations}
where $C_{0,N}=C_0 N^{7/4}$  as it appears in~\ref{lem:V-V/N^2}~\ref{Cond:Delta-T-N}).
Since $\rho_N=\rho_0/N$, we observe that  $C_{0,N}\rho_N^{7/4}=C_0\rho_0^{7/4}$, so 
$\delta_{0,N}=\rho_0/(2N)$ if $N\geqslant N_0$
with $N_0$ large enough (depending only on $L_0$, $C_0$ and $\rho_0$), and we have
\[
r_{0,N}=\frac{C_6 }{N^{1/5}},\quad\text{with $C_6=2\cdot (40 C_0\rho_0^2)^{4/5}$.}
\]
 This implies that 
\[
	\ell'\geqslant \frac{\delta}{N}\bigg(C_3\frac{N^3}{q^4}\frac{N^{1/5}}{C_6}\bigg)^{5/4}
	=C_5\frac{N^3}{q^5}\delta,\quad\text{with $C_5=(C_3/C_6)^{5/4}$.}
\]
Since $C_5$ depends only on $V$, this proves that $B_q(\ell,\ell')$ is a suitable $q$-adapted box.
\\[3pt]
$\bullet$~Now we assume that $q-1\geqslant q_0N$ and we prove the estimates of the proposition
on $r_{q-1,N}$ and $r_{q+1,N}$. Notice that the assumption  $q\geqslant q_0N+(q_0'-q_0)N_0$
 implies that $q-1\geqslant q_0N$, hence $\rho_N< r_{q+1}<r_{q-1,N}\leqslant 2\rho_N$.

Furthermore, since there exist $(q-1)$-adapted and $(q+1)$-adapted boxes repectively  
centered at $a_{q-1,N}$ and $a_{q+1,N}$, the image  points $\varphi^{q}(a_{q+1,N})$ and 
$\varphi^q(a_{q-1,N})=\varphi(a_{q-1,N})$ do not belong to ${\BdeN}$, so $a_{q-1,N}$ 
and $a_{q+1,N}$ are not in $B_q(\ell,\ell')$. This implies that 
\[
	B_q(\ell,\ell')\subset \{(\theta,r)\mid r_{q+1}<r<r_{q-1,N} \}
\]
and the proof of the proposition is complete.
\end{proof}

\begin{lemma}\label{lem:qualitative-disjoint-boxes} 
Assume $q\geqslant 2$ and $0<\delta<\min (\rho_0,L_0)$. We set  $m_q={a}_{q,N}-(\ell,\ell')$ 
and $M_q={a}_{q,N}+(\ell,\ell')$. Then there exists $\ell>0$ and $\ell'>0$ satisfying the 
following properties. 
\begin{description}
\item[\qquad (a)]  $\ell < \delta/(2N)$ and $\rho_N+\ell'<r_{q,N}$;
\item[\qquad (b)] $\delta/N < \theta(\varphi(m_q))$; 
\item[\qquad (c)] $\theta(\varphi^{q-1}(M_q)) <1-\delta/N$;
\item[\qquad (d)] $\abs{\theta(\varphi^{q}(m_q))-1}<\delta/(2N)$ and 
	$\abs{ \theta(\varphi^{q}(M_q))-1} <\delta/(2N)$.\vspace{3pt}
\end{description}
Furthermore, if these conditions are fulfilled then  $B_q=B_q(\ell,\ell')$ is a $q$-adaped box. 
\end{lemma}
\begin{proof} 
We abreviate $a_q:=a_{q,N}$ and $r_q:=r_{q,N}$. Note that $r_q>\rho_N>\delta/N$ since 
we have chosen $0<\delta<\rho_0/2$. Assume first that $\ell=\ell'=0$, so $m_q=M_q= a_q$.
Condition~\ref{Cond:torsion}) in Lemma~\ref{lem:system-P_V} implies that  either 
$r_q \leqslant L_0$ and $\theta(\varphi({a}_q))=r_q$,  or $r_q>L_0$ and 
$\theta(\varphi(a_q))> L_0$. In both cases, we have $\theta(\varphi( a_q))>\delta/N$. In the 
same way,  either $r_q \leqslant L_0$ and $\theta(\varphi^{q-1}(a_q))=1-r_q$, or $r_q>L_0$ 
and $\theta(\varphi^{q-1}( a_q)) < 1-L_0$.  In both cases, we have $\theta(\varphi^{q-1}( a_q))
<1-\delta/N$. At last, we have $\theta(\varphi^{q}(a_q))=1$. Thus Conditions $(a)$, $(b)$, $(c)$ 
and $(d)$ are fulfilled for  $\ell=\ell'=0$. Since these conditions are open, this implies that this 
already holds for  $(\ell,\ell')$ close enough to zero. \\[5pt]
Now we check is that  these conditions imply that $B_q$ is $q$-adapated.  Let assume that 
$(\ell,\ell')$ satisfies $(a)$, $(b)$, $(c)$ and $(d)$.~\\[5pt]
-- We first observe that we have $B_q\subset \BdeN$ since $\ell < \delta/(2N)$.  Furthermore, 
	we may observe that the condition  $\rho_0/N+\ell'<r_q$ implies that $B_q$ is completely 
	above  $\cS_0$.~\\[3pt]
-- We now  prove~\ref{Cond:disjonction-q-box}) in Definition~\ref{def:q-box}. 
	Notice that $\theta(\varphi^t(m))$ is increasing with $t>0$ if $m\in {B}_q(\ell,\ell')$ since 
	$m$ is above the separatrix $\cS_0$. Therefore~(b) shows that $\theta(\varphi^t
	(m_q)) > \delta/N$ for $t\geqslant 1$. From this, 
	 Lemma~\ref{lem:system-P_V}.~\ref{Cond:torsion}) implies that 
	$\theta(\varphi^t(m)) > \delta/N$ for all $m\in {B}_q(\ell,\ell')$ and $t \geqslant 1$. 
	In the same way, (c) shows that $\theta(\varphi^{t}(M_q))< 1-\delta/N$ for $t\leqslant q-1$,  
	so  Lemma~\ref{lem:system-P_V}.~\ref{Cond:torsion})
	 implies that $\theta(\varphi^{t}(m))< 1-\delta/N$ for all 
	$m\in B_q(\ell,\ell')$ and $t\leqslant q-1$. This proves that $\delta/N <\theta(\varphi^{t}(m))
	<1-\delta/N$, for $1\leqslant t\leqslant q-1$ and $m\in B_q(\ell,\ell')$, which 
	implies~\ref{Cond:disjonction-q-box}).~\\[3pt]
-- At last, we prove~\ref{Cond:injection-q-box}) in Definition~\ref{def:q-box}. Notice 
	that~(d) shows that $\varphi^{q}(m_q)$ and $\varphi^{q}(M_q)$ belong to the band $(1,0)+
	{\cal B}_{\delta/(2N)}$. Therefore Condition~\ref{Cond:torsion}) in 
	Lemma~\ref{lem:system-P_V} implies that $1-\delta/(2N) <\theta(\varphi^{q}(m) )< 1+
	\delta/(2N)$ for all $m\in B_q(\ell,\ell')$, which implies~\ref{Cond:injection-q-box}).\\[5pt]
Thus Conditions~(a), (b), (c) and (d) imply the required properties for $B_q$ and the proof 
of the proposition is complete.
\end{proof}
\begin{lemma}\label{lem:system-P_V}
The system defined by $P_V$ satisfies the following properties.
\begin{enumerate}
\item\label{Cond:torsion}
 	Assume $t>0$, $\theta_0\in\R$, $r_0\geqslant 0$ and $P_V(\theta_0,r_0)=0$. Then 
	$\theta\circ\varphi^{tP_V}(\theta_0,r)$ is an  increasing function of  $r \in [r_0,+\infty)$. %
\item \label{Cond:period}
	$T_V$ is an increasing bijection from $(0,+\infty)$ onto itself;
\item \label{Cond:period-asympt}
	For  each integer $k\geqslant 0$, we have 
	\[
				T_V^{(k)}(e)\sim\frac{ 1}{e^{k+\frac{1}{4}}}
							\bigg(\begin{smallmatrix}-\frac{1}{2}\\ ~k\end{smallmatrix}\bigg)
				\int_0^{+\infty}\frac{dx}{(1+x^4)^{k+1/2}}~\text{as $e>0$ tends to zero.}
	\]
\item\label{Cond:Delta_P}
 	If we set $\Delta T(r_2,r_1)=T_V(P_V(0,r_1))-T_V(P_V(0,r_2))$ then we have 
\[
	0<\Delta T(r_2,r_1) \leqslant C_0\frac{r_1(r_2-r_1)}{\left(r_1-\rho_0\right)^{5/4}},
	~\text{provided  $ \rho_0<r_1<r_2 \leqslant 3 \rho_0$,}
\]
\end{enumerate}
where $C_0$ is a positive constant (which depends only on $V$).
\end{lemma}
\begin{proof}
{\bfseries~\ref{Cond:torsion}).}~For every  $t>0$, we have 
\[
	t=\int_{ \theta_0}^{\theta\circ\varphi^t(\theta_0,r)}\frac{du}{\sqrt{2(e(r)-V(u))}}.
\]
	Since the  energy $e(r)=P_V(\theta_0,r)=\frac{1}{2}r^2+V(\theta_0)$ is increasing with 
	$r \geqslant 0$, it follows that ${\theta}\circ\varphi^t(\theta_0,r)$ is increasing with  
	$r\in [r_0,+\infty)$. \\[3pt]
{\bfseries \ref{Cond:period}).}~It  follows  from~\eqref{eq:T_V} that $T_V$ is a continuous 
	and decreasing function and converges to zero at infinity. We have $T_V(0)=+\infty$
	because $V(\theta)=-(\theta-1/2)^4$ for $\abs{\theta-1/2}\leqslant \theta^\star$.~\\[3pt]
{\bfseries \ref{Cond:period-asympt}).} ~We compute  
	\begin{align*}
	\displaystyle T^{(k)}_V(e)
		& \displaystyle 
			=\bigg(\begin{smallmatrix}-\frac{1}{2}\\ ~~k\end{smallmatrix}\bigg)
				\int_\T \frac{d\theta}{\sqrt{2}(e-V(\theta))^{\frac{1}{2}+k}}\\ %
		& 
			= \frac{1}{\sqrt{2}}
			\bigg(\begin{smallmatrix}-\frac{1}{2}\\ ~~k\end{smallmatrix}\bigg)
			\int_{-\frac{1}{2}+\theta^\star}^{\frac{1}{2}-\theta^\star}\frac{d\theta}{(e-V
			(\theta))^{k+\frac{1}{2}}}
			+
			\bigg(\begin{smallmatrix}-\frac{1}{2}\\ ~~k\end{smallmatrix}\bigg)
			\sqrt{2}\int_0^{\theta^\star} \frac{d\theta}{(e+\theta^4)^{k+\frac{1}{2}}}.
	\end{align*}
Since  $\int_{-\frac{1}{2}+\theta^\star}^{\frac{1}{2}-\theta^\star}
	 \frac{d\theta}{(e-V(\theta))^{k+\frac{1}{2}}}\leqslant \int_{-\frac{1}{2}
	+\theta^\star}^{\frac{1}{2}-\theta^\star}\frac{d\theta}{(-V(\theta))^{k+\frac{1}{2}}}$ 
	is bounded independantly of $e$, it follows that 
\[
	T^{(k)}_V(e)
	\sim \sqrt{2}\bigg(\begin{smallmatrix}-\frac{1}{2}\\ ~~k\end{smallmatrix}\bigg)
	\int_0^{\theta^\star} \frac{d\theta}{(e+\theta^4)^{k+\frac{1}{2}}}=
 	\sqrt{2}\bigg(\begin{smallmatrix}-\frac{1}{2}\\ ~~k\end{smallmatrix}\bigg)
	\frac{1}{e^{k+\frac{1}{4}}}
	\int_0^{\theta^\star/e^{1/4}} \frac{d\theta}{(1+\theta^4)^{k+\frac{1}{2}}}
\]
	as $e$ tends to zero, and this completes the  proof of~\ref{Cond:period-asympt}).~\\[3pt]
{\bfseries \ref{Cond:Delta_P}).} Notice  that $P_V(0,r)=(r^2-\rho_0^2)/2<4 \rho_0^2$ if 
	$\rho_0<r<3\rho_0$.  Furthermore, Condition~\ref{Cond:period-asympt}) shows that there 
	exists a constant $B$ such that  $0<-T_V'(e) \leqslant B e^{-5/4}$  if $0< e\leqslant 4 
	\rho_0^2$~; we set $\tau(r)=\Delta T(r,r_1)$,  for $r_1 \leqslant r \leqslant r_2$, so we have:
\[
	  0\leqslant \tau'(r) =-rT_V'(H(0,r)) \leqslant Br((r^2-\rho_0^2)/2)^{-5/4}\leqslant 
		Br_1((r_1^2-\rho_0^2)/2)^{-5/4},
\]
	hence the estimate as claimed, with $C_0=B\rho_0^{-5/4}$ since $(r_1^2-\rho_0^2)/2\geqslant \rho_0(r_1-\rho_0)$.
\end{proof}
\begin{prop}\label{prop:size-box-N=1}
There exist positive constants $\delta_0$ and $r_0$ (which depends only on $V$)
such that for each integer $q\geqslant 2$ satisfying  $\rho_0<r_{q,1} \leqslant 2\rho_0$ and 
for $0<\delta<\delta_0$, if we  set 
\[
	\ell=\delta/4\quad \text{and}\quad\ell'=\delta\cdot\bigg(\frac{r_{q,1}-\rho_0}{r_0}\bigg)^{5/4}
\] 
then $r_{q,1}-\ell'>\rho_0$ and  $B_q(\ell,\ell')$ is $q$-adapted
with respect to $\cal B_\delta$ and  the system $P_V$. 
\end{prop}
\begin{proof} Here we abreviate $r_q:=r_{q,1}$ and $a_q:=a_{q,1}$
and we prove that $\ell$, $\ell'$, $m_q=a_q-(\ell,\ell')$ and $M_q=a_q+(\ell,\ell')$
satisfy Conditions~(a),(b),(c) and~(d) in Lemma~\ref{lem:qualitative-disjoint-boxes} with $N=1$, 
provided that $r_0$ and $\delta_0$ satisfy suitable conditions.
We recall that we assume that $\delta<\rho_0/2$, so we may already set the constraint 
\begin{equation}\label{Cond:step0-size-box}
	\delta_0\leqslant\rho_0/2.
\end{equation}
We decompose the proof of the lemma  in five steps~: first we prove that 
$B_q(\ell,\ell')$ is above the separatrix $r=\rho_0$.
Then we  estimate the  variation of 
the period inside $B_q(\ell,\ell')$ and, at last,  we check~$(b)$, $(c)$ and $(d)$.~\\[3pt]
{\bf Step 1~:} we prove that  $\ell'< (r_q-\rho_0)/2$. Notice that $0<r_q-\rho_0\leqslant \rho_0$, 
	so we have 
\begin{equation}
\label{Cond:step1-size-box}
	\frac{\ell'}{r_q-\rho_0}=\frac{\delta}{r_0^{5/4}}(r_q-\rho_0)^{1/4} 
	\leqslant\frac{\rho_0^{1/4}}{r_0^{5/4}}\delta
	<\frac{1}{2},\quad\text{provided that  $\delta_0\leqslant\frac{r_0^{5/4}}{2\rho_0^{1/4}}$.}
\end{equation}
	This implies that $r_q-\ell'>\frac{1}{2}(r_q+\rho_0)>\rho_0$. Since we have
	$\ell=\delta/4<\delta/2$, we obtain Condition~(a) of 
	Lemma~\ref{lem:qualitative-disjoint-boxes}. ~\\[3pt]
{\bf Step 2~:} we prove that $\ell+ (r_q+\ell')\tau<\delta/2$, where $\ell=\delta/4$ and 
 $\tau=\Delta T(r_q+\ell',r_q-\ell')$.  
	Step~1 shows  
	that $\ell'<\frac{1}{2}(r_q-\rho_0)\leqslant \frac{1}{2}\rho_0$.
	Set $r_q'=r_q-\ell'$; since $\rho_0<r_q-\ell'<r_q+\ell' \leqslant 3\rho_0$, 
	Lemma~\ref{lem:system-P_V}.~\ref{Cond:Delta_P}) shows that 
	$
	\tau\leqslant 2C_0{\ell'r_q'}{(r_q'-\rho_0)^{-5/4}}$.
 	 Since $r_q'-\rho_0=r_q-\rho_0-\ell'>\frac{r_q-\rho_0}{2}$
	and $r_q+\ell' <\frac{5}{2}\rho_0$, since $r_q' \leqslant r_q \leqslant 2\rho_0$ , it follows that 
	\[
	\tau (r_q+\ell') < 5\tau\rho_0/2\leqslant \frac{10~C_0\ell'\rho_0^2}{((r_q-\rho_0)/2)^{5/4}}
	\leqslant \frac{10~ C_0\ell' \rho_0^2}{((r_q-\rho_0)/2)^{5/4}}
	=\frac{10~C_0 \rho_0^2}{(r_0/2)^{5/4}}\delta\leqslant \frac{\delta}{4}, 
	\]
 which proves the claimed  estimate, provided 
\begin{equation}\label{Cond:step2-size-box}
	\frac{10C_0 \rho_0^2}{(r_0/2)^{5/4}}\leqslant \frac{1}{4}.
\end{equation}
~\\[3pt]
{\bf Step 3~:}  we prove~(d). Set $a_q^+=(\ell,r_q)$, $a_q^-=(-\ell,r_q)$
	and  notice that 
	$\theta(\varphi^q(a_q^+)=1+\ell$, $\theta(\varphi^q(a_q^-)=1-\ell$, provided that 
	$\ell\leqslant L_0$; we assume that 
\begin{equation}\label{Cond:step3-size-box}
	\delta_0\leqslant L_0.
\end{equation}
	 Since ${\theta}(M_q)=\theta(a_q^+)$,  $\theta(m_q)=\theta(a_q^-)$
	and since Step~2 implies that $\ell+\tau(r_q+\ell')\leqslant \delta/2\leqslant L_0$, 
	Lemma~\ref{lem:system-P_V}.~\ref{Cond:torsion}) implies that 
	\[
	\bigg\{
		\begin{array}{l}
			1+\ell \leqslant {\theta}(\varphi^q({M}_q) )\leqslant 1+\ell +\tau(r_q+\ell');\\
			1-\ell -\tau(r_q-\ell') \leqslant {\theta}(\varphi^q({m}_q) )\leqslant 1-\ell.
		\end{array}
	\]
	This implies  that 
	$1+\delta/4 \leqslant {\theta}(\varphi^q({M}_q) )\leqslant 1+\delta/2$
	and $1-\delta/2 \leqslant {\theta}(\varphi^q({m}_q) )\leqslant 1-\delta/4$,  
	which proves~(d).~\\[3pt] 
{\bf Step 4~:}  we prove~(c). Recall that Step~3 implies that 
	$1 \leqslant {\theta}(\varphi^q({M}_q) )\leqslant 1+\delta/2\leqslant 1+L_0$, 
	so either $\delta-(r_q+\ell')<-L_0$ and ${\theta}(\varphi^{q-1}({M}_q) )<1-L_0$,  
	or ${\theta}(\varphi^{q-1}({M}_q) )<1+\delta-(r_q+\ell')$.  Since $L_0>\delta$ and 
	$r_q+\ell'>\rho_0 \geqslant 2\delta$, we obtain in both cases that  
	${\theta}(\varphi^{q-1}({M}_q) )<1-\delta$, which completes the proof of~(c).~\\[3pt]
{\bf Step 5~:}  we prove~(b). Either $-\ell+(r_q-\ell')>L_0$, so 
	${\theta}(\varphi({m}_q) )>L_0\geqslant \delta$, or ${\theta}(\varphi({m}_q))>-\ell+(r_q-\ell')
	>-\ell+\rho_0>2\delta-\ell=7\delta/4>\delta$. In both cases, this proves~(b).~\\[5pt]
Thus we obtain Conditions~(a), (b), (c) and (d) in Lemma~\ref{lem:qualitative-disjoint-boxes}  provided  
Conditions~\eqref{Cond:step0-size-box}, \eqref{Cond:step1-size-box}, \eqref{Cond:step2-size-box}
and~\eqref{Cond:step3-size-box} are satisfied, hence the  proposition holds true if we set
\begin{subequations}\label{subeq:r_0-d_0}
\begin{align}
r_0:=2 \cdot (40 C_0\rho_0^2)^{4/5},   \\[3pt]
\delta_0=\min
\bigg(L_0;\frac{\rho_0}{2};\frac{r_0^{5/4}}{2\rho_0^{1/4}}\bigg)=
\min
\bigg(L_0;{\rho_0}/{2};20\cdot 2^{5/4}\rho_0^{7/4} C_0\bigg),  
\end{align}
\end{subequations}
where $C_0$ is the constant in Lemma~\ref{lem:system-P_V}.~\ref{Cond:Delta_P}).
Thus the constants $r_0$ and $\delta_0$ depend only on $V$ and 
 the proof of the  proposition is complete. 
\end{proof}
\begin{lemma}\label{lem:V-V/N^2}
Assume $N$ and  $q$ are  positive real numbers. Then the following holds true. 
\begin{enumerate}
\item \label{Cond:T_V-N}
 				For all $e>0$, we have $T_{V/N^2}(e)=NT_V(N^2e)$;%
\item   \label{Cond:P_V-N}
				for all $r\in\R$, we have  $N^2P_{V/N^2}(0,r)=P_V(0,Nr)$;%
\item  \label{Cond:e_q}
		$e_{q,N}=\frac{1}{N^2}e_{q/N,1}$ and  $r_{q,N}=\frac{1}{N}r_{q/N,1}$;%
\item\label{Cond:T^k-N}
for each  integer $k\geqslant 0$, we have  
	 $T^{(k)}_{V/N^2}(e_{q,N})=N^{2k+1}	 T^{(k)}_{V}(e_{q/N,1})
	\sim (-1)^k\beta_k \frac{q^{4k+1}}{N^{2k}}$ uniformly 
	as ${q}/{N}$ tends to infinity, where $\beta_k$ is a positive constant (which depends
	only on $k$);%
\item\label{Cond:Delta-T-N}
	with    $\Delta_NT(r_2,r_1)=T_{V/N^2}(P_{V/N^2}(0,r_1))
	-T_{V/N^2}(P_{V/N^2}(0,r_2))$, we have\\[3pt]
 	$0<\Delta T_N(r_2,r_1)\leqslant C_0{N}^{7/4}
	\frac{~r_1(r_2-r_1)}{\left(r_1-\rho_N\right)^{5/4}}$, provided 
	$\rho_N <r_1<r_2 \leqslant 3 \rho_N$,
\end{enumerate}
 where $C_0$ is a positive constant (which depends only on $V$).
\end{lemma}
\begin{proof}
The first two conditions   follow directly from the formulas~\eqref{eq:P_V} 
and~\eqref{eq:T_V}. \\
\ref{Cond:e_q})~We use~\ref{Cond:T_V-N}) to compute 
\[
	T_V(e_{q/N,1})=q/N=T_{V/N^2}(e_{q,N})/N=T_V(N^2e_{q,N}),
\]
so  $e_{q/N,1}=N^2e_{q,N}$ for $T_V$ is one to one. 
In a similar way, we use~\ref{Cond:P_V-N}) to compute
\[
P_V(0,Nr_{q,N})=N^2P_{V/N^2}(0,r_{q,N})=N^2 e_{q,N}e_{q/N,1}=P_V(0,r_{q/N,1}),
\]
so $Nr_{q,N}=r_{q/N,1}$ since the function  $P_V(0,\cdot )$ is one to one, and this 
proves~\ref{Cond:e_q}).
\\[3pt]
\ref{Cond:T^k-N})
	 It follows from~\ref{Cond:T_V-N})   that 
\begin{align*}
	T^{(k)}_{V/N^2}(e_{q,N})=N\frac{d^k}{de^k}_{\mid_{e=e_{q,N}}}
		\!\!\!\!\big(T_V(N^2e)\big)=N^{2k+1}T_V^{(k)}(N^2 e_{q,N})\\  %
	=N^{2k+1}T^{(k)}_{V}(e_{q/N,1})
\end{align*}
and  this completes the proof of the equality in~\ref{Cond:T^k-N}). Furthermore, 
since $e_{q/N,1}$ tends to zero as $q/N$ tends to infinity, 
Statement~\ref{Cond:period-asympt})
implies that
\begin{align*}
		N^{2k+1}	T_V^{(k)}(e_{q/N,1})	\sim \frac{ 	N^{2k+1}}{(e_{q/N,1})^{k+\frac{1}{4}}}
							\bigg(\begin{smallmatrix}\!-\frac{1}{2}\\ \!~~k\end{smallmatrix}\bigg)
							\int_0^{+\infty}\frac{dx}{(1+x^4)^{k+1/2}}\\ %
~\text{and}~{q}/{N}=T_{V}(e_{q/N,1})\sim\frac{ 1}{e_{q/N,1}^{{1}/{4}}}
														\int_0^{+\infty}\frac{dx}{(1+x^4)^{1/2}}.
\end{align*}
This implies~\ref{Cond:T^k-N}) with 
\[
\displaystyle
\beta_k=\left(\begin{smallmatrix}k-\frac{1}{2}\\ \!~~k\end{smallmatrix}\right)
				\displaystyle\int_0^{+\infty}\!\!\!\!\!\!\!\!\frac{dx}{(1+x^4)^{k+1/2}}
		\bigg\slash \left(\displaystyle\int_0^{+\infty}\!\!\!\!\!\frac{dx}{(1+x^4)^{1/2}}
\right)^{4k+1}\!\!\!\!.
\]
\ref{Cond:Delta-T-N})  follows from 
Lemma~\ref{lem:system-P_V}.~\ref{Cond:Delta_P}).
 This holds true because
this asumptions on $r_1$, $r_2$  and~\ref {Cond:P_V-N})
imply that  $\rho_0< Nr_k \leqslant 3\rho_0$, $k=1,2$,  and 
 \begin{align*}
T_{V/N^2}(P_{V/N^2}(0,r_k))=NT_{V}\Big(N^2P_{V/N^2}(0,r_k)\Big)=
NT_V( P_V(0,Nr_k)),\\[3pt]
\text{so}\quad
0<\Delta_N(r_1,r_2)=N\Delta(Nr_1,Nr_2)\leqslant C_0\frac{r_1(r_2-r_1)N^{3-5/4}}{
\big(r_1-\rho_N\big)^{5/4}},
\end{align*}
which proves~\ref{Cond:Delta-T-N})
and completes the proof  of the lemma.
\end{proof}
%
%
%
%
%
\subsubsection{Local form \texorpdfstring{of~$G^q$}{}}
 Here and in the following unless mentioned otherwise, the numbers $\ell$, $\ell'$, 
$N_0$ and $q_1$ are   as in Proposition~\ref{prop:quantitative-disjoint-boxes},
assuming $q\geqslant q_0'N$ and $N\geqslant N_0$.
We abreviate $G=G_{N,\mu}=\Phi^{\mu W_N}\circ \Phi^{P_{V/N^2}}
$ (see Section~\ref{sec:proofpseudopend-def}) and  $B_{q}=B_q(\ell,\ell')$.
\begin{prop}\label{prop:local_form_FqNM}
In $B_{q}\subset \A$, the iterated map $G^q$ coincides with
\[
\FNqm \defeq \Phi^{\mu  W_N} \circ \Phi^{A_{q,N}},
\]
where $\displaystyle A_{q,N}(r) := \int_{e_{q,N}}^{P_{V/N^2}(0,r)}\big(q-T_{V/N^2}(h)\big)~dh
$ on $B_q$. 
\end{prop}
\begin{proof}
We abreviate $\varphi=\Phi^{P_{V/N^2}}$ and $V_N=V/N^2$. 
For $(\theta,r)\in \R^2$ satisfying  $P_{V_N}(\theta,r)>0$,  let $\Psi(\theta,r)=(\tau,h)$ denote the 
time-energy coordinates with  
\[
\tau=\int_0^\theta \frac{du}{\sqrt{2\big(P_{V_N}(\theta,r)-V_N(u)\big)}}\quad;\quad
h=P_{V_N}(\theta,r).
\]
Note that $\Psi\circ\varphi\circ\Psi^{-1}(\tau,h)=(\tau+1,h)$ and $\Psi(m+\theta,r)=
\Psi(\theta,r)+\big(m\, T_{V_N}(h),0\big)$, for all $m\in\Z$ and  $h=P_{V_N}(\theta,r)$.
Since  $V_N$  is constant on ${\B}_{L_0}$, we have 
\[
\tau =\int_0^\theta \frac{du}{\sqrt{2\big(P_{V_N}(\theta,r)-V_N(\theta)\big)}}
=\frac{\theta}{r},\quad \text{for $\abs{\theta}\leqslant L_0$.}
\]
 Since we have  
$\varphi^q(B_q)\subset {\B}_{\delta/2}\subset {\B}_{L_0}$
and  $\varphi^q$ preserves $P_{V_N}$, which is an increasing function depending 
only on $r$ in ${\B}_{L_0}$,  there exist a continuous function 	
$\theta_q~:B_q\to (-L_0,L_0)$ and a constant integer $m_0\in\Z$ such that $\varphi^q(\theta,r)=
	(m_0+\theta_q(\theta,r),r)$ on ${B}_q\subset\R^2$.  Furthermore, we have $m_0=1$
as it may be  checked
at $a_{q,N}$,  since we have $\varphi^q(a_{q,N})=(1,r_{q,N})$ and   $\theta_q(a_{q,N})=0$. 
Therefore, we have proved that $\varphi^q$ coincides on $B_q$ with 
\[
\varphi^q(\theta,r)=(1+\theta_q(\theta,r),r), ~\text{so}~\Psi\circ\varphi^q(\theta,r)=
\big(\theta_q(\theta,r)/r+T_{V_N}(h), h\big),~\text{with $h=P_{V_N}(\theta,r)$.} 
\]
Since we also have $\Psi\circ\varphi^q(\theta,r)=\Psi(\theta,r)+(q,0)=(\theta/r+q, h)$, 
this implies that 
\[
\theta_q(\theta,r)=\theta + r (q-T_{V_N}(h))=\theta-\partial_r A_{q,N}(r).
\]
Thus, on the one hand, 
$\varphi^q$ coincides on $B_q$ with  the time-$1$ flow $\Phi^{A_q}$ of the system
\[
\left\{
\begin{array}{l}
\theta'=-\partial_r A_{q,N}(r);\\
r'=0=\partial_\theta A_{q,N}(r).
\end{array}
\right.
\]
On the other hand, we have $\varphi^k(B_q)\cap {\BdeN}=\varnothing$, 
so $G^k=\varphi^k$ for $1\leqslant k \leqslant q-1$, and $G^q=\Phi^{\mu W_N}\circ \varphi^q$
on $B_q$.\\
From these two conclusions, the proposition follows. 
\end{proof}

\subsubsection{The Taylor expansion \texorpdfstring{of~$G^q$ at~$a_q$}{of the q iteration of G}}
%
This section carries out the first step in the proof of the existence of invariant 
curves in $B_q$ for  the map $G^q=\FNqm$.
The goal is to prove that we can find   complex coordinates  in which 
$a_{q,N}$ is mapped to zero and $F_{q,N,\mu}$ takes the form 
\begin{equation}
\label{eq:taylor-expansion}
\FNqm(z)=\lambda \Big(z+\sum_{k=2}^n P_k(z)+\varepsilon(z)\Big).
\end{equation}
This is achieved in Proposition~\ref{prop:complex-coord} 
and Corollary~\ref{Cor:Taylor-expansion} below.
Here  $n$ is an arbitrary large integer (but not depending on $q$, $N$ and $\mu$), 
$P_k$ is a  homogeneous polynomial of degree $k$ for $2\leqslant k \leqslant n$
and the error term $\varepsilon$   is  small enough  up to 
$n$ derivatives.
Note that the change of coordinates need not be symplectic in our setting. 
\\[10pt]
\noindent
{\bfseries Asymptotic behaviours of $\mathbf{A_{q,N}}$ on $\mathbf{B_q}$}.--
To achieve~\eqref{eq:taylor-expansion}  in a quantitative way, 
we must control the derivatives 
of the map $A_{q,N}$ 
near $a_{q,N}$. For that purpose, it is convenient to introduce the following notation.
%
\begin{Not}
\label{sec:notation-equivalents}
Here $E$ denotes any set of parameters;
for $f_1:E\to\R_+$ and $f_2~:E\to \R_+$,  we write $f_1=\jO_E(f_2)$, 
or $f_1(p)=\jO_E(f_2(p))$, 
or $f_1(p)=\jO(f_2(p))$ uniformly for $p\in E$,  if there 
	exists a constant $C>0$ (which does not depend on $p$) such that 
	\[
		\forall p\in E, \quad  f_1(p)\leqslant C f_2(p).
	\]
	We write $f_1\underset{E}{\asymp}f_2$, or $f_1(p)\underset{E}{\asymp}f_2(p)$,
 or $f_1(p)\asymp f_2(p)$ uniformly for $p\in E$, 
  if we have $f_1=\jO_E(f_2)$ and $f_2=\jO_E(f_1)$.\vspace{5pt}
\end{Not}
%
\noindent
We recall that for $q\geqslant q_0' N$, $B_q$ is contained in the annulus 
$\big\{ r_{q+1,N}\leqslant r \leqslant r_{q-1,N}\big\}\subset \A$.
\begin{prop}
\label{prop:asymptotic-T-Aq}
We have 
\begin{eqnarray}
\label{eq:T1}
\forall n\geqslant 0,  &(-1)^n T_{V/N^2}^{(n)}\big(P_{V/N^2}(0,r)\big)
																				\underset{~E_1}{\asymp}q^{4n+1}/N^{2n};\\
\label{eq:Aq1} 
\forall n\geqslant 1,  &  
	\abs{A_{q,N}^{(n)}(r)} =\jO_{E_1} \big( q^{4n-3}/N^{3n-2}\big);\\ %
\label{eq:Aq2} \forall n\geqslant 2, &  \quad 
	 (-1)^nA_{q,N}^{(n)}(r) \underset{~E_n}{\asymp}  q^{4n-3}/N^{3n-2}, %
\end{eqnarray}
where $E_n=\big\{(q,N,r)\mid q\geqslant q_n N,~N\geqslant N_0, ~r_{q+1,N}
\leqslant r \leqslant r_{q-1,N}\big\}$ and $q_n\geqslant q_0'$, for $n\geqslant 1$,  is a 
 positive constant which  depends only on $V$.
Furthermore, for  $n\geqslant 2$, we have 
\[
A_{q,N}^{(n)}(r_{q,N})\sim (-1)^n\beta_{n-1}\rho_0^n q^{4n-3}/N^{3n-2},
\quad\text{as $q/N$ tends to infinity,}
\]
where $\beta_{n-1}$ is a positive constant as defined 
in Lemma~\ref{lem:V-V/N^2}~\ref{Cond:T^k-N}).
\end{prop}
\begin{proof}
$\bullet$~We set $q_1\geqslant \max(2;q_0')$ and we prove~\eqref{eq:T1}. \\[3pt] 
-- First we assume that $N=1$. A direct computation shows that $(-1)^n T_V^{(n)}(h)>0$ for $h>0$.
Furthermore, Lemma~\ref{lem:V-V/N^2}~\ref{Cond:T^k-N})
 implies  that 
\[
	T_V^{(n)}(e) \sim 
		(-1)^n{\beta_n} T_V(e)^{4n+1}, \quad\text{as $e$ tends to zero.}
\]
This shows that  there exist two positive constants 
$c_n$ and $d_n$, for each $n\geqslant 1$, such that 
\[
\forall e\in (0; e_0),~\quad  c_n ~T_V(e)^{4n+1}\leqslant 
 (-1)^n T_V^{(n)}(e)\leqslant d_n  ~ T_V(e)^{4n+1},\quad \text{with  $e_0=P_V(a_{q_1-1,1})$. }
\]
Moreover, if $r_{q-1,1}\leqslant r \leqslant r_{q+1,1}$ then we have 
$e_{q+1,1}\leqslant P_V(0,r)\leqslant e_{q-1,1}\leqslant e_0$. 
 Since we have $T_V(e_{q-1,1})=q-1\asymp q\asymp q+1=T_{V}(e_{q+1,1})=q+1$ uniformly for 
 $(q,1,r)\in E_1$, it follows that
\[
\forall n\geqslant 0, \quad 
T_V^{(n)}(P_V(0,r))\asymp q^{4n+1} ~\quad\text{uniformly for  $(q,1,r)\in E_1$.}
\]
-- For $N\geqslant N_0$ and $q\geqslant q_1 N$, we use Lemma~\ref{lem:V-V/N^2};
we observe that if we assume that $r_{q-1,N}\leqslant r \leqslant r_{q+1,N}$ 
then we have 
\[
	r_{q/N+1,1}\leqslant r_{(q+1)/N}=  Nr_{q+1,N}\leqslant Nr \leqslant  Nr_{q-1,N} = 
	r_{(q-1)/N,1}\leqslant r_{q/N-1,1}.
\]
Therefore, if  $e=P_{V/N^2}(0,r)=P_V(0,Nr)/N^2$ then we have $e_{q/N+1,1}\leqslant N^2e 
\leqslant  e_{q/N-1,1}$,  hence Lemma~\ref{lem:V-V/N^2}~\ref{Cond:T^k-N})  and the discussion 
above when  $N=1$ show that 
\[
T_{V/N^2}^{(n)}(e)=N^{2n+1}T_V^{(n)}(N^2e)\asymp N^{2n+1}T_V^{(n)}(N^2e_{q/N,1})
\asymp q^{4n+1}/N^{2n} 
\]
uniformly for $(q,N,r)\in E_1$. This proves~\eqref{eq:T1}. 
\\[8pt]
$\bullet$~We prove~\eqref{eq:Aq1}. First we assume that $N=1$.\\[3pt]
We set $T_0(r)=q-T_V(P_V(0,r))$ and $T_k(t)=-T_V^{(k)}(P_V(0,r))$, for $k\geqslant 1$;
we observe that $\abs{T_0}\leqslant 1$ and  the point above shows that $(-1)^{k+1}T_k(r)
\asymp q^{4n+1}$ uniformly for $(q,1,r)\in E_1$.
An immediate  induction over $p\geqslant 1$ shows that 
\begin{align}
\label{eq:A_q(r)-2p}
	A_{q,1}^{(2p-1)}(r)=\sum_{k=0}^{p-1} C_{k,2p-1} r^{2k+1} T_{p+k-1}(r)
	\quad ; \quad 
	A_{q,1}^{(2p)}(r)=\sum_{k=0}^{p} C_{k,2p} r^{2k} T_{p+k-1}(r),\\ %
\notag
\text{with}~\left|
	\begin{array}{ll}
	C_{k,2p+1}=C_{k,2p}+(2k+2)~C_{k+1,2p} & \text{if}~p \geqslant 1 
	~\text{and}~0 \leqslant k \leqslant p-1,\\ %
	C_{k,2p+2}=C_{k,2p+1}+(2k+1)~C_{k,2p+1} & \text{if}~p \geqslant 1 
	~\text{and}~1\leqslant k \leqslant p,\\ %
	C_{p-1,2p-1}=C_{p,2p}=1 & \text{if}~p \geqslant 1.
	\end{array}\right.
\end{align}
Furthermore, we have 
\begin{equation}
\label{eq:T_k-asymp}
r^{2k}(-1)^{p+k}T_{p+k-1}(r)\asymp r^{2k+1}(-1)^{p+k}T_{p+k-1}(r)\asymp q^{4p+4k-3}
\end{equation}
uniformly for $(q,1,r)\in E_1$, since we have $\rho_0\leqslant r \leqslant 2\rho_0$.
Since we have $q\geqslant q_1\geqslant 2$ on $E_1$, it follows that 
$q^{4p+4k-3}=\jO_{E_1}(q^{4p+4k_0-3})$ for $0\leqslant k\leqslant k_0$ and $p\geqslant 1$, 
hence 
\[
\abs{A_{q,1}^{(2p)}(r)}=\jO_{E_1}(q^{8p-3})~\text{and}~\abs{A_{q,1}^{(2p-1)}(r)}=\jO_{E_1}(q^{8p-7})
~\text{for $p\geqslant 1$,}
\]
which proves~\eqref{eq:Aq1} on $E_1\cap \{N=1\}$.
Since we have $(q,N,r)\in E_1$ if and only if $(q/N,1,r/N)\in E_1$, this
 extends immediatly to~\eqref{eq:Aq1} for any $N\geqslant N_0$ according 
to Lemma~\ref{lem:A_q-N} bellow.
\\[8pt]
$\bullet$~We prove~\eqref{eq:Aq2}. First we assume that $N=1$.\\[3pt]
It follows from~\eqref{eq:T_k-asymp} and~\eqref{eq:A_q(r)-2p} above 
that there exist positive constants $c_{k,\ell}$ and $d_{k,\ell}$ (depending only on $V$)
such that 
\begin{subequations}
\begin{align}
\sum_{k=0}^{p-1} c_{k,2p-1} (-1)^{p+k} q^{4(p+k)-3}
&\leqslant  A_{q,1}^{(2p-1)}(r)\leqslant  
\sum_{k=0}^{p-1} d_{k,2p-1} (-1)^{p+k} q^{4(p+k)-3}; 
 \\
\sum_{k=0}^{p} c_{k,2p} (-1)^{p+k}q^{4(p+k)-3}
&\leqslant A_{q,1}^{(2p)}(r)\leqslant  
\sum_{k=0}^{p} d_{k,2p} (-1)^{p+k}q^{4(p+k)-3}. 
\end{align}
\end{subequations}
This implies that 
\begin{align*}
\demi  d_{p-1,2p-1} q^{8p-7}
&\leqslant  -A_{q,1}^{(2p-1)}(r)\leqslant  2 c_{p-1,2p-1}  q^{8p-7}
,\quad\text{for $q\geqslant q_{2p-1}$};
 \\
\demi c_{p,2p} q^{8p-3}
&\leqslant A_{q,1}^{(2p)}(r)\leqslant  
2d_{p,2p} q^{8p-3}, \quad\text{for $q\geqslant q_{2p}$};
\end{align*}
 where $q_{2p}\geqslant q_{2p-1}\geqslant q_{1}$ are large enough
(depending only on $V$).
Therefore we have proved that $-A_{q,1}^{(2p-1)}(r)\asymp  q^{8p-7}$
and $A_{q,1}^{(2p)}(r)\asymp q^{8p-3}$ uniformly for $(q,1,r)$ in $E_{2p-1}$
or $E_{2p}$ respectively, which is~\eqref{eq:Aq2} on  $\{N=1\}$. 
This  extends immediatly to~\eqref{eq:Aq1} for any $N\geqslant N_0$ according 
to Lemma~\ref{lem:A_q-N} bellow.
\\[8pt]
$\bullet$~Since we have $C_{p-1,2p-1}=C_{p,2p}=1$ in~\eqref{eq:T_k-asymp},
we obtain with~\eqref{eq:A_q(r)-2p}   that 
$A_{q,1}^{2p}(r_{q,1})\sim r_{q,1}^{2p}T_{2p-1}(r_{q,1})$ and 
$A_{q,1}^{2p-1}(r_{q,1})\sim r_{q,1}^{2p-1}T_{2p-2}(r_{q,1})$ as $q$ tends to infinity.
But Lemma~\ref{lem:V-V/N^2}~\ref{Cond:T^k-N}) shows that 
\[
T_{n-1}(r_{q,1})=-T^{(n-1)}(P_V(0,r_{q,1}))\sim (-1)^n\beta_{n-1}q^{4n-3}.
\]
Since Proposition~\ref{prop:quantitative-disjoint-boxes} shows that 
$r_{q,1}\sim \rho_0$ as $q$ tends to infinity, we obtain that  $A_{q,1}^{(n)}(r_{q,1})
\sim  (-1)^n\rho_0^n\beta_{n-1}q^{4n-3} $ when $N=1$. The announced equivalent
for general $N\geqslant 1$ follows using Lemma~\ref{lem:A_q-N} bellow
and this completes the proof of the proposition.
\end{proof}

\begin{lemma}
\label{lem:A_q-N}
We have $A_{q,N}(r)=\frac{1}{N}A_{q/N,1}(Nr)$.
\end{lemma}
\begin{proof}
Using  Lemma~\ref{lem:V-V/N^2}, we compute 
\begin{multline*}
A_{q,N}(r)
	=\int_{e_{q,N}}^{P_{V/N^2}(0,r)} (q-T_{V/N^2}(h))~dh
	=\int_{e_{q,N})}^{P_{V/N^2}(0,r)} (q-NT_V(N^2h))~dh\\
	=\frac{1}{N}\int_{N^2e_{q,N}}^{N^2P_{V/N^2}(0,r)}\left(\frac{q}{N}-T_V(h)\right)~dh
 	=\frac{1}{N}\int_{e_{q/N,1}}^{P_V(0,Nr)} \left(\frac{q}{N}-T_V(h)\right)~dh\\
	=\frac{1}{N}A_{q/N,1}(Nr).
\end{multline*}
This proves the   formula of the lemma.\vspace{10pt}
\end{proof}
%
\noindent
{\bfseries Lineart part of $\mathbf{\FNqm}$}.--
We recall that  $B_{q}$ denotes a $q$-adapted box with respect to $P_{V/N^2}$
and $\BdeN$, as it appears in Propostion~\ref{prop:quantitative-disjoint-boxes}.
\begin{prop}\label{prop:Taylor-expansion-linear-part}
Set $\sigma_{q,N}(\begin{smallmatrix} \theta\\ R\end{smallmatrix})=(\theta,r_{q,N}+R)$.   
There exist  a constant $\alpha_{q,N}$ and a function $S_{q,N}(R)$  satisfying for all 
$\mu>0$ and $(\theta,R)\in \sigma_{q,N}^{-1}(B_{q})$
\[
\left|
	\begin{array}{l}
		\sigma_{q,N}^{-1}\circ F_{q,\mu,N}\circ \sigma_{q,N}
		\left(\begin{smallmatrix} \theta \\ R\end{smallmatrix}\right)
		=\left(\begin{smallmatrix} 1 & \alpha_{q,N}\\
									-\mu &1-\mu \alpha_{q,N}\end{smallmatrix}\right)
			\left(\begin{smallmatrix} \theta \\ R\end{smallmatrix}\right)
			+S_{q,N}(R) \left(\begin{smallmatrix} ~~1 \\ -\mu \end{smallmatrix}\right),\\
		S_{q,N}(0)=S_{q,N}'(0)=0.
	\end{array}
\right.
\]

and the following estimates hold true.
\[
\begin{array}{l}
\displaystyle
\alpha_{q,N} \underset{\,E_2}{\asymp}
\frac{q^5}{N^4}\ ; \quad \abs{S_{q,N}(R)}
\underset{\,E_0}{\asymp}R^2 \frac{q^5}{N^4}
\ ;\quad  
	\abs{S_{ q,N}'(R)}\underset{\,E_1}{\asymp} \abs{R\,}\frac{q^9}{N^7}\ ;\quad 
\abs{R\,}=\jO_{E_1}( q^3/N^5)\ ;\\ %
\displaystyle
	\forall n \geqslant 2, \quad  
	(-1)^{n+1}{S_{q,N}^{(n)}(R)} \underset{~E_n}{\asymp} {q^{4n+1}}/{N^{3n+1}}\,;
\end{array}
\]
with $E_n=\{(q,N,R)\mid q\geqslant q_1 N,~N\geqslant N_0, 
(0, r_{q,N}+R)\in B_q\}$ and  $q_n$ a   positive constants which 
depend only  on $V$, for $n\geqslant 1$.\\[3pt]
Furthermore, we have 
\[
\alpha_{q,N}\sim \beta_1\rho_0^2\frac{q^{5}}{N^{4}},~\quad  
{S_{q,N}^{(2)}(0)}\sim -\beta_2\rho_0^3\frac{q^{9}}{N^{7}}\quad~\text{and}~\quad 
{S_{q,N}^{(3)}(0)} \sim \beta_3\rho_0^4\frac{q^{13}}{N^{10}}
\]
 as $q/N$ tends to infinity.
\end{prop}
\begin{proof} $\bullet$~Proposition~\ref{prop:local_form_FqNM}
shows that $\Phi^{A_{q,N}}(B_q)=\varphi^q(B_q)\subset{\BdeeN}$, so
$\mu W_N(\theta)=\dem \mu\theta^2$ on  $\Phi^{A_{q,N}}(B_q)$. This implies that 
$\FNqm$ coincides on $B_q$ with 
\[
\FNqm(\theta,r)=\Big(\theta+A_{q,N}'(r),r-\mu\big(\theta+A_{q,N}'(r)\big)\Big).
\]
Setting
\begin{equation}\label{eq;S_qN}
\alpha_{q,N}:=A''_{q,N}(r_{q,N})~\text{and}~
S_{q,N}(R):=A'_{q,N}(r_{q,N}+R)-A'_{q,N}(r_{q,N})-\alpha_{q,N}R,
\end{equation}
the anounced formula for $\FNqm$ follows from a direct computation.
\\[10pt]
$\bullet$~We prove the estimates of the proposition. Since we have $\sigma_{q,N}(\theta,R)\in B_q$,
it follows that  $r_{q+1,N}\leqslant r_{q,N}+R \leqslant r_{q-1,N}$.
 Now we apply Proposition~\ref{prop:asymptotic-T-Aq}:
\begin{itemize}
\item[--] with $n=2$, Estimate~\eqref{eq:Aq2}
 shows that $\alpha_{q,N}=A''_{q,N}(r_{q,N})\asymp q^5/N^4$.
\item[--]  For $\ell\geqslant 0$, we have  
\[
\begin{array}{l}
\displaystyle
S_{q,N}(R)=\frac{R^2}{2}A_{q,N}''(r_{q,N}+\eta_0 R);~
S_{q,N}'(R)={R}~A_{q,N}'''(r_{q,N}+\eta_1 R);\\[5pt]
\displaystyle
S_{q,N}^{(\ell)}(R)=A^{(\ell+1)}_{q,N}(r_{q,N}+\eta_\ell R),~\text{for $\ell\geqslant 2$;}
\end{array}
\]
for some $0<\eta_\ell<1$ (depending on $R$),
hence Estimates~\eqref{eq:Aq1} shows that 
\[
	S_{q,N}(R)~\asymp R^2  \frac{q^5}{N^4};~
	\abs{S'_{q,N}(R)}\asymp \abs{R} \frac{q^9}{N^7};~
	(-1)^{\ell+1}S^{(\ell)}_{q,N}(R)\asymp \frac{q^{4\ell+1}}{N^{3\ell+1}},
	~\text{for $\ell\geqslant 2$.}
\]
\end{itemize}
Since we have $\abs{R}\leqslant \ell'$ and $\ell'\asymp N^3/q^5$ according to 
Proposition~\ref{prop:quantitative-disjoint-boxes}, the proof of  the announced estimates
is complete.
\\[10pt]
$\bullet$~
At  last, we observe that 
\[
\alpha_{q,N}=A_{q,N}''(r_{q,N}),\quad S_{q,N}''(0)=A_{q,N}'''(r_q)
\quad\text{and}\quad
S_{q,N}^{(3)}(0)=A_{q,N}^{(4)}(r_{q,N}).
\]
Therefore  these quantities as $q/N$ tends to infinity may be estimated   immediately from 
the last estimate of  Proposition~\ref{prop:asymptotic-T-Aq}, which completes the proof 
of the proposition.
\end{proof}

\begin{lemma}
We have $\alpha_{q,N}=N\alpha_{q/N,1}$ and $ S_{q,N}(R)=NS_{q/N,1}(NR)$.
\end{lemma}
\begin{proof}
We have $\alpha_{q,N}=A_{q,N}''(r_{q,N})$. Therefore Lemma~\ref{lem:A_q-N}
and Lemma~\ref{lem:V-V/N^2}~\ref{Cond:e_q}) imply that 
\[
\alpha_{q,N}=N A_{q/N,1}''(Nr_{q,N})=N A_{q/N,1}''(r_{q/N,1})=N\alpha_{q/N,1}.
\]
In a similar way, Lemma~\ref{lem:A_q-N} implies that 
\begin{align*}
S_{q,N}(R)&=A_{q,N}'(r_{q,N}+R)-A_{q,N}'(r_{q,N})-\alpha_{q,N}R\\
&=N A_{q/N,1}''(Nr_{q,N}+NR)-N A_{q/N,1}''(Nr_{q,N})-N\alpha_{q/N,1}R\\
&=N A_{q/N,1}''(r_{q/N,1}+NR)-N A_{q/N,1}''(r_{q/N,1})-N\alpha_{q/N,1}R\\
&=N S_{q/N,1}(NR).
\end{align*}
This proves the second identity of the lemma and the proof is complete.
\end{proof}

~\\[10pt]
\noindent
{\bfseries Diagonalization of the lineart part and Taylor expansion}
\begin{Not}\label{not:E_n}
For $n\in\N^*$ and $\beta>0$,  we set
\begin{subequations}
\begin{align}
\label{eq:mu-small}
	E_\beta &\defeq \big\{
 			(q,N,\mu)\mid 0<\mu \alpha_{q,N}<1~\text{and}~
			q\geqslant \beta N,~N\geqslant N_0
			\big\},
\\[1ex]
\label{eq:mu-small2}
	E_{\beta,n} &\defeq \Big\{(q,N,\mu)\in E_\beta\mid\mu\alpha_{q,N}<\frac{1}{(n+1)^2}\Big\}
\end{align}
\end{subequations}
(with the notation of Proposition~\ref{prop:Taylor-expansion-linear-part} for~$\alpha_{q,N}$).
Unless mentioned otherwise, we shall abreviate  $f_1\asymp f_2$ if there exists a positive 
constant $\beta$ (not depending on $q$, $N$, $\mu$) satisfying  
$f_1\underset{E_{\beta,n}}{\asymp} f_2$.
\end{Not}

\begin{Not}\label{not:lambda}
Let $\lambda\in\C$ satisfy  the following two conditions
\begin{equation}\label{eq:lambda}
	\lambda+\lambda^{-1}=2-\mu \alpha_{q,N},\qquad
	\lambda=\exp(i\gamma_0)~\text{with}~ -\frac{\pi}{3}<\gamma_0<0.
\end{equation}
so we have $\abs{\lambda^p-1}\asymp\abs{\lambda-1}$ uniformly on $E_{\beta,n}$ for 
$1\leqslant p \leqslant 2n+2$ (see Lemma~\ref{lem:non-resonance} below).
\end{Not}
\noindent
It  follows immediately from~\eqref{eq:lambda} that 
\[
	1-\lambda=2\sin^2(\gamma_0/2)-i\sin\gamma_0~\text{and }~\abs{\lambda-1}^2
		=\mu\alpha_{q,N}=(1-\cos\gamma_0)^2+\sin^2\gamma_0=2(1-\cos\gamma_0),
\]
hence  $\sin^2\gamma_0={\mu\alpha_{q,N}(1-\mu\alpha_{q,N}/4)}$
and~\eqref{eq:mu-small} implies that 
\[
\frac{\sqrt{3}}{2}\abs{\lambda-1}\leqslant
-\sin\gamma_0\leqslant \abs{\lambda-1},
\quad\text{ so 
$-\sin\gamma_0~{\asymp}~ \abs{\lambda-1}=\sqrt{\mu\alpha_{q,N}}$.}
\]
\begin{Not}
For all $z\in\C$ we set $\Psi(z)=\sigma_{q,N}\circ \psi(z)$ with $\sigma_{q,N}(\theta,R)
=(\theta, r_{q,N}+R)$ and  
\[
	\psi(z)=B \left(\begin{array}{c}z \\ \overline{z}\end{array}\right)\in\R^2, \quad
	~B=\left(\begin{smallmatrix}\frac{\alpha}{\lambda-1} 
								& \frac{\alpha}{\overline{\lambda}-1}\\ 1 & 1\end{smallmatrix}\right)
~\text{and $\alpha=\alpha_{q,N}$.}
\]
\end{Not}
\begin{prop} \label{prop:complex-coord}
Assume $n\geqslant  2$  and set $\omega={q^4}/{N^3}$. 
Then there exists a positive constant $\beta>0$ such that for each $(q,N,\mu)\in E_\beta$
there exist  
$\rho>0$ and $\kappa>0$,  $a_{\nu}\in\R$, for $2\leqslant\nu
\leqslant n$, a function  $g~:[-2\rho~;2\rho]\to\R$ satisfying the following properties.
\begin{enumerate}
\item \label{Cond:complex-coordinate-size-rho}
	  $\Dm(0;\rho)\subset \Psi^{-1}(B_{q})$ and 
																$\rho\asymp \frac{\abs{\lambda-1}}{q\omega}$;%
\item\label{Cond:complex-coordinate-g}
 	$\Psi^{-1}\circ F_{q,N,\mu}\circ  \Psi(z)=\lambda\biggl(z+i\abs{\lambda-1}
																		g(z+\overline{z})\biggr)$ on $\Dm(0;\rho)$;%
\item\label{Cond:complex-coordinate-area}
	$\Psi^\star(dr\wedge d\theta)=\frac{\kappa}{2i}dz\wedge d\bar z$ and $\kappa{\asymp}
						\frac{\abs{\lambda-1}}{\mu}\asymp\frac{\omega q/N}{\abs{\lambda-1}}$;%
\item\label{Cond:complex-coordinate-taylor}
	$g(x)=\sum\limits_{\nu=2}^n a_\nu x^\nu+\varepsilon(x)$, where  
	$\left\{\begin{array}[c]{ll}(-1)^{\nu} a_\nu{\asymp}\omega^{\nu-1}, 
	&\text{for $2 \leqslant\nu\leqslant n$},\\
	\abs{\varepsilon^{(k)}(x)}
	=\jO_E(\omega^{n}\abs{x}^{n+1-k}),&\text{for $0 \leqslant k\leqslant n$,}
	\end{array}\right.$
\end{enumerate}	
 with $E=\{(x,q,N,\mu)\mid 
\abs{x}\leqslant 2\rho, ~(q,N,\mu)\in E_\beta\}$
and $E_\beta$ as in~\eqref{eq:mu-small}.
\end{prop}
\begin{proof}
we abreviate $F=F_{q,N,\mu}$, $S_{q,N}=S$ and $\alpha=\alpha_{q,N}$.~\\[5pt]
\ref{Cond:complex-coordinate-size-rho})~We recall that $\abs{\lambda-1}=\sqrt{\alpha\mu}$, so 
	$\abs{2\Re(\frac{\alpha z}{\lambda-1})}\leqslant 2\sqrt{\alpha/\mu}\abs{z}$. This shows 
	that 
	\[
		\forall (\ell,\ell')\in\R_+^2, \quad
		\abs{z}<\min\left(\frac{\ell}{2}\sqrt{\frac{\mu}{\alpha}},\frac{\ell'}{2} \right)
		\Rightarrow \psi(z)\in (-\ell,\ell)\times (-\ell',\ell').
	\]
	Proposition~\ref{prop:quantitative-disjoint-boxes}  and 
	Proposition~\ref{prop:Taylor-expansion-linear-part} show that  $\ell\asymp \frac{1}{N}$,  
	$\ell'\asymp\frac{N^3}{q^5}=\frac{1}{q\omega}$ and $\alpha\asymp\frac{q^5}{N^4}$, so 
	\[
		\frac{\ell}{2}\sqrt{\frac{\mu}{\alpha}}\asymp \sqrt{\alpha\mu}\frac{\ell}{\alpha}\asymp
		\abs{\lambda-1}\frac{N^3}{q^5}=\abs{\lambda-1}\frac{1}{q\omega}.
	\]
	Therefore, with  $\rho\asymp \min\Big(\frac{\abs{\lambda-1}}{q\omega};
	\frac{1}{q\omega}\Big)=\frac{\abs{\lambda-1}}{q\omega}$, we 
	obtain~\ref{Cond:complex-coordinate-size-rho}).~\\[5pt]
\ref{Cond:complex-coordinate-g})~Since 
	$\left(\begin{smallmatrix}\frac{\alpha}{\lambda-1} \\ 1
																										\end{smallmatrix}\right)$ 
	and $\left(\begin{smallmatrix}\frac{\alpha}{\overline{\lambda}-1} \\ 1
																										\end{smallmatrix}\right)$ 
	are two eigenvectors of the operator $B$ corresponding to the eigenvalues $\lambda$ and 
	$\overline{\lambda}$, we obtain by a direct computation that 
	\[
	\Psi^{-1}\circ F\circ \Psi(z)=\lambda z +S(z+\overline{z}) ~\psi^{-1}\!
	\left(\begin{smallmatrix}~~1 \\ -\mu\end{smallmatrix}\right).
	\]
	We notice that $\psi(i\lambda)=i(\lambda-\overline{\lambda})
	\left(\begin{smallmatrix}~~-\frac{\alpha}{\abs{\lambda-1}^2} \\ 1\end{smallmatrix}\right)
	=-2\sin\gamma_0\left(\begin{smallmatrix}-\frac{1}{\mu} \\~~1\end{smallmatrix}\right)
	=\frac{2}{\mu}\sin\gamma_0\left(\begin{smallmatrix}~~1 \\-\mu\end{smallmatrix}\right)$, hence 
\[
	\psi^{-1}\!
	\left(\begin{smallmatrix}~~1 \\ -\mu\end{smallmatrix}\right)=
	\frac{i\lambda\mu\,}{2\sin\gamma_0}.
\] 
	This implies~\ref{Cond:complex-coordinate-g}), with
	\begin{equation}
		g(z)=\frac{\mu}{2\abs{\lambda-1}\sin\gamma_0}S(z).\label{eq:g}
	\end{equation}
	\\[5pt]
\ref{Cond:complex-coordinate-area})~We have $\Psi^\star (dr\wedge d\theta)
	=\frac{~1}{2i}\mathrm{det}
	\big(\psi(1); \psi(i)\big) ~dz\wedge d\bar z $ and 
	\[
	\mathrm{det}(\psi(1); \psi(i))\asymp 
	\left|\begin{array}{cc}
	\Re\big(\frac{\alpha}{\lambda-1}\big) & \Re\big(\frac{i \alpha}{\lambda-1}\big) \\
	1 & 0
	\end{array}\right|
	=-\frac{\alpha\sin\gamma_0}{\abs{\lambda-1}^2}. 
	\] 
	Since $-\sin\gamma_0\asymp \abs{\lambda-1}=\sqrt{\alpha\mu}$ and $\alpha\asymp 
	q\omega/N$, this proves~\ref{Cond:complex-coordinate-area}).~\\[5pt]
\ref{Cond:complex-coordinate-taylor})~Using~\eqref{eq:g} and the Taylor expansion $S(x)=
	\sum\limits_{\nu=2}^n \frac{1}{\nu !} S^{(\nu)}(0) x^\nu +R(x)$,  we set 
	\begin{equation}\label{eq:a_nu}
	a_\nu= \frac{\mu S^{(\nu)}(0)}{2\nu !\abs{\lambda-1}\sin\gamma_0}~\text{and}~
	\varepsilon(x)=\frac{\mu R(x)}{2\abs{\lambda-1}\sin\gamma_0},~\text{so}
	~g(x)=\!\!\sum\limits_{\nu=2}^n a_\nu x^\nu+\varepsilon(x).
	\end{equation}
	Since $(-1)^{\nu-1}S^{(\nu)}(0){\asymp} {q^{4\nu+1}}/{N^{3\nu+1}}$, 
	$-\sin\gamma_0\asymp\abs{\lambda-1}$, $\alpha\asymp q^5/N^4$, 
	we have
	\[
	(-1)^{\nu}a_\nu\underset{E_{n}}{\asymp}  \frac{q^{4\nu+1}}{N^{3\nu+1}} 
	\frac{\mu}{\abs{\lambda-1}^2}=\frac{q^{4\nu+1}}{N^{3\nu+1}} \frac{1}{\alpha}
	\asymp\frac{q^{4\nu+1}}{N^{3\nu+1}} \frac{N^4}{q^5}=\frac{q^{4\nu-4}}{N^{3\nu-3}}
	=\omega^{\nu-1}.
	\]
	Thus all that remains is to prove the estimates on $\varepsilon(x)$. Notice that $\psi(\rho)\in 
	B_q(\ell,\ell')$, so $2\rho\leqslant \ell'$. Furthermore,  for $0 \leqslant j \leqslant n$, the 
	derivative $R^{(j)}(x)$ is the remainder of the Taylor expansion at zero of $S^{(j)}(x)$ up 
	to order $n-j$. Therefore the Taylor expansion theorem and 
	Proposition~\ref{prop:Taylor-expansion-linear-part} show that  for $x\in [-2\rho;2\rho]$ we 
	have
	\[
	\abs{R^{(n+1-j)}(x)}\leqslant\frac{\abs{x}^{n+1-j}}{(n+1-j)!}\max_{\abs{x}\leqslant 2\rho}
	\abs{S^{(n+1)}(y)}\asymp \abs{y}^{n+1-j}
	\frac{q^{4(n+1)+1}}{N^{3(n+1)+1}}.
	\]
	 Moreover the estimates 
	$\abs{\sin\gamma_0}\asymp \abs{\lambda-1}$ and $\alpha\asymp q^5/N^4$ imply that  
	\[
	\frac{\mu}{\abs{\lambda-1}~\abs{\sin\gamma_0}}\asymp 
	\frac{\mu}{\abs{\lambda-1}^2}=\frac{1}{\alpha}\asymp \frac{N^4}{q^5}, \text{hence}~
	\abs{\varepsilon^{(j)}(x)}=\jO_{E} (
		q^{4n}/N^{3n}\abs{x}^{n+1-j}),
	\]
	 and the proof  of~\ref{Cond:complex-coordinate-taylor})  is complete.
\end{proof}
\begin{cor}[Taylor expansion]\label{Cor:Taylor-expansion}
Assume $n\geqslant  0$; we set $\omega={q^4}/{N^3}$. 
Then 
for each $(q,N,\mu)\in E_{\beta,n}$,
there exist  $\lambda\in\C$ (with $\lambda=\exp(i\gamma_0)$ and $-\pi/3<\gamma_0<0$),
$\rho>0$,  $a_{\nu}\in\R$, for $2\leqslant 2\nu
\leqslant 2n+2$ and a function  $\varepsilon~:[-2\rho~;2\rho]\to\C$ satisfying
\begin{align}
(\Psi^{-1}\circ \FNqm\circ \Psi)(z)
	&=\lambda \Big(z+i\abs{\lambda-1}\sum\limits_{\ell=2}^{2n+2}a_\ell(z+\bar z)^\ell\Big)
								+\varepsilon(z+\overline{z})
\\
& \notag \text{with  $\abs{\varepsilon^{(k)}(x)}
	=\jO_E(\omega^{2n+2}\abs{\lambda-1}\abs{x}^{2n+3-k})$, for $0 \leqslant k\leqslant 2 n+2$,}\\
& \text{and}~  E=\{(x,q,N,\mu)\mid 
\abs{x}\leqslant 2\rho, ~(q,N,\mu)\in E_\beta\}.\notag
\end{align}
Furthermore we have  the following properties.
\begin{enumerate}
\item $\abs{\lambda^p-1}\asymp\abs{\lambda-1}$ for $1\leqslant p\leqslant 2n+2$
and $\rho\omega\asymp\abs{\lambda-1}/q$;
\item $(-1)^{\ell}a_\ell\asymp \omega^{\ell-1}$, for $2\leqslant \ell \leqslant 2n+2$;
\item\label{Cond:a_2-a_3}
  $2 a_2^2+3 a_3\abs{\lambda-1} R(\lambda)\asymp \omega^2$
 uniformly on $E_{\beta,n}$,
with $\displaystyle R(\lambda)=i~\frac{1+\lambda}{1-\lambda}~\frac{2+\lambda+2\lambda^2}{1+\lambda+\lambda^2}$.
\end{enumerate}
\end{cor}
\begin{proof} We prove~\ref{Cond:a_2-a_3}), wich is the only condition which 
 does not follow directly from Proposition~\ref{prop:complex-coord}.
Let $(q,N,\mu)$ be in $E_{\beta,n}$, so $0<\alpha\mu<1/(n+1)^2$. Let 
$0<\alpha_0<\frac{\pi}{6}$ satisfy
\[
\sin(\alpha_0)=\frac{1}{2n+2}.
\]
Since $\alpha\mu=\abs{\lambda-1}^2$ and $-\sin(\gamma_0/2)=\frac{1}{2}\abs{\lambda-1}<\frac{1}{2n+2}$,
we obtain that $-\gamma_0/2<\alpha_0$, so 
\[
\cos(\alpha_0)\abs{\lambda-1}<\cos(\gamma_0/2)\abs{\lambda-1}=-\sin(\gamma_0)<\abs{\lambda-1}
\]
As $q/N$ tends to infinity, Proposition~\ref{prop:Taylor-expansion-linear-part}
and~\eqref{eq:a_nu} show that 
	\[
	\begin{array}{l}
	a_2^2\displaystyle =\frac{\mu^2 S''(0)^2}{16\sin^2\gamma_0\abs{\lambda-1}^2}=
	\frac{\mu^2 S''(0)^2}{16\cos^2(\frac{\gamma_0}{2})\abs{\lambda-1}^4}
	=\frac{ S''(0)^2}{16\alpha^2\cos^2(\frac{\gamma_0}{2})}\sim 
									\frac{ (\beta_2/\beta_1)^2}{16\cos^2(\frac{\gamma_0}{2})}\rho_0^2 \omega^2,
\vspace{5pt}\\ %
	-a_3\displaystyle=\frac{\mu S'''(0)}{-12 \sin\gamma_0~\abs{\lambda-1}}=
	\frac{\mu S'''(0)}{12 \cos(\frac{\gamma_0}{2})\abs{\lambda-1}^2}
	=\frac{ S'''(0)}{12 \alpha\cos(\frac{\gamma_0}{2})}
								\sim	\frac{ \beta_3/\beta_1}{12\cos(\frac{\gamma_0}{2})}\rho_0^2\omega^2,  %
	\end{array}
\]
 Now we compute
\[
	R(\lambda)=-\frac{\cos(\frac{\gamma_0}{2})}{\sin(\frac{\gamma_0}{2})} 
	\frac{1+4\cos\gamma_0}{1+2\cos\gamma_0},~\text{hence}
\abs{\lambda-1}R(\lambda)=2\cos(\frac{\gamma_0}{2})\frac{8\cos^2(\frac{\gamma_0}{2})-3}
	{4\cos^2(\frac{\gamma_0}{2})-1}.
\]	
Since we have $\cos^2(\frac{\gamma_0}{2})\geqslant 1-\big(\frac{1}{2n+2}\big)^2\geqslant
\frac{15}{16}$ for $n\geqslant 1$, it follows that 	for $q/N$ large enough we have
\begin{align*}
\displaystyle
\frac{72}{11}=4\cdot \frac{\scriptstyle 18}{\scriptstyle 11}
\leqslant
	\frac{2\abs{\lambda-1}R(\lambda)}{\cos(\frac{\gamma_0}{2})}=4~
	\frac{8\cos^2(\frac{\gamma_0}{2})-3}{4\cos^2(\frac{\gamma_0}{2})-1} 
	\leqslant\frac{20}{3},~\text{hence}\\
	\Big(\frac{\scriptstyle18}{\scriptstyle 11}\beta_2^2-\beta_1\beta_3\Big)
	\frac{\rho_0^2\omega^2}{8\beta_1^2}
\leqslant \frac{1}{2}
	\Big(\frac{\frac{\scriptstyle 72}{\scriptstyle 11}\beta_2^2}{ 16}-\frac{3\beta_1\beta_3}{12}\Big)
	\frac{\rho_0^2\omega^2}{\beta_1^2\cos(\frac{\gamma_0}{2})}
\leqslant {2}a_2^2\abs{\lambda-1}R(\lambda)+3a_3\leqslant \frac{\scriptstyle 5 a_2^2}{\scriptstyle 12}
	\asymp \omega^2.
\end{align*}
	This holds true and implies the lemma because we can 
	evaluate $18\beta_2^2-11\beta_1\beta_3>0$.
\end{proof}

\begin{lemma}\label{lem:non-resonance}
If $\abs{\lambda-1}<\frac{1}{n+1}$ then $\abs{\lambda-1}\leqslant \abs{\lambda^p-1}$
 for $1\leqslant p \leqslant 2n+2$.
\end{lemma} 
\begin{proof} Notice that $\abs{\sin\gamma_0}\leqslant \abs{\lambda-1}<\frac{1}{n+1}$, so 
$\gamma_0>\frac{-\pi}{2n+2}$. Therefore we have~\\[5pt]
$\displaystyle~\qquad\qquad \frac{\abs{\lambda^p-1}}{\abs{\lambda-1}} \geqslant
\Re\Big(\sum_{j=0}^{p-1}\lambda^j\Big)=\sum_{j=0}^{p-1}\cos (j\gamma_0)
\geqslant\sum_{j=0}^{p-1}\cos\bigg(\frac{j\pi}{2n+2}\bigg) \geqslant 1$.
\end{proof}
\subsection{Normalisations}    \label{sec:normalisGq}
\noindent
The goal of this section is to prove that we can find  nearly  symplectic  coordinates  in which 
$F_{q,N,\mu}$ takes the form 
\begin{equation}
\label{eq:Herman-normal-form}
F_{q,N,\mu}(z)=\lambda z\exp\bigl(2\pi i\abs{z}^2+\varepsilon(z)\bigr),
\end{equation}
where the error term $\varepsilon$ is  a real valued function and  is  small enough  up to 
enough derivatives.
For this purpose, our first step is to specify a suitable  change of coordinates
  in which  $F_{q,N,\mu}$ appears as 
  a Birkhoff's normal form up to some order, namely 
\begin{equation}
\label{eq:birkhoff}
F_{q,N,\mu}(z)=
\lambda z\big(1+\sum_{p=1}^{n}b_p \abs{z}^{2p}\big)+\widetilde\varepsilon(z).
\end{equation}
Note that the change of coordinates does not need to be symplectic in our setting. 
%
%
%
\subsubsection{Notations and statements}   \label{sec:Ok}
To achieve~\eqref{eq:birkhoff} and~\eqref{eq:Herman-normal-form} in a quantitative way, we must deal with smooth functions on 
$\Dm^\star(0,\tau)=\{z\in\C\mid 0 <\abs{z}\leqslant\tau\}$ 
(but not necessarily smooth at zero) and control their behaviour near zero. To this end we 
introduce the following notations. 

\begin{Not} 
In the following, we use the operators  $\bar\partial=\frac{1}{2}(\partial_s+i\partial_t)$
and  $\partial=\frac{1}{2}(\partial_s-i\partial_t)$, with $z=s+it$ and $(s,t)\in\R^2$.
Assume  $\tau>0$ and $k\in\N$.  A smooth function $f~:\Dm^\star(0,\tau)\to\C$ is said to be 
controlled up to the $k$ derivatives, by  $C\geqslant 0$ at  order $\ell\in\R$,  and we write 
$f\in \jO_k(\ell;C,\tau)$ or $f(z)=\jO_k(\ell;C,\rho)$  if 
\[
 \forall  z\in\Dm^\star(0,\tau), \quad 
\abs{\partial^\alpha\bar\partial^\beta f(z)} \leqslant C \abs{z}^{\ell-\alpha-\beta},
\quad\text{for all  $(\alpha,\beta)\in\N^2$ such that $\alpha+\beta\leqslant k$.}
\]
\end{Not}
\begin{Not}
 For $(k,m)\in\N^2$, $\rho>0$, two sets $E$ and $E'$ satisfying 
	$E\subset E'\times\C$,  two   function $f_1~:E \to\C$ and $f_2~:E'\to \R_+$, and a 
	function $\rho~:E'\to\R_+$,  we write $f_1(\cdot,z)= \jO_{k,E'}(m;f_2,\rho)$  if there 
	exists two  constants $C\geqslant 0$ and $c>0$ satisfying 
	\begin{align*}
	\forall x\in E', \quad E'\times\Dm(0;c\rho(x))\subset E,\\
	~\text{and}~f_1(x,z)=\jO_k(m;C f_2(x),c\rho(x)).
	\end{align*}
\end{Not}
\noindent
All the  properties of the spaces $\jO_k$  we need are listed in Appendix~\ref{app:O_k}.\\
At last, we need to introduce  analogous definitions in polar coordinates.
\begin{Not}
Assume $\rho>0$, $\ell\in \R$ and $k\in\Z$. We  recall that $\T=\R/\Z$.
\begin{itemize}
\item A smooth function $f~:(0;\rho]\times \T\to\C$ is said to be  
	controlled up to  $k$ derivatives, by  $C\geqslant 0$ at  order $\ell\in\R$, 
	and we write $f\in \jO_k^\T(\ell;C,\tau)$ 
	or $f(r,\theta)=\jO_k^\T(\ell;C,\rho)$ if 
	\[
  		\abs{\partial_r^\alpha\partial_\theta^\beta f(r,\theta)}  \leqslant C  r^{\ell-\alpha},
                                              ~\text{for $0<r\leqslant \rho$,  $\theta\in\T$ and 
																								$\alpha+\beta\leqslant k$.}
	\]
\item For two sets $E$ and $E'$ satisfying $E\subset E'\times\R_+\times\T$, two functions 
	$f_1~:E \to\C$ and $f_2~:E'\to \R_+$, and a function $\rho~:E'\to\R_+$, we write 
	$f_1(\cdot,r,\theta)= \jO_{k,E'}^\T(\ell;f_2,\rho)$  if there exists two  constants 
	$C\geqslant 0$ and $c>0$ satisfying 
	\begin{align*}
	\forall x\in E', \quad E'\times(0;c\rho(x)]\times \T\subset E
	\quad \text{and}\quad f_1(x,r,\theta)=\jO_k^\T(\ell;C f_2(x),c\rho(x)).
	\end{align*}
\end{itemize}
\end{Not}
\noindent
 Basically, we can can rephrase 
Proposition~\ref{prop:complex-coord}~\ref{Cond:complex-coordinate-taylor}) as follows
\begin{equation}
g(z+\overline{z})=\sum_{\nu=2}^n a_\nu (z+\overline{z})^\nu+\jO_{n,E_{\beta,n}}(n+1;\omega^n,\rho),
\end{equation}
where $E_{\beta,n}$ is defined by~\eqref{eq:mu-small2}. 
The constants $n$ and $\beta$ do not depend on $(q,N,\mu)$.  
Here we introduce   $E_{\beta,n}$ rather than  $E_{\beta}$  (see~\eqref{eq:mu-small2})
for  suitable  estimates on  the 
non resonant part of the conjugation of the transformation $F_{q,N,\mu}$ to its Birkhoff's 
normal form (see Proposition~\ref{prop:Birkhoff-invariants-in-complex-coordinates} below).
The constant $\beta>0$
is chosen so  $q/N$  is large enough for  appropriate estimates of $a_2$ and $a_3$
(see Corollary~\ref{Cor:Taylor-expansion}.\ref{Cond:a_2-a_3}) above).

\begin{Not}
From now on, unless mentioned otherwise, we shall abreviate $\jO_{k,n}:=
\jO_{k,E_{\beta,n}}$.\\
\noindent
In the following, $h=h(\cdot,q,N,\mu)$ denotes  any  family of symplectic maps  
from  $\Dm(0;\rho)$ into $\C$.  We assume that for  $2n+2\geqslant k\geqslant 1$ we have
\begin{equation}
h(z)=\lambda \Big(z+i\abs{\lambda-1}\sum_{\ell=2}^{2n+2}a_\ell(z+\bar z)^\ell\Big)
									+\jO_{k,n}(2n+3,\abs{\lambda-1}\omega^{2n+2},\rho),
\end{equation}
\noindent
where $\omega=q^4/N^3$,  $\rho>0$, $\lambda=\exp(i\gamma_0)$, $a_\nu\in\C$ satisfy all
the conditions in  Corollary~\ref{Cor:Taylor-expansion}.
\end{Not}
%
%
%
\subsubsection*{Birkhoff normal form}
The next proposition is a the quantitative version of~\eqref{eq:birkhoff}. It
recalls a classical result of normal form theory. We construct  polynomial  
 coordinates in which  the symplectic map $h$  is put in   its  Birkhoff normal form up to a reminder of 
arbitrarily high order. The proof follows Moser's strategy and is inductive in its nature: a sum 
of homogeneous polynomials is used to normalize the Taylor expansion of $h$ order 
by order. But  for our purpose, we need to achieve this keeping a  quantitative track of the 
operations involved. Therefore  we provide  below a complete proof of the statement. 

Before proceeding to the precise statement, we need to introduce a few more notations. We shall 
consider a  coordinates change of the form 
\begin{equation}\label{eq:birkhoff-polynomial}
	u=\Phi(z)=z+\sum_{\nu}\varphi_\nu z^\nu, 
	~\text{where  $z^{(\nu_1,\nu_2)}=z^{\nu_1}\overline{z}^{\nu_2}$}
\end{equation}
Here the index $\nu=(\nu_1,\nu_2)\in\N^2$ in the sum above runs over all the couples such 
that $2\leqslant \abs{\nu}\leqslant 2n+2$, with $\abs{\nu}=\nu_1+\nu_2$. For such polynomial, 
we denote by $[\Phi]_\nu$ the $\nu$-component $\varphi_\nu z^\nu=\varphi_{\nu_1,\nu_2}
z^{\nu_1}\overline{z}^{\mu_2}$. By extension, for any smooth function $F$, we denote by 
$[F]_\nu$ the $\nu$-component  of its Taylor expansion  at zero. For any integer $p$, it is also 
convenient to denote by $[F]_p$ its  $p$-homogeneous part. Thus, for $p\geqslant 2$, we have
\[
[\Phi]_p(z)=\sum\limits_{\abs{\nu}=p}\varphi_\nu z^\nu
~\text{and}~[f]_p(z)=i\abs{\lambda-1}a_p (z+\overline{z})^p.
\]
\begin{prop}[Birkhoff normal form] 
\label{prop:Birkhoff-invariants-in-complex-coordinates}
Assume $2n+2\geqslant k \geqslant 2$. Then for each $(q,N,\mu)\in E_{\beta,n}$
there exist
$\rho'>0$, $b_j\in\C$ for  $1\leqslant j \leqslant n$, and 
$\varphi_\nu\in\C$ for $2\leqslant \abs{\nu}\leqslant 2n+2$ satisfying the 
following conditions.
\begin{enumerate}
\item\label{cond:birkhoff-conj}
	 The polynomial $\Phi$ in~\eqref{eq:birkhoff-polynomial} defines a  diffeomorphism 
								from a neighbourhood of zero  onto a set that contains $\Dm(0,\rho')$ and 
\[
\Phi\circ h\circ\Phi^{-1}(u)=\lambda u\Big(1+
	i\sum\limits_{p=1}^{n}b_p \abs{u}^{2p}\Big)+\jO_{k,n}(2n+3;\abs{\lambda-1}
																											\omega^{2n+2},\rho);
\]%
\item \label{cond:birkhoff-J}   $J(\Phi)=\abs{\partial \Phi}^2-\abs{\bar\partial\Phi}^2$
												 contains no  $(j,j)$-component for  $1\leqslant j \leqslant n$;%
\item\label{cond:birkhoff-Im} $\Im(\varphi_\nu)=0$ if $\nu=({j+1,j})$ 
																									with $1\leqslant j \leqslant n$;%
\item  $\rho'{\asymp}\rho$;
\item \label{cond:birkhoff-bk} $b_1{\asymp}\abs{\lambda-1}
					\omega^2$ and $\abs{b_j}=\jO_{  E_{\beta,n}}(\abs{\lambda-1}\omega^{2j})$ for  
																										$2\leqslant j \leqslant n$;%
\item \label{cond:birkhoff-z}$\abs{\Phi^{-1}(z)}\asymp\abs{z}$ on $\Dm(0;\rho')$; %
\item\label{cond:birkhoff-phi} $\abs{\varphi_\nu}=\jO_{E_{n,\beta}}(\omega^{\abs{\nu}-1})$
																				for $2\leqslant \abs{\nu}\leqslant 2n+2$; %
\end{enumerate}
\end{prop}
%
%
%
%
\subsubsection*{Herman normal form}
This is the quantitative version of~\eqref{eq:Herman-normal-form}. We first state
this result in complex coordinates.
\begin{prop}\label{prop:conjugate-Herman-C}
Assume $2\leqslant k\leqslant 2n$.
Then there exist a diffeomorphism $\psi$ from $\Dm(0;\rho_1)$ into a set containing 
$\Dm(0;\rho_2)$ and a function $\varepsilon~:\Dm(0;\rho_2)\to \R$ with the following 
properties.
\begin{enumerate}
\item $\psi\circ h\circ \psi^{-1}(z)=\lambda z\exp\big(2\pi i\abs{z}^2+\varepsilon(z)\big)$;
\item $\varepsilon$ is a real valued function and $\varepsilon(z)
			=\jO_{k,n}(2n;\abs{\lambda-1}^{-n},\rho_2)$;
\item $\rho_1\asymp \rho$,  $\rho_2\asymp \rho \omega \sqrt{\abs{\lambda-1}}$ and 
			$\abs{\psi(z)}\asymp \abs{z}\omega\sqrt{\abs{\lambda-1}}$.
\end{enumerate}
\end{prop}
\noindent
We also give an equivalent result in polar coordinates, in order
 to apply the invariant curve theorem.
\begin{prop}[Herman normal form]\label{prop:conjugate-Herman}
There exist  $\rho'>0$,  a diffeomorphism $\Psi$ from $(0;\rho')\times\T$ into $\Dm(0;\rho)$ 
and a  function $\bar\varepsilon~:(0;\rho')\times \T\to \R$ with the following properties. 
\begin{enumerate}
\item\label{Cond:herman-polar} 
	$\Psi^{-1}\circ h\circ \Psi(r;\theta)=\Big( r+\bar\varepsilon(\theta,r);
																	\frac{\gamma_0}{2\pi}+\theta+r\Big)$;%
\item\label{Cond:herman-reminder} 
	$\bar\varepsilon$ is a real valued function and $\bar\varepsilon(r,\theta)
									=\jO_{k,n}^\T(n+1;\abs{\lambda-1}^{-n},\rho')$;%
\item\label{Cond:herman-area}
	 $\rho'\asymp  \abs{\lambda-1}^3/q^2$,  $\abs{\Psi(r,\theta)}\asymp \frac{1}{\omega}
		\sqrt{\frac{r}{\abs{\lambda-1}}}$ and $\displaystyle \mathrm{area}
		\Big(	\Psi\big((0;r)\times \T\big)	\Big)\asymp \frac{r}{\omega^2\abs{\lambda-1}}$.%
\end{enumerate}
\end{prop}
%
%
%
\subsubsection{Proof of Proposition~\ref{prop:Birkhoff-invariants-in-complex-coordinates}}
\begin{proof} We construct a polynomial $\Phi$ of degree $2n+2$ such that $\Phi\circ h$ is 
of the form
\begin{equation}\label{eq:birkhoff-conjugation}
	\Phi\circ h=\lambda\Phi\big(1+i\sum_{\ell=1}^{n} b_\ell\abs{\Phi}^{2\ell}\big)
	+\jO(\abs{z}^{2n+3}).
\end{equation}
Taking the $\ell$-homogeneous part of this,  for $2\leqslant \ell \leqslant 2n+2$,  this is equivalent to
\begin{equation}\label{eq:birkhoff-conjugation-2}
	\lambda[\Phi]_\ell(z)-[\Phi]_\ell(\lambda z)=[h]_\ell+\sum_{j=2}^{\ell-1} \bigg[[\Phi]_j
	\circ h\bigg]_\ell-i\lambda \sum_{j=1}^{(\ell-1)/2} b_j\bigg[\Phi\abs{\Phi}^{2j}\bigg]_\ell.
\end{equation}
{\bf Computation of $\mathbf{[\Phi]_2}$.}  Since $h(z)=\lambda z+i\lambda \abs{\lambda-1}
a_2(z+\bar z)^2+\jO(\abs{z}^3)$, we have  
\[
	[h]_2=i\lambda a_2\abs{\lambda-1}(z+\bar z)^2,
\] 
so~\eqref{eq:birkhoff-conjugation-2} implies that 
\[
	\lambda[\Phi]_2(z)-[\Phi]_2(\lambda z)=i\lambda a_2\abs{\lambda-1}(z+\bar z)^2.
\]
From this, we obtain that 
\[
	(\lambda-\lambda^2)\varphi_{2,0}=i\lambda \abs{\lambda-1}a_2~;~
	(\lambda-1)\varphi_{1,1}=2i\lambda \abs{\lambda-1}a_2~;~
	(\lambda-\bar\lambda^2)\varphi_{0,2}=i\lambda \abs{\lambda-1}a_2,
\]
\[
~\text{hence}~\quad
\left\{\begin{array} {l}
	\varphi_{2,0}=ia_2\abs{1-\lambda}/(1-\lambda),\\ %
	\varphi_{1,1}=2ia_2\lambda \abs{1-\lambda}/(\lambda-1),\\ %
	\varphi_{0,2}=ia_2\lambda^3\abs{1-\lambda}/(\lambda^3-1).
\end{array}\right.
\]
This shows in particular that if $\abs{\nu}=2$ then $\abs{\varphi_\nu}
{\asymp}\abs{a_2}{\asymp} \omega$.~\\[5pt]
{\bf Computation of $\mathbf{b_1}$.}  Equation~\eqref{eq:birkhoff-conjugation-2} implies 
that 
\[
	\lambda[\Phi]_3(z)-[\Phi]_3(\lambda z)=[h]_3+\bigg[[\Phi]_2\circ h\bigg]_3
	-i\lambda  b_1\bigg[\Phi\abs{\Phi}^{2j}\bigg]_3.
\]
Therefore, we obtain that 
\[
	\lambda[\Phi]_3(z)-[\Phi]_3(\lambda z)=i\lambda\abs{1-\lambda}a_3 (z+\bar z)^3
	+\bigg[[\Phi]_2(\lambda z+i\lambda\abs{1-\lambda }a_2(z+\bar z)^2)\bigg]_3\!\!\!
	-i\lambda b_1 z\abs{z}^2.
\]
Taking the $(2,1)$~part of this, we obtain  that 
\begin{align*}
i\lambda b_1=3i\lambda\abs{1-\lambda}a_3 
	+\Big[[\Phi]_2(\lambda z+i\lambda\abs{1-\lambda }a_2(z+\bar z)^2)\Big]_{2,1}
\\
	\text{with}~\Big[[\Phi]_2(\lambda z+i\lambda\abs{1-\lambda }a_2(z+\bar z)^2)\Big]_{2,1}%
	\!\!\!\!\!= \varphi_{2,0}\lambda^2 2i\abs{1-\lambda}(2a_2)
																				+\varphi_{1,1}\abs{1-\lambda}(-ia_2)\\ %
	\qquad\qquad\qquad\qquad+\varphi_{0,2}\bar\lambda^2\abs{1-\lambda}(-2ia_2),\\
	=-2a_2^2\lambda \abs{1-\lambda}^2~\frac{1+\lambda}{1-\lambda}~\frac{2+\lambda
																				+2\lambda^2}{1+\lambda+\lambda^2}.%
\end{align*}
This with Corollary~\ref{Cor:Taylor-expansion}~\ref{Cond:a_2-a_3}) implies that 
\[
	\frac{b_1}{\abs{\lambda-1}}=3a_3+2a_2^2\abs{\lambda-1}~R(\lambda)
	\asymp \omega^2.
\]
{\bf Computation of $\mathbf{[\Phi]_3}$.}  With  $\nu=(p,q)\in\N^2$ satisfying 
$p+q=3$ and $\nu\neq (2,1)$, Equation~\eqref{eq:birkhoff-conjugation-2} implies that 
\begin{align*}
	\lambda[\Phi]_\nu-\lambda^{p-q}[\Phi]_\nu&=[h]_\nu+\bigg[[\Phi]_2\circ h\bigg]_\nu
	=i\lambda\abs{\lambda-1} a_3 \big(\begin{smallmatrix}3\\ \nu\end{smallmatrix}\big) 
																				z^\nu+\bigg[[\Phi]_2\circ h\bigg]_\nu\\ %
	&=i\lambda\abs{\lambda-1} a_3 \big(\begin{smallmatrix}3\\ \nu\end{smallmatrix}\big) 
	z^\nu\!\!+\varphi_{2,0}\lambda^2\big(2z i\abs{\lambda-1}a_2(z+\bar z)^2\big)_\nu\\ %
	&									\phantom{=i\lambda\abs{\lambda-1} a_3 \big(\begin{smallmatrix}
																						3\\ \nu\end{smallmatrix}\big) z^\nu}
	+2\varphi_{1,1} ~\Re \big(\bar z i\abs{\lambda-1}a_2(z+\bar z)^2\big)_\nu\\ %
	&						\phantom{=i\lambda\abs{\lambda-1} a_3 \big(\begin{smallmatrix}3\\ \nu
																									\end{smallmatrix}\big) z^\nu}
	+\varphi_{0,2} \bar\lambda^2\big(-2\bar z i\abs{\lambda-1}a_2(z+\bar z)^2\big)_\nu.
\end{align*}
Since $\abs{\lambda^{p-q}-\lambda}=\abs{\lambda^{p-q-1}-1}\geqslant \abs{\lambda-1}$
and $\abs{p-q-1}\leqslant 4$, Lemma~\ref{lem:non-resonance} implies that $\abs{\varphi_\nu}
=\jO_{ E_{\beta,n}}(a_2^2)=\jO_{  E_{\beta,n}}(\omega^2)$.~\\[2pt]
We now compute $\varphi_{2,1}$. We write  
\begin{align*}
	 \partial \Phi(z) & = 1+2\varphi_{2,0}z+\varphi_{1,1}\bar z+2\varphi_{2,1}\abs{z}^2
												+3\varphi_{3,0}z^2+\varphi_{1,2}\bar z^2+\jO(\abs{z}^3),\\ %
	\bar\partial \Phi(z) & = 2\varphi_{0,2}\bar z+\varphi_{1,1} z+2\varphi_{2,1}\abs{z}^2
													+3\varphi_{0,3}\bar z^2+\varphi_{2,1} z^2+\jO(\abs{z}^3).
\end{align*}
Therefore Constraint~\ref{cond:birkhoff-J}) of the proposition implies that 
\[
	0=(J(\Phi))_{1,1}(z)=\abs{z}^2
	\bigg(4\abs{\varphi_{2,0}}^2+\abs{\varphi_{1,1}}^2+4\Re(\varphi_{2,1})-
																4\abs{\varphi_{0,2}}^2-\abs{\varphi_{1,1}}^2\bigg).
\]
This with Constraint~\ref{cond:birkhoff-Im}) shows that $\varphi_{2,1}=\abs{\varphi_{0,2}}^2
-\abs{\varphi_{2,0}}^2$. In particular, it implies that $\abs{\varphi_{2,1}}
=\jO_{E_{\beta,n}}(\omega^2)$.~\\[5pt]
{\bf Estimates of $\mathbf{\Phi}$ and $\mathbf{b_k}$.} We compute $[\Phi]_p$ by induction 
over $p$. Let assume that $4\leqslant p \leqslant 2n+2$ and that we have computed 
$[\Phi]_{\leqslant p-1}$. We also assume that $\abs{\varphi_\mu}
=\jO_{ E_{\beta,n}}(\omega^{\abs{\mu}-1})$ if $\abs{\mu}\leqslant p-1$ and $\abs{b_k}
=\jO_{ E_{\beta,n}}(\abs{\lambda-1}\omega^{2k})$ if $2k\leqslant p-2$.   
Equation~\eqref{eq:birkhoff-conjugation-2} implies that 
\begin{align}
	\lambda[\Phi]_p(z)-[\Phi]_p(\lambda z)
	&=[h]_p+\sum_{\ell=2}^{p-1}\bigg[[\Phi]_\ell\bigg(\lambda z+i\lambda \abs{\lambda-1}
										\sum_{m=2}^{p-\ell+1} a_m (z+\bar z)^m\bigg)\bigg]_p\notag\\ %
	&\label{eq:Phi-lambda} \phantom{[h]_p+\sum_{\ell=2}^{p-1}\bigg[[\Phi]_\ell
												\bigg( \lambda z+\lambda i}-i\lambda \sum_{\ell=1}^{(p-1)/2} 
	b_\ell\bigg[\Phi\abs{\Phi}^{2\ell}\bigg]_p.
\end{align}
We estimate each term of the right hand side of this. Let $\nu\in\N^2$
satisfy $\abs{\nu}=p$.~\\[3pt]
--We have $[h]_\nu=i\lambda\abs{\lambda-1}a_{\abs{\nu}}z^\nu$ and $\abs{a_{\abs\nu}}
	~\abs{\lambda-1}=\jO_{ E_{\beta,n}}\big(\abs{\lambda-1}\omega^{\abs{\nu}-1}\big)$.~\\[3pt]
-- For $1\leqslant \ell< (p-1)/2$, we have
	\begin{equation}\label{eq:phi-2l+1}
	\big[\Phi \abs{\Phi}^{2\ell}\big]_\nu=z^\nu\bigg(\sum_{\underset{\bar{}}{\nu}} 
	\varphi_{\nu_0}\prod_{i=1}^\ell \varphi_{\nu_i} \bar \varphi_{\nu_{i+\ell}}\bigg),
	\end{equation}
	where the summation index $\underset{\bar{}}{\nu}$  above runs over all the tuples 
	$\underset{\bar{}}{\nu}=(\nu_0,\nu_1,\ldots,\nu_{2l})\in(\N^2)^{2\ell+1}$ satisfying 
	the condition 
	\[
	\nu=\nu_0+\sum_{j=1}^\ell (\nu_j+\bar\nu_{j+\ell}),\qquad
	~\text{with $\overline{(r,s)}=(s,r)$ for all $(r,s)\in\N^2$.} 
	\]
	Thus each term in the sum in~\eqref{eq:phi-2l+1} belongs to $\jO_{ E_{\beta,n}}
	\big(\omega^{n_N}\big)$, with 
	\[
			n_N=\abs{\nu_0}-1+\sum_{j=1}^\ell (2\abs{\nu_j}-2)=\abs{\nu}-(1+2\ell).
	\]
	Since $\abs{b_\ell}=O_{ E_{\beta,n}}\big(\abs{\lambda-1}\omega^{2\ell}\big)$, it follows 
	that  $b_\ell\big[\Phi\abs{\Phi}^{2\ell}\big]_\nu$ is of the form $c(\nu,\Phi,\ell)
	z^\nu$, with $\abs{c(\nu,\Phi,\ell)}=O_{ E_{\beta,n}}\big(\abs{\lambda-1}
	\omega^{\abs{\nu}-1}\big)$.~\\[5pt]
-- For $\mu\in\N^2$ satisfying $2\leqslant \abs{\mu}\leqslant p-1$, we estimate the term
	\begin{align*}
		\bigg[[\Phi]_\mu\big(\lambda( z+i\abs{\lambda-1}\sum_{m} a_m (z+\bar z)^m)\big)\bigg]_\nu
		\!\!\!\!&
		= \varphi_\mu\bigg[\!\big(\lambda (z+i\abs{\lambda-1}\sum_{\nu'} a_{\abs{\nu'}}
		\big(\begin{smallmatrix}\abs{\nu'}\\ \nu'\end{smallmatrix}\big) z^{\nu'})
		\big)^{\!\!\mu}\bigg]_\nu\!\!\!,
	\end{align*}
	where  $m $  in the sum in  the  left hand side runs over the intergers satisfying $2\leqslant 
	m\leqslant N_0=p-\abs{\mu}+1$ and $\nu'\in\N^2$ in the sum of the right hand side runs 
	over the couple satisfying $2\leqslant \abs{\nu'}\leqslant N_0$. A direct computation shows 
	that the right hand side of this equality is of the form 
	\[
 	\varphi_\mu\lambda^\mu z^\nu\sum_{(\mu_0,\underset{\bar{}}{\mu})}\prod_{\nu'\in\N}
	\bigg(i\abs{\lambda-1}a_{\abs{\nu'}}
	\big(\begin{smallmatrix}\abs{\nu'}\\ \nu'\end{smallmatrix}\big) \bigg)^{\mu_{\nu'}}\!\!\!,
	\]
	where $\N$ denote the set of the couples $\nu'\in\N^2$ satisfying $\abs{\nu'}\geqslant 2$ 
	and  the index in the sum above runs over all the tuples  $(\mu_0,\underset{\bar{}}{\mu})$, 
	with $\mu_0\in\N^2$ and $\underset{\bar{}}{\mu}=(\mu_{\nu'})\in (\N^2)^{\N}$ 
	satisfying 
	\[
		\mu_0+\sum_{\nu'\in\N} \mu_{\nu'}\cdot \nu'=\nu ~\text{and}~\mu_0
		+\sum_{\nu'\in\N} \mu_{\nu'}=\mu,
	\]
	with $(r,s)\cdot (r',s')=(rr'+ss',rs'+r's)$ for all $(r,s,r',s')\in\N^4$. Note that we have 
	$\mu_{\nu'}\neq 0$ for at leat one index $\nu'$ since otherwise we should have $\nu=\mu_0
	=\mu$, which is impossible since $\abs{\mu}<p=\abs{\nu}$.  On the other hand,  we have  
	$\mu_{\nu'}\neq 0$ for at most  $\abs{\mu}\leqslant p-1\leqslant 2n+1$ indices  $\nu'$.
	This implies that 
	\[
		\prod_{\nu'\in\N}\abs{\lambda-1}^{\mu_{\nu'}}=\jO_{ E_{\beta,n}}(\abs{\lambda-1}).
	\]
	Furthermore, we have $\abs{a_{\abs{\nu'}}}=\jO_{ E_{\beta,n}}(\omega^{\abs{\nu'}-1})$, so 
	\[
	\prod_{\nu'\in\N}\abs{a_{\abs{\nu'}}}^{\mu_{\nu'}}=\jO_{ E_{\beta,n}}
	\bigg(\prod_{\nu'\in\N}(\omega^{\abs{\nu'}-1})^{\abs{\mu_{\nu'}}}\bigg)
	=\jO_{ E_{\beta,n}}\big(\omega^{\abs{\nu}-\abs{\mu}}\big).
	\]
	Since we have  $\abs{\varphi_\mu}=\jO_{ E_{\beta,n}}(\omega^{\abs{\mu}-1})$, it follows 
	that 
	\[
	\bigg\vert\varphi_\mu\lambda^\mu \sum_{(\mu_0,\underset{\bar{}}{\mu})}
	\prod_{\nu'\in\N}\bigg(i\abs{\lambda-1}a_{\abs{\nu'}}
	\big(\begin{smallmatrix}\abs{\nu'}\\ \nu'\end{smallmatrix}\big) \bigg)^{\mu_{\nu'}}
	\bigg\vert=\jO_{ E_{\beta,n}}(\abs{\lambda-1} \omega^{\abs{\nu}-1}).
	\]
	Therefore, we obtain that $\abs{\lambda-\lambda^\nu}~\abs{\varphi_\nu}
	=\jO_{ E_{\beta,n}}(\abs{\lambda-1} \omega^{\abs{\nu}-1})$. With Lemma~\ref{lem:non-resonance}, 
	this  shows that  $\abs{\varphi_\nu}=\jO_{ E_{\beta,n}}(\omega^{\abs{\nu}-1})$ (if $\nu\neq 
	(\ell_0+1,\ell_0)$ when $p=2\ell_0+1$).~\\[3pt]
-- If $p=2\ell_0+1$ and $\nu_0=(\ell_0+1,\ell_0)$ then~\eqref{eq:Phi-lambda}
	shows that 
	\[
	i\lambda b_{\ell_0} z\abs{z}^{2\ell_0}=[h]_{\nu_0}+\sum_{\ell=2}^{p-1}
	\bigg[[\Phi]_\ell\circ h\bigg]_{\nu_0}-i\lambda \sum_{\ell=1}^{\ell_0-1} b_\ell
	\bigg[\Phi\abs{\Phi}^{2\ell}\bigg]_{\nu_0},
	\]
	hence $\abs{b_{\ell_0}}=\jO_{ E_{\beta,n}}(\abs{\lambda-1} \omega^{\abs{\nu_0}-1})
	=\jO_{ E_{\beta,n}}(\abs{\lambda-1} \omega^{2\ell_0})$. \\[3pt]
	Furthermore, a direct computation shows that 
	\[
	\abs{z}^{-2\ell_0}\big(J(\Phi)\big)_{(\ell_0,\ell_0)}=2(\ell_0+1)\Re(\varphi_{\nu_0})
	+\!\!\!\!\sum_{\begin{smallmatrix}
							r'+s=\ell_0+1\\ r+s'=\ell_0+1\\
							(r,s)\neq (1,0), ~(r',s')\neq (1,0)
					\end{smallmatrix}}\!\!\!\!  (rr'-ss')\big(\overline{\varphi_{r,s}}\varphi_{r',s'}\big).
	\]
	Since in the sum above we have $\abs{\overline{\varphi_{r,s}}\varphi_{r',s'}}=
\jO_{ E_{\beta,n}}(\omega^{r+s+r'+s'-2})=\jO_{ E_{\beta,n}}(\omega^{2\ell_0})$,  
	Constraints~\ref{cond:birkhoff-J}) and~\ref{cond:birkhoff-Im}) show that 
	$\abs{\varphi_{\nu_0}}=\jO_{ E_{\beta,n}}(\omega^{2\ell_0})$.

	Thus we have proved the announced estimates at the rank $p$, and so at any order,  and the 
	proof of the points~\ref{cond:birkhoff-J}),
	~\ref{cond:birkhoff-Im}), ~\ref{cond:birkhoff-bk}) and~\ref{cond:birkhoff-phi}) 
 	of the proposition is complete.~\\[5pt]
{\bf Estimate of $\mathbf{\Phi^{-1}}$.}~We shall apply the (inverse) axiom of 
Lemma~\ref{lem:O_k} (see Appendix~\ref{app:O_k}) to the polynomial $\Phi$ with 
$\varepsilon=\abs{\lambda-1}(\sqrt{3/2}-1)/q$, $C=0$ and $\tau\leqslant \rho$ satisfying 
\begin{equation}\label{eq:tau-phi-birkhoff}
		\norm{\partial \Phi-1}_\tau+\norm{\bar\partial\Phi}_\tau\leqslant \varepsilon.
\end{equation}
Note that $\varepsilon\leqslant \sqrt{3/2}-1$, so $2\varepsilon+\varepsilon^2\leqslant 1/2$.
Furthermore, we may choose $\tau\asymp \rho$. Indeed, we have 
\[
	\norm{\partial \Phi-1}_\tau+\norm{\bar\partial\Phi}_\tau =\!\!\!\!\!\!\!\!\sum_{2\leqslant 
	\abs{\nu}\leqslant 2n+2}\!\!\!\!\!\!\abs{\nu}~\abs{\varphi_\nu}\tau^{\abs{\nu}-1}
	=\sum_{\ell=1}^{2n+1}\jO_{ E_{\beta,n}}\big((\omega\tau)^{\ell}\big)
	=\jO_{ E_{\beta,n}}\big(\omega\tau\big) ~\text{if $\omega\tau\leqslant 1$.}
\]
 We recall that $\rho\asymp \frac{\abs{\lambda-1}}{q\omega}$, so $\omega\rho\asymp 
\frac{\abs{\lambda-1}}{q}$.  This implies that there exists $\tau>0$ 
verifying~\eqref{eq:tau-phi-birkhoff} with $\tau\leqslant\rho$ and $\tau\asymp\rho$.We set 
$\rho'=(1-\varepsilon)\tau$, so $\rho'\asymp \rho$. The (inverse) axiom of Lemma~\ref{lem:O_k} 
shows that $\Phi^{-1}$ exists from  $\Dm(0;\rho')$ into $\Dm(0;\rho)$ and that there exists  
$Q{(q,N,\mu)}(z)$, a polynomial in $z$ of degree $2n+2$ and valuation $2$, satisfying   
\begin{equation}\label{eq:Phi-inv}
	\left\{\begin{array}{l}
		\Phi^{-1}(z)=z+Q(z)+\jO_{k, E_{\beta,n}}(2n+3;\varepsilon/\rho^{2n+2},\rho),\\ %
		\norm{\partial Q}_\rho+\norm{\bar\partial Q}_\rho=\jO_{ E_{\beta,n}}(\varepsilon).
	\end{array}\right.
\end{equation}
At last,  ~\eqref{eq:tau-phi-birkhoff} implies that $(1-\varepsilon)\abs{z}\leqslant\abs{\Phi(z)}
\leqslant (1+\varepsilon)\abs{z}$ on $\Dm(0,\tau)$ and this  completes the proof 
of~\ref{cond:birkhoff-z}) in  the proposition.~\\[5pt]
{\bf Estimate of the reminder.} We have shown that there exists a polynomial $\Phi$ of degree 
$2n+2$ that verifies~\eqref{eq:birkhoff-conjugation}. We set 
\[
	P(z)=\lambda z\bigg(1+i\sum_{\ell=1}^{n} b_\ell \abs{z}^{2\ell}\bigg) ~\text{and}~
	V(z)=\lambda\bigg(z+i\abs{\lambda-1}\sum_{m=2}^{2n+2}a_m(z+\bar z)^m\bigg), 
\]
so $\Phi\circ h(z)=P\circ \Phi(z)+\jO(\abs{z}^{2n+3})$ and $h(z)=V(z)+\jO_{k, n}(2n+3;
\abs{\lambda-1}\omega^{2n+2},\rho)$.~\\[3pt]
-- We prove in two steps that $\Phi\circ h-P\circ\Phi$ belongs to $\jO_{k, n}(2n+3;\abs{\lambda-1}
	\omega^{2n+2},\rho)$.~\\[3pt]
{\bfseries $\mathbf{\circ}$~Step~1.} 
	We estimate $\Phi\circ h$. We write $h=V+\varepsilon_0$ and
\begin{equation}\label{eq:hadamard-Vf}
	\Phi\circ h=\Phi\circ V+\int_0^1 \bigg(\partial \Phi(V+t\varepsilon_0)\cdot \varepsilon_0
	+\bar\partial \Phi(V+t\varepsilon_0)\cdot\overline{\varepsilon_0}\bigg)~dt.
\end{equation}
The ($Z$-product)  and the (restriction) axioms of Lemma~\ref{lem:O_k} imply that  
\[
	V(z)=\jO_k(0;(k+1)^{2n+1}\norm{V}_\rho,\rho)~\text{and}~
	\varepsilon_0(z)=\jO_k(0;\abs{\lambda-1}\omega^{2n+2}\rho^{2n+3},\rho).
\]
Since we have $\abs{a_m}\rho^m=\jO_{ E_{\beta,n}}(\omega^{m-1}\rho^m)$ and 
$\omega\rho=\jO_{ E_{\beta,n}}(1)$,  it follows that $V$ and $\varepsilon_0$ both lie in  
$\jO_{k, n}(0;\rho,\rho)$. Therefore the (product) axiom shows that  
\begin{align*}
	\varphi_{\nu}(V+t\varepsilon_0)^{\nu'}
	=\jO_{k, E_{\beta,n}'}(0;\abs{\varphi_\nu}\rho^{\abs{\nu}-1},\rho)
&~\text{uniformely for   $(q,N,\mu,t)\in E_{\beta,n}'= E_{\beta,n}\times [0;1]$, }
\\
&\text{with $\nu=(\nu_1,\nu_2)$}\begin{array}[t]{rl}
\text{and}&\nu'=(\nu_1-1,\nu_2)\\
\text{or}&\nu'=(\nu_1,\nu_2-1).
\end{array}
\end{align*}
This implies that $\partial\Phi\circ (V+t\varepsilon_0)$ and $\bar\partial\Phi\circ 
(V+t\varepsilon_0)$ belong to $\jO_{k, E_{\beta,n}'}(0;1,\rho)$, since we have 
$\abs{\varphi_\nu}\rho^{\abs{\nu}-1}=\jO_{ E_{\beta,n}}\big((\omega\rho)^{\abs{\nu}-1}\big)
=\jO_{ E_{\beta,n}}(1)$.  

From this, since  $\varepsilon_0(z)=\jO_k(2n+3;\abs{\lambda-1}
\omega^{2n+2},\rho)$, ~\eqref{eq:hadamard-Vf} and the (product) axiom show that 
\[ 
	\Phi\circ h=\Phi\circ V+\jO_{k, n}(2n+3;\abs{\lambda-1}\omega^{2n+2},\rho).
\]
Furthermore,  the ($Z$-product) axiom  shows that 
\[
	[\Phi\circ V]_{\geqslant 2n+3}(z)=\!\!\!\!\!\!\sum_{2\leqslant \abs{\nu}\leqslant 2n+2}
	\!\!\!\!\!\!\abs{\varphi_\nu}\sum_{\underline{m}} \jO_{k}(\abs{\underline{m}};
	(k+1)^{\abs{\underline{m}}-1}\abs{\lambda-1}^{p_{\underline{m}}}  
	\abs{a_{\underline{m}}},\rho),
\]
where the indices in the sum  runs over all the tuples $\underline{m}=(m_\ell)\in
(\N^2)^{\abs{\nu}}$ satisfying $m_\ell=(1,0)$ if $\abs{m_\ell}=1$, 
$1 \leqslant \abs{m_\ell}\leqslant 2n+2$ for $1\leqslant \ell \leqslant \abs{\nu}$ and 
$\abs{\underline{m}}>2n+2$, with $\abs{m}=\sum\limits_{\ell \geqslant 0} \abs{m_\ell}$,  
and where
\[
	p_{\underline{m}}=\mathrm{card}\{\ell\mid  \abs{m_\ell}>1\}~;~
	\abs{a_{\underline{m}}}=\prod_{\ell, \abs{m_\ell}>1}\abs{a_{m_\ell}}
	\big(\begin{smallmatrix}\abs{m_\ell}\\ m_\ell\end{smallmatrix}\big).
\]
We have $p_{\underline{m}}\geqslant 1$ and $\abs{\underline{m}}
\leqslant (2n+2)^2$, so the (restriction) axiom shows that 
\[
	\jO_{k}(\abs{\underline{m}};(k+1)^{\abs{\underline{m}}-1}
	\abs{\lambda-1}^{p_{\underline{m}}}  \abs{a_{\underline{m}}},\rho) \subset 
 	\jO_{k, E_{\beta,n}}(2n+3;\abs{\lambda-1}  \abs{a_{\underline{m}}} 
	~\rho^{\abs{\underline{m}}-(2n+3)},\rho).
\]
A direct computation shows that $\abs{a_{\underline{m}}}
=\jO_{ E_{\beta,n}}(\omega^{\abs{\underline{m}}-\abs{\nu}})$, so 
\[
	\abs{\varphi_\nu}\abs{a_{\underline{m}}}
	=\jO_{ E_{\beta,n}}(\omega^{\abs{\underline{m}}-1}\rho^{\abs{\underline{m}-(2n+3)}})
	=\jO_{ E_{\beta,n}}((\omega\rho)^{\abs{\underline{m}}-(2n+3)}\omega^{2n+2})
	=\jO_{ E_{\beta,n}}(\omega^{2n+2}).
\]
From this, it follows that $[\Phi\circ V]_{\geqslant 2n+3}$ lies in 
$\jO_{k,n}(2n+3;\abs{\lambda-1} \omega^{2n+2},\rho)$. Thus we have proved that 
\[
	\Phi\circ h=[\Phi\circ V]_{\leqslant 2n+2}
						+\jO_{k, n}(2n+3;\abs{\lambda-1} \omega^{2n+2},\rho).
\]
{\bf $\mathbf{\circ }$~Step~2.} We estimate $P\circ \Phi$. The ($Z$-product) axiom shows 
that  
\[
	P\circ \Phi(z)=[P\circ \Phi]_{\leqslant 2n+2}(z)
			+\sum_{\ell=1}^n\abs{b_\ell}\sum_{\underline{\nu}}
			\jO_k\bigg(\abs{\nu};\abs{\varphi_{\underline{\nu}}}(k+1)^{\abs{\nu}-1},\rho\bigg),
\]
where the index in the sum runs avor all the $\underline{\nu}=(\nu_0,\ldots,\nu_{2\ell})
\in(\N^2)^{2\ell+1}$ such that $\abs{\nu}\geqslant 2n+3$, with $\abs{\underline{\nu}}
=\sum\limits_{j=0}^{2\ell}\abs{\nu_j}$, and where
\[
	\abs{\varphi_{\underline{\nu}}}=\prod_{j=0}^{2\ell}\abs{\varphi_{\nu_j}}
	\subset\jO_{ E_{\beta,n}}\bigg(\prod_{j=0}^{2\ell}\omega^{\abs{\nu_j}-1}\bigg)
	=\jO_{ E_{\beta,n}}\bigg(\omega^{\abs{\underline{\nu}}-(2\ell+1)}\bigg).
\]
Since we have $\abs{b_\ell}=\jO_{k}(\abs{\lambda-1}\omega^{2\ell})$, we obtain that 
$\abs{b_\ell}~\abs{\varphi_{\underline{\nu}}}=O_{k}(\abs{\lambda-1}
\omega^{\abs{\underline{\nu}}-1})$. This implies that
\begin{align*}
	\abs{b_\ell}\jO_k\Big(\abs{\nu};\abs{\varphi_{\underline{\nu}}}(k+1)^{\abs{\nu}-1}
	,\rho\Big) 
&\subset \jO_{k, n}\big(\abs{\nu};\abs{\lambda-1}\omega^{\abs{\underline{\nu}}-1},
																																	\rho\big)\\ %
&\subset \jO_{k, n}\big(2n+3;\abs{\lambda-1}\omega^{\abs{\underline{\nu}}-1}
																						\rho^{\abs{\nu}-(2n+3)},\rho\big)\\ %
&\subset\jO_{k, n}\big(2n+3;\abs{\lambda-1}\omega^{2n+2}(\omega\rho)^{\abs{\nu}
																												-(2n+3)},\rho\big)\\ %
&\subset \jO_{k, n}\big(2n+3;\abs{\lambda-1}\omega^{2n+2},\rho\big).
\end{align*}
Thus we have proved that 
\[
	P\circ \Phi(z)=[P\circ \Phi]_{ \leqslant 2n+2}(z)
						+\jO_{k, n}\big(2n+3;\abs{\lambda-1}\omega^{2n+2},\rho\big).
\]
Since $[P\circ \Phi]_{ \leqslant 2n+2}=[\Phi\circ h]_{ \leqslant 2n+2}$ by construction of 
$\Phi$, we obtain that 
\[
	\Phi\circ h(z)=P\circ \Phi(z)+\jO_{k, n}\big(2n+3;\abs{\lambda-1}\omega^{2n+2},
							\rho\big)
\]
and this completes the proof of the announced estimate of $\Phi\circ h-P\circ \Phi$.~\\[3pt]
-- Now we may compute the estimate of the reminder $\Phi\circ h\circ \Phi^{-1}-P$.
Equation~\eqref{eq:Phi-inv}, the ($Z$-product) and  (restriction)  axioms imply that 
\[
	\Phi^{-1}(z)=\jO_{k, n}(1;1,\rho).
\]
Thus the (product) axiom shows that 
\[
	(\Phi\circ h-P\circ\Phi)\circ \Phi^{-1}(z)=\jO_{k, n}\big(2n+3;\abs{\lambda-1}
	\omega^{2n+2},\rho\big),
\]
which proves~\ref{cond:birkhoff-conj})  in 
Proposition~\ref{prop:Birkhoff-invariants-in-complex-coordinates}
 and this ends  the proof of the proposition.
\end{proof}
%
%
%
\subsubsection{Proof of Proposition~\ref{prop:conjugate-Herman-C}}
Proposition~\ref{prop:conjugate-Herman-C}  follows immediately from the Birkhoff normal 
form of $h$ (Proposition~\ref{prop:Birkhoff-invariants-in-complex-coordinates}) with 
Lemma~\ref{Lem:exponential-birkhoff} and Lemma~\ref{lem:Herman : exp(i|z|^2)}.

Through the whole section, we  assume that $2n\geqslant k\geqslant 1$, $\abs{b_m}
=\jO_{ E_{\beta,n}}(\abs{\lambda-1}\omega^{2m})$ for $1\leqslant m\leqslant n$ and 
$\omega\rho\asymp \abs{\lambda-1}/q$.  We also recall that $\lambda^p\neq 1$ for 
$1\leqslant p \leqslant 2n+2$ and  $\Phi$ is a diffeomorphism satisfying 
$\abs{\Phi(z)}\asymp \abs{z}$ and 
\[
	\Phi\circ h\circ \Phi^{-1}(z)=\lambda zV(\abs{z}^2)
									+\jO_{k,n}(2n+3,\abs{\lambda-1}\omega^{2n+2}, \rho), 
~\text{with $V(s)=1+i\sum\limits_{m=1}^{n} b_{m} s^m$.}
\]
\begin{lemma}\label{Lem:exponential-birkhoff}
 If  $h$ is symplectic on a neighbourhood of zero and  the Taylor expansion of the Jacobian 
$J(\Phi)=\abs{\partial \Phi}^2-\abs{\bar\partial\Phi}^2$ does not contain any power of the 
form $\abs{z}^{2\ell}$,  for $ 1 \leqslant \ell \leqslant n$,
then we have $J(\Phi)(z)=1+\jO(\abs{z}^{2n+2})$ and 
\[
	\Phi\circ h \circ \Phi^{-1}(z)=\lambda z\exp\biggl(i~\sum\limits_{\ell=1}^n 
	\gamma_\ell \abs{z}^{2\ell}+\jO_{k,n}\big(2n+2;\abs{\lambda-1}
	\omega^{2n+2},\rho\big)\biggr),
\]
where $\gamma_\ell\in\R$ for each $(q,N,\mu)\in  E_{\beta,n}$ and $\abs{\gamma_\ell}
=\jO_{\tilde E_n}(\abs{\lambda-1}\omega^{2\ell})$ for $1\leqslant \ell \leqslant n$.\\
 In particular, we have $\gamma_1=b_1$.
\end{lemma}
\begin{proof}
we split the  proof of the lemma  into $4$~steps. Steps~1,~2 and~4 
follow Moser's arguments.~\\[5pt]
\textbf{Step~1:} We set $g(z)=\lambda zV(\abs{z}^2)$; we check that  $J(g)(z)$ is a 
polynomial in  $\abs{z}^2$.~\\
Indeed, we have $\partial g(z)=\lambda(V(\abs{z}^2)+\abs{z}^2V'(\abs{z}^2))$ and 
$\bar\partial g(z)=\lambda z^2 V'(\abs{z}^2)$, so 
\[
	J(g)(z)=\abs{\partial g(z)}^2-\abs{\bar\partial g(z)}^2=
	\abs{V}^2(\abs{z}^2)+2\abs{z}^2 \Re(\overline{V}V')(\abs{z}^2)
	=\frac{d}{dx}_{\mid_{x=\abs{z}^2}}\!\!\!\!\!\!x\abs{V(x)}^2.
\]
\textbf{Step~2:} We prove that $J(\Phi)(z)=1+\jO(\abs{z}^{2n+2})$. 
Since  $\Phi(z)=\jO_{k,n}(1;1, \rho)$, we have 
\[
	(\Phi\circ h)(z)=g(\Phi(z))+\jO_{k,n}(2n+3;\abs{\lambda-1}\omega^{2n+2},\rho).
\]
Since $h$ is symplectic, we have $J(\Phi\circ h)=J(\Phi)\circ h$ and from the form of 
$\Phi\circ h$ above, it follows that 
\begin{equation}\label{eq:J(phi)}
		J(\Phi)\circ f=(J(g)\circ \Phi)\cdot J(\Phi)+\jO(\abs{z}^{2n+2}),
\end{equation}
We write
\[
	J(\Phi)(z)=1+J_p(z)+\jO(\abs{z}^{p+1}) ~\text{and}~\frac{d}{dx}\bigl(x\abs{V(x)}^2\bigr)
	=1+\alpha x^q+\jO(x^{q+1}),
\]  
where $J_p$ is an homogeneous polynomial of degree $p \leqslant 2n+1$, $\alpha \in\R$ 
and $0<q \leqslant n+1$, so $J(g)(z)=1+\alpha \abs{z}^{2q}
										+\jO(\abs{z}^{2q+2})=J(g)(\Phi(z))+\jO(\abs{z}^{2q+1})$.~\\[5pt]
Since  $h(z)=\lambda z+\jO(\abs{z}^2)$, ~\eqref{eq:J(phi)} implies that 
\begin{equation}\label{eq:jet-J}
	1+J_p(\lambda z)=(1+\alpha \abs{z}^{2q}+\jO(\abs{z}^{2q+1}))\cdot(1+J_p(z))
	+\jO(\abs{z}^{p+1}).
\end{equation}
\begin{minipage}{\textwidth}
Now we use the hypothesis  that $J_p$ does not contain any power of $\abs{z}^2$:
\begin{itemize}
\item[--] If $2q>p$ then~\eqref{eq:jet-J} implies that $J_p(z)=J_p(\lambda z)$, so $J_p=0$;%
\item[--]If $2q=p$ then~\eqref{eq:jet-J} implies that  $J_p(\lambda z)\!-J_p( z)
																			=\alpha \abs{z}^{2q}$, so $J_p=\alpha=0$;%
\item[--] If $2q<p$ then~\eqref{eq:jet-J} implies that $\alpha \abs{z}^{2q}=
																						\jO(\abs{z}^{2q+1}))$, so $\alpha=0$.
\end{itemize}
\end{minipage}
~\\[3pt]
In any case we obtain that  $J_p=0$ if $p \leqslant 2n+1$, so $J(\Phi)(z)=1+\jO(\abs{z}^{2n+2})$.
Furthermore, we should note that this also implies that $J(g)(x)=1+\jO(\abs{x}^{n+1})$.~\\[5pt]
\textbf{Step~3:} we prove the  existence of the complex coefficients $\gamma_\ell$, for $1\leqslant \ell
\leqslant n$.
We have $z^{-1}=\jO_k(-1; k!,\rho)$. Therefore the (product) axiom shows that
\[
	\Phi\circ h\circ \Phi^{-1}(z)=\lambda z\big(V(\abs{z}^2)+\varepsilon_0(z)\big), ~\text{with 
				$\varepsilon_0(z)=\jO_{k,n}(2n+2;\abs{\lambda-1}\omega^{2n+2},\rho)$.}
\]  
Set $L(z)=\log(1+z)$,  where $\log$  denotes the principal value of the logarithm; we estimate 
$L(-1+V(\abs{z}^2)+\varepsilon_0(z))$. 
Since $\abs{b_m}=\jO_{ E_{\beta,n}}(\abs{\lambda-1}\omega^{2m}) \subset \jO_{ E_{\beta,n}}(1)$ 
and $\omega\rho=\jO_{ E_{\beta,n}}(\abs{\lambda-1}/q)$, we have 
\[
	\abs{-1+V(\abs{z^2})+\epsilon_0(z)}=\jO_{ E_{\beta,n}}(\abs{\lambda-1}(\omega\rho)^2)
															  =\jO_{ E_{\beta,n}}(\abs{\lambda-1}^3/q^2).
\]
Therefore, up to changing $\rho$ to $\rho'$, with $\rho'\leqslant \rho$ and $\rho'\asymp \rho$
small enough, we may assume that $\abs{V(\abs{z^2})+\epsilon_0(z)-1}\leqslant 1/2$ on 
$\Dm(0;\rho)$. Thus $\log(V(\abs{z^2})+\epsilon_0(z))$ is well defined. Moreover, we have 
$\abs{\partial^{2n+2}L(z)}\leqslant (2n+2)!2^{2n+2}$ on $\Dm(0,1/2)$. Therefore the 
(Taylor expansion) axiom shows that 
\[
	L(z)=\sum_{\ell=0}^{2n}\frac{(-1)^\ell}{\ell+1}z^{\ell+1}+L_1(z),
								~\text{with $L_1(z)=\jO_{k,n}(2n+2;1,1/2)$.}
\]
Since  $\abs{b_m}=\jO_{ E_{\beta,n}}(\abs{\lambda-1}\omega^{2m})$ and $\omega\rho
=\jO_{ E_{\beta,n}}(1)$, we have 
\[
V(\abs{z}^2)+\varepsilon_0(z)=1+\jO_{k,n}(1;\abs{\lambda-1}\omega,\rho).
\]
Since $2n+2\geqslant k$, we may apply the (composition) axiom, which  shows that 
\[
L_1(V(\abs{z}^2)+\varepsilon_0(z)-1)=\jO_{k,n}(2n+2;\abs{\lambda-1}^{2n+2}
\omega^{2n+2},\rho).
\]
Now we write $L_0=L-L_1$ and $P(z)=V(\abs{z}^2)-1$; we estimate $L_0(P(z)+
\varepsilon_0(z))$. Since $P(z)=\jO_{k,n}(0;\abs{\lambda-1},\rho)$ and 
$\varepsilon_0(z)=\jO_{k,n}(2n+2;\abs{\lambda-1}\omega^{2n+2},\rho)$, the 
(product) axiom shows that 
\[
	(P(z)+\epsilon_0(z))^\ell=P^\ell(z)+\jO_{k,n}(2n+2;\abs{\lambda-1}^\ell
	\omega^{2n+2},\rho).
\]
Furthermore, we have 
\[
	(P(z))^\ell=(i)^\ell\sum_{j=\ell}^{n\ell}\abs{z}^{2j}\sum_{\underline{m}\in N_j} 
	b_{\underline{m}},
\]
where the summation index $\underline{m}=(m_1,\ldots,m_\ell)\in N_j\subset \N^\ell$ 
above runs over all the tuples satisfying $\sum_{q=1}^\ell m_q=j$ and where 
$b_{\underline{m}}=\prod_{q=1}^\ell b_{m_q}$. Therefore  we have $\abs{b_{\underline{m}}}
=\jO_{ E_{\beta,n}}(\abs{\lambda-1}^{\ell}\omega^{2j})$ for $\underline{m}\in N_j$, so the 
(restriction) axiom shows that there exist $\gamma_{j,\ell}\in\C$ for each $(q,N,\mu)\in  E_{\beta,n}$
 satisfying 
$\abs{\gamma_{j,\ell}}=\jO_{\tilde E_n}(\abs{\lambda-1}^\ell\omega^{2j})$ and  
\[
	(P(z))^\ell=i\sum_{j=\ell}^{n} \gamma_{j,\ell} \abs{z}^{2j}+\jO_{k,\tilde E_n}(2n+2;
	\abs{\lambda-1}^{\ell}\omega^{2n+2},\rho).
\]
Thus, setting $\gamma_j=\sum\limits_{\ell=1}^{j} \frac{(-1)^{\ell-1}}{\ell}\gamma_{j,\ell}$,
it follows  from the estimates above that
\[
\left\vert\begin{array}{l}
				\abs{\gamma_j}=\jO_{\tilde E_n}(\abs{\lambda-1}\omega^{2j});\\[5pt]%
				L_0(P(z)+\varepsilon_0(z))=i\sum\limits_{j=1}^{n}\gamma_j \abs{z}^{2j}
				+\jO_{k,n}(2n+2;\abs{\lambda-1}\omega^{2n+2},\rho).
\end{array}\right.
\]
In particular, we have $\lambda z(1+ib_1\abs{z}^2)+\jO(\abs{z}^5)=\lambda z
\exp(i\gamma_1\abs{z}^2+\jO(\abs{z}^4))$, so $ib_1=i\gamma_1$.~\\[5pt]
{\bfseries Step~4:} 
all that remains is to check that $\gamma_\ell\in\R$. We have $J(g)(x)
=\frac{d}{dx}(x\abs{V(x)}^2)=1+\jO(x^{n+1})$, so  $\abs{V(x)}^2=1+\jO(x^{n+1})$ and 
\[
	\bigg\vert   \exp\bigl(\sum_{j=1}^n i\gamma_\ell \abs{z}^{2\ell}\bigr)    \bigg\vert^2
	=V(\abs{z}^2)+\jO(\abs{z}^{2n+2})=1+\jO(\abs{z}^{2n+2}).
\]
This last estimate holds true if and only if $\Re(i\gamma_\ell)=0$ for $1\leqslant\ell
\leqslant n$, which means that $\gamma_\ell\in\R$.
This ends the proof of the existence of the coefficients  $\gamma_j$ with the announced 
properties and the proof of the lemma is complete. 
\end{proof}

\begin{lemma}\label{lem:Herman : exp(i|z|^2)} 
There exist $\rho'>0$  and $\rho''>0$,  with $\rho'\leqslant \rho$, a function $\varepsilon~:
\Dm(0;\rho'')\to\R$ and a diffeomorphism $\varphi$ from $\Dm(0;\rho')$
into an open set that contains $\Dm(0;\rho'')$ satisfying
\begin{enumerate}
\item  $\rho'\asymp \rho\asymp \rho''/\sqrt{{b_1}}$ and  $\abs{\varphi(z)}\asymp 
									\abs{z}\sqrt{{b_1}}$	uniformely on $E_{\beta,n}\times \Dm(0;\rho')$;%
\item $(\varphi\circ \Phi\circ h\circ \Phi^{-1}\circ \varphi^{-1})(z)
									=\lambda z\exp(-2\pi i\abs{z}^2+\varepsilon(z))$ on $\Dm(0;\rho'')$;%
\item $\varepsilon(z)=\jO_{k,n}\big(2n,\abs{\lambda-1}^{-n},\rho\sqrt{{b_1}}\big)$.
\end{enumerate}
\end{lemma}
\begin{proof} Lemma~\ref{Lem:exponential-birkhoff}
shows that 
\[
	\Phi\circ h\circ \Phi^{-1}(z)=\lambda z \exp\Big(i \sum_{\ell=1}^{n}\gamma_\ell 
																		\abs{z}^{2\ell}+\varepsilon_0(z)\Big),
~\text{with  $\varepsilon_0\in \jO_{k,\tilde E_n}(2n+2, 
																			\abs{\lambda-1}\omega^{2n+2},\rho)$.}
\]
With 
$\displaystyle\varphi_0(z)=z\Big(1+\sum_{\ell=2}^{n}\frac{\gamma_\ell}{\gamma_1}
\abs{z}^{2\ell-2}+\frac{1}{\gamma_1}\Im\big({\varepsilon_0(z)}/{\abs{z}^2}\big)\Big)^{1/2}$, 
this is equivalent to the equation
\begin{equation}
\label{eq:exp-Phi} \Phi\circ h\circ \Phi^{-1}(z)=\lambda z \exp
						\Big(i\gamma_1\abs{\varphi_0(z)}^2+\Re\big(\varepsilon_0(z)\big)\Big).
\end{equation}
{\bf We estimate $ \mathbf{\Phi\circ h\circ \Phi^{-1}\circ \varphi_0^{-1} (z)}$.} ~\\[3pt]
We set $\varepsilon_1(z)=(1/\gamma_1)~\Im(\abs{z}^{-2}\varepsilon_0(z))$. 
Since $z^{-1}=\jO_k(-1; k!,\rho)$, the (product) axiom shows that 
$\varepsilon_1(z)=\jO_{k,n}(2n;\omega^{2n},\rho)\subset 
\jO_k(0;(\rho\omega)^{2n},\rho)$. Since $\rho\omega=\jO_{E_{\beta,n}}(1)$ and 
$\abs{\gamma_\ell/\gamma_1}=\jO_{E_{\beta,n}}(\omega^{2\ell-2})~$, it follows that 
\[
	\sum_{\ell=2}^{n} \frac{\abs{\gamma_\ell}}{\abs{\gamma_1}}\abs{z}^{2\ell-2}
				+\big\vert \varepsilon_1(z)\big\vert =\jO_{ E_{\beta,n}}((\rho\omega)^2).
\]
This implies that there exists $\rho'\leqslant \rho$ such that $\rho'\asymp \rho$ and 
\[
\sum_{\ell=2}^{n} \frac{\abs{\gamma_\ell}}{\abs{\gamma_1}}\abs{z}^{2\ell-2}
	+\big\vert \varepsilon_1(z)\big\vert \leqslant \frac{1}{2}~\text{on $\Dm(0;\rho')$.}
\]
In particular,  $\varphi_0(z)$ is well defined for $z\in\Dm(0,\rho')$.
Furthermore, the (Taylor) axiom shows that $\sqrt{1+z}=\sum_{\ell=0}^{k} 
\big(\begin{smallmatrix}1/2\\ \ell\end{smallmatrix}\big)~z^\ell+\jO_{k,n}(k+1; 1,1/2)$, 
so,  by composition,  there exist coefficients $c_\ell\in\R$,  $1 \leqslant \ell \leqslant n$ and 
a real values reminder $\varepsilon_2$ such that 
\begin{equation}\label{eq:estimate-eps2}
	\Big(1+\sum_{\ell=2}^{n} \frac{\gamma_\ell\abs{z}^{2\ell-2}\!\!\!\!\!\!\!\!}{\gamma_1}
	+\varepsilon_1(z)\Big)^{\dem}%
							=1+\sum_{\ell=1}^{n-1} c_\ell \abs{z}^{2\ell}+\varepsilon_2(z),
	~\text{with}\left\{\begin{array}{l}
											\abs{c_\ell}=\jO_{ E_{\beta,n}}(\omega^{2\ell}),\\
											\varepsilon_2(z)=\jO_{k,n}(2n;\omega^{2n},\rho').
									\end{array}\right.
\end{equation}
Thus we have proved that 
\begin{equation}\label{eq:estimate-phi_0}
\varphi_0(z)=z+\sum_{\ell=1}^{n-1} c_\ell z\abs{z}^{2\ell}+\jO_{k,n}
																											(2n+1;\omega^{2n},\rho').
\end{equation}
We write $\varphi_0(z)=P_0(z)+z\varepsilon_2(z)$; up to shrinking $\rho'$ (keeping 
$\rho'\asymp \rho$), we may assume that there exist $c>0$ and $C>0$ such that 
$c\asymp 1$,  $z\varepsilon_2(z)=\jO_k(2n+1;C\omega^{2n},\rho')$ and 
\[
	\norm{\partial P_0-1}_{\rho'}+\norm{\bar\partial P_0}_{\rho'}+C(\rho')^{2n}\omega^{2n}
	\leqslant c\omega^2\rho^2\leqslant \frac{1}{8}.
\]
It follows with  the (inverse) axiom applied with $\varepsilon=c\omega^2\rho^2$ that  
$\varphi_0$ is a diffeomorphism from $\Dm(0,\rho')$ onto an open set that contains a disk of 
the form $\Dm(0,\rho'')$, such that  $\rho''\asymp \rho$ and there exists a polynomial $Q$ of 
degree $2n$ and valuation $2$ such that 
\[
	\varphi_0^{-1}(z)=z+Q(z)+\jO_{k,n}(2n+1;\varepsilon/(\rho'')^{2n},\rho''),
	~\text{with $\norm{Q}_{\rho''}=\jO_{ E_{\beta,n}}(\rho''\varepsilon)$.}
\]
$Q$ is a sum of monomials of the form $d_\nu z^\nu$, with $2 \leqslant\abs{\nu}\leqslant 2n$, 
and we have 
\begin{align*}
	\abs{d_\nu}\rho''^{\abs{\nu}}\leqslant \norm{Q}_{\rho''}
																		=\jO_{ E_{\beta,n}}(\rho''^3\omega^2), \\ %
\text{so $\abs{d_\nu}=\jO_{ E_{\beta,n}}(\rho''^{3-\abs{\nu}}\omega^2)$ and $d_\nu z^{\nu}
																		=\jO_{k,n}(2; \rho''\omega^2 ,\rho'')$.}
\end{align*}
In a similar way, since we have $\rho''\asymp \rho$ and $\varepsilon=\jO_{ E_{\beta,n}}
(\rho^2\omega^2)$, the (restriction) axiom shows that  $\jO_{k,n}(2n+1;
\varepsilon/(\rho'')^{2n},\rho'')\subset \jO_{k,n}(2;\omega^2 
\rho^2 \rho^{2n-1}/\rho^{2n},\rho'')=\jO_{k,n}(2;\omega^2\rho'',\rho'')$. Since 
$\rho\omega=\jO_{ E_{\beta,n}}(1)$,  we obtain that 
\begin{equation}\label{eq:estimate-phi_0-inv}
\varphi_0^{-1}(z)=z+\jO_{k,n}(2;\omega^2\rho'',\rho'')=z+\jO_{k,n}
(2;\omega,\rho'')=\jO_{k,n}(1;1,\rho'').
\end{equation}
This with~\eqref{eq:exp-Phi} and the (composition) axiom shows 
that 
\begin{align}\label{eq:f_0}
\Phi\circ h\circ \Phi^{-1}\circ \varphi_0^{-1} (z)=\lambda \varphi_0^{-1}(z)
													\exp\bigg(i\gamma_1\abs{z}^2+\varepsilon_3(z)\bigg),\\ %
\label{eq:estimate-eps3}~\text{with}~\varepsilon_3(z)=\Re(\varepsilon_0\circ 
	\varphi_0^{-1}(z))=\jO_{k,n}\big(2n+2;\abs{\lambda-1}\omega^{2n+2},\rho\big).%
\end{align}
We should also note that the (restriction) axiom shows that 
\[
i\gamma_1\abs{z}^2+\varepsilon_3(z)
=\jO_{k,n}(0;\omega^2\rho^2,\rho)\subset\jO_{k,n}(0;1,\rho).
\] Since 
$\exp(z)=\sum_{\ell=0}^{k} z^\ell/\ell!+\jO_{k,n}(k+1;1,1)$,  this implies that 
$\exp(i\gamma_1\abs{z}^2+\varepsilon_3(z))=1+\jO_{k,n}(0;\omega^2\rho^2,\rho)
\subset \jO_{k,n}(0;1,\rho)$. From this,~\eqref{eq:f_0} 
and~\eqref{eq:estimate-phi_0-inv}, one can deduce that 
\begin{equation}\label{eq:estimate-f_0}
			\Phi\circ h\circ \Phi^{-1}\circ \varphi_0^{-1} (z)=\jO_{k,n}(1;1,\rho).
\end{equation}
{\bfseries We  estimate $\mathbf{\varphi_0\circ \Phi\circ h\circ \Phi^{-1}\circ \varphi_0^{-1} (z)}$.}~\\[3pt]
We write $f_0(z)=\Phi\circ h\circ \Phi^{-1}\circ \varphi_0^{-1} (z)$ and 
$P_0(z)=zP_1(z)$, with $P_1(z)=1+\sum_{\ell=1}^n c_\ell\abs{z}^{2\ell} $ and $\varphi_0(z)
=z\big(P_1(z)+\varepsilon_2(z)\big)$. Since $P_1$ depends only on $\abs{z}$ and 
$\abs{f_0}(z)=\abs{\varphi_0^{-1}(z)}
\exp(\varepsilon_3(z))$,  we have
\[
		P_1(f_0(z))=P_1(\varphi_0^{-1}(z)e^{\varepsilon_3(z)}).
\]
Since $\abs{\lambda-1}\omega^{2n+2}\rho^{2n+2}=\jO_{ E_{\beta,n}}(1)$, 
~\eqref{eq:estimate-eps3} shows with the ($P$-composition) and (restriction)~axioms
 that 
$\exp(\varepsilon_3(z))=1+\jO_{k,n}(2n+2;\abs{\lambda-1}\omega^{2n+2},\rho)$. 
Therefore the (product)~axiom and ~\eqref{eq:estimate-phi_0-inv}  imply  that 
\[
	\varphi_0^{-1}(z)e^{\varepsilon_3(z)}=\varphi_0^{-1}(z)+\jO_{k,n}(2n+3;%
																						\abs{\lambda-1}\omega^{2n+2},\rho).
\]
From this and the estimates  $\abs{c_\ell}=\jO_{ E_{\beta,n}}(\omega^{2\ell})$, one may 
check that 
\begin{align}
\notag
P_1(\varphi_0^{-1}(z)e^{\varepsilon_3(z)})=P_1(\varphi_0^{-1}(z))+\varepsilon_5(z), \\ %
\label{eq:estimate-eps5}
\text{with $\varepsilon_5(z)=\jO_{k,n}(2n+4;\abs{\lambda-1}\omega^{2n+4},\rho))
																\subset \jO_{k,n}(2n;\omega^{2n},\rho)$.} %
\end{align}
 Since $P_1$ is real valued,  so is $\varepsilon_5$. Moreover, ~\eqref{eq:estimate-eps2},
the (composition) axiom and~\eqref{eq:estimate-phi_0-inv} imply  that
\begin{equation}\label{eq:eps2-pho-inv}
				\varepsilon_2\circ \varphi_0^{-1}(z)=\jO_{k,n}(2n;\omega^{2n},\rho');
\end{equation}
the estimate  $\abs{c_\ell}=\jO_{ E_{\beta,n}}(\omega^{2\ell})$  and~\eqref{eq:estimate-phi_0-inv}
imply  that 
\[
	c_\ell\abs{\varphi_0^{-1}(z)}^{2\ell}
		=\jO_{k,n}(2\ell;\omega^{2\ell},\rho)
		\subset\jO_{k,n}(0;(\omega\rho)^{2\ell},\rho).
\]
 Thus we obtain that  $P_1(\varphi_0^{-1}(z))+\varepsilon_2\circ\varphi_0^{-1}(z)=1
+\jO_{k,n}(0;(\rho\omega)^2,\rho)$.\\
 Since $\rho\omega=\jO_{ E_{\beta,n}}(1)$, we 
may assume  that $(P_1+\varepsilon_2)
\circ \varphi_0^{-1}(z)$ does not vanish, up to shrinking $\rho''$ (keeping $\rho''\asymp \rho$),
and
\begin{equation}\label{eq:inv-e7}
\big((P_1+\varepsilon_2)\circ \varphi_0^{-1}(z)\big)^{-1}=1+\jO_{k,n}
(0;(\rho\omega)^2,\rho'')\subset \jO_{k,_n}(0;1,\rho).  
\end{equation}
From this discussion, it follows that 
\begin{align*}
   \varphi_0\circ f_0(z)
& = f_0(z)\Big(P_1\big(\varphi_0^{-1}(z)e^{\varepsilon_3(z)}\big)+\varepsilon_4(z)\Big)\\ %
& =\lambda \varphi_0^{-1}(z)\exp\big(i\gamma_1\abs{z}^2+\varepsilon_3(z)\big)
				\Big(P_1\big(\varphi_0^{-1}(z)\big)+\varepsilon_5(z)+\varepsilon_4(z)\Big) \\ %
&=\lambda \varphi_0^{-1}(z) \exp\big(i\gamma_1\abs{z}^2+\varepsilon_3(z)\big)\Big(P_1
	\big(\varphi_0^{-1}(z)\big)+\varepsilon_2\circ\varphi_0^{-1}(z)+\varepsilon_6(z)\Big)\\ %
&=\lambda \varphi_0^{-1}(z)\big(P_1+\varepsilon_2\big)\big(\varphi_0^{-1}(z)\big)
				\exp\big(i\gamma_1\abs{z}^2+\varepsilon_3(z)\big)\big(1+\varepsilon_7(z)\big)\\ %
&=\lambda z\exp\big(i\gamma_1\abs{z}^2+\varepsilon_3(z)\big)\big(1+\varepsilon_7(z)\big)%
\end{align*}
with $\varepsilon_4=\varepsilon_2\circ f_0$, 
$\varepsilon_6=\varepsilon_5+\varepsilon_4-\varepsilon_2\circ\varphi_0^{-1}$ and 
$\varepsilon_7=\varepsilon_6\cdot \big((P_1+\varepsilon_2)\circ \varphi_0^{-1}\big)^{-1}$.
Since $P_1$, $\varepsilon_2$, $\varepsilon_3$ and $\varepsilon_5$ are real valued, so is 
$\varepsilon_7$. Furthermore, the (composition) axiom applied to~\eqref{eq:estimate-eps2}
and~\eqref{eq:estimate-f_0} shows that  $\varepsilon_4$ lies in $\jO_{k,n}
(2n;\omega^{2n},\rho)$, and so $\varepsilon_6$ according to~\eqref{eq:estimate-eps5} 
and~\eqref{eq:eps2-pho-inv}. With the (product) axiom, this estimate of $\varepsilon_6$
and~\eqref{eq:inv-e7}  imply that $\varepsilon_7(z)=\jO_{k,n}(2n;\omega^{2n},\rho)$.

Since $\log(1+z)=\sum_{\ell=0}^{k-1}(-1)^\ell/(\ell+1)z^{\ell+1}+\jO_{k,n}(k+1;1,1/2)$, 
the ($P$-composition) shows that we may set $\varepsilon_8(z)=\log(1+\varepsilon_7(z))$,
up to shrinking $\rho''$ (keeping $\rho''\asymp \rho$), so $\varepsilon_8$ is real valued and 
$\varepsilon_8(z)=\jO_{k,n}(2n;\omega^{2n},\rho)$ and 
\[
            \varphi_0\circ f_0(z)=\lambda z\exp\big(i\gamma_1\abs{z}^2+\varepsilon_9(z)\big),
\]
with $\varepsilon_9(z)=\varepsilon_3(z)+\varepsilon_8(z)=\jO_{k,n}
(2n;\omega^{2n},\rho)$ according to~\eqref{eq:estimate-eps3} and the (restriction) axiom; 
we note that $\varepsilon_9$ is real valued.~\\[3pt]
{\bfseries Definition of $\mathbf{\varphi}$.} We set $\varphi_1(z)=z\sqrt{{\gamma_1}/2\pi}$
on a disk of the form $\Dm(0;c\rho/\sqrt{{\gamma_1}})$ with $c>0$ small enough (keeping  
$c\asymp 1$) so $\varepsilon=\varepsilon_9\circ \varphi_1^{-1}$ is well defined; 
we set $\varphi=\varphi_1\circ \varphi_0$. Thus we have 
\[
	\varphi\circ\Phi\circ h \circ \Phi^{-1}\circ \varphi^{-1}(z)
		=\varphi_1\circ \varphi_0\circ f_0\circ \varphi_1^{-1}(z)
		=\lambda z \exp\big(-2\pi i\abs{z}^2+\varepsilon(z)\big).
\]
Since we have $z\sqrt{{2\pi }/{{\gamma_1}}}=\jO_{k,\tilde E_n}(1;\frac{1}{\sqrt{{\gamma_1}}},
\rho\sqrt{{\gamma_1}})=\jO_{k,\tilde E_n}(1;\frac{1}{\omega}\abs{\lambda-1}^{-\frac{1}{2}}),
\rho\sqrt{\gamma_1})$, the (composition) axiom shows that 
\[
	\varepsilon(z)=\varepsilon_9(z\sqrt{{2\pi}/{ \gamma_1}})=\jO_{k,\tilde E_n}(2n;
																	1/{\abs{\lambda-1}}^{n},\rho\sqrt\gamma_1)
\]
and this completes the proof of the lemma.
\end{proof}

%
%
%
\subsubsection{Proof of Proposition~\ref{prop:conjugate-Herman}}
We now focus on the tranformation of $h$
 in polar coordinates, as stated in Proposition~\ref{prop:conjugate-Herman}.
\begin{proof}
\ref{Cond:herman-polar})~We consider the diffeomorphism $\psi$ of 
	Proposition~\ref{prop:conjugate-Herman-C} and set $\Psi=\psi^{-1}\circ p$, with $p(r,\theta)
	=\sqrt{r} \exp(2\pi i\theta)$. Since $\psi^{-1}$ is a diffeomorphism from $\Dm(0;\rho_2)$ 
	onto its image, so is $\psi$ from $(0;\rho_2^2]\times \T$.  This proves the point~1 of the 
	corollary, with $\rho'=\rho_2^2$ and $\bar\varepsilon(r,\theta)=r\big(\exp(2\varepsilon\circ 
	p(r,\theta))-1\big)$.~\\[3pt]
\ref{Cond:herman-reminder})~The function $\bar\varepsilon$ is real valued by construction. We now estimate 
	$\bar\varepsilon$. Proposition~\ref{prop:conjugate-Herman-C} shows that $\rho_2\asymp \rho
	\omega\abs{\lambda-1}^{1/2}$,  so $\rho_2^{2n}\abs{\lambda-1}^{-n}\asymp 
	(\rho\omega)^{2n}=\jO_{E_{\beta,n}}(1)$. Therefore, up to shrinking $\rho_2$ (keeping 
	$\rho_2\asymp \rho\omega\abs{\lambda-1}^{1/2}$), we may assume that $2\varepsilon
	(\Dm(0;\rho_2))\subset\Dm(0;1)$. The ($P$-composition) axiom of Lemma~\ref{lem:O_k}, 
	applied to $\exp(z)-1=P(z)+\jO_{k,n}(2n;1,1)$ composed with 
	\[
		2\varepsilon(z)=\jO_{k,n}(2n;C_1/\rho_2^{2n},\rho_2),~\text{where $C_1
		=\abs{\lambda-1}^{-n}\rho_2^{2n}$ and $P(z)=\sum\limits_{\ell=1}^{2n-1} 
		z^\ell/\ell !$},
	\] 
	shows that $\exp(2\varepsilon(z))=1+\jO_{k,n}(2n;C_{01}/\rho_2^{2n},\rho_2)$,  
	with
	\[
	C_{01}=C_1^{2n}+\norm{P}_{{}_{C_1}}=\jO_{E_{\beta,n}}
	\bigg(\sum_{\ell=1}^{2n}C_1^{\ell}\bigg)=\jO_{E_{\beta,n}}(C_1), 
	~\text{since $C_1=\jO_{E_{\beta,n}}(1)$.}
	\]
	Thus we have proved that $\exp(2\varepsilon(z))-1=\jO_{k,n}(2n;%
	\abs{\lambda-1}^{-n},\rho_2)$, hence the ($Z$-product) axiom 
	of~\ref{lem:O_k} applied twice shows that 
\[
	\abs{z}^2\big(\exp(2\varepsilon(z))-1)=
	\jO_{k,n}(2n+2;\abs{\lambda-1}^{-n},\rho_2).
\] 
 Applying the ($\T$-composition) axiom of 
	Lemma~\ref{lem:O_kT} to this map composed with $p$, the  estimate of $\bar\varepsilon$ 
	of the lemma follows, with $\rho'=\rho_2^2$.~\\[3pt]
\ref{Cond:herman-area})~We already know that $\rho'=\rho_2^2\asymp \rho^2\omega^2\abs{\lambda-1}
	\asymp \abs{\lambda-1}^3/q^2$. Furhermore, we have $\abs{p(r,\theta)}=\sqrt{r}$ and 
	Proposition~\ref{prop:conjugate-Herman-C} shows that 
\[
	\abs{\psi^{-1}(z)}\asymp 
	\frac{\abs{z}}{\omega\sqrt{\abs{\lambda-1}}},
~\text{hence}~
	\abs{\Psi(r,\theta)}\asymp 
	\frac{\sqrt{r}}{\omega\sqrt{\abs{\lambda-1}}}.
\] 
	At last, we observe that $\psi(0)=0$. 
	Therefore the diffeomorphism $\psi^{-1}$ maps any Jordan curve with $0$ in the interior 
	onto a Jordan curve with $0$ in the interior. But  the estimate of $\abs{\psi^{-1}(z)}$ above 
	implies that the Jordan curve $\psi^{-1}(\partial\Dm(0,\sqrt{r}))$ lies between two circles 
	centered at zero with radii comparable to   $\frac{1}{\omega}\sqrt{r/\abs{\lambda-1}}$.
	Therefore we have 
	\[
	\mathrm{area}(\psi^{-1}(\Dm(0;\sqrt{r})))\asymp \bigg(\frac{\sqrt{r}}{\omega
	\sqrt{\abs{\lambda-1}}}\bigg)^2 =\frac{r}{\omega^2\abs{\lambda-1}}.
	\]
	This implies the last estimate of the proposition since $p\big((0;r]\times\T\big)
	=\Dm(0;\sqrt{r})\setminus\{0\}$ and the proof of Proposition~\ref{prop:conjugate-Herman}
 is complete.
\end{proof}


%

\subsection{The invariant  curve theorem}

The following statement is taken from \cite{herman1}, VII.11.3 and VII.11.11.A.1.

\label{sec:stateHermanthm}

\begin{thmnonb}[Herman \cite{herman1}]
Assume $\delta>0$ and set $\A_\delta=\T\times [-\delta,\delta]$. Let $\gamma\in\R$ and 
$\Gamma>0$ satisfy
\begin{equation}\label{eq:constant-type}
	0<\Gamma\leqslant \inf_{q\geqslant 1,\, p\in\Z}
        \big\{q\abs{q\gamma-p} \}.
\end{equation}
Then there exist two constants $c_1>0$ and $C_2\geqslant 0$ such that for any embedding  
$F:\A_\delta\to\A$  of the form
\[
F(\theta,r)=(\theta+\gamma+r,r+\varphi(\theta,r)),\quad\text{with}~\varphi\in
C^4(\A_\delta)
\]
satisfying the essential circle intersection property and 
\begin{equation}\label{eq:herma-cond-c1}
	\max_{1\le i+j\le 4}\norm{\partial_r^i\partial_\theta^j\varphi}_{C^0(\A_\delta)}
	\leqslant  c_1\Gamma^2,
\end{equation}
 there is an unique function  $\psi\in W^{3,2}(\T)$ and a
 diffeomorphism  $f\in C^1(\T)$
with rotation number $\gamma$ such that $F(\theta,\psi(\theta))=\big(f(\theta),\psi\big(f(\theta)
\big)\big)$ and we have 
\begin{equation}\label{eq:hermann-W32}
	\norm{\psi}_{W^{3,2}(\T)} \leqslant C_2\Gamma^{-1}\max_{1\le i+j\le 4}
	\norm{\partial_r^i\partial_\theta^j\varphi}_{C^0(\A_\delta)}.
\end{equation}
If $\delta/\Gamma\geqslant 0.123$ then $c_1=14$ and $C_2=0.097$ are suitable for  
any $\gamma$ satisfying~\eqref{eq:constant-type}.  
\end{thmnonb}
\noindent
This theorem requires a few comments.~\\[3pt]
1)~We recall that  $F~:\A_\delta\to\A$ satisfies {\em the essential circle interection 
	property} provided that each simple essential curve $\jC\subset \A_\delta$ (homotopy 
	equivalent to the circle $\{r=0\}$) satisfies $F(\jC)\cap\jC\not =\varnothing$. In our case, since 
	$F_{q,N,\mu}$ is symplectic and fixes $a_q$, this condition is automatically 
	fulfilled.~\\[3pt]
2) Here $W^{3,2}(\T)$ denotes the Sobolev space of all distributions $\psi\in\D'(\T)\cap 
	L^2(\T)$ with derivatives $\mathrm{D}^k\psi$ in $ L^2(\T)$ for $0 \leqslant k \leqslant 3$, 
  	endowed with the norm 
	\[
	\norm{\psi}_{W^{3,2}(\T)}=\left(\norm{\mathrm{D}^3\psi}_{L^2(\T)}^2 
	+\abs{\widehat{\psi}(0)}^2\right)^{1/2},~\text{with}~\widehat{\psi}(0)=\int_0^1\psi(t)~dt.
	\]
	In particular, it is standard to prove that  $W^{3,2}(\T)$ embeds in $C^{0}(\T)$ (see 
	Proposition IV.3.7 in~\cite{herman1} ) and 
	\[
	\norm{\psi-\widehat{\psi}(0)}_{C^0(\T)}\leqslant \frac{1}{12\sqrt{210}}
	\norm{\mathrm{D}^3\psi}_{L^2(\T)}.
	\]
	Since here $f(\theta)=\theta+\gamma+\psi(\theta)$ is a diffeomorphism of the circle with rotation 
	number $\gamma$, one can prove that $f-\gamma$ has a fixed point (see \cite{herman2}). This 
	implies that $\psi$ should vanished at some point $x_0\in\T$. Therefore we have 
	$\abs{\hat{\psi}(0)}=\abs{\psi(x_0)-\hat{\psi}(0)}\leqslant
	\norm{\psi-\hat{\psi}(0)}_{C^0(\T)}$, hence 
	\begin{equation}\label{eq:C0-estimate-W3-2}
		\norm{\psi}_{C^0(\T)}\leqslant 2\norm{\psi-\widehat{\psi}(0)}_{C^0(\T)}
		\leqslant \frac{1}{6\sqrt{210}}\norm{\mathrm{D}^3\psi}_{L^2(\T)}\leqslant
		\frac{1}{6\sqrt{210}}		\norm{\psi}_{W^{3,2}(\T)}.
	\end{equation}
3) One says that the number $\gamma\in\R$ is of constant type with Markoff constant at least  
	$\Gamma$ exactly when  it satisfies~\eqref{eq:constant-type}.  All we need to know 
	about it   is the following result (see IV.3.5 in~\cite{herman1}).

\begin{lemma}[Herman]\label{lem:constant-type}
There exists a constant $c>0$ such that for  all $0<\eta<1/2$, if $[a,b]\subset [0,1]$ satisfies 
$\abs{b-a}\geqslant \eta$ then $[a,b]$ contains infinitely many numbers of constant type with 
Markoff constant at least $c\eta$.
\end{lemma}
Notice that  we may shrink  $c$  as we  need; in the following, we shall take 
$1/c \geqslant 0.123$ in order to apply Herman's theorem.

\subsection{Conclusion of the proof of Theorem~\ref{th:varpseudopend}}
\label{secconclupfThmF}
 Now we have all the ingredients to prove~Theorem~\ref{th:varpseudopend}, 
which follows immediately from the following proposition. We recall that $B_{q,N}$
is a $q$-adapted box for $\GNm$ and that $\GNm^q=\FNqm$ on $B_{q,N}$.
The set $E_{\beta,n}$ which appears in the statement is the one
defined in~\eqref{eq:mu-small2} on p.~\pageref{eq:mu-small2}.

\begin{prop}
There exists a real number $\eta_0>0$  and, 
for each $(q,N,\mu)\in E_{\beta,n}$, a disc 
$\Omega_{q,N,\mu}\subset B_{q,N}$ satisfying the following conditions.
\begin{enumerate}
\item  If  $\mu q^5/N^4<\eta_0$ then $F_{q,N,\mu}(\Omega_{q,N,\mu})=\Omega_{q,N,\mu}$;
\item $\mathrm{area}\big(\Omega_{q,N,\mu}\big)\asymp \frac{\mu}{N^2}$. 
\end{enumerate}
\end{prop}
\begin{proof} We set $n\geqslant 8$, $k=4$, $\omega=q^4/N^3$; we
  recall that the number~$\alpha_{q,N}$ introduced in
  Proposition~\ref{prop:Taylor-expansion-linear-part} (on p.~\pageref{prop:Taylor-expansion-linear-part}) satisfies
  $\alpha_{q,N}\asymp q^5/N^4$, and $\lambda=\exp(i\gamma_0)$ satisfies
  $\lambda+\lambda^{-1}=2-\mu\alpha_{q,N}$ and
  $\abs{\lambda-1}=\mu\alpha_{q,N}$.~\\[3pt]
$\bullet$~Since $2n+2\geqslant k$, Proposition~\ref{prop:complex-coord} applied to 
	$F_{q,N,\mu}$  and $2n+2$ shows that for each $(q,N,\mu)\in  E_{\beta,n}$ there exist
$\rho>0$  and a map $\Psi_0$ from a disk $\Dm(0;\rho)$ into $B_{q,N}$ such that 
	\[
		\Psi_0^{-1}\circ F_{q,N,\mu}\circ \Psi_0(z)=\lambda\bigg(
								\sum_{\nu=2}^{2n+2} a_\nu (z+\bar z)^{\nu}
								+\jO_{k,n}\big(2n+3,\omega^{2n+2},\rho\big)
								\bigg),
	\]
	with $(-1)^{\nu-1}a_\nu\asymp \omega^{\nu-1}$ and $\rho\asymp 
	\frac{\abs{\lambda-1}}{q\omega}$. Since $2n\geqslant k=4$, we may  apply 
	Proposition~\ref{prop:conjugate-Herman}: for each $(q,N,\mu)\in E_{\beta,n}$, there exist
	$\rho'>0$ and  a map $\Psi_1$ from $\T\times (0;\rho')$ into 
	$\Dm(0;\rho_{q,N,\mu})$  satisfying
	\[
		\Psi_1^{-1}\circ \Psi_0^{-1}\circ F_{q,N,\mu}\circ \Psi_0 \circ \Psi_1(\theta,r)
							=\bigg(\frac{\gamma_0}{2\pi}+\theta+r,r+\bar\varepsilon(r,\theta)\bigg),
	\]   
	where  $\bar\varepsilon(r,\theta)=\jO_k^\T(n+1;\abs{\lambda-1}^{-n},\rho')$ and 
	$\rho'\asymp \abs{\lambda-1}^3/q^2$. ~\\[3pt] 
$\bullet$~
	Since $\abs{\gamma_0}\asymp\abs{\lambda-1}\asymp \mu q^5/N^4$, we may choose $\eta_0>0$ small enough 
	so $\rho'+\frac{\abs{\gamma_0}}{2\pi}<1$. Lemma~\ref{lem:constant-type} shows that 
	there exists $r_0\in [\rho'/3;2\rho'/3]$ such that $\gamma=\frac{\gamma_0}{2\pi}+r_0$ is of 
	constant type with Markoff constant $\Gamma\geqslant  c\rho'/3$. We set 
	$\Psi_2(\theta,r')=(\theta,r_0+r')$ and $\delta=\rho'/3$, so $\Psi_2$ defines an embedding 
	from $\A_\delta$ into $\T\times (0;\rho')$ such that 
	\[
	F(\theta,r'):=\Psi_2^{-1}\circ \Psi_1^{-1}\circ \Psi_0^{-1}\circ F_{q,N,\mu}\circ \Psi_0\circ\Psi_1\circ
	\Psi_2(\theta,r')=\bigg(\gamma+\theta+r',r'+\tilde\varepsilon(\theta,r')\bigg),
	\]   
	with $\tilde\varepsilon(\theta,r')=\bar\varepsilon(r_0+r',\theta)$.~\\[3pt]
$\bullet$ We apply Herman's theorem of Section~\ref{sec:stateHermanthm}
 to $F$ on $\A_\delta$ with 
	$c_1=14$, $C_2=0.097$ and $\Gamma=c\rho'/3$. These constants $c_1$ and $C_2$ 
	are suitable because $\delta/\Gamma=\frac{1}{c}\geqslant 0.123$, so we have 	
	\[
	\max_{1\le i+j\le 4}\norm{\partial_r^i\partial_\theta^j\tilde\varepsilon}_{C^0(\A_\delta)}
	\leqslant \max_{1\le i+j\le 4}\norm{\partial_r^i\partial_\theta^j\bar 
	\varepsilon}_{C^0((0;\rho')\times\T)}=\jO_{ E_{\beta,n}}(\abs{\lambda-1}^{-n}\rho'^{n-3}).
	\]
	Since we have $\abs{\lambda-1}\asymp \mu q ^5/N^4$, $\abs{\lambda-1}^{-n}\rho'^{n-3}
	\asymp\abs{\lambda-1}^{2n-9}/q^{2n-6}$ 
	and $\Gamma^2\asymp\rho'^2\asymp \abs{\lambda-1}^6/q^4$ and since we assume 
	$n\geqslant  8$, we may choose $\eta_0>0$ small enough so~\eqref{eq:herma-cond-c1} is 
	satisfied for all $(q,N,\mu) \in  E_{\beta,n}$ verifying  $\mu q^5/N^4 \leqslant \eta_0$.  

	\psset{xunit=1cm,yunit=1cm}
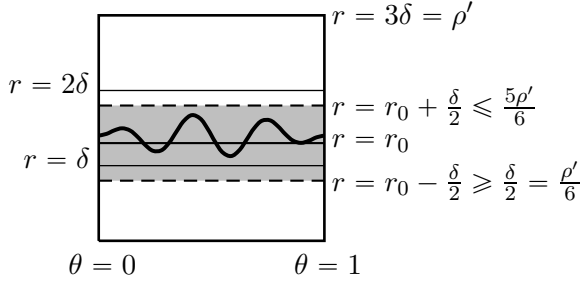
\begin{pspicture}(-0.8,-1)(3,4)
\psframe*[linecolor=lightgray](0,0.8)(3,1.8)
\psline[linewidth=1pt]{}(0,0)(0,3)(3,3)(3,0)(0,0)
\psline[linewidth=0.5pt]{}(0,1)(3,1)
\psline[linewidth=0.5pt]{}(0,2)(3,2)
\psline[linewidth=0.8pt]{}(0,1.3)(3,1.3)
\psline[linewidth=0.8pt,linestyle=dashed]{}(0,0.8)(3,0.8)
\psline[linewidth=0.8pt,linestyle=dashed]{}(0,1.8)(3,1.8)

\def\f{ x  360 mul sin 8 div x  mul  3 x sub  mul   1.4 add }

\psplot[plotstyle=curve,linewidth=1.5pt,plotpoints=30]{0}{3}{\f}
\rput[l](-0.4,-0.3){$\theta=0$}
\rput[l](2.6,-0.3){$\theta=1$}
\rput[l](3.1,1.3){$r=r_0$}
\rput[l](3.1,0.8){$r=r_0-\tfrac{\delta}{2}\geqslant \tfrac{\delta}{2}=\tfrac{\rho'}{6}$}
\rput[l](3.1,1.8){$r=r_0+\tfrac{\delta}{2}\leqslant \tfrac{5\rho'}{6}$}
\rput[l](-1,1.1){$r=\delta$}
\rput[l](-1.2,2.1){$r=2\delta$}
\rput[l](3.1,3){$r=3\delta=\rho'$}

\end{pspicture}

	Herman's theorem shows that there exists a map $\psi~:\T\to\R$ 
	such that $\jC=\{(\theta,\psi(\theta))\}$ is globally invariant  by 
	$h(\theta,r)=\big(\gamma+\theta+r',r'+
	\tilde{\varepsilon}(\theta,r')\big)$ and~\eqref{eq:hermann-W32} holds true. This 
	with~\eqref{eq:C0-estimate-W3-2} implies that 
	\[
	\norm{\psi}_{C^0(\T)}=\jO_{ E_{\beta,n}}\bigg(\Gamma^{-1}\max_{1\le i+j\le 4}
	\norm{\partial_r^i\partial_\theta^j\tilde \varepsilon}_{C^0(\A_\delta)}\bigg)
	=\jO_{ E_{\beta,n}}(\abs{\lambda-1}^{-n}\rho'^{n-4}).
	\]
	Since $\abs{\lambda-1}^{-n}\rho'^{n-4}\asymp {\abs{\lambda-1}^{2n-12}}/{q^{2n-8}}$,
	$\delta\asymp\abs{\lambda-1}^3/q^2$ and  $n\geqslant 8$, we may 
	choose $\eta_0>0$ small enough so $\norm{\psi}_{C^0(\A_\delta)}\leqslant\delta/2$ for 
	all  $(q,N,\mu) \in E_{\beta,n}$ verifying $\mu q^5/N^4\leqslant \eta_0$.
	\vspace{3pt}
	
	Thus  we have proved that the Jordan curve $\jC=\{\Psi_1\circ \Psi_2(\theta,\psi(\theta))\}$ 
	is invariant by  $\Psi_0^{-1}\circ F_{q,N,\mu}\circ \Psi_0$. Since $\rho'/3 \leqslant r_0 
	\leqslant 2/3\rho'$ and $\norm{\psi}_{C^0(\T)}\leqslant \rho'/6$,  we have 
	$\rho'/6 \leqslant r_0+\psi(\theta)\leqslant 5\rho'/6$ on $\T$, so the estimate of $\Psi_1$ 
	in Proposition~\ref{prop:conjugate-Herman} shows that 
	\[
	\mathrm{area} \big(\Int(\jC)\big)\asymp\frac{\rho'}{\omega^2\abs{\lambda-1}}\asymp 
	\frac{\abs{\lambda-1}^2}{q^2\omega^2}.
	\]
	Therefore  $\Omega_{q,N,\mu}=\Psi_0(\Int(\jC))$ is invariant by $F_{q,N,\mu}$
	and  the 
	point~\ref{Cond:complex-coordinate-area}) in  Proposition~\ref{prop:complex-coord} 
	indicates that 
	\[
	\mathrm{area}(\Omega_{q,N,\mu})\asymp \kappa \frac{\abs{\lambda-1}^2}{q^2\omega^2}
	\asymp \frac{\abs{\lambda-1}}{q\omega N}\asymp\frac{\mu \alpha_{q,N}}{q\omega N}
	=\frac{\mu}{N^2};
	\]
 	this completes the proof of the proposition.
\end{proof}


\vfil

\pagebreak

\section{Coupling devices, multi-dimensional periodic domains,
  wandering domains}\label{SecPfWander}
%

At this point of the paper, it only remains to be proven 
Theorem~\ref{th:lowerbounds} stated in Section~\ref{secthmlowbdsC}
and Part~(ii) of Theorem~\ref{th:perdomains} stated in Section~\ref{secthmperdomD}.
Both proofs will make use of a ``coupling lemma'' which is the object
of Section~\ref{Sec:Couplem}.


\subsection{Coupling devices}\label{Sec:Couplem}


We quote here almost exactly Lemma~3.2 of \cite{ms},
which was itself a simple adaptation of a result already present in \cite{hms}.
Though very simple, this coupling lemma plays a crucial role in our constructions.


\begin{lemma} \label{Lem:cl}
Let $m,m'\ge1$ be integers.
Let $F \col \A^m \righttoleftarrow$
and $G \col \A^{m'} \righttoleftarrow$
be two diffeomorphisms,
and let $f \col \A^m\to\R$ and $g \col \A^{m'}\to\R$ be two Hamiltonian functions
which generate complete vector fields.

Suppose moreover that we are given $q\ge1$ integer and $\jV\subset\A^{m'}$
such that~$\jV$ is $q$-periodic for~$G$ (\ie $\jV=G^q(\jV)$)
and the ``synchronization conditions''
\beq\label{eq:sync1}
g(x')=1,\quad
\dd g(x')=0,\quad
g(G^s(x'))=0,\quad
\dd g(G^s(x'))=0,\quad 1\le s\le q-1,
\eeq
hold for all $x'\in \jV$.

Then $f\otimes g$ generates a complete Hamiltonian vector field
and the diffeomorphism
$\cF \defeq \Phi^{f\otimes g} \circ (F\times G) \col \A^{m+m'}
\righttoleftarrow$ 
satisfies
\beq \label{eq:iterates}
\cF^{\ell q+s}(x,x') = 
\Big( F^s\circ\big(\Phi^f\circ F^q\big)^\ell(x), \  G^{\ell q+s}(x') \Big),
\qquad x\in\A^m,\ x'\in \jV,
\eeq
for all integers $\ell,s\in\Z$ such that $0 \le s \le q-1$.
\end{lemma}


We have denoted by~$f\otimes g$ the function $(x,x')\mapsto f(x)g(x')$,
and by~$F\times G$ the product diffeomorphism $(x,x')\mapsto \big(F(x),G(x')\big)$.

\begin{proof} 
See the proof of Lemma~3.2 in \cite[p.~1631]{ms}.
The point is that
\[
\Phi^{f\otimes g}(x,x')=\big(\Phi^{g(x')\,f}(x),\Phi^{f(x)\, g}(x')\big),\qquad x\in\A^m,\ x'\in\A^{m'},
\]
(as proved in \cite{hms}, using the invariance of both~$f$ and~$g$ by
the Hamiltonian vector field generated by $f\otimes g$), so the
synchronization conditions~\eqref{eq:sync1} easily imply~\eqref{eq:iterates}.
\end{proof}


Notice that, under the assumptions of Lemma~\ref{Lem:cl}, the union
\beq   \label{eqtijVdisjun}
\ti\jV \defeq \jV \sqcup G(\jV) \sqcup \ldots \sqcup G^{q-1}(\jV)
\eeq
is a disjoint union because, for any $s\in\{1,\ldots,q-1\}$, the synchronization conditions~\eqref{eq:sync1} say that
$\jV \subset g\ii(1)$ 
and $G^s(\jV) \subset g\ii(0)$.
Thus any $x' \in \ti\jV$ can be written $x' = G^s(x'_0)$ with uniquely
determined $s\in\{0,1,\ldots,q-1\}$ and $x'_0 \in \jV$;
then~\eqref{eq:iterates} shows that
$\cF^{s}\big( F^{-s}(x),x'_0 \big) = (x,x')$ and that
\[
\cF^k(x,x') =
\Big( F^{s_1}\circ\big(\Phi^f\circ F^q\big)^{\ell_1}\circ F^{-s}(x), \  G^{k}(x') \Big),
\qquad x\in\A^m,\ x'\in \ti\jV,
\]
with $k+s = \ell_1 q + s_1$.
In particular, the set $\A^m\times\ti\jV$ is invariant under~$\cF$ and the second projection
makes
$G\big\vert _{\ti\jV}$ a factor of $\cF\big\vert _{\A^m\times\ti\jV}$.
Note also that~\eqref{eq:iterates} yields
\beq    \label{eqqiterate}
\cF^q\big\vert _{\A^m\times\jV} = 
\big( \Phi^f\circ F^q \big) \times \big( G^q \big\vert _\jV \big).
\eeq
Assuming furthermore that there is a subset $\jU \subset \A^m$ which
is periodic or wandering for $\Phi^f\circ F^q$, we easily obtain that
$\jU \times \jV \subset \A^{m+m'}$ is periodic or wandering for~$\cF$.
This is essentially the content of the folowing two corollaries.


\begin{cor}  \label{cor:couplingper}
Let $\cF = \Phi^{f\otimes g} \circ (F\times G) \col
\A^{m+m'}\,\to\,\A^{m+m'}$
with $m$, $m'$, $F$, $G$, $f$, $g$, $q$
and $\jV \subset \A^{m'}$ as in Lemma~\ref{Lem:cl}
(in particular $\jV$ is $q$-periodic for~$G$ and the synchronization 
conditions~\eqref{eq:sync1} hold).

Assume now that the diffeomorphism $\Phi^f\circ F^q$ admits a
$p$-periodic subset $\jU\subset\A^{m}$, with a certain integer $p\ge1$.
Assume moreover 
%
%
that there exist sets $\jB \subset \jB_* \subset \A^m$
and $\jB' \subset \jB'_* \subset \A^{m'}$ such that
\begin{align}
\label{eq:nonov1}
%
%
\jU &\subset \jB,
& & \hspace{-.25em} (\Phi^f\circ F^{q})^k(\jU)\cap \jB_* =\varnothing
&  & \hspace{-1.9em} \text{for\ } 1\leq k \leq p-1,\\[1ex]
\label{eq:nonov2}
%
%
\jV &\subset \jB',
& & \hspace{-.25em} G^k(\jV)\cap \jB'_* =\varnothing
& & \hspace{-1.9em} \text{for\ } 1\leq k\leq q-1.\\
\intertext{Then the product set $\jU\times\jV \subset \A^{m+m'}$ is
$(pq)$-periodic for the diffeomorphism~$\cF$ and%
%
%
}
\label{eq:nonov3}
%
%
\jU \times \jV &\subset \jB \times \jB',
& & \hspace{-.25em} \cF^k(\jU\times\jV)\cap (\jB_*\times\jB'_*)=\varnothing
& & \hspace{-.7em} \text{for\ } 1\leq k\leq pq-1.
\end{align}
\end{cor}


\begin{proof} 
Let $\psi \defeq \Phi^f\circ F^q$.
By~\eqref{eqqiterate},
$\cF^{pq}(\jU\times \jV)=\Big(\psi^p(\jU),\  G^{p
  q}(\jV)\Big)=\jU\times \jV$
and this $(pq)$-periodic set is obviously contained in $\jB \times \jB'$.

Suppose that $k\in\Z$ and 
$\cF^k(\jU\times\jV)\cap (\jB_*\times\jB'_*) \neq \varnothing$.
We thus can find $(x,x') \in \jU\times\jV$ such that 
$z \defeq \cF^k(x,x') \in \jB_*\times\jB'_*$.
By~\eqref{eq:iterates}, the second projection of~$z$ is $G^k(x')$, 
in view of~\eqref{eq:nonov2} this implies that $k\in q\Z$,
say $k=\ell q$.
But, again by~\eqref{eq:iterates}, the first projection of~$z$ is thus
$\psi^\ell(x)$, 
and~\eqref{eq:nonov1} then implies $\ell\in p\Z$.
Therefore $k\in pq\Z$ and~\eqref{eq:nonov3} is proved.
\end{proof}


\begin{figure}	

\centering{{\begin{picture}(0,0)%
\includegraphics{0_dessin_intro.pstex}%
\end{picture}%
\setlength{\unitlength}{4144sp}%
\begingroup\makeatletter\ifx\SetFigFont\undefined%
\gdef\SetFigFont#1#2#3#4#5{%
  \reset@font\fontsize{#1}{#2pt}%
  \fontfamily{#3}\fontseries{#4}\fontshape{#5}%
  \selectfont}%
\fi\endgroup%
\begin{picture}(4900,1678)(496,-1677)
\put(1081,-871){\makebox(0,0)[lb]{\smash{\SetFigFont{10}{12.0}{\rmdefault}{\mddefault}{\updefault}$\Phi^ f$}}}
\put(3466,-1051){\makebox(0,0)[lb]{\smash{\SetFigFont{10}{12.0}{\rmdefault}{\mddefault}{\updefault}$G$}}}
\put(496,-134){\makebox(0,0)[lb]{\smash{\SetFigFont{10}{12.0}{\rmdefault}{\mddefault}{\updefault}$\A^{m}$}}}
\put(1567,-749){\makebox(0,0)[lb]{\smash{\SetFigFont{10}{12.0}{\rmdefault}{\mddefault}{\updefault}$F$}}}
\put(3106,-151){\makebox(0,0)[lb]{\smash{\SetFigFont{10}{12.0}{\rmdefault}{\mddefault}{\updefault}$\A^{m'}$}}}
\put(4106,-376){\makebox(0,0)[lb]{\smash{\SetFigFont{10}{12.0}{\rmdefault}{\mddefault}{\updefault}$\jV$}}}
\put(1426,-1261){\makebox(0,0)[lb]{\smash{\SetFigFont{10}{12.0}{\rmdefault}{\mddefault}{\updefault}$\jU$}}}
\put(3451,-533){\makebox(0,0)[lb]{\smash{\SetFigFont{10}{12.0}{\rmdefault}{\mddefault}{\updefault}$G(\jV)$}}}
\put(4741,-617){\makebox(0,0)[lb]{\smash{\SetFigFont{10}{12.0}{\rmdefault}{\mddefault}{\updefault}$G^ {q-1}(\jV)$}}}
\end{picture}
}}

\caption{Coupling of a wandering domain~$\jU$
  in~$\A$ and a periodic domain~$\jV$ in~$\A^{n-1}$ 
\label{figcl}}

\end{figure}


\begin{cor} \label{cor:couplingwand}
Let $\cF = \Phi^{f\otimes g} \circ (F\times G) \col
\A^{m+m'}\,\to\,\A^{m+m'}$
with $m$, $m'$, $F$, $G$, $f$, $g$, $q$
and $\jV \subset \A^{m'}$ as in Lemma~\ref{Lem:cl}
(in particular $\jV$ is $q$-periodic for~$G$ and the synchronization 
conditions~\eqref{eq:sync1} hold).

Assume now that the diffeomorphism $\Phi^f\circ F^q$ admits a
wandering subset $\jU\subset\A^m$.
Then the product set $\jU\times \jV \subset \A^{m+m'}$ is wandering for the
diffeomorphism~$\cF$.
\end{cor}

See Figure~\ref{figcl}.


\begin{proof} 
Let $\jW \defeq \jU\times \jV$. We show that $\cF^{k}(\jW) \cap \jW =
\varnothing$ for arbitrary $k\in\Z\setminus\{0\}$. 

Suppose first that $k\notin q\Z$. 
Then $\jV \cap G^k(\jV) = \varnothing$ as already observed in~\eqref{eqtijVdisjun}.
Thus $x'\in \jV$ implies $G^{k}(x')\notin \jV$, whence 
\[
\cF^{k}(x,x')\notin \jW
\quad \text{for all $x\in\A^m$}
\]
by~\eqref{eq:iterates},
\ie $\cF^{k}(\A^m\times \jV) \cap (\A^m\times \jV) = \varnothing$.
In particular, $\cF^{k}(\jW) \cap \jW= \varnothing$ when $k\notin q\Z$. 

Suppose now that $k = \ell q$ with $\ell\in\Z\setminus\{0\}$.
We have $\cF^{k}(\jW)=\big( (\Phi^f\circ F^q)^\ell(\jU), \jV\big)$ by~\eqref{eqqiterate}
and~$\jU$ is wandering for $\Phi^f\circ F^q$,
hence $(\Phi^f\circ F^q)^\ell(\jU) \cap \jU= \varnothing$,
therefore $\cF^{k}(\jW) \cap \jW = \varnothing$ again.
\end{proof}




\subsection{Proof of Part~(ii) of Theorem~\ref{th:perdomains} (periodic domains
  in~$\A^{n-1}$)}\label{ssec:perdomains}


\subsubsection{Overview of the method}


For $n\ge3$,
we must construct an arbitrarily close to integrable system 
in $\Pa2\big(\Phi^{\Demi(r_2^2+\cdots+r_n^2)}\big)$
possessing a periodic polydisc of arbitrarily large period
in~$\A^{n-1}$; 
the Gromov capacity of this polydisc must be bounded from below as in~\eqref{eq:capaD}
and ``localization conditions'' of the
form~\eqref{eq:disjoint1}--\eqref{eq:disjoint2} must hold for its orbit.
The near-integrable system will be obtained by applying
Corollary~\ref{cor:couplingper} with $m=1$ and $m'=n-2$.
The period of the polydisc will be of the form $Q = pq$, with
\[
p \defeq \ell p_{j+2}, \quad
\text{$\ell \in \N$ arbitrarily large}, \qquad
q \defeq p_{j+3}\cdots p_{j+n}
\]
(recall that $(p_j)_{j\ge1}$ is the prime number sequence),
so that~$Q$ will be an integer multiple of $N_j \defeq 
p_{j+2} \, p_{j+3}\cdots p_{j+n}$,
and the deviation of the system from
$\Phi^{\Demi(r_2^2+\cdots+r_n^2)}$ will be $O(1/N_j^2)$.

To apply Corollary~\ref{cor:couplingper}, we must define a system~$F$,
a function~$f$ and a $p$-periodic domain~$\jU$ for $\Phi^f \circ F$ in the first factor,~$\A$, 
and a system~$G$, a function~$g$ and a $q$-periodic domain~$\jV$
for~$G$ in the second factor,~$\A^{n-2}$.

\emph{On the first factor}, we will make use of
Theorem~\ref{th:varpseudopend} 
(in a way very similar to the proof of Theorem~\ref{th:perdomains}(i)
in Section~\ref{secPfThperi})
to produce a system $\Psi \in \Pa1\big(\Phi^{\Demi r_2^2}\big)$
possessing a $p$-periodic disc~$\ti\jU$ in~$\A$, whose area admits a suitable
bound from below and whose orbit is suitably localized.
A simple rescaling of the action variable~$r_2$ by the factor~$q$ will
then yield a system of the form
\beq
\psi = \Phi^f \circ F^q
\quad \text{with
$f\in G^{\al,L}(\T)$ small, \ $F\in\Pa1\big(\Phi^{\Demi r_2^2}\big)$,}
\eeq
possessing a $p$-periodic disc~$\jU$.
The smallness of $\normD{f}_{\al,L}$ will be controlled by the choice of the ``tuning
parameter''~$\mu$ at the moment of using Theorem~\ref{th:varpseudopend}.

\emph{On the second factor}, we will use a near-integrable system of the form
\beq   \label{eqformedeG}
G = G\zz 3\times\cdots \times G\zz n 
\in \Pa1\big(\Phi^{\Demi (r_3^2+\cdots+r_n^2)}\big)
\eeq
where, for each~$\ka$,
$G\zz \ka \in \Pa1\big(\Phi^{\Demi r_\ka^2}\big)$ has a
$p_{j+\ka}$-periodic disc~$\jV\zz \ka$ with area
suitably bounded from below and orbit suitably localized.
Since $p_{j+3}, \ldots, p_{j+n}$ are pairwise coprime and their
product is~$q$, we shall have 
$\jV = \jV\zz 3 \times \cdots \times \jV\zz n$ $q$-periodic for~$G$.
Lemma~\ref{lembump} of the Appendix~\ref{secBumpGev} will then yield a
``bump function'' $g\in G^{\al,L}(\A^{n-2})$ satisfying the
synchronization conditions relative to~$\jV$ and~$G$.

According to Corollary~\ref{cor:couplingper}, the polydisc
$\jU\times\jV$ will thus be $(pq)$-periodic for $\Phi^{f\otimes g} \circ (F\times G)$,
which will be the desired near-integrable system.
Notice that $\normD{g}_{\al,L}$ will be exponentially large, so we
need to choose properly the tuning parameter~$\mu$ in the first step,
so as to compensate the largeness of $\normD{g}_{\al,L}$ by the
smallness of $\normD{f}_{\al,L}$ and ensure
\beq
\de^{\al,L}\big( \Phi^{f\otimes g} \circ (F\times G), 
\Phi^{\Demi(r_2^2+\cdots+r_n^2)} \big) = O(1/N_j^2).
\eeq


\subsubsection{A $p$-periodic polydisc for a near-integrable system of the form
  $\Phi^f\circ F^q$ in~$\A$}


Let $\al>1$ and $L>0$ be real.
%

We give ourselves reals $\rho_0 >2$, $L_0,\th^\star>0$ such that
$L_0<\Demi-\th^\star$, $\de\defeq1$
and, as in Section~\ref{secPfThperi},
by means of Lemma~\ref{lembump}
we pick $1$-periodic functions~$V$, $(W_M)_{M\in\N^*}$ in
$G^{\al,L}(\R)$ which satisfy the assumptions (i)--(v) of
Theorem~\ref{th:varpseudopend}.
In particular,
\beq   \label{eqdefWM}
W_M(\th) \defeq \Demi \eta_M(\th) \big(\dist(\th,\Z)\big)^2,
\quad
\norm{W_M}_{\al,L} \le C_0\, \exp\Big(c(\al,L)\, M^{\frac{1}{\al-1}} \Big)
\quad \text{for all $M \in \N^*$,}
\eeq
with some positive reals~$C_0$ and $c(\al,L)$.

We get $C_1,C_2,C_3,C_4>0$ fulfilling the conclusions of
Theorem~\ref{th:varpseudopend}:
setting
\[
P_{V/M^2}(\th,r) \defeq \demi r^2 + \frac{1}{M^2} V(\th), \qquad
\GMm \defeq  \Phi^{\mu W_M} \circ \Phi^{P_{V/M^2}} 
\]
for every integer $M\ge1$ and real $\mu>0$
(as in~\eqref{eq:P_V}--\eqref{eqdefGNm}),
Theorem~\ref{th:varpseudopend} says that~$\GMm$ has a $p$-periodic disc~$\DMp$
for each integer $p \ge C_1 M$ provided $\mu < C_2 {M^4}/{p^5}$,
with area
\beq    \label{ineqarea}
\area(\DMp) \ge C_3 \frac{\mu}{M^2}
\eeq
and orbit localized as in~\eqref{eqlocalizDNm}.


Let $n\ge3$, $j\ge1$ and $\ell \ge C_1$ be integers, and
\beq \label{eqdefqNj}
q \defeq p_{j+3}\cdots p_{j+n}, \qquad N_j \defeq p_{j+2} \, q,
\qquad p \defeq \ell p_{j+2},
\eeq
so that $Q \defeq pq = \ell N_j$ is an arbitrary multiple $\ge C_1 N_j$ of~$N_j$.
We define
\beq   \label{eqdefmujellmin}
\mu_{j,\ell} \defeq \min \bigg\{%
\frac{C_2}{2 \ell^5 p_{j+2}}, 
\frac{1}{ (2 p_{j+2})^n } \exp\Big( - (n-1) c(\al,L) (2 p_{j+2})^{\frac{1}{\al-1}} \Big)
\bigg\}.
\eeq
Notice that
$\mu_{j,\ell}  < C_2 {p_{j+2}^4}/{p^5}$.
We may thus consider the map 
\beq
\Gpj = \Phi^{\mu_{j,\ell} W_{p_{j+2}}} \circ \Phi^{\Demi r^2 + p_{j+2}^{-2}V}
\eeq
which has a well-defined $p$-periodic disc $\ti\jU_{j,\ell} \defeq \Dpj$.


\begin{lemma}\label{lem:scalingq}
Let
\[
\sig \col (\th,r) \in \A \mapsto (\th,qr) \in \A.
\]
Then, for any Hamiltonian function of the form $(\th,r)\in\A \mapsto
h(r)+v(\th)$ with $h(qr) = q^2 h(r)$, one has
\[
\sig\ii \circ \Phi^{h+v} \circ \sig = \Phi^{ q( h + q^{-2}v ) }.
\]
\end{lemma}


\begin{proof}
This is a simple scaling property of the Hamiltonian flow
already used in \cite{hms}. 
Since~$\sig$ is not symplectic but conformal-symplectic, one needs to
rescale the action variable~$r$ \emph{and} the time: the identity
$\sig\ii \circ \Phi^{t(h+v)} \circ \sig = \Phi^{ qt( h + q^{-2}v ) }$
is easily checked by differentiating both sides \wrt~$t$.
\end{proof}


Applying Lemma~\ref{lem:scalingq} with $h(r)=\Demi r^2$, we get
\beq    \label{eqdefFj}
\sig\ii \circ \Phi^{\Demi r^2 + p_{j+2}^{-2}V} \sig = F_j^q, \qquad
F_j \defeq \Phi^{\Demi r^2 + \frac{1}{N_j^2}V}
\eeq
and, with $h=0$,
\beq     \label{eqdeffjell}
\sig\ii \circ \Phi^{\mu_{j,\ell} W_{p_{j+2}}} \sig = \Phi^{f_{j,\ell}}, \qquad
f_{j,\ell} \defeq q\ii\mu_{j,\ell} W_{p_{j+2}}.
\eeq
Therefore, the map
\beq
\Phi^{f_{j,\ell}} \circ F_j^q = \sig\ii \circ \Gpj \circ \sig 
\eeq
has a $p$-periodic disc $\jU_{j,\ell} \defeq \sig\ii(\ti\jU_{j,\ell})$.
Inequality~\eqref{ineqarea} entails
\beq   \label{ineqareajUellj}
\area(\jU_{j,\ell}) \ge C_3 \frac{\mu_{j,\ell}}{q \, p_{j+2}^2}
%
%
\eeq
and, because of~\eqref{eqlocalizDNm},
\beq     \label{eqlocalizjUjell}
\jU_{j,\ell} \subset 
\Bdeep\cap \A^+_{4/N_j}, \qquad
(\Phi^{f_{j,\ell}} \circ F_j^q)^k(\jU_{j,\ell}) \cap \Bdep = \varnothing 
\quad \text{for $1\le k \le p-1$.}
\eeq


\subsubsection{A $q$-periodic polydisc for a near-integrable system~$G$ in~$\A^{n-2}$}


We now need a near-integrable system $G=G_j$ of the
form~\eqref{eqformedeG} possessing a
$q$-periodic polydisc in~$\A^{n-2}$.
We shall take each factor of the form described in
%
%
\begin{prop}   \label{prop:perellipse}
For any integer $p\ge2$ and positive real $\nu < 1/p$, the
exact-symplectic map of~$\A$
%
%
\beq \label{eq:defLa}
\La_{p,\nu}=\Phi^{\nu W_{p}}\circ\Phi^{\Demi r^2}
\eeq
(with the same sequence of functions $(W_M)_{M\in\N^*}$ as in~\eqref{eqdefWM})
has a $p$-periodic disc~$E_{p,\nu}$ such that
\beq\label{eq:areaellipse}
\area\big(E_{p,\nu}\big)=\frac{\pi}{128} \frac{\nu}{p}
\eeq
and
\beq\label{eq:localellipse}
E_{p,\nu} \subset \jB_{1/2p},\qquad 
\La_{p,\nu}^k(E_{p,\nu})\cap\jB_{1/p} = \varnothing
\quad \text{for $1\le k \le p-1$.}
\eeq
\end{prop}


\smallskip

Indeed, we shall define
\beq \label{eqdefGjLanujka}
G_j \defeq \La_{p_{j+3},\nu_j\zz 3} \times \cdots \times \La_{p_{j+n},\nu_j\zz n},
\qquad
\text{with}\ens\;
\nu_j\zz \ka \defeq \frac{1}{N_j^2 \normD{W_{p_{j+\ka}}}_{\al,L}}
\quad\text{for $\ka=3,\ldots,n$}
\eeq
and, since $p_{j+3}, \ldots, p_{j+n}$ are pairwise coprime and
their product is~$q$, 
\beq   \label{eqdefVj}
\jV_j \defeq E_{p_{j+3},\nu_j\zz 3} \times \cdots \times E_{p_{j+n},\nu_j\zz n}
\subset \jB_{1/2p_{j+3}} \times \cdots \times \jB_{1/2p_{j+n}} 
\eeq
will be a $q$-periodic polydisc for~$G_j$ whose iterates are polydiscs
satisfying
\beq    \label{eqlocalizorbVj}
G_j^k(\jV_j) \cap \big(
\jB_{1/p_{j+3}} \times \cdots \times \jB_{1/p_{j+n}} 
\big) = \varnothing
\quad \text{for $1\le k \le q-1$.}
\eeq

\smallskip


\begin{proof}[Proof of Proposition~\ref{prop:perellipse}]
The disc~$E_{p,\nu}$ will be a $p$-periodic filled ellipse centred at
$O_p \defeq \big( \angD{0},1/p \big) \in \A$.
Recall that
\[
\Phi^{\Demi r^2}\big( \angD{\th},r \big) = \big( \angD{\th+r},r \big), \qquad
\Phi^{\nu W_p}\big( \angD{\th},r \big) = \big( \angD{\th},r-\nu W_p'(\th) \big).
\]
Let us set
\[
B_p \defeq \textstyle
\big\{\, \big( \angD{x}, \frac{1}{p}+y \big) \,\big\vert\,
\abs{x} \le \frac{1}{8p}, \;
\abs{y} \le \frac{1}{8p^2}
\,\big\}
= O_p + \big[-\frac{1}{8p},\frac{1}{8p}\,\big] \times \big[-\frac{1}{8p^2},\frac{1}{8p^2}\,\big].
\]
We will sometimes omit the canonical projection $\angD{\,\cdot\,} \col
\R\to\T$ in our notations and consider $(x,y)$ as local coordinates near~$O_p$.

We first note that~$B_p$ is a $p$-adapted box for $h(r)=\dem r^2$ and
$\B_{\frac{1}{2p}}$ in the sense of Definition~\ref{def:q-box} of Section~\ref{sec:Defqbox}.
Indeed, for $O_p+(x,y)\in B_p$, we have 
$\Phi^{th}\big(O_p+(x,y)\big) = 
\big( \angD{x + t\big(\frac{1}{p}+y\big)}, \frac{1}{p} + y \big)$
and a straightforward computation shows that
\[ \textstyle
1\le t\le p-1 \ens\Rightarrow\ens
\frac{1}{2p} <
-\frac{1}{8p} + \frac{1}{p} -\frac{1}{8p^2}
\le x + t\big(\frac{1}{p}+y\big) \le
\frac{1}{8p} + (p-1)\big( \frac{1}{p} + \frac{1}{8p^2} \big)
< 1-\frac{1}{2p},
\]
hence 
\beq   \label{eqDefi}
1\le t\le p-1 \quad\Rightarrow\quad
\Phi^{\Demi t r^2}\big(O_p+(x,y)\big) \notin \ov{\jB_{\frac{1}{2p}}},
\eeq
while the first component of $\Phi^{ph}\big(O_p+(x,y)\big)$ is
$\angD{x+1+py} = \angD{x+py}$
and $\abs{x+py} \le \frac{1}{4p}$,
hence 
\beq   \label{eqDefii}
\Phi^{\Demi p r^2}\big(O_p+(x,y)\big) \in \ov{\jB_{\frac{1}{4p}}}.
\eeq


We now observe that the restrictions to~$B_p$ of $\La_{p,\nu}$ and its
iterates up to the $p$th
\[
A_{k,p,\nu} \defeq {\La^k_{p,\nu}}\big\vert _{B_p}, \qquad
k= 0,\ldots,p
\]
are affine in the coordinates $(x,y)$, and even linear for the $p$th iterate.
Indeed, one checks by induction on $k\in\{0,\ldots,p-1\}$
that $A_{k,p,\nu} = \Phi^{\Demi k r^2} \big\vert _{B_p}$:
this clearly holds for $k=0$ and, assuming it for
$0\leqslant k\leqslant p-2$, we have
\[
A_{k+1,p,\nu} = \Phi^{\nu W_p}\circ \Phi^{\Demi r^2} \circ 
\Phi^{\Demi k r^2} \big\vert _{B_p} 
= \Phi^{\Demi (k+1)r^2} \big\vert _{B_p} 
\]
because, by~\eqref{eqDefi},  
$\Phi^{\Demi (k+1)r^2}(B_p)$ lies away from the support of~$W_p$.
Now, \eqref{eqDefii} says that $\Phi^{\Demi p r^2}(B_p)$ is contained
in $\ov{\jB_{\frac{1}{4p}}}$
and this is a part of~$\A$ which we may identify with 
$\big[ -\frac{1}{4p},\frac{1}{4p} \big] \times \R \subset \R\times\R$,
in which $W_p\equiv \Demi \th^2$ in the coordinates $(\th,r)$, whence for $k=p$
\[
A_{p,p,\nu} = \Phi^{\nu W_p} \circ \Phi^{\Demi r^2} \circ A_{p-1,p,\nu}
= \Phi^{\nu W_p} \circ \Phi^{\Demi p r^2} \big\vert _{B_p}  = 
\Phi^{\Demi\nu\th^2} \circ \Phi^{\Demi p r^2} \big\vert _{B_p} .
\]
We thus end up with
\beq\label{eq:explambda}
A_{k,p,\nu}\big(O_p+(x,y)\big) = O_p+\left\vert\!\!
\begin{array}{ll}
\big( \angD{x+k(y+\frac{1}{p})}, y \big) &\text{if $0\leqslant k\leqslant p-1$,}\\[1ex]
\big( \angD{x+py}, y-\nu (x+py) \big) &\text{if $k=p$.}
\end{array} \right.
\eeq


Let us consider the linear transformation $A \col \R^2\to\R^2$ defined by
\[
A
\begin{pmatrix} x \\ y \end{pmatrix}
\defeq 
\begin{pmatrix} x+py \\ y-\nu (x+py) \end{pmatrix} =
\begin{pmatrix}
  \; 1   &    p       \\
-\nu & 1-\nu p
\end{pmatrix}
\begin{pmatrix} x \\ y \end{pmatrix}.
\] 
According to~\eqref{eq:explambda}, 
if~$\cE$ is a filled ellipse centred at the origin and invariant
by~$A$ and $O_p+\cE \subset B_p$,
then $O_p+\cE$ is a $p$-periodic disc for $\La_{p,\nu}$ which satisfies~\eqref{eq:localellipse}.
Elementary linear algebra shows that
\[
A = P
\begin{pmatrix*}[r]
 \cos\ga  &  \sin\ga  \\
-\sin\ga  &  \cos\ga
\end{pmatrix*}
P\ii,
\]
where 
$0 < \ga \defeq \arccos\big( 1 - \frac{\nu p}{2} \big) <
\frac{\pi}{3}$
(recall that $0<\nu p<1$) and
\[
P
\begin{pmatrix} X \\ Y \end{pmatrix}
\defeq \frac{1}{p \sin\ga}
\begin{pmatrix}
        p        &      0       \\
-1+\cos\ga & \sin\ga
\end{pmatrix}
\begin{pmatrix} X \\ Y \end{pmatrix}.
\] 
Hence, for each $r>0$, 
$
\cE(r) \defeq P \Big( \big\{
\left(\begin{smallmatrix} X \\ Y \end{smallmatrix}\right) \in \R^2 \mid
X^2 + Y^2 < r^2 \big\} \Big)
$
is a filled ellipse of area $\pi r^2$, centred at the origin and invariant by~$A$.
We choose
\[
E_{p,\nu} \defeq O_p + \cE(r_{p,\nu}),
\qquad
r_{p,\nu} \defeq \tfrac{1}{8}\big(\tfrac{\nu}{2p}\big)^{1/2}.
\]
Using $\sin\ga > (\nu p/2)^{1/2}$ and $\sin\ga > 1-\cos\ga$, the
property $O_p + \cE(r_{p,\nu}) \subset B_p$ is easily checked and the
desired conclusions are fulfilled, including~\eqref{eq:areaellipse}.
\end{proof}


\subsubsection{Applying Corollary~\ref{cor:couplingper}}


From now on, taking advantage of the Prime Number Theorem,
we assume that the parameter~$j$ is large enough so that
\beq  \label{ineqTNP}
p_{j+2} < p_{j+3} < \cdots < p_{j+n} \le 2 p_{j+2}
\eeq
(this is the interest of having taken successive prime numbers for our
$n-1$ pairwise coprime integers).
Recall that the other parameter is $\ell\ge C_1$, so that $Q=pq$ is an
arbitrary multiple $\ge C_1 N_j$ of~$N_j$.

On the one hand, in $\A^m=\A$, we have a $p$-periodic disc
$\jU_{j,\ell}$ for the map $\Phi^{f_{j,\ell}} \circ F_j^q$ defined
by~\eqref{eqdefFj}--\eqref{eqdeffjell}, satisfying the localization
condition~\eqref{eqlocalizjUjell}.
On the other hand, in $\A^{m'} = \A^{n-2}$, we have a $q$-periodic
polydisc~$\jV_j$ for the map~$G_j$ defined
by~\eqref{eqdefGjLanujka}, with localization conditions 
\eqref{eqdefVj}--\eqref{eqlocalizorbVj}.
We can thus apply Corollary~\ref{cor:couplingper} with
\begin{gather}   \label{eqdefgjsynchr}
g_j \defeq \eta_{p_{j+3}} \otimes \cdots \otimes \eta_{p_{j+n}}  \\[1.5ex]
\notag
\begin{aligned}
&\jB \defeq \Bdeep\cap \A^+_{4/N_j}, & \quad
&\jB' \defeq \jB_{1/2p_{j+3}} \times \cdots \times \jB_{1/2p_{j+n}},
\\[1ex]
&\jB_* \defeq \Bdep, & 
&\jB'_* \defeq \jB_{1/p_{j+3}} \times \cdots \times \jB_{1/p_{j+n}} 
\end{aligned}
\end{gather}
and get a map
\beq   \label{eqfinaldefPsi}
\Psi \defeq \Phi^{f_{j,\ell}\otimes g_j} \circ (F_{j,\ell} \times G_j)
=
\Phi^{f_{j,\ell}\otimes g_j} \circ
\Phi^{ \nu_j\zz 3 W_{p_{j+3}} + \cdots + \nu_j\zz n W_{p_{j+n}} }
\circ \Phi^{ \Demi(r_2^2+\cdots+r_n^2) + \frac{1}{N_j^2}V }
\eeq
possessing a $Q$-periodic polydisc
\[
\jD \defeq \jU_{j,\ell} \times \jV_j \subset \jB \times \jB',
\]
(recall that $Q = pq$), with
\[
\Psi^k(\jD) \cap (\jB_* \times \jB'_*) = \varnothing
\quad \text{for $1\le k \le Q-1$.}
\]
The last two properties coincide
with~\eqref{eq:disjoint1}--\eqref{eq:disjoint2}, since all iterates
of~$\jD$ are in fact polydiscs (in view of~\eqref{eq:iterates}
and~\eqref{eqdefGjLanujka}).


To end the proof of Theorem~\ref{th:perdomains}(ii),
we just need to check that
\begin{enumerate}[(i)]
\item 
  the map $\Psi$, which clearly belongs to
  $\Pa2\big(\Phi^{\Demi(r_2^2+\cdots+r_n^2)}\big)$, is indeed close to
  integrable, namely
\beq   \label{ineqestimdevfinal}
\de^{\al,L}\big( \Psi, \Phi^{ \Demi(r_2^2+\cdots+r_n^2) } \big)
\le \frac{C_0+n-2+\normD{V}_{\al,L}}{N_j^2},
\eeq
\\
\item 
  the Gromov capacity of~$\jD$ is not too small, namely
\beq    \label{ineqfinalCG}
\CG(\jD) \ge \ti C \min\bigg\{%
\frac{1}{Q^5} N_j^{4-\frac{2}{n-1}}, \,
N_j^{-2-\frac{2}{n-1}}
\exp\Big( - \ti c \, N_j^{\frac{1}{(n-1)(\al-1)}} \Big)
\bigg\},
\eeq
with $\ti C \defeq \min\big\{
\frac{C_2 C_3}{2},
\frac{\pi}{256 C_0},
2^{-n-1} C_3
\big\}$
and $\ti c \defeq 2^{\frac{1}{\al-1}} (n-1) c(\al,L)$.
\end{enumerate}
Indeed, this will yield~\eqref{eq:devPsij} and~\eqref{eq:capaD},
up to an obvious change of notation for the constants ``$C_2$''
and~``$C_3$'',
by taking for~$c$ a large enough function of $\ti c$, $\al$ and~$n$.


\begin{proof}[Proof of~(i)]
In view of~\eqref{eqdeffjell} and~\eqref{eqdefgjsynchr},
we can estimate the norm of $f_{j,\ell} \otimes g_j$ thanks
to~\eqref{eqdefWM} and~\eqref{ineqnormetap}:
\[
\normD{ f_{j,\ell} \otimes g_j }_{\al,L} \le q\ii \mu_{j,\ell} C_0 \,
\exp\Big( p_{j+2}^{\frac{1}{\al-1}} + \cdots + p_{j+n}^{\frac{1}{\al-1}} \Big)
\]
which, by~\eqref{eqdefmujellmin} and~\eqref{ineqTNP}, is 
$ \le \frac{C_0}{ (2 p_{j+2})^n q }
\le \frac{C_0}{N_j^2} $.
In view of~\eqref{eqfinaldefPsi} and the choice of $\nu_j\zz 3,
\ldots,\nu_j\zz n$ in~\eqref{eqdefGjLanujka},
this yields~\eqref{ineqestimdevfinal}.
\end{proof}


\begin{proof}[Proof of~(ii)]
Since $N_j = p_{j+2} \, q$ and $q = p_{j+3} \cdots p_{j+n}$,
\eqref{ineqTNP} yields $p_{j+2}^{n-1} < N_j < (2 p_{j+2})^{n-1}$, whence
\[
\Demi N_j^{ \frac{1}{n-1} } < p_{j+2} < N_j^{ \frac{1}{n-1} },
\qquad
N_j^{ \frac{n-2}{n-1} } < q < 2 \, N_j^{ \frac{n-2}{n-1} }.
\]
Now, by~\eqref{eqGromovCapProduct}, 
$\CG(\jD) = \min\Big\{
\area(\jU_{j,\ell}), \area(E_{p_{j+3},\nu_j\zz 3}), \ldots, \area(E_{p_{j+n},\nu_j\zz n})
\Big\}$.
In view of~\eqref{eq:areaellipse}, for $3 \le \ka \le n$,
\begin{align*}
\area(E_{p_{j+\ka},\nu_j\zz \ka}) = \frac{\pi}{128} \frac{\nu_j\zz \ka}{p_{j+\ka}} 
& \ge \frac{\pi}{256 C_0} \frac{1}{ N_j^2 \, p_{j+2}} 
      \exp\Big( -c(\al,L)\, p_{j+2}^{\frac{1}{\al-1}} \Big) \\[1ex]
& \ge \frac{\pi}{256 C_0} \frac{1}{ N_j^{ 2+\frac{1}{n-1} } }
      \exp\Big( -c(\al,L)\, N_j^{\frac{1}{(n-1)(\al-1)}} \Big),
\end{align*}
while~\eqref{ineqareajUellj} yields
$\area(\jU_{j,\ell}) \ge \min\{ A , B \}$ with
\[
A \defeq \frac{C_2 C_3}{2 \, \ell^5 \, p_{j+2}^3 \, q} 
= \frac{C_2 C_3 N_j^2 q^2}{2 Q^5} 
\ge \frac{C_2 C_3}{2} \frac{1}{Q^5} N_j^{4-\frac{2}{n-1}}
\]
and $\dst B \defeq \frac{C_3}{ 2^n p_{j+2}^{n+2} \, q } 
\exp\Big( - (n-1) c(\al,L) (2 p_{j+2})^{\frac{1}{\al-1}} \Big)$
larger than
\[
\frac{C_3}{ 2^{n+1} N_j^{2+\frac{2}{n-1}}} 
\exp\Big( - 2^{\frac{1}{\al-1}} (n-1) c(\al,L)\, N_j^{\frac{1}{(n-1)(\al-1)}} \Big),
\]
whence~\eqref{ineqfinalCG} follows.
\end{proof}


The proof of Theorem~\ref{th:perdomains}(ii) is now complete.


\subsection{Proof of Theorem \ref{th:lowerbounds} (lower bounds for
  wandering domains in $\A^n$)}   \label{secpfthmlower}


\subsubsection{Overview of the proof}


Here is the more precise statement which, as explained in
Section~\ref{paragabetterC}, implies Theorem~\ref{th:lowerbounds}.
\label{sec:statementthmCp}
\begin{thmCp}
Let $n\ge2$ be integer. Let $\al>1$ and $L>0$ be real, 
and let $h(r) \defeq \frac{1}{2}(r_1^2+\cdots+r_n^2)$. 
Then there exist a positive real~$c_*$ and 
a sequence $(\Phi_j)_{j\geq0}$ of exact symplectic diffeomorphisms
of~$\A^n$
which belong to $\P^{\al,L}_2(h)$ if $n=2$ 
and to $\P^{\al,L}_3(h)$ if $n\geq3$,
such that each $\Phi_j$ admits a wandering polydisc $\jW_j \subset \A_3^n$
and 
\beq   \label{ineqCGjWj}
\eps_j \defeq \de^{\al,L}(\Phi_j,\Phi^h)
\xrightarrow[j\to\infty]{}0
\quad \text{and} \quad
\CG(\jW_j)\ge \exp\Big(-c_* \Big(\frac{1}{\eps_j}\Big)^{\frac{1}{2(n-1)(\al-1)}}\Big).
\eeq
\end{thmCp}

\medskip


The rest of Section~\ref{secpfthmlower} is devoted to the proof of Theorem~\ref{th:lowerbounds}'.

The idea is as follows.
Each near-integrable system~$\Phi_j$ and wandering polydisc~$\jW_j$ will be obtained by means of
Corollary~\ref{cor:couplingwand} in the form
\beq
\Phi_j \defeq \Phi^{f\otimes g} \circ (F\times G), \qquad
\jW_j \defeq \jU \times \jV,
\eeq
where 
\beq   \label{eqdefGPsijqj}
G \defeq \Psi_{j,q_j} \col \A^{n-1} \righttoleftarrow, \qquad
\jV \defeq \jD_{j,q_j} \subset \A^{n-1} 
\eeq
will be provided by Theorem~\ref{th:perdomains},
with a suitably chosen large integer~$q_j$, 
and the function~$g$ will be chosen so as to satisfy the
synchronization conditions~\eqref{eq:sync1} for the orbit of $\jD_{j,q_j}$,
while 
\beq   \label{eqdefFfjU}
F \defeq \Phi^{\Demi r_1^2} \col \A \righttoleftarrow, 
\quad f \defeq \frac{1}{q_j} U, \quad
\jU \defeq W_{q_j} \subset \A
\eeq
stem from
%
\begin{prop}\label{prop:standmap}
Let $0<\rho<1/2$ and let $U\in C^\infty(\T)$ be a function such that
\beq   \label{eqcondUlin}
U'\big( \angD{x} \big) = -1+x \qquad
\text{for $x \in [-\rho,\rho]$.}
\eeq
Then there exist a real $C_0>0$ such that, for each integer $q\ge1$, the diffeomorphism
$\Phi^{\pfrac{1}{q} U}\circ\big(\Phi^{\ppdemi r^2}\big)^q$
%
%
of~$\A$ admits a wandering disc $W_q \subset \A_3$ such that
\beq   \label{eq:estimarea}
\area(W_q) = \frac{C_0}{q}.
\eeq
\end{prop}

We give the proof of Proposition~\ref{prop:standmap} in
Section~\ref{Sec:Ffq}.
Then, in Section~\ref{ssec:proofn}, we indicate how to
choose~$q_j$ and check that $\Phi_j$ and~$\jW_j$ have all the desired properties.


\subsubsection{Standard maps with  wandering discs in~$\A$---Proof of Proposition~\ref{prop:standmap}}
\label{Sec:Ffq}



%
To prove Proposition~\ref{prop:standmap}, we first consider the so-called ``standard map'' 
\[
\jS \defeq \Phi^U\circ\Phi^{\ppdemi r^2} \col \A \righttoleftarrow
\]
\ie
\[
\jS(\th,r) = \big(\th+r,\,r-U'(\th+r)\big), \qquad (\th,r)\in \T\times \R.
\]
Since $\jS(\th,r+1)=\jS(\th,r)$, by passing to the quotient, $\jS$
induces a map $\jS^*  \col \T\times\T \righttoleftarrow$.
Our assumption on~$U$ entails that the origin
$\big( \angD{0},\angD{0} \big)$ of~$\T^2$ is a fixed point of~$\jS^*$,
in a neighbourhood of which~$\jS^*$ is linear:
\[
\jS^*\big( \angD{x}, \angD{y} \big) = \big( \angD{x+y}, \angD{-x} \big),\qquad 
x\in [ -\tfrac{\rho}{2}, \tfrac{\rho}{2}], \quad y\in [ -\tfrac{\rho}{2}, \tfrac{\rho}{2}].
%
%
\]
The eigenvalues being~$\ex^{\pm \ima \frac{\pi}{3}}$, the origin is an elliptic fixed point 
surrounded by invariant ellipses. 
Let~$W^*$ denote any invariant filled ellipse contained in the
projection onto~$\T^2$ of $[ -\tfrac{\rho}{2}, \tfrac{\rho}{2}] \times
[ -\tfrac{\rho}{2}, \tfrac{\rho}{2}]$,
and let $C_0 \defeq \area(W^*)$.

We define~$W$ to be the lift of~$W^*$ in~$\A$ which contains the
point $O \defeq \big( \angD{0}, 0 \big)$.
Since $U'(0)=-1$, one sees that $\jS^k(O) =  \big( \angD{0}, k \big)$
for all $k\in\Z$,
hence the orbit of~$W$ under~$\jS$ consists of pairwise disjoint
filled ellipses centred at the points $\jS^k(O)$:
\[
\jS^k(W) = \big( \angD{0}, k \big) + W, \qquad k\in\Z.
\]
In particular, $W$ is a wandering disc for~$\jS$.

We now obtain a wandering disc for
\[
\jS_q \defeq \Phi^{\pfrac{1}{q} U}\circ\big(\Phi^{\ppdemi r^2}\big)^q,
\]
for any integer $q\ge1$, by means of the scaling
\[
\sig \col (\th,r) \in \A \mapsto (\th,qr) \in \A.
\]
Indeed, Lemma~\ref{lem:scalingq} with $h(r) = \Demi r^2$ and $v=0$
yields $\sig\ii\circ\Phi^{\Demi r^2}\circ\sig = \Phi^{\Demi q r^2}$
and, with $h=0$ and $v=U$,
$\sig\ii\circ\Phi^{U}\circ\sig = \Phi^{q\ii U}$,
whence
\[
\jS_q = \sig\ii \circ \jS \circ \sig
\]
and $W_q \defeq \sig\ii(W)$ is a wandering disc for~$\jS_q$.
Clearly, $\area(W_q) = C_0/q$ and the proof of Proposition~\ref{prop:standmap} is complete.


\begin{rem}
The diffeomorphism $\jS$ is  ``dynamically far'' from  the integrable map $\Phi^{\ppdemi r^2}$.
Indeed, $\jS$ cannot possess any essential invariant curve $\jC$, otherwise the orbit of each point in the 
complement $\A\setm\jC$ would be contained in a single connected components of $\A\setm\jC$, 
and this is not the case for the orbit of $(0,0)$. As a consequence, $\jS_q$ has no essential invariant curve.
However, when  $q\to\infty$, $\jS_q$ is a small perturbation of the integrable map $\Phi^{\ppdemi qr^2}$.
This is not in contradiction with the KAM theorem: the torsion of $\Phi^{\ppdemi qr^2}$ tends to infinity 
when $q\to \infty$,  which makes the KAM  threshold  tend to $0$.
\end{rem}


\subsubsection{Proof of Theorem~\ref{th:lowerbounds}'}
\label{ssec:proofn}


Let $n\ge2$ be integer. Let $\al>1$ and $L>0$ be real.


On the one hand, 
Theorem~\ref{th:perdomains} yields reals $c,C_1,C_2,C_3>0$ and a
sequence $(\Psi_{j,q})$ in $\Paal(\Phi^{\Demi(r_2^2+\cdots+r_n^2)})$ such that
$\de^{\al,L}(\Psi_{j,q},\Phi^{\Demi(r_2^2+\cdots+r_n^2)}) \le
\frac{C_2}{N_j^2}$,
where $N_j=p_{j+2}\cdots p_{j+n}$ is arbitrarily large,
and each $\Psi_{j,q}$ for~$q$ arbitrary integer multiple of~$N_j$ not
smaller than $C_1 N_j$ has a $q$-periodic polydisc
$\jD_{j,q} \subset \A^{n-1}$ satisfying~\eqref{eq:capaDdeux} or~\eqref{eq:capaD}.
In view of the localization conditions~\eqref{eq:disjointdeux}
or~\eqref{eq:disjoint1}--\eqref{eq:disjoint2} satisfied by the orbit
of~$\jD_{j,q}$ independently of~$q$, we define
\[
g_j \defeq \eta_{p_{j+2}} \otimes \cdots \otimes \eta_{p_{j+n}}
\]
(making use of the ``bump functions'' of Lemma~\ref{lembump}),
so that~$g_j$ satisfies the synchronization conditions~\eqref{eq:sync1} for the
orbit of $\jD_{j,q}$ under~$\Psi_{j,q}$.


On the other hand,
choosing $\rho \defeq 1/6$ and $U \in G^{\al,L}(\T)$ satisfying~\eqref{eqcondUlin}
(\eg $U(\th) \defeq \eta_3(\th) (-x+\Demi x^2)$, where $x$ is the lift
of~$\th$ in $(-\Demi,\Demi]$),
we get from Proposition~\ref{prop:standmap} a wandering disc $W_q
\subset \A$ for $\Phi^{\pfrac{1}{q} U}\circ\big(\Phi^{\ppdemi
  r^2}\big)^q$
for each integer $q\ge1$.


We take $j$ large enough so that~\eqref{ineqTNP} holds (thanks to the
Prime Number Theorem),
and define
\beq
q_j \defeq M_j N_j, \qquad
M_j \defeq \big[N_j \normD{U}_{\al,L} \normD{g_j}_{\al,L}\big]+1,
\eeq
where $[\ ]$ denotes the integer part.
Applying Corollary~\ref{cor:couplingwand}  with the
data~\eqref{eqdefGPsijqj}--\eqref{eqdefFfjU}, we obtain a wandering
domain 
\[
\jW_j \defeq W_{q_j} \times \jD_{j,q_j}
\]
for the map
\[
\Phi_j \defeq \Phi^{\frac{1}{q_j} U \otimes g_j} \circ 
\big( \Phi^{\Demi r_1^2} \times \Psi_{j,q_j} \big).
\]
If $n=2$, then $\Psi_{j,q_j} \in
\Pa1(\Phi^{\Demi(r_2^2+\cdots+r_n^2)})$, hence
$\Phi_j \in \Pa2(\Phi^{\Demi(r_1^2+r_2^2+\cdots+r_n^2)})$.
If $n=3$, then $\Psi_{j,q_j} \in
\Pa2(\Phi^{\Demi(r_2^2+\cdots+r_n^2)})$, hence
$\Phi_j \in \Pa3(\Phi^{\Demi(r_1^2+r_2^2+\cdots+r_n^2)})$.
In all cases, 
\[
\eps_j \defeq 
\de^{\al,L}(\Phi_j, \Phi^{\ppdemi(r_1^2+r_2^2+\cdots+r_n^2)}) \le
\normD{\tfrac{1}{q_j} U\otimes g_j}_{\al,L} +
\de^{\al,L}\big(\Psi_{j,q_j},\Phi^{\ppdemi(r_2^2+\cdots+r_n^2)}\big)
\le \frac{1+C_2}{N_j^2}.
\]


We conclude by bounding from below the Gromov capacity of~$\jW_j$
which, according to~\eqref{eqGromovCapProduct}, is
\[
\CG(\jW_j) = \min \big\{ \area(W_{q_j}), \CG(\jD_{j,q_j}) \big\}.
\]
We have $q_j \le 2 \normD{U}_{\al,L} N_j^2 \normD{g_j}_{\al,L}$ and,
by~\eqref{ineqnormetap} and~\eqref{ineqTNP},
\[
\normD{g_j}_{\al,L} \le \exp\Big( c(\al,L) \big( p_{j+2}^{\frac{1}{\al-1}} + 
\cdots + p_{j+n}^{\frac{1}{\al-1}} \big) \Big)
\le \exp\Big( (n-1) c(\al,L) (2 p_{j+2})^{\frac{1}{\al-1}} \Big).
\]
Since $p_{j+2}^{n-1} < N_j$, we thus can find $C,\ti c>0$ independent
of~$j$ such that
\beq   \label{ineqqjCtic}
q_j \le C \, \exp\Big( \ti c \, N_j^{\frac{1}{(n-1)(\al-1)}} \Big).
\eeq
By~\eqref{eq:estimarea}, this yields
\[
\area(W_{q_j}) \ge \frac{C_0}{C} \exp\Big( - \ti c \, N_j^{\frac{1}{(n-1)(\al-1)}} \Big).
\]
On the other hand,
\[
\CG(\jD_{j,q_j}) \ge C_3 \min \bigg\{%
\frac{1}{q_j^5} N_j^{4-\frac{2}{n-1}}, \,
\exp\Big( -c N_j^{\frac{1}{(n-1)(\al-1)}} \Big)
\bigg\}
\]
and, again by~\eqref{ineqqjCtic}, one can find $C',\ti c'>0$ independent
of~$j$ such that
\[
\frac{1}{q_j^5} N_j^{4-\frac{2}{n-1}} \ge
C'\, \exp\Big( -\ti c' N_j^{\frac{1}{(n-1)(\al-1)}} \Big).
\]
We end up with
\[
\CG(\jW_j) \ge \min\big\{ \tfrac{C_0}{C}, C_3 C', C_3 \big\} \,
\exp\Big( -\max\{\ti c,\ti c',c\} N_j^{\frac{1}{(n-1)(\al-1)}} \Big)
\]
and thus can find $c_*>0$ independent
of~$j$ such that~\eqref{ineqCGjWj} holds.

This concludes the proof of Theorem~\ref{th:lowerbounds}'.


\vfil

\pagebreak

\appendix
\appendixpage
\addappheadtotoc

\section{\texorpdfstring{Algebraic operations in $\jO_k$}{Algebraic operations in
O}}
\label{app:O_k}
For $\ell\in\R$, we denote by $[\ell]\in\Z$ the integral part such that $[\ell]\leqslant \ell
<[\ell]+1$. For all $\nu=(\nu_1,\nu_2)\in\N^2$, we set $\abs{\nu}=\nu_1+\nu_2$, 
$\partial^\nu f=\partial^{\nu_1}\bar\partial^{\nu_2}f$, $z^\nu=z^{\nu_1}{\overline{z}^{\nu_2}}$ 
and $\nu!=\nu_1!\nu_2!$.~\\
  If $P$ is a polynomial of the form 
	$P(z)=\sum\limits_{\begin{smallmatrix}
\nu\in\N^2\\
\abs{\nu}\leqslant n-1
\end{smallmatrix}
} a_\nu z^\nu$
then we set $\norm{P}_\tau=\sum\limits_{\abs{\nu}\leqslant n-1}\abs{a_\nu}\tau^{\abs{\nu}}$.
\begin{lemma} \label{lem:O_k}
Assume $(n,m,\ell,\ell_1,\ell_2)\in\N^2\times\R^3$. Then the  spaces
$\jO_k$ of Section~\ref{sec:Ok} satisfy the 
following axioms.
\begin{description}
\item[(Restriction)] $\jO_{k}(\ell;C,\tau)\subset\jO_{k-1}(\ell;C,\tau)\cap\jO_{k}(\ell-1;C\tau,
	\tau)$;%
\item[(Derivative)] If $f\in \jO_{k}(\ell;C,\tau)$ then $\partial^\alpha\bar\partial^\beta f\in 
	\jO_{k-\alpha-\beta}(\ell-\alpha-\beta; C,\tau)$ for all $(\alpha,\beta)\in\N^2$ such that 
	$\alpha+\beta \leqslant k$;%
\item[(Primitive)] If $f\in\jO_0(\ell+1;C_0,\tau)$, $\partial f\in\jO_{k}(\ell;C_1,\tau)$  and 
	$\bar\partial f\in\jO_{k}(\ell;C_2,\tau)$ then $f\in  \jO_{k+1}(\ell+1;C,\tau)$, with 
	$C \leqslant \max(C_0,C_1,C_2)$;%
\item[(Product)] If $f \in\jO_{k}(\ell_1;C_1,\tau)$ and $g\in\jO_{k}(\ell_2;C_2,\tau)$ then we 
	have  $fg\in\jO_{k}(\ell_1+\ell_2;2^kC_1C_2,\tau)$;%
\item[($\mathbf{Z}$-Product)] If $f\in  \jO_{k}(\ell;C,\tau)$ then $zf(z)
	=\jO_{k}(\ell+1;C(k+1),\tau)$;%
\item[(Polynomial)] If $P$ is a polynomial of degree $m$ and if $0\leqslant n\leqslant m$ then 
	we have $P=[P]_{\leqslant n-1}+\jO_k(n;(k+1)^{m-1}\norm{P}_\tau/\tau^{n},\tau)$;%
\item[(P-product)] If $f(z)=P(z)+\jO_k(n+1,C/\tau^{n+1},\tau)$ and $g=Q(z)+\jO_k(n+1,
	C'/\tau^{n+1},\tau)$, where $P$ et $Q$ are two polynomials of degree $n$, then we have 
	\begin{align*}
		(fg)(z)=R(z)+\jO_k(n+1;C''/\tau^{n+1},\tau), 
	\end{align*}
	where $C''\leqslant (k+1)^{n}(\norm{P}_\tau C'+\norm{Q}_\tau C)+2^kCC'+(k+1)^{2n}
	\norm{P}_\tau\norm{Q}_\tau$ and  $R$ is a polynomial of degree $n$ satisfying 
	$\norm{R}_\tau\leqslant \norm{P}_\tau\norm{Q}_\tau$;%
\item[(Lipschitz)] If  $\partial f\in\jO_{k}(n;C_1,\tau)$ and $\bar\partial f\in\jO_{k}(n;C_2,\tau)$  
then $f$ is Lipschitz continuous near zero. Furthermore, $f-f(0)\in\jO_{k+1}(n+1;C,\tau)$, with 
$C=C_0+C_1$ if $n=0$ and $C=\max(C_1,C_2)$ if $n \geqslant 1$;%
\item[(Composition)] Assume  $n\geqslant k$. If $h\in\jO_k(n;C_0,\tau_0)$,  $f\in\jO_k(m;
	C_1,\tau_1)$ and $f(\Dm(0,\tau_1))\subset \Dm(0;\tau_0)$ then $h\circ f\in\jO_k(nm;
	\alpha_kC_0C_1^n ,\tau_1)$, with $\alpha_k =  2^{\frac{k(k+1)}{2}}$;%
\item[($P$-Composition)] Assume  $n+1\geqslant k$. If $h(z)=P(z)+\jO_k(n+1;C_0,\tau)$ and  
	$f(z)=Q(z)+\jO_k(n+1;C_1/{\rho^{n+1}},\rho)$, where $P$ and  $Q$ are polynomial of 
	degree $n$ with $Q(0)=0$ and $f(\Dm(0,\rho))\subset \Dm(0;\tau)$ then there exist a 
	polynomial $R$ of degree $n$ and a constant $C_{01}$ satisfying $\norm{R}_{\rho}
	\leqslant \norm{P}_{\norm{Q}_{\rho}}$ and 
	\begin{align*}
	& h\circ f(z)  =R(z)+ \jO_k(n+1; C_{01}/\rho^{n+1} , \rho),\\
	~\quad&\text{with}~~C_{01}\leqslant 2^{k(k+1)/2}C_0 \big(C_1+(k+1)^n
	\norm{Q}_{\rho}\big)^{n+1}+\norm{P}_{2^kC_1+(k+1)^n\norm{Q}_{\rho}}; 
\end{align*}
\item[(Taylor Expansion)] Assume $0<\eta<1$ and $n\geqslant k$. If $f$ is holomorphic on 
	$\Dm(0,\tau)$ and $\partial^nf\in\jO_{0}(0;C,\tau)$ then $\displaystyle f(z)=T_f^n(z)+
	\jO_k(n;C/(n-k)!,\tau)$, with $T_f^n(z)=\sum\limits_{\ell=0}^{n-1}\frac{1}{\ell!}
	\partial^\ell f(0)~z^\ell$;%
\item[(Inverse)] Assume $m>k\geqslant 1$. Then  there exists two constant $\beta_{m}\geqslant 
	0$ and $B_m\geqslant 0$ such that  if  $\Phi(z)=z+P(z)+\jO_k(m;C,\tau)$, where $P$ is a 
	polynomial of degree $m-1$ and valuation $2$ satisfying $\norm{\partial P}_\tau+
	\norm{\bar\partial P}_\tau+C\tau^{m-1} \leqslant \varepsilon$, with $2\varepsilon+
	\varepsilon^2 \leqslant 1/2$, then $\Phi$  is a diffeomorphism from $\Dm(0,\tau)$ onto a 
	set containing  $\Dm(0,\tau(1-\varepsilon))$ such that  $\Phi^{-1}$ is of the form 
	\[
	\Phi^{-1}(z)=z+Q(z)+\jO_k(m;B_{m} \varepsilon/\tau^{m-1}, \tau(1-\varepsilon)),
	\]
	where $Q$ is a polynomial of degree $m-1$ with valuation~$2$ and $\norm{Q}_\tau
	\leqslant \beta_m\varepsilon\tau$.
\end{description}
\end{lemma}
\begin{proof}~The proof of the axioms from~(restriction) to ($Z$-product) follows directly
from the definition  or by easy inductions over $k$.~\\[5pt]
$\bullet$~We prove the (Polynomial) axiom.  Let's write 
	\[
	P(z)=[P]_{\leqslant n-1}+\sum_{n\leqslant \abs{\nu}\leqslant m} a_\nu z^\nu.
	\]
	Since $z$ and $\overline{z}$ belong to $\jO_k(1;1,\tau)$, the ($Z$-product) axiom implies 
	by an easy induction that $z^\nu=\jO_k(\abs{\nu};(k+1)^{\abs{\nu}-1},\tau)$. Therefore, for 
	$\nu\in\N^2$ satisfying $m\geqslant \abs{\nu}\geqslant n$,  the (restriction) axiom shows 
	that 
	\[
	z^\nu=\jO_k\big(\abs{\nu};(k+1)^{\abs{\nu}-1},\tau\big)\subset\jO_k\big(n;(k+1)^{m-1}
	\tau^{\abs{\nu}-n},\tau\big).
	\]
	This implies that $[P]_{\geqslant n}=\jO_k(n;(k+1)^{m-1}C/\tau^n,\tau)$, with $C\leqslant 
	\sum\limits_{n\leqslant \abs{\nu}\leqslant m} \abs{a_\nu}\tau^{\abs{\nu}}$, so $C\leqslant
	\norm{P}_\tau$, and the proof of the axiom is complete.~\\[5pt]
$\bullet$~We prove the (Lipschitz) axiom. Assume that $\partial f\in\jO_{k}(n;C_1,\tau)$ and 
	$\bar\partial f\in\jO_{k}(n;C_2,\tau)$. This implies that 
	\[
	\abs{df(z)\cdot Z}=\abs{\partial f(z) Z+\bar\partial f(z)\overline{Z}}\leqslant(C_1+C_2)
	\abs{z}^n\abs{Z}.
	\]
	Therefore the derivative of $h$ is bounded by $(C_1+C_2)\tau^n$ on $\Dm(0,\tau)$, so $f$ 
	extends to a Lipschitz-continuous function on this disc. Furthermore, we obtain that 
	\[
	\abs{f(z)-f(0)}\leqslant \frac{C_1+C_2}{n+1}\abs{z}^{n+1}\leqslant C_0 \abs{z}^{n+1},
	\]
	with $C_0=C_1+C_2$ if $n=0$ and $C_0=\max(C_1,C_2)$ if $n\geqslant 1$. By the 
	(primitive) axiom, this completes the proof of the (Lipschitz) axiom.~\\[3pt]
$\bullet$~We prove the (composition) axiom by induction over $k$. The condition on 
	$f(\Dm(0;\tau_1))$ shows that $h\circ f$ is well defined  on $\Dm(0;\tau_1)$ and since 
	$n\geqslant 0$, we have $\abs{h\circ f(z)}\leqslant C_0(C_1\abs{z}^m)^n=C_0C_1^n 
	\abs{z}^{nm}$. Thus  we obtain that $h\circ f\in\jO_0(nm; \alpha_0 C_0C_1^n, \tau_1)$, 
	with $\alpha_0\leqslant 1$, and the axiom is proved  for $k=0$.~\\
	Assume that $k\geqslant 1$, so $\partial h(f)$ and $\bar\partial h(f)$ are in $\jO_{k-1}((n-1)m; 
	\alpha_{k-1}C_0C_1^{n-1},\tau_1)$.  nI that case  the (product) axiom shows that 
	\[
	\begin{cases}
		 \partial( h\circ f)=\partial h(f)\cdot \partial f +\bar\partial h(f)\cdot\overline{\bar\partial f}
										\in\jO_{k-1}(nm-1;  2^k \alpha_{k-1}C_0C_1^{n} , \tau_1),\\ %
		\bar\partial( h\circ f)=\partial h(f)\cdot\bar\partial f+\bar\partial h(f)\cdot\overline{\partial f}
				\in\jO_{k-1}(nm-1;  2^k\alpha_{k-1}C_0C_1^{n},\tau_1). 
	\end{cases}
\]
This implies that $h\circ f\in\jO_k(nm; \alpha_k C_0C_1^n, \tau_1)$, with $\alpha_k\leqslant 
\max(\alpha_0, 2^k\alpha_{k-1})$ and the (composition) axiom follows immediately.~\\[3pt]
%
$\bullet$~We prove the ($P$-composition) axiom. We write $h=P+\varepsilon_0$ and $f=Q
	+\varepsilon_1$, so $h\circ f=P\circ f+ \varepsilon_0\circ f$. Note that $1\in\jO_k(0;1,\rho)$, 
	so the $Z$-product axiom shows that $z^\nu=\jO_k(\abs{\nu};(k+1)^{\abs{\nu}},\rho)$ for 
	$\nu\in\N^2$.  Now we write 
	\begin{align*}
	Q(z)=\sum_{1\leqslant\abs{\nu}\leqslant n}b_\nu z^\nu&=\sum_{1\leqslant\abs{\nu}\leqslant n}
				\jO_k(\abs{\nu};(k+1	)^{\abs{\nu}}\abs{b_\nu},\rho)\\ %
	& 	=\sum_{1\leqslant \abs{\nu}\leqslant n}\jO_k(1;(k+1)^{\abs{\nu}}\abs{b_\nu}
																\rho^{\abs{\nu}-1},\rho)\\ %
	&		=\jO_k\bigg(1;(k+1)^n \frac{\norm{Q}_\rho}{\rho},\rho\bigg).
	\end{align*}
	Thus we obtain that $f(z)=\jO_k\big(1;\frac{(k+1)^n \norm{Q}_\rho+C_1}{\rho},\rho\big)$ 
	and the (composition) axiom implies that
	\begin{equation}\label{eq:P-comp-eps}
		\varepsilon_0\circ f(z)=\jO_k\bigg(n+1;2^{\frac{k(k+1)}{2}}C_0\bigg(\frac{(k+1)^n
																\norm{Q}_\rho+C_1}{\rho}\bigg)^{n+1},\rho\bigg).
	\end{equation}
	Now we estimate $P\circ f$.~\\
 	For  $1\leqslant j \leqslant \ell\leqslant n$ we have $Q^j=\jO_k(0;(k+1)^{nj}
	\norm{Q}_\rho^j,\rho)$, so 
\begin{align*}
f^\ell	& = (Q+\varepsilon_0)^\ell= Q^\ell+\sum_{j=0}^{\ell-1}
															\big(\begin{smallmatrix}\ell\\ j\end{smallmatrix}\big) 
															Q^j\varepsilon_0^{\ell-j}\\ %
			& =Q^\ell+\sum_{j=0}^{\ell-1} 
							\bigg(\begin{smallmatrix}	\ell\\ j	\end{smallmatrix}\bigg) 
							\jO_k\bigg((\ell-j)(n+1);\bigg((k+1)^n\norm{Q}_\rho\bigg)^j
							\bigg(\frac{2^kC_1}{\rho^{n+1}}\bigg)^{\ell-j},\rho\bigg)\\ %
&=Q^\ell+\jO_k\bigg(n+1;\sum_{j=0}^{\ell-1}
							  \bigg(\begin{smallmatrix}	\ell\\ j \end{smallmatrix}\bigg)
							  \big((k+1)^n\norm{Q}_\rho\big)^j\big({2^kC_1}\big)^{\ell-j}/\rho^{n+1},
							\rho\bigg). %
\end{align*}
	We can decompose $Q^\ell=[Q^\ell]_{\leqslant n}+[Q^\ell]_{>n}$ in its part of degree $n$ 
	and its part of valuation $n+1$. Since $[Q^\ell]_{>n}\in\jO_k(n+1;(k+1)^{\ell n}
	\norm{[Q^\ell]_{>n}}_\rho/\rho^{n+1})$ and $\norm{Q^\ell}_\rho\leqslant
	\norm{Q}_\rho^\ell$, it follows that 
	\[
		f^\ell=[Q^\ell]_{\leqslant n}+\jO_k\bigg(n+1;\big((k+1)^n\norm{Q}_\rho
																	+2^kC_1\big)^{\ell}/\rho^{n+1},\rho\bigg).
	\]
	Now we write $P(z)=\sum\limits_{\abs{\nu}\leqslant n}a_\nu z^\nu$ and we note that 
	$\norm{~[Q^\nu]_{\leqslant n}}_\rho\leqslant (\norm{Q}_\rho)^{\abs{\nu}}$, so 
	\begin{align*}
	P\circ f&=\sum_{\abs{\nu}\leqslant n}a_\nu [Q^\nu]_{\leqslant n}+\jO_k\bigg(n+1;
	\sum_{\abs{\nu}\leqslant n}\abs{a_\nu}\frac{\big((k+1)^n\norm{Q}_\rho
	+2^kC_1\big)^{\abs{\nu}}}{\rho^{n+1}},\rho\bigg)\\ %
&=R+\jO_k\bigg(n+1;\frac{1}{\rho^{n+1}}\norm{P}_{(k+1)^n\norm{Q}_\rho+2^kC_1},
																		\rho\bigg),\\ %
&~\text{with}~R=\sum_{\abs{\nu}\leqslant n}a_\nu  [Q^\nu]_{\leqslant n}
	~\text{and}~\norm{R}_\rho\leqslant \sum_{\abs{\nu}\leqslant n}\abs{a_\nu}
	(\norm{Q}_\rho)^{\abs{\nu}}=\norm{P}_{\norm{Q}_\rho}.
	\end{align*}
	This with~\eqref{eq:P-comp-eps} implies the ($P$-composition) axiom.~\\[3pt]
$\bullet$~The (Taylor expansion) axiom directly follows from the Taylor expansion theorem,   
	which  shows that 
	\[
	\left\vert
 				\partial^jf(z)-\sum_{\ell=0}^{n-1-j}\frac{1}{\ell!}\partial^\ell f(0)z^\ell
	\right\vert\leqslant \frac{\abs{z}^{n-j}}{(n-j)!}C.
	\]
~\\[3pt]
$\bullet$~ The proof of the (inverse) axiom proceeds in several steps. We first prove the 
	existence of the diffeomorphism,    then we estimate its derivatives.~\\[3pt]
{\bf Existence of $\mathbf{\Phi^{-1}}$.} Assume that  $\abs{w}\leqslant \tau(1-\varepsilon)$ 
	and set $\varphi_w(z)=w-\Phi(z)+z$.  This implies that  $\norm{d\varphi_w(z)}\leqslant
	\norm{\partial P}_\tau+\norm{\bar\partial P}_{\tau}+2C\tau^{m-1}\leqslant 2\varepsilon
	\leqslant 1/2$ and $\abs{\varphi_w(z)} \leqslant \abs{w}+C\tau^m\leqslant\abs{w}+ 
	\varepsilon\tau\leqslant \tau$ on $\Dm(0,\tau)$. Therefore $\varphi_w~:\Dm(0,\tau)\to
	\Dm(0,\tau)$ is $(2\varepsilon)$-Lipschitz. By Picard's theorem, it follows that the equation 
	$\varphi_w(z)=z$,  so $\Phi(z)=w$, has an unique solution $z\in\Dm(0,\tau)$ if $w\in
	\Dm(0,\tau(1-\varepsilon))$. Thus  $\Dm(0,\tau(1-\varepsilon))$ lies in  $\Phi(\Dm(0,\tau))$, 
	so $\Phi^{-1}$ is a diffeomorphism from $\Dm(0,\tau(1-\varepsilon))$ into 
	$\Dm(0,\tau)$.~\\[5pt]
	We now prove by induction over $k\geqslant 0$ the existence of a  constant $\beta_{k,m}$ 
	and of $Q$ such that such that $\Phi^{-1}(z)=z+Q(z)+\jO_k(m;\beta_{k,m}\varepsilon/
	\tau^{m-1}, \tau(1-\varepsilon))$. We set $\Phi_0(z)=\Phi(z)-z=w-z$, $\Psi(w)=\Phi^{-1}(w)
	=z$ and $\Psi_0(w)=\Psi(w)-w$.~\\[5pt]
{\bf The boot strapping equation.}
	We write the derivatives of $\Psi\circ \Phi(z)=z$ as 
	\[
	\begin{cases}
		1 = (\partial\Psi)\circ\Phi\cdot\partial\Phi+(\bar\partial \Psi)\circ\Phi\cdot\overline{\bar
																								    					 \partial\Phi},\\ %
		0 = (\partial\Psi)\circ\Phi\cdot\bar\partial\Phi+(\bar\partial\Psi)\circ\Phi\cdot
																										  \overline{\partial \Phi}. %
	\end{cases}
	\]
	Therefore we have 
 	\begin{equation}\label{eq:boot-strapping}
	\biggl(\begin{array}{c} \partial \Psi\\ \bar\partial \Psi\end{array}
	\biggr)\circ\Phi
	=
	\biggl(\begin{array}{cc}
					1+\partial\Phi_0 & \overline{\bar\partial\Phi_0}\\
					\bar\partial\Phi_0 & 1+\overline{\partial\Phi_0}
				\end{array}\biggr)^{-1}
	\biggl(\begin{array}{c}1\\ 0\end{array}\biggr)
	=
	\frac{1}{J(\Phi)}\biggl(\begin{array}{c}
												1+\overline{\partial\Phi_0}\\
															-\bar\partial\Phi_0
											\end{array}\biggr), 
	\end{equation}
	where $J(\Phi)=\abs{1+\partial\Phi_0}^2-\abs{\bar\partial\Phi_0}^2
	=1+2\Re(\partial\Phi_0) +\abs{\partial\Phi_0}^2-\abs{\bar\partial\Phi_0}^2$. 

	We now estimate the right hand side of Equation~\eqref{eq:boot-strapping}. First we 
	observe that $\partial \Phi_0$ and $\bar\partial \Phi_0$ are bounded on $\Dm(0;\tau)$ by 
	$\norm{\partial P}_\tau+C\tau^{m-1}$ and  $\norm{\bar\partial P}_\tau+C\tau^{m-1}$ 
	respectively, so they both lie in $\jO_{0}(0;\varepsilon,\tau)$.  It follows that 
	\[
				J(\Phi)=1+\jO_0(0;2\varepsilon+\varepsilon^2,\tau)\subset 1+\jO_{0}(0; 1/2, \tau).
	\]
	We have $\partial \Phi_0=\partial P+\jO_{k-1}(m-1;C,\tau)$ and $\bar\partial \Phi_0
	=\bar\partial P+\jO_{k-1}(m-1;C,\tau)$, so the ($P$-product) shows that 
	\begin{align*}
	\abs{\partial \Phi_0}^2=P_0+\jO_{k-1}\big(m-1; \frac{C_0}{\tau^{m-1}},\tau\big);
		-\abs{\bar\partial\Phi_0}^2=P_1+\jO_{k-1}\big(m-1;\frac{C_0'}{\tau^{m-1}},\tau\big),\\ %
	~\text{with $\norm{P_0}_\tau\leqslant\norm{\partial P}_\tau^2$, $\norm{P_1}_\tau\leqslant 
																							\norm{\bar\partial P}_\tau^2$ },\\ %
	C_0\leqslant 2 k^{m-1}C\tau^{m-1}\norm{\partial P}_\tau+2^{k-1}C^2\tau^{2m-2}
																					+k^{2m-2}\norm{\partial P}_\tau^2,\\ %
	\text{and}~C_0'\leqslant 2 k^{m-1}C\tau^{m-1}\norm{\bar\partial P}_\tau
										+2^{k-1}C^2\tau^{2m-2}+k^{2m-2}\norm{\bar\partial P}_\tau^2. %
	\end{align*}
	Therefore the polynomial  $J_0=2\Re(\partial P)+P_0-P_1$ has degree $m-2$ and satisfies 
	\begin{align*}
	J(\Phi)=1+J_0+\jO_{k-1}\big(m-1;C'/\tau^{m-1},\tau\big),\\ %
	~\text{with}~ \norm{J_0}_\tau\leqslant 2\norm{\partial P}_\tau+\norm{\partial P}_\tau^2+
					\norm{\bar\partial P}_\tau^2 \leqslant 2\norm{\partial P}_\tau+\varepsilon^2 \\ %
	\text{and}~C'\leqslant 2C\tau^{m-1}+C_0+C_0'\leqslant 2C\tau^{m-1}+C_0''. %
	\end{align*}
	Here we have 
	\begin{align*}
	C_0''&\leqslant C_0+C_0'\\
 	& \leqslant  2k^{m-1}C\tau^{m-1}\big(\norm{\partial P}_\tau+\norm{\bar\partial P}_\tau
		\big)+2^k C^2\tau^{2m-2}+k^{2m-2}\big(\norm{\partial P}_\tau^2
																		+\norm{\bar\partial P}_\tau^2\big)\\ %
	& \leqslant \bigg(k^{m-1}\big( \norm{\partial P}_\tau+\norm{\bar\partial P}_\tau\big)
																				+2^{k/2}C\tau^{m-1}\bigg)^2\\ %
	& \leqslant\max\big(2^k,k^{2m-2}\big)~\varepsilon^2. %
	\end{align*}
	Now we estimate $1/J(\Phi)$. The (Taylor expansion) axiom shows that 
	\[
	\frac{1}{1+z}=\sum_{\ell=0}^{m-2} (-1)^\ell z^\ell+\jO_{k-1}(m-1; 2^m(m-1)!/(m-k)!,1/2).
	\]
	Therefore the ($P$-composition) axiom applied to $(1+z)^{-1}-1$ and $J(\Phi)-1$ shows 
	that there exist a polynomial $J_1$ of degree $m-2$ and a constant $C''$ such that 
	\begin{align*}
	&\phantom{J_1}1/J(\Phi)=1+J_1+\jO_{k-1}\big(m-1;C''/\tau^{m-1},\tau\big),\text{with}~\\ %
	\norm{J_1}_\tau &\leqslant \sum_{\ell=1}^{m-2}\norm{J_0}_\tau^\ell\leqslant 2
											\norm{J_0}_\tau\leqslant 2(2\varepsilon+\epsilon^2)~\text{and}\\ %
	C''&\leqslant 2^{\frac{k(k-1)}{2}}2^{m}\frac{(m-1)!}{(m-k)!}\big(C'+k^{m-2}
			\norm{J_0}_\tau\big)^{m-1}\!\!+\sum_{\ell=1}^{m-2}\big(2^{k-1}C'+k^{m-2}
																										\norm{J_0}_\tau\big)^\ell \\ %
	& \leqslant   2^{\frac{k(k-1)}{2}}2^{m}(m-1)!k^{(m-2)(m-1)}\sum_{\ell=1}^{m-1}
																							\big(	C'+\norm{J_0}_\tau\big)^\ell.%
	\end{align*}
	Note that $C'+\norm{J_0}_\tau\leqslant 2C\tau^{m-1}+C_0''+2\norm{\partial P}_\tau+
	\varepsilon^2\leqslant 2\varepsilon+\max(2^k,1+k^{2m-2})~\varepsilon^2$. If $k\geqslant 
	2$ then $C'+\norm{J_0}_\tau\leqslant k^{2m-2}(\varepsilon+\epsilon^2)$. If $m>k=1$ 
	then  $C'+\norm{J_0}_\tau\leqslant 2(\varepsilon+\epsilon^2)$. It follows in both cases that 
	\begin{align*}
	C''&\leqslant 2^{\frac{k(k-1)}{2}}2^{m+1}(m-1)!m^{(3m-4)(m-1)}(\varepsilon+\epsilon^2)\\ %
	& \leqslant  2^{\frac{m^2+m+2}{2}}(m-1)!m^{(3m-4)(m-1)}(\varepsilon+\epsilon^2)\\ %
	& \leqslant D_m\varepsilon, \quad~\text{with $D_m=2^{\frac{(m+1)^2}{2}}(m-1)!
																												m^{(3m-4)(m-1)}$}.%
	\end{align*}
	If we set $Q_0=J_1+\partial P+[J_1\partial P]_{\leqslant m-2}$ and $Q_1=-\bar\partial 
	P-[J_1 \bar\partial P]_{\leqslant m-2}$ then the  estimates above,  the ($P$-product) axiom 
	and Equation~\eqref{eq:boot-strapping} imply that
	\begin{align}	\label{eq:boot-partial}
	(\partial\Psi)\circ\Phi & =1+Q_0(z)+\jO_{k-1}(m-1; D_m'\varepsilon/\tau^{m-1},\tau),\\ %
							\label{eq:boot-barpartial}
	(\bar\partial\Psi)\circ\Phi & =Q_1(z)+\jO_{k-1}(m-1; D_m''\varepsilon/\tau^{m-1},\tau),%
	\end{align}
	where $\norm{Q_0}_\tau\leqslant \norm{J_1}_\tau+\norm{\partial P}_\tau+\norm{J_1}_\tau
	~\norm{\partial P}_\tau\leqslant 2(2\varepsilon+\varepsilon^2)+\varepsilon+2\varepsilon
	(2\varepsilon+\varepsilon^2)\leqslant \frac{13}{2}\varepsilon$ and 
	$\norm{Q_1}_\tau\leqslant \norm{\bar\partial P}_\tau+ \norm{\bar\partial P}_\tau~
	\norm{J_1}_\tau\leqslant 2\norm{\bar\partial P}_\tau\leqslant 2\varepsilon$, as long as we 
	have
	\begin{align*}
	D_m'\varepsilon & \geqslant C\tau^{m-1}+C''+2^{k-1} C''C\tau^{m-1}\\ %
	& \qquad +k^{m-2}\big(\norm{\partial P}_\tau C''+C\tau^{m-1}\norm{J_1}_\tau\big)
														+k^{2(m-1)}\norm{\partial P}_\tau\norm{J_1}_\tau\\ %
	& \geqslant \varepsilon+D_m\varepsilon+2^{k-1}D_m\varepsilon^2+k^{m-2}D_m
												\varepsilon^2+k^{2(m-2)}2\epsilon(2\epsilon+\epsilon^2)\\ %
	& \geqslant \big((1+D_m+k^{2(m-2)})+D_m(2^{k-1}+k^{m-2})/4\big)\varepsilon,\\ %
	D_m''\varepsilon & \geqslant  C\tau^{m-1}+2^{k-1}C\tau^{m-1}C''\\ %
	& \quad+k^{m-2}(\norm{\bar\partial P}_\tau C''+\norm{J_1}_\tau C\tau^{m-1})
												+k^{2(m-2)}\norm{J_1}_\tau~\norm{\bar\partial P}_\tau\\ %
	& \geqslant  \varepsilon+2^{k-1}D_m\varepsilon^2+k^{m-2}D_m\varepsilon^2
																		+k^{2(m-2)}2\epsilon(2\epsilon+\epsilon^2).%
	\end{align*}
	Therefore we may take
	\[
	D_m'=D_m''+D_m=1+m^{2(m-2)}+D_m\big(1+(2^{m-1}+m^{m-2})/4\big).
	\]
~\\[5pt]
{\bf Computation of $\mathbf{\beta_{m}}$ and $B_m$.}
	We have 
	\[
	\abs{z-w}=\abs{\Phi_0(z)}\leqslant (\norm{\partial P}_\tau+\norm{\bar\partial P}_{\tau}+
	2C\tau^{m-1})~\abs{z}\leqslant 2\varepsilon\abs{z}, ~\text{so $(1-2\varepsilon)\abs{z} 
	\leqslant \abs{w} $.}
	\] 
	It follows that $\abs{\Psi(w)-w}=\abs{z-w}\leqslant 2\varepsilon\abs{z}\leqslant 2\varepsilon
	(1-2\varepsilon)^{-1} \abs{w}\leqslant 4\varepsilon\abs{w}$. Therefore we have $\abs{\Psi(w)}
	\leqslant (1+4\varepsilon)\abs{w}\leqslant 2\abs{w}$. Furthermore, 
	Equations~\eqref{eq:boot-partial} and~\eqref{eq:boot-barpartial} imply that 
	\[
	\left\{\begin{array}{rl}
	(\partial\Phi)\circ\Psi&=1+\jO_0\big(1;(\norm{Q_0}_\tau+D_m'\varepsilon)/\tau,\tau\big),\\ %
	(\bar\partial\Phi)\circ\Psi&=\jO_0\big(1;(\norm{Q_1}_\tau+D_m''\varepsilon)/\tau,\tau\big).%
	\end{array}\right.
	\]
	Thus we obtain that 
	\[
	\left\{\begin{array}{rl}
 		\partial\Phi &=1+\jO_0\big(1;(1+4\varepsilon)(\norm{Q_0}_\tau+D_m'\varepsilon)/\tau,
																										(1-\varepsilon)\tau\big),\\ %
		\bar\partial\Phi&=\jO_0\big(1;(1+4\varepsilon)(\norm{Q_1}_\tau+D_m''\varepsilon)/\tau,
																											(1-\varepsilon)\tau\big).%
	\end{array}\right.
	\]
	This with the (Lipschitz) axiom implies that
	\[
	\Phi^{-1}(z)=z+\jO_1(2;B_0^2\varepsilon, \tau(1- \varepsilon)), 
	~\text{with}~B_0^2=(1+4\varepsilon)(\norm{Q_0}_\tau/\varepsilon+D_m')\leqslant 13+2D_m'.
	\]
	Now we set $\beta_m^2=0$.  Given $2\leqslant \ell\leqslant m-1$, we assume that there 
	exist $B_m^\ell\geqslant 0$,  $\beta_m^\ell\geqslant 0$  and a polynomial of degree $\ell-1$
	and valuation $2$ satisfying $\norm{q_\ell}_\tau\leqslant \beta_m^\ell\tau\varepsilon$ and 
 	$\Phi^{-1}(z)=z+q_\ell(z)+\jO_{\min(\ell-1,k)}\big(\ell;B_m^\ell \varepsilon/\tau^{\ell-1};
	(1-\varepsilon)\tau\big)$. Note that Equations~\eqref{eq:boot-partial} 
	and~\eqref{eq:boot-barpartial}, the (polynomial) and (restriction)  axioms imply that 
	\[
	\left\{\begin{array}{rl}
		(\partial\Psi)\circ\Phi&=1+[Q_0]_{\leqslant \ell}+\jO_{\min(k,\ell)-1}
		\big(\ell;(\min(k,\ell)^{\ell-1}\norm{Q_0}_\tau+D_m'\varepsilon)/\tau^\ell,\tau\big),\\ %
		(\bar\partial\Psi)\circ\Phi&=[Q_1]_{\leqslant \ell}+\jO_{\min(k,\ell)-1}
			\big(\ell;(\min(k,\ell)^{\ell-1}\norm{Q_0}_\tau+D_m''\varepsilon)/\tau^\ell,\tau\big).%
	\end{array}\right.
	\]
	The ($P$-composition) axiom and  Equations~\eqref{eq:boot-partial} 
	and~\eqref{eq:boot-barpartial} imply that there exist  $C_{0\ell}\geqslant 0$ and 
	$C_{1\ell}\geqslant 0$, two polynomials $R_{0\ell}$ and $R_{1\ell}$ of degree $\ell-1$ 
	satisfying 
 	\[
	\left\{\begin{array}{rl}
	 \partial\Psi &=1+R_{0\ell}+\jO_{\min(\ell,k)-1}(\ell;C_{0\ell}/\tau^{\ell},(1-\varepsilon)\tau),\\ %
	\bar\partial\Psi&=R_{1\ell}+\jO_{\min(\ell,k)-1}(\ell;C_{1\ell}/\tau^{\ell},(1-\varepsilon)\tau),%
	\end{array}\right.
	\]
	with $\norm{R_{0\ell}}_\tau\leqslant \norm{Q_0}_{\tau+\norm{q_\ell}_\tau}\leqslant
	(1+\beta_m^\ell\varepsilon)^{m-2} \norm{Q_0}_{\tau}\leqslant\frac{13}{2}
	(1+\beta_m^\ell/4)^{m-2} ~\varepsilon$ and $\norm{R_{1\ell}}_\tau\leqslant 
	\norm{Q_1}_{\tau+\norm{q_\ell}_\tau}\leqslant2(1+\beta_m^\ell/4)^{m-2}\varepsilon$,  
	and where 
	\begin{align*}
	&&C_{0\ell}\leqslant 2^{\frac{\ell(\ell-1)}{2}}\frac{\min(k,\ell)^{\ell-1}\norm{Q_0}_\tau
		+D_m'\varepsilon}{\tau^\ell}\big(B_m^\ell \varepsilon\tau+\min(k,\ell)^{\ell-1}
																					(\tau+\norm{q_\ell}_\tau)\big)^\ell\\ %
	&& + \norm{[Q_0]_{\leqslant \ell}}_{2^{\min(k,\ell)-1}B_m^\ell\tau\varepsilon+
																\min(k,\ell)^{\ell-1}(\tau+\norm{q_\ell}_\tau)}\\ %
	& &\leqslant  2^{\frac{\ell(\ell-1)}{2}}\big({\scriptstyle\frac{13}{2}}\min(k,\ell)^{\ell-1}
	+D_m'\big)\big(B_m^\ell \varepsilon+\min(k,\ell)^{\ell-1}(1+\beta_m^\ell\varepsilon)
																											\big)^\ell\varepsilon\\ %
	&&+\bigg(2^{\min(k,\ell)-1}B_m^\ell\varepsilon+\min(k,\ell)^{\ell-1}(1+\beta_m^\ell
																		\varepsilon)\bigg)^{\ell}\norm{Q_0}_\tau\\ %
	&&\leqslant C_{0\ell}'\varepsilon\\ %
	&&\text{with}~C_{0\ell}'= 2^{\frac{\ell(\ell-1)}{2}}\big({\scriptstyle\frac{13}{2}}
		\min(k,\ell)^{\ell-1}+D_m'\big)\big(B_m^\ell/4+\min(k,\ell)^{\ell-1}(1+\beta_m^\ell/4)
																																\big)^\ell\\ %
	&&+{\scriptstyle\frac{13}{2}}\bigg(2^{\min(k,\ell)-1}B_m^\ell/4+\min(k,\ell)^{\ell-1}
																							(1+\beta_m^\ell/4)\bigg)^{\ell}.%
	\end{align*}
	In a similar way, we obtain that $C_{1\ell}\leqslant C_{1\ell}'\varepsilon$, with 
	\begin{align*}
	&&C_{1\ell}'= 2^{\frac{\ell(\ell-1)}{2}}\big(2\min(k,\ell)^{\ell-1}+D_m''\big)
									\big(B_m^\ell /4+\min(k,\ell)^{\ell-1}(1+\beta_m^\ell/4)\big)^\ell\\ %
	&&+2\bigg(2^{\min(k,\ell)-1}B_m^\ell/4+\min(k,\ell)^{\ell-1}(1+\beta_m^\ell/4)
																															\bigg)^{\ell}.%
	\end{align*}
	Let apply  the (Lispschitz) axiom to $\Psi(z)-q_{\ell+1}(z)$, where $q_{\ell+1}$ denotes 
	the polynomial of degree $\ell$ and valuation $2$ such that $\partial q_{\ell+1}=R_{0\ell}$ 
	and $\bar\partial q_{\ell+1}=R_{1\ell}$. We obtain that  
	\[
	\Psi(z)=z+q_{\ell+1}(z)+\jO_{\min(\ell,k)}\big(\ell+1;\max(C_{0\ell},C_{1\ell})/\tau^{\ell},
	(1-\varepsilon)\tau\big).
	\]
	Since $\norm{q_{\ell+1}}_\tau\leqslant \big(\norm{R_{0\ell}}_\tau+\norm{R_{1\ell}}_\tau
	\big)\tau$, we may set $\beta_{m}^{\ell+1}=(2+\frac{13}{2})(1+\beta_m^\ell/4)^{m-2}$ 
	and $B_m^{\ell+1}=\max(C_{0\ell}',C_{1\ell}')$. This proves the (inverse) axiom by 
	induction over $m$, with $\beta_m=\beta_m^m$ and $B_m=B_m^m$.
\end{proof}

We have also used the following properties of the space $\jO_k^\T$. Since the proofs follow 
easily from the definitions and are very similar to those of the spaces $\jO_k$, we omit them.
\begin{lemma}\label{lem:O_kT}
Assume $(k,m,n,\ell,\ell_1,\ell_2)\in\N^3\times \R^3$. Then the spaces $\jO_k^\T$
satisfy the following axioms.
\begin{description}
\item[($\T$-Derivative)] If $p\in \jO_{k}^\T(\ell;C,\tau)$ then $\partial_r^\alpha
	\partial_\theta^\beta p\in \jO_{k-\alpha-\beta}(\ell-\alpha; C,\tau)$ for all $(\alpha,\beta)\in
																		\N^2$ such that $\alpha+\beta \leqslant k$;%
\item[($\T$-Primitive)] If $p\in\jO_0^\T(\ell+1;C_0,\tau)$, $\partial_r p\in\jO_{k}^\T
	(\ell;C_1,\tau)$ and $\partial_\theta p\in\jO_{k}^\T(\ell;C_2,\tau)$ then we have $p\in  
	\jO_{k+1}^\T(\ell+1;C,\tau)$, with $C \leqslant \max(C_0,C_1,C_2)$;%
\item[($\T$-Product)] If $p\in\jO_{k}^\T(\ell_1;C_1,\tau)$ and $g\in\jO_{k}^\T(\ell_2;
	C_2,\tau)$ then we have $pq\in\jO_{k}(\ell_1+\ell_2;2^kC_1C_2,\tau)$;%
\item[($\T$-Composition)] Assume $n\geqslant k$, $\rho>0$ and $\tau>0$. Let $f~:
	\Dm(0,\rho)\to \C$ and $p~:(0;\tau]\times\T\to\C$ satisfy $p((0;\tau]\times\T)
	\subset\Dm(0,\rho)$, $f(z)=\jO_k(n;C,\rho)$ and $p=\jO_k^\T(\ell;C_1,\tau)$. Then we 
	have $f\circ p (r,\theta)=\jO_k^\T(n\ell;2^{\frac{k(k+1)}{2}}C C_1^n,\tau)$.
\end{description}
\end{lemma}


\vfil

\pagebreak

\section{Estimates on Gevrey maps}
\label{app:GevMapsComp}

We begin with preliminaries on the composition of Gevrey functions
(Section~\ref{secappremind}) and the flow of a Gevrey near-integrable
Hamiltonian (Section~\ref{secGevflow}),
then we prove Proposition~\ref{propPsiPhih} in Section~\ref{subsecpfpropPsiPhih}.
In all this part we omit the index~$\al$ in the Gevrey norms, writing
for instance
$\normD{\,\cdot\,}_{L,R}$ instead of $\normD{\,\cdot\,}_{\al,L,R}$.

We end Appendix~\ref{app:GevMapsComp} with a reminder of a result on
Gevrey ``bump'' functions proved in \cite{ms}
(Section~\ref{secBumpGev}), used in Section~\ref{secPfThperi} as well
as in Sections~\ref{ssec:perdomains} and~\ref{ssec:proofn}.

\subsection{Reminder on Gevrey maps and their composition}	\label{secappremind}

Let $n\ge1$, $L,R>0$, and $\ph \in G^{\al,L}(\A^n_R)$. 
We first recall the analogue of the Cauchy inequalities for the Gevrey norms~\eqref{eq:defGalL}:
if $0 < \La < L$,
then all the partial derivatives of~$\ph$ belong to $G^{\al,\La}(\A^n_R)$ and,
for each $k\in\N$,
\begin{equation}	\label{ineqGevCauch}
\sum_{m\in \N^{2n};\ |m|=k} \normD{\pa^m\ph}_{\La,R} \le 
\frac{k!^\al}{(L-\La)^{k\al}} \normD{\ph}_{L,R}
\end{equation}
(Lemma~A.2 from \cite{hms}).

To state the result on composition, we introduce a new notation:
\begin{equation}	\label{eqdefcNst}
\cN^*_{L,R}(\ph) \defeq 
\sum_{\ell\in\N^{2n},\ \ell\neq0} \frac{L^{|\ell|\al}}{\ell!^\al} 
\normD{\pa^\ell\ph}_{C^0(\A^n_R)},
\end{equation}
so that
$\normD{\ph}_{L,R} = \normD{\ph}_{C^0(\A^n_R)} + \cN^*_{L,R}(\ph)$.
Then, Proposition~A.1 of \cite{hms} yields
\begin{prop}	\label{propCompos}
Let $n\ge1$, $R,R_0>0$, $\La,L>0$, and consider a map $\phi \col \A^n_R \to
\A^n_{R_0}$, the $2n$ components of which belong\footnote{%
In fact, the first~$n$ components are of the form
$\ph \col \A^n_R \to \T$ and, for them, what we mean is that there is a lift
$\ti\ph \col \R^n\times\ov B_R \to \R$ such that
$\ti\ph_{|\FF^n_R} \in G^{\al,\La}(\FF^n_R)$, with
$\FF^n_R \defeq [-1,1]^n\times\ov B_R$; observe that $\cN^*_{\La,R}(\ph)$
stays well-defined.
}
to $G^{\al,\La}(\A^n_R)$ and satisfy
\begin{equation}	\label{ineqcNphij}
\cN^*_{\La,R}(\phi_1), \ldots, \cN^*_{\La,R}(\phi_{2n})
\le \frac{L^\al}{(2n)^{\al-1}}.
\end{equation}
Then, for any $Y \in G^{\al,L}(\A^n_{R_0})$, we have 
$Y\circ\phi\in G^{\al,\La}(\A^n_R)$ and
$\normD{Y\circ\phi}_{\La,R} \le \normD{Y}_{L,R_0}$.
\end{prop}


When testing inequalities~\eqref{ineqcNphij} to apply this result, the following may be useful:
\begin{lemma}   \label{lemusefulNst}
Let $n\ge1$ and $R>0$. Suppose $0 < \La < L$ and $\ph\in
G^{\al,L}(\A^n_R)$. Then
\begin{equation}	\label{ineqcNst}
\cN^*_{\La,R}(\ph) \le \frac{\La^\al}{(L-\La)^\al} \normD{\ph}_{L,R}.
\end{equation}
\end{lemma}

\begin{proof}
Bounding $\cN^*_{\La,R}(\ph)$ by the sum 
$\sum_\mu\sum_m \frac{\La^{(1+|m|)\al}}{(\mu+m)!^\al} 
\normD{\pa^{(\mu+m)}\ph}_{C^0(\A^n_R)}$
over all multi-indices $\mu,m$ with $|\mu|=1$ and using $(\mu+m)!\ge m!$, we get
$\cN^*_{\La,R}(\ph) \le \La^\al \sum_\mu \normD{\pa^\mu \ph}_{\La,R}$
and we conclude by~\eqref{ineqGevCauch}.
\end{proof}

\subsection{A lemma on the flow of a Gevrey near-integrable Hamiltonian}
\label{secGevflow}

\begin{lemma}	\label{lemGevnearintflow}
Let $n\ge1$, $\al\ge1$, $L,R_0>0$, and $h\in G^{\al,L}(\ov B_{R_0})$.
Let $R,\La$ be such that 
\begin{equation}	\label{ineqRLa}
0<R<R_0, \qquad
0 < \La < 
\big(1 + 2^{4\al} L^{-2\al} \normD{h}_{L,R_0}\big)^{-1/\al}
\frac{L}{2\cdot (2n)^{(\al-1)/\al} }.
\end{equation}
Then there exist $\eps_0,C>0$ such that, for any $u\in G^{\al,L}(\A^n_{R_0})$ with
$\normD{u}_{L,R_0} < \eps_0$
and any $t\in[0,1]$,
the time-$t$ map $\Phi^{t(h+u)} \col \A^n_R \to \A^n_{R_0}$ is well-defined and satisfies
\begin{equation}	\label{ineqPhihuPhih}
\NormD{\Phi^{t(h+u)}-\Phi^{th}}_{\La,R} \le C \normD{u}_{L,R_0}
\end{equation}
(with the notation~\eqref{eqdefNotNorm}).
The numbers~$\eps_0$ and~$C$ can be chosen as depending on~$h$ only through
$\normD{h}_{L,R_0}$ and being respectively decreasing and increasing
functions of this quantity.
\end{lemma}


\begin{rem}
In fact 
$\NormD{\Phi^{t(h+u)}-\Phi^{th}}_{\La,R} \le C t \normD{u}_{L,R_0}$
for all $t\in[0,1]$
(as can be seen by applying~\eqref{ineqPhihuPhih} to~$th$ and~$tu$ themselves).
\end{rem}


\begin{proof}
\textbf{a)} Let $n,\al,L,R_0,R,\La$ be as in the hypothesis, let $h\in G^{\al,L}(\ov
B_{R_0})$.
We set 
\begin{equation}	\label{eqdefK}
L' \defeq \frac{\La+L}{2}, \qquad
K \defeq \max\Big\{ 2\cdot 2^{3\al} (L-\La)^{-2\al} \normD{h}_{L,R_0} , 1 \Big\}.
\end{equation}

Let $u\in G^{\al,L}(\A^n_{R_0})$ and $\eps \defeq \normD{u}_{L,R_0}$.
We shall work in the phase space $\R^n\times\ov B_{R_0}$, denoting the variables
by
\[
x = (\th;r) = (\th_1,\ldots,\th_n;r_1,\ldots,r_n) = (x_1,\ldots,x_{2n}).
\]
We can consider that $h$ and $h+u$ generate Hamiltonian vector fields $\ti
X_{h}$ and~$\ti X_{h+u}$ which are defined on $\R^n\times\ov B_{R_0}$ and 
$1$-periodic in each of the first~$n$ variables.
The flow of~$\ti X_h$ is
\begin{equation}	\label{eqdeftiPhith}
\ti\Phi^{t h}(\th;r) = (\th + t \na h(r); r).
\end{equation}
It is $\Z^n$-equivariant, in the sense that 
$\ti\Phi^{t h}(\th+\ell;r) = \ti\Phi^{t h}(\th;r) + (\ell;0)$ for any $\ell\in\Z^n$.
We shall study the flow over the time-interval $[0,1]$ of the vector field
\begin{equation}	\label{eqdeftiXhu}
\ti X_{h+u}(\th;r) = \big( \na h(r) + \na\ccr u(\th;r); -\na\cth u(\th;r)\big),
\qquad (\th;r) \in \R^n \times \ov B_{R_0},
\end{equation}
as the solution of a fixed-point equation in a complete metric space for which the
contraction principle applies.
We shall find a unique solution 
\[
\ti\Phi^{t(h+u)} \col \R^n\times\ov B_R \to \R^n\times\ov B_{R_0},
\qquad t\in[0,1],
\]
which is a $\Z^n$-equivariant lift to $\R^n\times\ov B_{R_0}$ of the flow
of~$X_{h+u}$ in~$\A^n_{R_0}$.

\medskip

\noindent\textbf{b)} 
Let $V \defeq \big( G^{\al,\La}(\A^n_R) \big)^n$. 
%
For any $\psi\in V$, we write
\[
\psi = (\psi_1,\ldots, \psi_n), \qquad
\normD{\psi}_V \defeq \normD{\psi_1}_{\La,R} + \cdots + \normD{\psi_n}_{\La,R}.
\]
Let $W \defeq V\times V$.
For any $\eta\in W$, we write
\[
\eta = \big(\eta\cth; \eta\ccr \big) 
= (\eta_1,\ldots,\eta_n;\eta_{n+1},\ldots,\eta_{2n}), \qquad
\normD{\eta}_W \defeq \frac{1}{K}\normD{\eta\cth}_V + \normD{\eta\ccr}_V.
\]
Let $E \defeq C^0\big( [0,1], W \big)$. 
For any $\xi \in E$, we set
\[
\normD{\xi}_E \defeq \max_{t\in[0,1]} \normD{\xi(t)}_W.
\]
We get a Banach space $(E,\normD{\cdot}_E)$.

Let us denote by $\ph$ the ``unperturbed'' flow over $[0,1]$, \ie
\[
\ph(t) \defeq \ti\Phi^{t h}, \qquad t\in[0,1]
\]
defined by~\eqref{eqdeftiPhith}.
For every $\xi\in E$ and $t\in[0,1]$, $\ph(t)+\xi(t)$ can be considered as a
$\Z^n$-equivariant map $\R^n\times\ov B_R \to \R^n \times \R^n$ (identifying
functions on $\A^n_R$ with functions on $\R^n\times\ov B_R$ which are
$1$-periodic in the first~$n$ variables).
We thus can view
\[
\ph + E \defeq \{\, \ph+\xi \mid \xi\in E \,\}
\]
as a complete metric space (with the distance $\dist(\ph+\xi,\ph+\xi^*) \defeq
\normD{\xi^*-\xi}_E$) where the flow $\ti\Phi^{t(h+u)}$ is to be found.
More specifically, we restrict ourselves to the closed ball
\begin{gather}
\notag
\cB_\rho \defeq \{\, \ph+\xi \mid \normD{\xi}_E \le \rho \,\}, \\[-1.5ex]
\intertext{with}
\label{eqdefrhoepsz}
\rho \defeq \frac{2^{\al+1} \eps}{ (L-\La)^\al }, \qquad 
\eps < \eps_0 \defeq \frac{(L-\La)^\al}{2^{\al+1}} 
\min\Big\{ R_0-R,\ \bet,\ \frac{(L-\La)^\al}{2^{2\al}K} \Big\}
\end{gather}
with
$\bet \defeq \dfrac{L'^\al}{(2n)^{\al-1}} - 
\Big(1+2^{2\al} (L-\La)^{-2\al}\normD{h}_{L,R}\Big)\La^\al$
(the positiveness of~$\bet$ is ensured by~\eqref{ineqRLa} because $L-\La>L/2$).
Observe that
\begin{equation}	\label{ineqrho}
\rho \le \min\{ R_0-R, \bet \}.
\end{equation}

\medskip

\noindent\textbf{c)} 
The flow $\Psi(t) = \ti\Phi^{t(h+u)}$ that we are searching is characterised by
$\Psi(0)=\Id$ and $\frac{d\Psi}{d t}(t) = \ti X_{h+u} \circ \Psi(t)$,
or
\[
\Psi(t) = \Id + \int_0^t \ti X_{h+u}\circ\Psi(\tau) \, d\tau.
\]
Let us first check that, with our choice of~$\rho$, the formula
\begin{equation}	\label{eqdefcF}
\cF(\Psi)(t) \defeq \Id + \int_0^t \ti X_{h+u}\circ\Psi(\tau) \, d\tau,
\qquad t\in[0,1]
\end{equation}
defines a functional $\cF \col \cB_\rho \to \ph+E$.

Assume $\Psi=\ph+\xi\in\cB_\rho$.
In view of~\eqref{eqdeftiXhu}, the components of $\ti X_{h+u}$ belong to
$G^{\al,L'}(\A^n_{R_0})$. We thus only need to check that, for each $\tau\in[0,1]$, 
$\Psi(\tau)$ maps $\A^n_R$ in $\A^n_{R_0}$ and its components satisfy
$\cN^*_{\La,R}\big(\Psi_j(\tau)\big) \le L'^\al / (2n)^{\al-1}$
so as to apply Proposition~\ref{propCompos}.

The first condition is met because $\Psi\ccr(\tau) = r + \xi\ccr(\tau)$ and the
components of 
$\xi\ccr(\tau) = \big( \xi_{n+1}(\tau),\ldots,\xi_{2n}(\tau) \big)$
satisfy 
\[
\sum_{j=n+1}^{2n} \normD{\xi_j(\tau)}^2_{C^0(\A_R)} \le
\normD{\xi(\tau)}^2_W \le \normD{\xi}^2_E \le \rho^2
\le (R_0-R)^2
\]
by~\eqref{ineqrho}.
The second condition is met because, for any $1\le j\le 2n$, on the one hand
$\cN^*_{\La,R}\big(\xi_j(\tau)\big) \le \normD{\xi_j(\tau)}_{\La,R} \le
\rho$,
and on the other hand
$\cN^*_{\La,R}\big(\ph_j(\tau)\big) = \La^\al$ for $j\ge n+1$,
while for $j\le n$, by~\eqref{ineqcNst} and~\eqref{ineqGevCauch},
\begin{multline*}
\cN^*_{\La,R}\big(\ph_j(\tau)\big) \le 
\La^\al + \tau\cN^*_{\La,R}\Big(\frac{\pa h}{\pa r_j}\Big) \\
\le
\La^\al + \frac{\La^\al}{(L'-\La)^\al}\normD*{\frac{\pa h}{\pa r_j}}_{L',R} 
\le \La^\al \big( 1 + 2^{2\al} (L-\La)^{-2\al} \normD{h}_{L,R} \big)
= \frac{L'^\al}{(2n)^{\al-1}} - \bet,
\end{multline*}
which is $\le \frac{L'^\al}{(2n)^{\al-1}} -
\rho$ by~\eqref{ineqrho}.

\medskip

\noindent\textbf{d)} 
Let us now check that $\cF(\cB_\rho) \subset \cB_\rho$.
For $\Psi = \ph + \xi \in \cB_\rho$, we write $\cF(\Psi) = \ph + \eta$
and observe that, in view of~\eqref{eqdeftiPhith} and~\eqref{eqdeftiXhu},
for each $t\in[0,1]$,
\begin{gather*}
\eta\cth(t) = \int_0^t \big(g(\tau) + \na\ccr u\circ\Psi(\tau)\big) \, d\tau, 
\qquad
\eta\ccr(t) = - \int_0^t \na\cth u\circ\Psi(\tau) \, d\tau \\[-2.5ex]
\intertext{with}
g(\tau)(\th;r) \defeq 
\na h \big( r + \xi\ccr(\tau)(\th;r) \big) - \na h(r), \qquad \tau\in[0,1].
\end{gather*}
We already checked that Proposition~\ref{propCompos} applies to $\frac{\pa
h}{\pa r_i}\circ\Psi(\tau)$ and $\frac{\pa u}{\pa x_j}\circ\Psi(\tau)$. It
yields
\[
\normD*{\frac{\pa u}{\pa x_j} \circ \Psi(\tau)}_{\La,R} \le
\normD*{\frac{\pa u}{\pa x_j}}_{L',R_0}
\qquad \text{for $1\le j \le 2n$, $\tau\in[0,1]$}
\]
whence
\[
\normD{\eta\cth(t)-g(t)}_V + \normD{\eta\ccr(t)}_V \le
\sum_{j=1}^{2n} \normD*{\frac{\pa u}{\pa x_j}}_{L',R_0}
\le \frac{2^\al \eps}{(L-\La)^\al}
\]
(by~\eqref{ineqGevCauch}, recalling that $\normD{u}_{L,R_0} = \eps$).
If $1\le i\le n$, we can also apply Proposition~\ref{propCompos} to 
\[
g_i(\tau) =
\frac{\pa h}{\pa r_i} \big( r + \xi\ccr(\tau) \big) - \frac{\pa h}{\pa r_i}(r)=
\left\langle \int_0^1 \na \frac{\pa h}{\pa r_i}\big( r + s \xi\ccr(\tau) \big)\,d s,
\xi\ccr(\tau) \right\rangle,
\]
whence
$\normD{g_i(\tau)}_{\La,R} \le \sum_{j=1}^n 
\normD{\frac{\pa^2 h}{\pa r_i\pa r_j}}_{L',R_0}
\normD{\xi_{n+j}(\tau)}_{\La,R}$
and
\[
\normD{g(\tau)}_V \le \sum_{j=1}^n 
\bigg( \sum_{i=1}^n \normD*{\frac{\pa^2 h}{\pa r_i\pa r_j}}_{L',R_0} \bigg)
\normD{\xi_{n+j}(\tau)}_{\La,R}
\le \frac{2^{3\al} \normD{h}_{L,R_0} \rho}{(L-\La)^{2\al}}
\]
by~\eqref{ineqGevCauch}.
Therefore, recalling that $K\ge1$, we get
\[
\frac{1}{K} \normD{\eta\cth(t)}_{\La,R} + \normD{\eta\ccr(t)}_{\La,R} 
\le \frac{2^{3\al} \normD{h}_{L,R_0} \rho}{K (L-\La)^{2\al}} + 
\frac{2^\al \eps}{(L-\La)^\al}.
\]
Both summands are $\le \rho/2$, by our choices of~$K$ and~$\eps_0$, 
\eqref{eqdefK} and~\eqref{eqdefrhoepsz},
whence $\normD{\eta}_E \le \rho$ and $\cF(\Psi)\in\cB_\rho$ as desired.

\medskip

\noindent\textbf{e)} 
Similar computations show that~$\cF$ induces a contraction on~$\cB_\rho$:
Let $\Psi = \ph + \xi, \ti\Psi = \ph+\ti\xi \in \cB_\rho$, 
and $\cF(\Psi) = \ph + \eta$, $\cF(\ti\Psi) = \ph + \ti\eta$.
We get
\begin{gather*}
\ti\eta\cth(t) - \eta\cth(t) = \int_0^t \big( G(\tau) + U\cth(\tau) \big) \, \d\tau,
\qquad
\ti\eta\ccr(t) - \eta\ccr(t) = - \int_0^t U\ccr(\tau) \, \d\tau,
\\[-2ex]
\intertext{with}
\begin{aligned}
G_i(\tau) &= \frac{\pa h}{\pa r_i} \big( r + \ti\xi\ccr(\tau) \big)
- \frac{\pa h}{\pa r_i} \big( r + \xi\ccr(\tau) \big),
\qquad i = 1,\ldots, n \\[1ex]
U_j(\tau) &= \frac{\pa u\quad}{\pa x_{j\pm n}} \circ\ti\Psi(\tau)
- \frac{\pa u\quad}{\pa x_{j\pm n}} \circ\Psi(\tau),
\qquad j = 1,\ldots, 2n
\end{aligned}
\end{gather*}
(where $j\pm n$ stands for $j+n$ if $j\le n$, for $j-n$ else).
We obtain
\[
\frac{1}{K}\normD{G(\tau)}_V \le \frac{1}{K}\sum_{j=1}^n 
\bigg( \sum_{i=1}^n \normD*{\frac{\pa^2 h}{\pa r_i\pa r_j}}_{L',R_0} \bigg)
\normD{\ti\xi_{n+j}(\tau)-\xi_{n+j}(\tau)}_{\La,R}
\le \la_0 \normD{\ti\xi-\xi}_E,
\]
with 
$\dst \la_0 \defeq \frac{2^{3\al} \normD{h}_{L,R_0}}{K(L-\La)^{2\al}} \le \demi$,
and 
\begin{multline*}
\normD{U(\tau)}_W \le \normD{U\cth(\tau)}_V + \normD{U\ccr(\tau)}_V
\le \sum_{i=1}^{2n} 
\bigg( \sum_{j=1}^{2n} \normD*{\frac{\pa^2 u\quad}{\pa x_i\pa x_{j\pm n}}}_{L',R_0} \bigg)
\normD{\ti\xi_i(\tau)-\xi_i(\tau)}_{\La,R} \\[1ex]
\le \frac{2^{3\al}\eps}{(L-\La)^{2\al}} 
\sum_{i=1}^{2n} \normD{\ti\xi_i(\tau)-\xi_i(\tau)}_{\La,R}
\le \la_1 \normD{\ti\xi-\xi}_E,
\end{multline*}
with
$\dst \la_1 \defeq \frac{2^{3\al}\eps_0 K}{(L-\La)^{2\al}} < \demi$
by~\eqref{eqdefrhoepsz},
whence the contraction property
\[
\normD{\ti\eta-\eta}_E \le (\la_0+\la_1) \normD{\ti\xi-\xi}_E
\]
with $\la_0+\la_1 <1$.

\medskip

\noindent\textbf{f)}
Finally, we get a unique fixed point $\Psi \in \cB_\rho$ for the
functional~$\cF$, which encodes the flow of~$\ti X_{h+u}$.
The difference $\De(t) \defeq \Psi(t)-\ph(t)$, when viewed as a map $\A^n_R \to \R^{2n}$, is
a lift of the difference of flows $\Phi^{t(h+u)}-\Phi^{th}$ and
\[
\sum_{j=1}^{2n} \normD{\De_j(t)}_{\La,R} \le K\rho 
= \frac{2^{\al+1} K}{ (L-\La)^\al } \eps.
\]
\end{proof}

\subsection{Proof of Proposition~\ref{propPsiPhih}}	\label{subsecpfpropPsiPhih}

Let $n,\al,L,R,R_0>0$ and~$h$ be as in the hypothesis of
Proposition~\ref{propPsiPhih}. 
Let $R' \defeq \frac{R+R_0}{2}$.
Lemma~\ref{lemGevnearintflow} yields $\eps_0,\La,C_*>0$ such that $\La\le L/2$
and,
for any $u\in G^{\al,L}(\A^n)$, 
\begin{equation}	\label{ineqlemma}
\normD{u}_{L,\infty} < \eps_0 \ens\Rightarrow\ens
\NormD{\ti\Phi^{h+u}-\ti\Phi^h}_{\La,R'}, \;
\NormD{\ti\Phi^u-\Id}_{\La,R'} \le C_* \normD{u}_{L,\infty},
\end{equation}
where $\ti\Phi^{h+u},\ti\Phi^u \col \R^n\times\R^n \to \R^n\times\R^n$ are the lifts
of $\Phi^{h+u},\Phi^u$ obtained by flowing along the lifts to $\R^n\times\R^n$
of the corresponding vector fields (which are complete in this case) and
$\NormD{\phi}_{\La,R'} \defeq \normD{\phi_1}_{\La,R'} + \cdots + \normD{\phi_{2n}}_{\La,R'}$
for a map $\phi \col \R^n\times\R^n \to \R^n\times\R^n$.
We set 
\begin{gather*}
\eps_*  \defeq 
\min\Big\{ \eps_0, \frac{R_0-R}{C_*},\frac{\La}{2\cdot(2n)^{\al-1}C_*} \Big\},
\\[1ex]
L_* \defeq 2^{-\frac{1}{\al}} (2n)^{-\frac{\al-1}{\al}} 
\big(1 + 2^{4\al} L^{-2\al} \normD{h}_{L,R}\big)^{-\frac{1}{\al}} 
\La.
\end{gather*}


Let $m\ge1$ and $\Psi \in \Palm$ be such that $\de_m^{\al,L}(\Psi,\Phi^h) < \eps_*$.
Let $\eps$ be such that 
\[
\de_m^{\al,L}(\Psi,\Phi^h) < \eps < \eps_*.
\]
We shall prove that there is a lift $\xi \col \A^n \to \R^n\times\R^n$ of
$\Psi-\Phi^h$ such that $\NormD{\xi}_{L_*,R} \le C_* \eps$,
which is sufficient to prove the proposition.


Let us choose $u_0,u_1,\ldots,u_m \in G^{\al,L}(\A^n)$ such that
$\Psi = \Phi^{u_m}\circ \cdots \circ \Phi^{u_1} \circ \Phi^{h+u_0}$ and
\[
\normD{u_0}_{L,\infty} + \normD{u_1}_{L,\infty} + \cdots + \normD{u_m}_{L,\infty} 
< \eps.
\]
We observe that the formulae
\begin{equation}	\label{eqdefinducxizz}
\xi\zz0 \defeq \ti\Phi^{h+u_0} - \ti\Phi^h,
\qquad
\xi\zz {j+1} \defeq \xi\zz j +  
\big(\ti\Phi^{u_{j+1}} - \Id\big) \circ (\Phi^h+\xi\zz j)
\end{equation}
inductively define $\xi\zz0,\xi\zz1,\ldots,\xi\zz m \col \A^n \to\R^n\times\R^n$
so that $\xi\zz m$ is a lift of $\Psi-\Phi^h$.
It is thus sufficient to check that
\begin{equation}	\label{ineqinduc}
\NormD{\xi\zz j}_{L_*,R} \le C_*\big(
\normD{u_0}_{L,\infty} + \normD{u_1}_{L,\infty} + \cdots + \normD{u_j}_{L,\infty}
\big)
\end{equation}
for $0 \le j \le m$.


In view of~\eqref{eqdefinducxizz},
inequality~\eqref{ineqinduc} holds for $j=0$ by~\eqref{ineqlemma}, because 
$\normD{u_0}_{L,\infty} \le \eps < \eps_0$ and $L_*<\La$, $R<R'$.


Assume that~\eqref{ineqinduc} holds for a given $j<m$. We observe that
$\NormD{\ti\Phi^{u_{j+1}}-\Id}_{\La,R'} \le C_* \normD{u_{j+1}}_{L,\infty}$
by~\eqref{ineqlemma}, because
$\normD{u_{j+1}}_{L,\infty} \le \eps < \eps_0$.
We can apply Proposition~\ref{propCompos} to check that the components of
$\big(\ti\Phi^{u_{j+1}}-\Id\big) \circ (\Phi^h+\xi\zz j)$
belong to $G^{\al,L_*}(\A^n_R)$ and bound their norms,
because the inequality
\[
\sum_{i=n+1}^{2N} \normD{\xi_i\zz j}_{C^0(\A^n)} \le \NormD{\xi\zz j}_{L_*,R}
\le C_* \eps < R_0-R
\]
ensures that $\Phi^h+\xi\zz j$ maps $\A^n_R$ in $\A^n_{R_0}$, and, for $1\le i\le
n$, both
$L_*^\al + 
\cN^*_{L_*,R}\big(\frac{\pa h}{\pa r_i}\big)
+ \cN^*_{L_*,R}\big(\xi_i\zz j\big)$
and
$L_*^\al + 
\cN^*_{L_*,R}\big(\xi_{n+i}\zz j\big)$
are $\le L_*^\al + \frac{2^{4\al}L_*^\al}{L^{2\al}}\normD{h}_{L,R}
+ \NormD{\xi\zz j}_{L_*,R}$
(applying~\eqref{ineqcNst} between $L^*$ and $\frac{L_*+L}{2}$
and~\eqref{ineqGevCauch} between $\frac{L_*+L}{2}$ and~$L$),
which is 
$\le \frac{\La^\al}{2 (2n)^{\al-1}} + C_* \eps < \frac{\La^\al}{(2n)^{\al-1}}$.
We thus get
$\NormD{\big(\ti\Phi^{u_{j+1}}-\Id\big) \circ (\Phi^h+\xi\zz j)}_{L_*,R}
\le \NormD{\ti\Phi^{u_{j+1}}-\Id}_{\La,R'} \le C_*
\normD{u_{j+1}}_{L,\infty}$,
which implies~\eqref{ineqinduc} for the index~$j+1$
by virtue of~\eqref{eqdefinducxizz}.

\subsection{Gevrey bump fuctions}	\label{secBumpGev}

We call ``bump function'' a function on~$\T$ which vanishes
identically outside a given interval~$I$ and whose value is~$1$ at
each point of a given subinterval of~$I$
(so this is in fact a ``flat-top bump function'').
Of course, such a function can only exist in a non-quasianalytic
functional space.

Dealing with Gevrey functions on~$\T$, we use the
notation~\eqref{eqdefGalLT}
and quote without proof Lemma~3.3 of \cite{ms} on the existence of
Gevrey bump functions on~$\T$:


\begin{lemma}   \label{lembump}
Let $\al>1$ and $L>0$.
Then there exists a real $c(\al,L)>0$ such that, for each real $p>2$, the
space $G^{\al,L}(\T)$ contains a function $\eta_p$ which satisfies
\[
-\frac{1}{2p} \le \th \le \frac{1}{2p}
\ens \Rightarrow \ens 
\eta_p(\angD{\th}) = 1, \qquad
\frac{1}{p} \le \th \le 1 - \frac{1}{p}
\ens \Rightarrow \ens 
\eta_p(\angD{\th})=0
\]
and
\beq   \label{ineqnormetap}
\norm{\eta_p}_{\al,L} \le \exp\Big(c(\al,L) \, p^{\frac{1}{\al-1}} \Big).
\eeq
\end{lemma}


\smallskip

The proof can be found in \cite[p.~1633]{ms}.


\smallskip

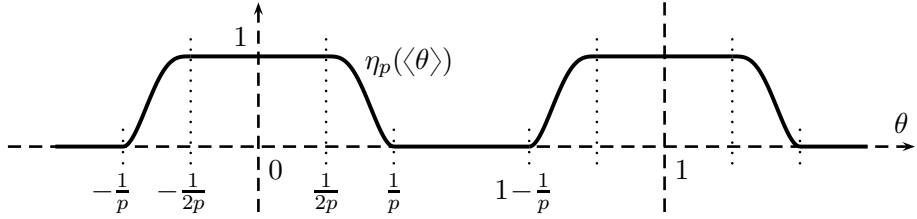
\begin{center}

\psset{xunit=0.9cm,yunit=1.2cm}
\begin{pspicture}(-1.7,-0.725)(11.7,1.8)
\psline[linewidth=1pt,linestyle=dashed]{->}(-1.7,0)(11.7,0)
\psline[linewidth=1pt,linestyle=dashed]{->}(2,-0.725)(2,1.6)
\psline[linewidth=1pt,linestyle=dashed](8,-0.725)(8,1.6)
\def\f{ x  4 exp  1   x   1 sub -4 mul  add  x 1 sub 2 exp 10 mul add x 1 sub 3 exp -20 mul add
 mul sqrt}
\def\h{ x 6 sub  4 exp  1   x   7 sub -4 mul  add  x 7 sub 2 exp 10 mul add x 7 sub 3 exp -20 mul add
 mul sqrt}
\def\g{4 x sub  4 exp  1   4 x sub    1 sub -4 mul  add  4 x sub  1 sub 2 exp 10 mul add  4 x sub  
1 sub 3 exp -20 mul add mul sqrt}
\def\ga{10 x sub  4 exp  1   10 x sub    1 sub -4 mul  add  10 x sub  1 sub 2 exp 10 mul add  10 x sub  
1 sub 3 exp -20 mul add mul sqrt}
\psplot[plotstyle=curve,linewidth=1.5pt,plotpoints=100]{0}{1}{\f}
\psplot[plotstyle=curve,linewidth=1.5pt,plotpoints=100]{6}{7}{\h}
\psplot[plotstyle=curve,linewidth=1.5pt,plotpoints=30]{3}{4}{\g}
\psplot[plotstyle=curve,linewidth=1.5pt,plotpoints=30]{9}{10}{\ga}
\psplot[plotstyle=curve,linewidth=1.5pt,plotpoints=30]{1}{3}{1}
\psplot[plotstyle=curve,linewidth=1.5pt,plotpoints=30]{7}{9}{1}
\psplot[plotstyle=curve,linewidth=1.5pt,plotpoints=30]{-1}{0}{0}
\psplot[plotstyle=curve,linewidth=1.5pt,plotpoints=30]{4}{6}{0}
\psplot[plotstyle=curve,linewidth=1.5pt,plotpoints=30]{10}{11}{0}
\psline[linewidth=1pt,linestyle=dotted](0,-0.2)(0,0.2)
\rput[t](0,-0.3){$-\tfrac{1}{p}\hspace{.8em}$}
\psline[linewidth=1pt,linestyle=dotted](1,-0.2)(1,1.2)
\rput[t](1,-0.3){$-\tfrac{1}{2p}\hspace{.6em}$}
\psline[linewidth=1pt,linestyle=dotted](4,-0.2)(4,0.2)
\rput[t](4,-0.3){$\tfrac{1}{p}$}
\psline[linewidth=1pt,linestyle=dotted](3,-0.2)(3,1.2)
\rput[t](3,-0.3){$\tfrac{1}{2p}$}
\psline[linewidth=1pt,linestyle=dotted](6,-0.2)(6,0.2)
\rput[t](6,-0.3){$1\!-\!\tfrac{1}{p}\hspace{.4em}$}
\psline[linewidth=1pt,linestyle=dotted](7,-0.2)(7,1.2)
\psline[linewidth=1pt,linestyle=dotted](10,-0.2)(10,0.2)
\psline[linewidth=1pt,linestyle=dotted](9,-0.2)(9,1.2)
\uput{.5em}[-45](2,-0.04){$0$}
\uput{.5em}[-45](8,-0.04){$1$}
\uput{.5em}[135](2,1){$1$}
\rput[b](11.5,0.15){$\th$}
\uput{.65em}[20](3.3,.8){$\eta_p(\angD{\th})$}
\end{pspicture}

\end{center}

\bigskip


Notice that, intuitively, higher values of~$p$ must produce a larger
norm, since the graph gets steeper between~$\frac{1}{2p}$ and~$\frac{1}{p}$
for instance, which makes the derivatives reach higher and higher
values.
In fact, one can prove that an exponential bound such as~\eqref{ineqnormetap} is optimal.



\vfil

\pagebreak

\section{Generating functions for exact symplectic $C^\infty$ maps}
\label{secGenfcns}

In this appendix we fix $n\ge1$ integer
and review the classical formalism of generating functions of
mixed sets of variables to define exact symplectic local
diffeomorphisms of~$\A^n$.


The coordinates in $\T^n\times\R^n$ will be denoted indifferently
$(\th,r)$, or $(\th_1,\ldots,\th_n,r_1,\ldots,r_n)$, or simply
$(x_1,\ldots,x_{2n})$.
For instance, the Liouville $1$-form which gives rise to the exact
symplectic structure on $\T^n\times\R^n$ can be written 
\[ \la = \sum_{i=1}^n r_i\,\dd \th_i = \sum_{i=1}^n x_{n+i}\,\dd x_i. \]
%
%
We denote the partial gradient operators by
\beq    \label{eqdefnazz}
\na\zz1 \defeq \begin{pmatrix} \pa_1 \\ \vdots \\ \pa_n \end{pmatrix},
\quad
\na\zz2 \defeq \begin{pmatrix} \pa_{n+1} \\ \vdots \\ \pa_{2n} \end{pmatrix},
\eeq
and view
$\dd\zz1 \defeq \begin{pmatrix} \pa_1 \cdots \pa_n \end{pmatrix}$
and
$\dd\zz2 \defeq \begin{pmatrix} \pa_{n+1} \cdots \pa_{2n} \end{pmatrix}$
as matrix-valued differential operators acting on vector-valued
functions.


Recall that $\angD{\,\cdot\,} \col \R^n \to \T^n$ denotes the canonical
projection.

%
\blm    \label{lemfindSigfromA}
Let $\Om,\Om' \subset \T^n\times\R^n$ be open, and denote by~$\un\Om$
and~$\un\Om'$ their lifts in $\R^n\times\R^n$. Suppose that $A \in
C^\infty(\Om')$ satisfies the property:
\beq \label{eqpartialdiffeo}
\text{The map}\ens
(\th,r') \mapsto (\th,r) = \big( \th, r' + \na\zz1A(\th,r') \big)
\ens\text{is a diffeomorphism from~$\un\Om'$ onto~$\un\Om$.}
\eeq
Denote by~$F\zz2$ the second group of components of the inverse
diffeomorphism, so that,
for each $(\th,r) \in \un\Om$,
\beq \label{eqcharaccrp}
r' = F\zz2(\th,r)
\quad\Leftrightarrow\quad
(\th,r') \in \un\Om' \ens\text{and}\ens \na\zz1S(\th,r') = r,
\eeq
where 
\beq \label{eqdefSfromA} 
S(\th,r') \defeq \sum_{i=1}^n \th_i r'_i + A(\th,r'), \qquad
(\th,r') \in \un\Om'. 
\eeq
Define $F\zz1(\th,r) \defeq
\th + {\na\zz2A\big( \th, F\zz2(\th,r) \big)} \in \R^n$
for $(\th,r) \in \un\Om$, so that
\beq \label{eqcharaccthp}
\th' = F\zz1(\th,r)
\quad\Leftrightarrow\quad
{\na\zz2S(\th,r')} = \th'.
\eeq
 
Then $\un F = \big( F\zz1, F\zz2 \big) \col \un\Om \to \R^n\times \R^n$
is $C^\infty$ and induces an exact symplectic local diffeomorphism
$\jF_A = \big( \angD*{F\zz1}, F\zz2 \big) \col \Om \to \A^n$.
The inverse Jacobian matrix of~$\jF_A$ at an arbitrary point
$(\th,r) \in \Om$ is a block matrix
$\begin{pmatrix} 
M^{1,1} & M^{1,2} \\[1ex]
M^{2,1} & M^{2,2}
\end{pmatrix}$
with
\beq    \label{eqinvJacFA}
M^{1,1} = (1_n + \dd\zz1\na\zz2 A)\ii, \quad
M^{1,2} = - (1_n + \dd\zz1\na\zz2 A)\ii \dd\zz2\na\zz2 A,
\eeq
where the partial derivatives of~$A$ are evaluated at $(\th,r') = \big(\th,F\zz2(\th,r)\big)$.
\elm

%
We shall see in the course of the proof that
\[ \jF_A^*\la - \la = \dd\Sig, \] 
where $\Sig \in C^\infty(\Om)$ is defined by
\beq \label{eqdefSig}
\Sig(\th,r) \defeq \ti\Sig\big( \th,F\zz2(\th,r) \big),
\qquad
\ti\Sig(\th,r') \defeq 
\sum_{i=1}^n r'_i\pa_{n+i}A(\th,r') - A(\th,r').
\eeq
By abuse of language, we will call any function~$A$
satisfying~\eqref{eqpartialdiffeo} a \emph{generating function
  for~$\Om'$}
(although it is the function $S(\th,r')$ that is usually called
generating function).
%

\begin{rem}
It is easy to check that, if $F=\jF_A$, then the set of all possible
generating functions of~$F$ coincide with the set of all functions
\[
(\th,r') \mapsto A(\th,r') + c+\sum_{i=1}^n\ell_i r'_i, \qquad
c\in\R, \quad \ell \in \Z^n.
\]
\end{rem}


\begin{proof}[Proof of Lemma~\ref{lemfindSigfromA}]
The Jacobian matrix of the diffeomorphism mentioned
in~\eqref{eqpartialdiffeo} can be written as the block matrix
\[ \begin{pmatrix} 
1_n & 0 \\[1ex]
\dd\zz1 \na\zz1 S(\th,r') & \dd\zz2 \na\zz1 S(\th,r') 
\end{pmatrix}. \]
The hypothesis entails that the matrix $\dd\zz2 \na\zz1 S(\th,r')$ is invertible
for each $(\th,r') \in \un\Om'$;
notice that its transpose is $\dd\zz1 \na\zz2S(\th,r')$.

The definition~\eqref{eqcharaccrp} of the map~$F\zz2$ shows that it
is~$C^\infty$ on~$\un\Om$, with
\[
\dd\zz1 \na\zz1 S + \dd\zz2\na\zz1 S \cdot \dd\zz1 F\zz2 \equiv 0,
\quad
\dd\zz2\na\zz1 S \cdot \dd\zz2 F\zz2 \equiv 1_n,
\]
where it is understood that the partial derivatives of~$F\zz2$ are
evaluated on~$(\th,r)$ and those of~$S$ on
$(\th,r') = \big( \th, F\zz2(\th,r) \big)$.
Moreover, $F\zz2$ can be equally viewed as a $C^\infty$ function
on~$\Om$ (it is $\Z^n$-periodic in~$\th$ because~$\na\zz1 S$ is).
By~\eqref{eqcharaccthp}, $F\zz1$ is~$C^\infty$ on~$\un\Om$ and
\[
\dd\zz1 F\zz1 =
\dd\zz1 \na\zz2 S + \dd\zz2\na\zz2 S \cdot \dd\zz1 F\zz2,
\quad
\dd\zz2 F\zz1 = \dd\zz2\na\zz2 S \cdot \dd\zz2 F\zz2.
\]
We have $F\zz1(\th+\ell,r) = F\zz1(\th,r)+\ell$ for all $\ell\in\Z^n$,
thus~$\un F$ induces a $C^\infty$ map $\jF_A \col \Om \to \A^n$.

\smallskip

A bit of calculus shows that the Jacobian matrix of~$\un F$,
which is 
$\begin{pmatrix} 
\dd\zz1 F\zz1 & \dd\zz2 F\zz1 \\[1ex]
\dd\zz1 F\zz2 & \dd\zz2 F\zz2
\end{pmatrix}$,
has the inverse
\[
\begin{pmatrix} 
1_n & 0 \\[1ex]
\dd\zz1 \na\zz1 S & \dd\zz2 \na\zz1 S
\end{pmatrix}
\begin{pmatrix} 
\dd\zz1 \na\zz2 S & 0 \\[1ex]
0 & \dd\zz1 \na\zz2 S
\end{pmatrix}\ii
\begin{pmatrix} 
1_n & - \dd\zz2 \na\zz2 S \\[1ex]
0 & \dd\zz1 \na\zz2 S
\end{pmatrix},
\]
so~$\un F$ and hence~$\jF_A$ are local diffeomorphisms,
and~\eqref{eqinvJacFA} is proved.


Let us denote by $F_1,\ldots,F_{2n}$ the components of~$\un F$, so that
\eqref{eqcharaccrp} and~\eqref{eqcharaccthp} entail
\beq \label{eqidentcrpcthp}
(\pa_i S)\big( \th, F\zz2(\th,r) \big) = r_i, \qquad
{(\pa_{n+i} S)\big( \th, F\zz2(\th,r) \big)} = F_i(\th,r),
\eeq
for $(\th,r)\in\un\Om$ and $i = 1,\ldots,n$,
and $\un F^*\la = \sum_{i=1}^n F_{n+i} \,\dd F_i$.
%
%
Let us define $\ti S \in C^\infty(\un\Om)$ by $\ti S(\th,r) \defeq S\big( \th, F\zz2(\th,r) \big)$.
Applying the chain rule and inserting~\eqref{eqidentcrpcthp}, we get
\begin{align*}
\dd \ti S &= \sum_{i=1}^n (\pa_i S)\big( \th, F\zz2(\th,r) \big) \,\dd\th_i
+ \sum_{i=1}^n (\pa_{n+i}S)\big( \th, F\zz2(\th,r) \big)\,\dd F_{n+i} \\[1ex]
&= \sum_{i=1}^n r_i \,\dd\th_i + \sum_{i=1}^n F_i(\th,r)\,\dd F_{n+i}
= \la - \un F^*\la + \dd \bigg(\sum_{i=1}^n F_{n+i} F_i\bigg).
\end{align*}
Thus $\un F^*\la - \la = \dd\chi$, with
\begin{align*}
\chi(\th,r) & \defeq \sum_{i=1}^n F_{n+i}(\th,r) F_i(\th,r) - \ti S(\th,r)\\[1ex]
& = \sum_{i=1}^n F_{n+i}(\th,r) (\pa_{n+i}S)\big(\th,F\zz2(\th,r)\big) 
- S\big( \th, F\zz2(\th,r) \big) \\[1ex]
&= \ti\chi\big( \th, F\zz2(\th,r) \big),
\end{align*}
where $\ti\chi(\th,r') = \sum_{i=1}^n r'_i (\pa_{n+i}S)(\th,r') - S(\th,r')$.
Inserting~\eqref{eqdefSfromA}, we see that $\chi\in C^\infty(\un\Om)$
is $\Z^n$-periodic in~$\th$ and induces the
function $\Sig \in C^\infty(\Om)$ defined by~\eqref{eqdefSig}.
\end{proof}


\blm    \label{lemfindAfromF}
Let $\Om \subset \A^n$ be open and connected.
%
%
Let $F \col \Om \to \T^n\times\R^n$ be an exact symplectic $C^\infty$
local diffeomorphism of the form
\[
F(\th,r) = \big( \th +\angD{f(\th,r)}, F\zz2(\th,r) \big), 
\qquad (\th,r) \in \Om,
\]
where $f,F\zz2 \in C^\infty(\Om,\R^n)$.
Assume that the map
\beq \label{eqdefinversPhi}
(\th,r) \mapsto (\th,r') = \big( \th, F\zz2(\th,r)\big)
\eeq
is a $C^\infty$-diffeormorphism from~$\Om$ onto an open set $\Om'
\subset \A^n$.

Then there exists a generating function $A \in C^\infty(\Om')$
for~$\Om$ such that $F = \jF_A$.
It can be obtained as follows:
let $\Phi \col \Om' \to \Om$ denote the inverse of the
diffeomorphism~\eqref{eqdefinversPhi} and set
\beq   \label{eqdefbet}
\bet \defeq \sum_{i=1}^n 
\big( \Phi_{n+i}(\th,r') - r'_i \big)\,\dd\th_i +
\sum_{i=1}^n f_i\circ\Phi(\th,r')\,\dd r'_i,
\eeq
then~$\bet$ is an exact $C^\infty$ $1$-form on~$\Om'$ and any $A\in
C^\infty(\Om')$ such that $\bet = \dd A$ satisfies $F = \jF_A$.
\elm


\begin{proof}
The $1$-form~$\bet$ can be written as
\beq    \label{eqnewformbet}
\bet = \Phi^*(\la-F^*\la) + \dd \bigg( \sum_{i=1}^n r'_i
(f_i\circ\Phi) \bigg)
\eeq
because $\Phi\zz1(\th,r') \equiv \th$ entails $\Phi^*\la =
\sum_{i=1}^n \Phi_{n+i}\,\dd \th_i$ and
$F\circ\Phi(\th,r') \equiv \big( \angD{\th + f\circ\Phi(\th,r')}, r' \big)$ 
entails $\Phi^*(F^*\la) = (F\circ\Phi)^*\la = \sum_{i=1}^n r'_i 
\,\dd(\th_i + f_i\circ\Phi)$.
%
%
Since~$F$ is exact symplectic, the $1$-form $F^*\la-\la$ is exact, and
the formula~\eqref{eqnewformbet} shows that~$\bet$ is thus exact too.

Pick any $A \in C^\infty(\Om')$ such that $\bet = \dd A$.
In view of~\eqref{eqdefbet}, the map $(\th,r') \mapsto \big(\th, r'+\na\zz1A(\th,r')\big)$
coincides with~$\Phi$ and is thus a $C^\infty$ diffeomorphism $\Om'\to\Om$, hence~$A$ is
a generating function for~$\Om$.
We have $\Phi\ii(\th,r) = \big( \th, F\zz2(\th,r) \big)$
and, again thanks to~\eqref{eqdefbet}, 
$\th + \na\zz2 A\big( \th, F\zz2(\th,r) \big) = \th + f(\th,r)$,
therefore $\jF_A = F$.
\end{proof}


\vfil

\pagebreak

\section{Proof of Lemma~\ref{sublemminvers}}
\label{apppfTechnicLem}
\subsection{Set-up}

Let us give ourselves an integer $n\ge1$, reals $\al\ge1$, $R,R_0,L_0>0$ such that $R<R_0$,
and a function $\eta \in G^{\al,L_0}([0,1])$. We assume that~$\eta$ is
not identically zero (otherwise there is nothing to be proved).
We set 
\begin{align}
\label{eqdefepsst}
\eps_* &\defeq \frac{1}{\normD{\eta}_{\al,L_0}} \min\Big\{
\frac{R_0-R}{2}, \frac{L_0^\al}{2^{\al+1} (2n+1)^{\al-1}}
\Big\},\\
\label{eqdefeeLmun}
L &\defeq \frac{L_0}{(2^{\al+1}(2n+1)^{\al-1})^{1/\al}}.
\end{align}
Given $\psi = (\psi_1,\ldots,\psi_n) \in G^{\al,L_0}(\A^n_{R_0},\R^n)$ such
that
\beq    \label{eqdefepsnormDpsi}
\eps \defeq \sum_{i=1}^n \normD{\psi_i}_{\al,L_0,R_0} \le \eps_*,
\eeq
we define for each $t\in [0,1]$ a $C^\infty$ map
\beq   \label{eqdefPsitetapsi}
\Psi_t \col (\th,r) \in \A^n_{R_0} \mapsto (\th,r') = \big( \th, r +
\eta(t) \psi(\th,r) \big) \in \A^n.
\eeq
Our aim is to prove that~$\Psi_t$ induces a $C^\infty$ diffeomorphism
from $\T^n\times B_{R_0}$ onto its image~$\Om_t$, 
to check that $\A^n_R \subset \Om_t$ 
and to study the inverse map.

\subsection{Diffeomorphism property}

Let $t\in[0,1]$.
The Jacobian matrix of~$\Psi_t$ at an arbitrary $(\th,r) \in
\A^n_{R_0}$ is the block matrix
$\begin{pmatrix} 
1_n & 0 \\
M & 1_n + N
\end{pmatrix}$,
where $M \defeq \eta(t) \dd\zz1\psi(\th,r)$ and $N \defeq \eta(t) \dd\zz2\psi(\th,r)$,
with the notations of Appendix~\ref{secGenfcns}.

The matrix norm of~$N$ subordinate to the Euclidean structure
of~$\R^n$ is $\le \sum_{1\le i,j \le n} \absD{N_{i,j}}$,
and ${N_{i,j}} = \eta(t) \pa_{n+j}\psi_i(\th,r)$,
thus this matrix norm is less than~$1$ by~\eqref{eqdefepsst} and~\eqref{eqdefepsnormDpsi}
(because $\sum_{i,j} \normD{\eta}_{C^0([0,1])} 
\normD{\pa_{n+j}\psi_i}_{C^0(\A^n_{R_0})} \le
\normD{\eta}_{\al,L_0} \sum_i L_0^{-\al} \normD{\psi_i}_{\al,L_0,R_0}
= L_0^{-\al} \eps \normD{\eta}_{\al,L_0}$)
and $1_n+N$ is invertible.
Therefore, by the Implicit Function Theorem, 
$\Psi_t$ is a $C^\infty$ local diffeomorphism on~$\A^n_{R_0}$.

Suppose that $(\th,r)$ and $(\th^*,r^*)$ have the same image
by~$\Psi_t$. Then $\th^* = \th$ and 
\[
r^* - r = -\eta(t) \big( \psi(\th,r^*) - \psi(\th,r) \big) =
-\eta(t) \int_0^1 \dd\zz2 \psi\big( \th, (1-s)r+sr^* \big) (r^*-r)
\,\dd s,
\]
whence $\normD{r^*-r} < \normD{r^*-r}$ by the above remark on
the matrix norm of $\eta(t) \dd\zz2\psi$, thus $r^* = r$.
Therefore, $\Psi_t$ is injective on~$\A^n_{R_0}$ and induces a
$C^\infty$ diffeomorphism from $\T^n\times B_{R_0}$ onto an open
subset~$\Om_t$ of~$\A^n$.

\subsection{Study of the inverse map}

We can write $\Psi_t\ii(\th,r') = \big( \th, r' + \chi(\th,r',t)
\big)$, with $\chi = (\chi_1,\ldots,\chi_n)$ and
$\chi_i(\,\cdot\,,\,\cdot\,,t) \in C^\infty(\Om_t)$ for each~$i$.
Given $(\th,r',t) \in \A^n\times[0,1]$, the point $(\th,r')$ belongs
to~$\Om_t$ if and only if there exists $u\in\R^n$ such that $r'+u \in
B_{R_0}$ and
\[
u = -\eta(t) \psi(\th,r'+u).
\]
This vector~$u$ is then unique and is $\chi(\th,r',t)$.
We must prove that $\A^n_R \subset \Om_t$, that the restriction of
the functions~$\chi_i$ to $\A^n_R \times [0,1]$ belong to
$G^{\al,L}(\A^n_R \times [0,1])$, and that their Gevrey norms satisfy~\eqref{ineqnormchiiepseta}.
All this follows from


\begin{sublem}
Consider the Banach space $V \defeq \big(G^{\al,L}(\A^n_R \times
[0,1]) \big)^n$, with the norm
\[
\normD{u}_V \defeq \normD{u_1}_{\al,L,R} + \cdots + \normD{u_n}_{\al,L,R}
\quad \text{for}\ens u = (u_1, \ldots, u_n) \in V.
\]
Let $\cB \defeq \{\,
u\in V \mid \normD{u}_V \le \eps\normD{\eta}_{\al,L_0} \,\}$.
Then, for any $u \in \cB$, the formula
\beq    \label{eqdefcFuv}
v(\th,r',t) \defeq -\eta(t) \psi\big(
\th, r' + u(\th,r',t) \big)
\eeq
makes sense for all $(\th,r',t) \in \A^n_R\times[0,1]$ and defines a
vector-valued function $v = \cF(u)$, which belongs to~$\cB$.
Moreover, the functional $\cF \col \cB \to \cB$ satisfies
\[ 
\normD{\cF(u^*)-\cF(u)}_V \le \dem \normD{u^*-u}_V,
\qquad u,u^*\in\cB.
\]
\end{sublem}


Indeed, the contraction~$\cF$ has a unique fixed point, which is
nothing but $\chi_{|\A^n_R \times [0,1]}$.


\begin{proof}[Proof of Sub-Lemma]
Let $u \in \cB$.
For each $(\th,r',t) \in \A^n_R\times[0,1]$, we have
$\normD{u(\th,r',t)} \le \sum \normD{u_i}_{C^0(\A^n_R \times[0,1])}
\le \normD{u}_V \le \eps_*\normD{\eta}_{\al,L_0} \le R_0-R$
by~\eqref{eqdefepsst} and~\eqref{eqdefepsnormDpsi}, thus
\beq    \label{eqdefUinAnRz} 
U(\th,r',t) \defeq \big( \th, r'+u(\th,r',t) \big) \in \A^n_{R_0}. 
\eeq
Therefore, the function $v \col \A^n_R\times[0,1] \to \R^n$ is well-defined as
$v = - \eta \cdot (\psi\circ U)$.

For each $i=1,\ldots,n$, we can apply Proposition~A.1 of \cite{hms} to the composition $\psi_i\circ U$:
the function $\psi_i\circ U$ belongs to $G^{\al,L}(\A^n_R\times[0,1])$
and $\normD{\psi_i\circ U}_{\al,L,R} \le \normD{\psi_i}_{\al,L_0,R_0}$ because
\beq    \label{ineqsumellnzLz}
\sum_{\ell\in\N^{2n+1},\, \ell\neq0} 
\frac{L^{\absD{\ell}\al}}{\ell !^\al} \normD{\pa^\ell U_k}_{C^0(\A^n_R\times[0,1])}
\le \frac{L_0^\al}{(2n+1)^{\al-1}},
\qquad k=1,\ldots,2n
\eeq
(indeed: for $k\le n$, the \lhs of~\eqref{ineqsumellnzLz} is
$L^\al = \frac{L_0^\al}{2^{\al+1}(2n+1)^{\al-1}}$ by~\eqref{eqdefeeLmun}, 
and for $k = n+j$ with $1\le j \le n$, the \lhs is
$\le L^\al + \normD{u_j}_{\al,L,R} \le L^\al + \eps\normD{\eta}_{\al,L_0}$,
and $\eps\normD{\eta}_{\al,L_0} \le
\frac{L_0^\al}{2^{\al+1}(2n+1)^{\al-1}}$).
Therefore, by the algebra norm property, $v_i \in G^{\al,L}(\A^n_R\times[0,1])$ and 
$\normD{v_i}_{\al,L,R} \le \normD{\eta}_{\al,L_0}
\normD{\psi_i}_{\al,L_0,R_0}$,
hence $v\in \cB$.


Let us now suppose that we are given $u,u^* \in \cB$ and consider the
difference between $v^* \defeq \cF(u^*)$ and $v \defeq \cF(u)$.
We have $v^*_i - v_i = \sum_{j=1}^n M_{i,j} (u^*_j-u_j)$ for each
$i=1,\ldots,n$, with
\[
M_{i,j}(\th,r',t) \defeq - \eta(t) \int_0^1 
\pa_{n+j} \psi_i \circ U_s(\th,r',t) \,\dd s,
\qquad j=1,\ldots,n,
\]
where, for each $s\in[0,1]$, $U_s(\th,r',t) \defeq \big( \th, r' + (1-s) u(\th,r',t)+s u^*(\th,r',t) \big) 
\in \A^n_{R_0}$ by~\eqref{eqdefUinAnRz}.

On the one hand $\pa_{n+j} \psi_i \in G^{\al,L_0/2}(\A^n_{R_0})$ and
\beq   \label{ineqCauchypsii}
\sum_{j=1}^n \normD{\pa_{n+j} \psi_i}_{\al,\frac{L_0}{2},R_0} \le
\frac{2^\al}{L_0^\al} \normD{\psi_i}_{\al,L_0,R_0}
\eeq
by~\eqref{ineqGevCauch}.
On the other hand
\[
\sum_{\ell\in\N^{2n+1},\, \ell\neq0} 
\frac{L^{\absD{\ell}\al}}{\ell !^\al} \normD{\pa^\ell U_{s,k}}_{C^0(\A^n_R\times[0,1])}
\le \frac{L_0^\al}{2^\al(2n+1)^{\al-1}},
\qquad k=1,\ldots,2n, \ens s\in[0,1]
\]
(same verification as for~\eqref{ineqsumellnzLz}).
Thus, we can apply again Proposition~A.1 of \cite{hms}: 
we get $\pa_{n+j} \psi_i \circ U_s \in G^{\al,L}(\A^n_R\times[0,1])$
and
$\normD{\pa_{n+j} \psi_i \circ U_s}_{\al,L,R} \le \normD{\pa_{n+j} \psi_i}_{\al,\frac{L_0}{2},R_0}$.

Therefore, $M_{i,j} \in G^{\al,L}(\A^n_R\times[0,1])$ for each $(i,j)$
and, in view of the algebra norm property and~\eqref{ineqCauchypsii},
\[
\sum_{i,j} \normD{M_{i,j}}_{\al,L,R} \le \normD{\eta}_{\al,L_0} 
\sum_i \frac{2^\al}{L_0^\al} \normD{\psi_i}_{\al,L_0,R_0}
= \frac{2^\al}{L_0^\al} \eps \normD{\eta}_{\al,L_0} \le \demi
\]
by~\eqref{eqdefepsst} and~\eqref{eqdefepsnormDpsi}, whence it follows that
$\normD{v^*-v}_V \le
\sum_{i,j} \normD{M_{i,j}}_{\al,L,R} \normD{u_j^*-u_j}_{\al,L,R}
\le \dem \normD{u^*-u}_V$.
\end{proof}


\vfil

\pagebreak


\end{document}